\documentclass[11pt, twoside, openleft]{book}
\usepackage[a4paper, left=1.5in, right=1.25in]{geometry}
\usepackage{amssymb}
\usepackage{amsmath}
\usepackage{microtype}  % Add to preamble, best fix [web:148][web:144]
\usepackage{amsthm}
\usepackage{commath}
\usepackage{textcomp}
\usepackage{mathrsfs}
\usepackage{enumerate}
\usepackage{amsfonts}
\usepackage{palatino}
\usepackage[dvipsnames]{xcolor}
\usepackage[colorlinks= true, linkcolor = Red, citecolor = Green, urlcolor = Blue, pdfstartview=]{hyperref}
\usepackage{mathtools}
\usepackage{graphicx}
\usepackage{wrapfig}
\usepackage{subcaption}
\usepackage{etoolbox}
\usepackage{comment}
\usepackage{blkarray}
\usepackage{sectsty}
\usepackage{stmaryrd}
\usepackage{esvect}
\usepackage{tikz-cd}
\usepackage{graphicx}
\usepackage[capitalise]{cleveref}
\usepackage{stackrel}
\usepackage{fancyhdr}
\usepackage{quiver}
\usepackage{emptypage}
\usetikzlibrary{nfold}
\usepackage{quotchap}
\usepackage{csquotes}
\usepackage[T1]{fontenc}

\usepackage{afterpage}
\newcommand\blankpage{%
    \null
    \thispagestyle{empty}%
    \addtocounter{page}{-1}%
    \newpage}

\makeatletter
\let\thetitle\@title
\makeatother
\frontmatter

\fancypagestyle{mainmatterstyle}{
  \fancyfoot{}
  \lhead{}\chead{}\rhead{}
  \lfoot{}\cfoot{\thepage}\rfoot{}
  \setlength{\footskip}{30pt}
  \setlength{\headheight}{12.77002pt}
  \fancyhead{}
  \fancyhead[ER]{\makebox[1em][r]{\thepage}}
  \fancyhead[ER]{\makebox[1em][r]{} \small \rightmark}
  \fancyhead[OL]{\makebox[1em][l]{}\hspace{5mm} \small\leftmark}
}
\usepackage[backend=bibtex,style=alphabetic,sorting=nty,backref=true,maxbibnames=99,maxalphanames=999]{biblatex}
\DefineBibliographyStrings{english}{%
  backrefpage = {\textbf{P.}},% originally "cited on page"
  backrefpages = {\textbf{P.}},% originally "cited on pages"
}

\theoremstyle{definition}
\theoremstyle{plain}
\newtheorem{definition}{Definition}[section]
\newtheorem{defn}{Definition}

\newtheorem{example}[definition]{Example}
\newtheorem{remark}[definition]{Remark} 
\newtheorem{lemma}[definition]{Lemma}
\newtheorem{prop}[definition]{Proposition}
\newtheorem{question}[definition]{Question}
\newtheorem{corollary}[definition]{Corollary}
\newtheorem{theorem}[definition]{Theorem}
\newtheorem{conj}{Conjecture}

\newtheorem*{theorem*}{Theorem}
\newtheorem*{conj*}{Conjecture}
\newtheorem*{corollary*}{Corollary}
\newtheorem*{lemma*}{Lemma}
\newtheorem*{prop*}{Proposition}
\newtheorem*{qstn*}{Question}
\newtheorem*{defn*}{Definition}

\renewcommand{\AA}{\mathbb{A}}
\renewcommand{\SS}{\mathbb{S}}
\renewcommand{\L}{\mathbb{L}}
\renewcommand{\frak}{\mathfrak{}}
\renewcommand{\rm}{\mathrm}

\newcommand{\tcolor}{\textcolor}

\newcommand{\CC}{\mathbb{C}}

\newcommand{\GG}{\mathbb{G}}

\newcommand{\HH}{\mathcal{H}}
\newcommand{\SH}{\mathrm{SH}}
\newcommand{\ZZ}{\mathbb{Z}}
\newcommand{\KK}{\mathcal{K}}
\newcommand{\Sch}{\mathrm{Sch}}
\newcommand{\Sm}{\mathrm{Sm}}

\newcommand{\pr}{\mathrm{pr}}

\newcommand{\XX}{\mathcal{X}}
\newcommand{\QQ}{\mathbb{Q}} 
\newcommand{\PP}{\mathbb{P}}
\newcommand{\RR}{\mathbb{R}}
\newcommand{\FF}{\mathbb{F}}

\newcommand{\bs}{\backslash}
\newcommand{\ul}{\underline}
\newcommand{\mcal}[1]{\mathcal{#1}}
\newcommand{\Abs}{\mathbb{A}^2\backslash\{0\}}
\newcommand{\Absn}{\AA^n\backslash\{0\}}
\newcommand{\barg}{\bar{g}}
\newcommand{\Lin}{\langle}
\newcommand{\Rin}{\rangle}
\newcommand{\abek}{\mathcal{A}b(k)}

\newcommand{\Spec}{\operatorname{Spec}}

\newcommand{\ML}{\operatorname{ML}}
\newcommand{\Fun}{\operatorname{Fun}}
\newcommand{\Hom}{\operatorname{Hom}}
\newcommand{\Map}{\operatorname{Map}}
\newcommand{\Sing}{\operatorname{Sing}}
\newcommand{\Set}{\operatorname{Set}}
\newcommand{\sSet}{\operatorname{sSet}}
\newcommand{\Spc}{\operatorname{Spc}}
\newcommand{\Spt}{\operatorname{Spt}}
\newcommand{\Shv}{\operatorname{Shv}}
\newcommand{\Psh}{\operatorname{Psh}}
\newcommand{\sPsh}{\operatorname{sPsh}}
\newcommand{\sShv}{\operatorname{sShv}}
\newcommand{\MW}{\operatorname{MW}}
\newcommand{\M}{\operatorname{M}}
\newcommand{\GW}{\operatorname{GW}}
\newcommand{\Gr}{\operatorname{Gr}}
\newcommand{\W}{\operatorname{W}}
\newcommand{\CH}{\operatorname{CH}}
\newcommand{\Pic}{\operatorname{Pic}}
\newcommand{\Th}{\operatorname{Th}}
\newcommand{\Aut}{\operatorname{Aut}}
\newcommand{\Ker}{\operatorname{Ker}}
\newcommand{\LND}{\operatorname{LND}}
\newcommand{\Id}{\operatorname{Id}}
\newcommand{\DM}{\operatorname{DM}}
\newcommand{\SmCor}{\operatorname{SmCor}}
\newcommand{\Cor}{\operatorname{Cor}}
\newcommand{\TN}{\operatorname{TN}}

\title{Relative $\AA^1$-Contractibility of Smooth Schemes And Exotic Motivic Spheres}
\date{}
\usepackage{pdfpages}
\usepackage{url}
\usepackage{lastpage}
\begin{document}
% \maketitle
\includepdf[pages={1,2}]{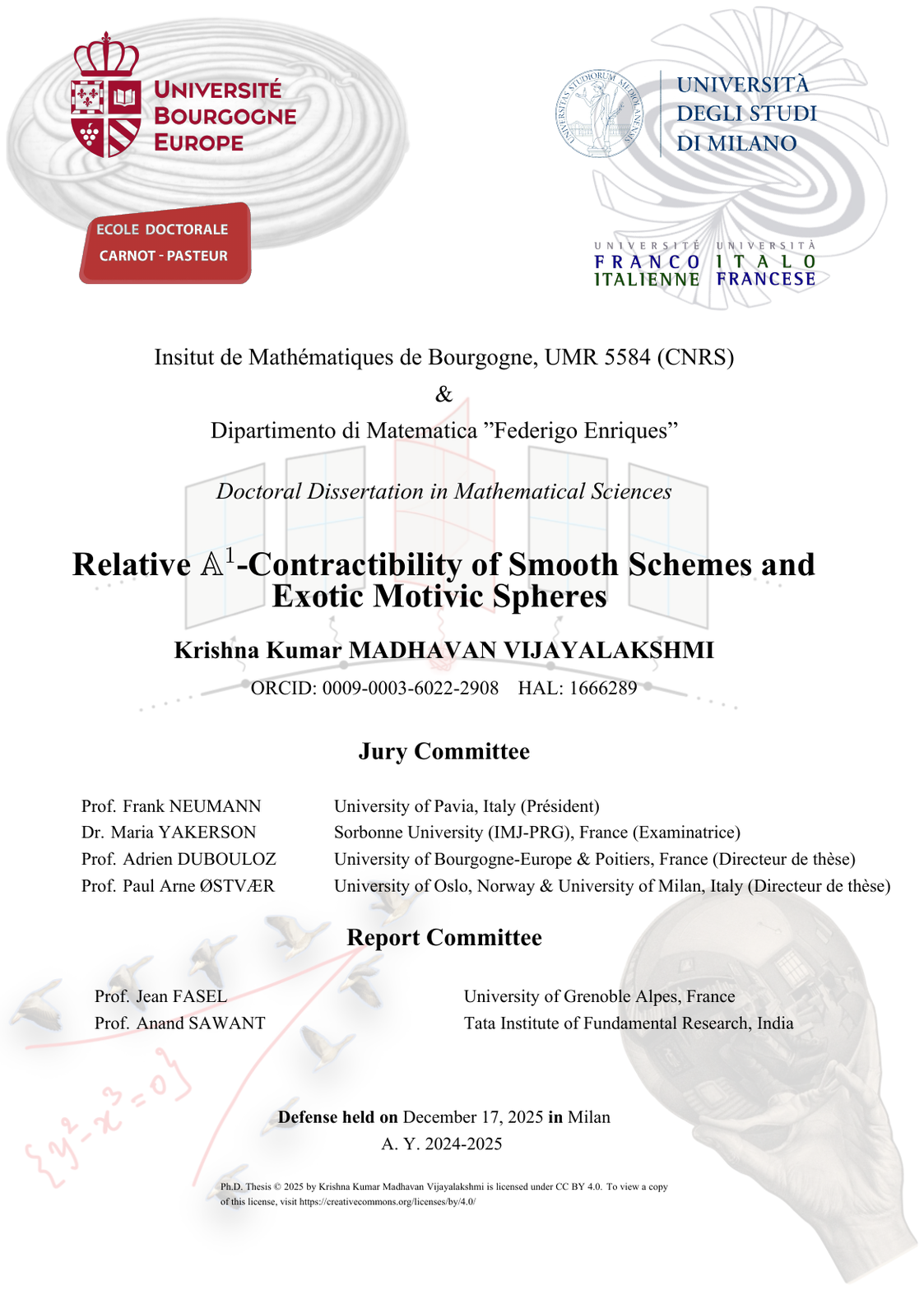}

\pagenumbering{roman}
\setcounter{page}{1}

\pagestyle{empty}
{
\normalsize
\hypersetup{linkcolor=black}
\tableofcontents
}
\newpage
\pagestyle{fancy}

\chapter{Abstract}
\markboth{Abstract}{}

One of the emerging problems in algebraic geometry is to characterize the affine $n$-space $\AA^n$ among smooth affine schemes up to $\AA^1$-contractibility: a property that makes motivic spaces be homotopy equivalent to a field $\Spec k$ in the motivic homotopy theory. Recent efforts show that this characterization holds in dimensions $n<3$ over certain fields. In our work, we extend this observation to "reasonable" arbitrary base schemes in relative dimensions $n< 3$, exploiting the Zariski local triviality and the triviality of the relative canonical sheaf. From dimensions $n\ge 3$, the existence of smooth “exotic” affine schemes, those that are $\AA^1$-contractible but not isomorphic to the affine $n$-space, has already been established. A well-studied family constitutes the Koras-Russell threefolds $\KK$ and their generalized prototypes $\XX_n$ in higher dimensions, whose $\AA^1$-contractibility has been so far proven over fields of characteristic zero. In this direction, we extend the relative $\AA^1$-contractibility of $\KK$ and $\XX_n$ over a Noetherian base scheme in arbitrary dimensions. 
\vspace{2mm}

Furthermore, using these prototypes, we study the existence of “exotic spheres” - $n$-dimensional smooth schemes that are $\AA^1$-homotopic, but not isomorphic to $\AA^n \bs \{0\}$, in motivic homotopy theory. This result can be seen as the "compact" analog of the study of exotic affine schemes. Our main result shows that in all dimensions $n\ge 4$, the quasi-affine varieties $\XX_n \bs \{\bullet\}$ give a model for the exotic motivic spheres over infinite perfect fields. The novelty here is that these constitute the first family of examples of smooth motivic spheres of dimension $n$, which are not isomorphic to $\AA^n \bs \{0\}$.
\vspace{8mm}

{\let\cleardoublepage\relax
\chapter*{Resum\'e}
\markboth{Abstract}{}

L'un des problèmes émergents en géométrie algébrique consiste à caractériser l'esp-ace affine $\AA^n$ parmi les schémas affines lisses à contractibilité $\AA^1$ près~: une propriété qui rend les espaces motiviques homotopiquement équivalents à un corps $\Spec k$ en homotopie motivique. Des travaux récents montrent que cette caractérisation est valable en dimension $n<3$, sur certains corps. Dans nos travaux, nous étendons cette observation à une large famille de schémas de base «~raisonnables~» en dimension relatives $n< 3$, en exploitant la trivialité locale de Zariski et la trivialité du faisceau canonique relatif. À partir de la dimensions $n\ge 3$, l'existence d'espaces affines «~exotiques~», c'est-à-dire des schémas affines lisses qui sont $\AA^1$-contractibles mais qui ne sont pas isomorphes à l'espace affine $\AA^n$, a déjà été établie. Une famille bien étudiée est celle des 3-variétés de Koras-Russell $\KK$ et leurs prototypes généralisés $\XX_n$ en dimensions supérieures, dont la $\AA^1$-contractibilité  a été prouvée jusqu'à présent sur des corps de caractéristique nulle. Dans cette direction, nous étendons ces résultats en établissant la $\AA^1$-contractibilité relative  de $\KK$ et $\XX_n$ sur un schéma de base noethérien en dimension arbitraire. 
\vspace{2mm}

De plus, en utilisant ces prototypes, nous étudions l'existence de «~sphères exotiques~» en homotopie motivique - des schémas lisses de dimension $n$ qui sont $\AA^1$-homotopes mais non isomorphes à $\AA^n \bs \{0\}$-. Ce résultat peut être considéré comme l'analogue «~compact~» de l'étude des espaces affines exotiques. Notre résultat principal montre qu'en dimension $n\ge 4$, les variétés quasi-affines $\XX_n \bs \{\bullet\}$ fournissent un modèle pour les sphères motiviques exotiques sur des corps parfaits infinis. La nouveauté ici est que celles-ci constituent la première famille d'exemples de sphères motiviques lisses de dimension $n$ qui ne sont pas isomorphes à $\AA^n \bs \{0\}$.

}

{\let\cleardoublepage\relax
\chapter*{Abstract}
\markboth{Abstract}{}

Uno dei problemi emergenti nella geometria algebrica è quello di caratterizzare lo spazio affine $n$-dimensionale $\AA^n$ tra schemi affini regolari a meno di $\AA^1$-contrazione: una proprietà che rende gli spazi motivici omotopicamente equivalenti a un campo $\Spec k$ nella teoria dell'omotopia motivica. Recenti studi dimostrano che questa caratterizzazione vale in dimensioni $n<3$ su determinati campi. Nel nostro lavoro, estendiamo questa osservazione a schemi “ragionevoli” di base arbitraria, in dimensioni relative $n< 3$, sfruttando la trivialità locale di Zariski e la trivialità del fascio canonico relativo. A partire dalle dimensioni $n\ge 3$, è già stata dimostrata l'esistenza di schemi affini regolari “esotici”, in altre parole, schemi che sono $\AA^1$-contraibili ma non isomorfi allo spazio affine $n$-dimensionale. Una famiglia ben studiata è costituita dalle tre-varietà $\KK$ di Koras-Russell e dai loro prototipi generalizzati $\XX_n$ in dimensioni superiori, la cui $\AA^1$-contraibilità  è stata finora dimostrata su campi di caratteristica zero. In questa direzione, estendiamo la $\AA^1$-contraibilità relativa di $\KK$ e $\XX_n$ su uno schema di base Noetheriano in dimensioni arbitrarie. 

\vspace{2mm}

Inoltre, utilizzando questi prototipi, studiamo l'esistenza di “sfere esotiche” - schemi lisci $n$-dimensionali che sono $\AA^1$-omotopici, ma non isomorfi ad $\AA^n \bs \{0\}$ nella teoria dell'omotopia motivica. Questo risultato può essere visto come l'analogo “compatto” dello studio degli schemi affini esotici. Il nostro risultato principale mostra che in tutte le dimensioni $n\ge 4$, le varietà quasi affini $\XX_n \bs \{\bullet\}$ forniscono un modello per le sfere motiviche esotiche su campi perfetti infiniti. Queste costituiscono la prima famiglia nota di esempi di sfere motiviche lisce di dimensione $n$, che non sono isomorfe a $\AA^n \bs \{0\}$.
}

\begin{savequote}
Mathematics is us reverse engineering the language of nature.
\qauthor{"The Almanack of Naval Ravikant" by Eric Jorgenson}
\end{savequote}
\chapter{Preface}\label{preface}
\markboth{Preface}{}

Our script resides in the realm of algebraic geometry: a unique blend of various fundamental fields of mathematics. In the following paragraphs, we will paint a picture that portrays the position of this script within the mathematical arena.
\vspace{2mm}

\textcolor{cyan}{Topology} is the study of geometry up to continuous deformations. It is unconcerned with the space being dense or sparse, curved or stretched (but it is sensitive to perforations). While geometry takes a quantitative view on spaces, topology takes a non-quantitative, continuous viewpoint: topology equips a space with a notion of \emph{nearness} without quantitatively spelling out how near they actually are. In this way, topology generalizes the notion of a metric space, whence generalizing the notion of distance between its elements. One of the most pivotal challenges in mathematical theory is to classify spaces up to a given \emph{type}. The classification of spaces up to diff(hom)-eomorphisms often pose limitations and topology offers a relaxed treatment through its notion of \textcolor{cyan}{homotopy}, which is an invariant that segregates spaces into families by their reachability via continuous maps: two continuous maps $f,g: X\to Y$ between topological spaces are \emph{homotopic} if there is a collection of continuous maps $H(x, i)$ (the homotopy function)
\begin{align*}
    &\hspace{15mm} H(x,i): X \times [0,1] \to Y \qquad \text{such that} \\
    & H(x,0) = f(x), \quad H(x,1) = g(x), \quad \forall\ x\in X
\end{align*}
\textcolor{cyan}{Algebra} is the art of discretizing these continuous data. \textcolor{cyan}{Algebraic topology}, thereby, develops frameworks to systematically transport topological challenges into the language of algebra - be it commutative, homological, homotopical, differential, or categorical.
{\small
\begingroup
\addtolength{\jot}{1em}
\begin{align*}
& \hspace{4mm} \text{Geometry} \xrightarrow{\makebox[2.5cm]{\text{deformation}}}  \text{Topology}\xrightarrow{\makebox[2.5cm]{\text{discretize}}}   \text{Algebra}  \\
& \frac{\text{Quantitative data}}{\text{bend, squeeze, stretch}} \xmapsto{\makebox[1cm]{}}\text{Continuous data} \xmapsto {\makebox[1cm]{}} \text{Discrete data}
\end{align*}
\endgroup}
Algebraic topology provides a home for tackling problems from geometry and topology facilitated by its numerous \emph{algebraic invariants}. This translation is enhanced by the notion of \emph{categories} and functors between categories. \textcolor{cyan}{Functors} are certain inherent maps that preserve the invariant nature between categories; an \emph{algebro-topological functor} associates a space to an algebraic invariant (groups, rings, modules,...). The \emph{building blocks} of algebraic topology can be broadly divided as follows:
\begin{align*}
    &\hspace{3mm} \text{Spheres}\quad S^n:= \{\vec{x} \in \RR^{n+1} : ||\vec{x}|| =1\} \hspace{16mm} \text{(Geometrical object)} \\
    & \text{Simplices}\quad \Delta^n:= \{\vec{x} \in \RR^{n+1} : \Sigma x_i = 1, x_i\ge 0\} \quad \text{(Combinatorial object).}
\end{align*}
Every "decent" topological space is fundamentally glued out of these blocks. Analyzing maps 'out of spheres' gives rise to the \emph{homotopy} theory and analyzing maps 'out of simplices' gives rise to \emph{(co-)homology} theory: two of the most fundamental algebro-topological invariants. The $n$-simplex (pl. simplices) is homeomorphic to another important object called the \emph{$n$-ball} given by $D^n:=\{\vec{x}\in \RR^n: ||\vec{x}|| \le 1\}$. In the view of homotopy theory, the following fact exemplifies the usefulness of such balls: every $n$-ball is homotopy equivalent to a point $\{\bullet\}$ (though the former is $n$-dimensional and the latter is 0-dimensional!). This crucial property means that almost all of the known algebraic invariants, when taken up to homotopy, become trivial on balls. This property makes it an ideal object for the choice of parameterizing homotopies in any category, as they themselves do not add any extra data! Spaces that are homotopy equivalent to a ball are called \textcolor{cyan}{contractible} spaces. In the view of algebro-topological theories, balls can be seen as \emph{contrasting objects} to spheres (e.g., the (co-)homology theory of the former is always trivial and the latter is never trivial). Therefore, within the framework of any homotopy theory, a solid comprehension of these objects is fundamental. Algebraic topology has a wealth of tools to study a given space:
\begin{displayquote}
{\small Cellular approximation, Siefert-Van Kampen \& Mayer-Vietoris type gadgets, Excision, Covering spaces, Hopf-Brouwer degree map, Obstruction theory \& Postnikov towers, Operads, Localization exact sequences, Universal coefficients,\\ K\"unneth formula, Poincar\'e duality, Spectral sequences, Characteristic classes, Hurewicz theorem, $\dots$}    
\end{displayquote}
Such a vast collection of machinery is transported to nearby areas of mathematics by invoking the categorical framework. Several of these tools behave well under 'decent' functorial transformations, and so this enables one to export tools from algebraic topology to the realm of \textcolor{cyan}{algebraic geometry} - a novel area where algebraic topology meets geometry in profound ways. It is where polynomial equations define the geometric objects. The building blocks of algebraic geometry are the \emph{affine varieties} upon which bigger blocks of \emph{quasi-projective varieties} and schemes are built. Every scheme is locally glued out of affine varieties (in the so-called \emph{Zariski topology}). At the heart of this conception is the \emph{Hilbert's Nullstellensatz} (German for "theorem of zeroes") that embellishes a dictionary between algebra and geometry
{\small
\begin{align*}
    & \hspace{37mm}\text{Algebra} \longleftrightarrow \text{Geometry} \\ 
    & \{ \text{Radical\ ideals\ of}\ \CC[x_1,\dots,x_n] \} \longleftrightarrow \{ \text{Zero \ sets\ in}\ \CC^n \} \\
    & \hspace{48.5mm} J \longmapsto \mcal{Z}(J), \\
    & \hspace{38.5mm} \rm{ \sqrt{\mcal{I}(W)} \longmapsfrom W } 
\end{align*}}
Roughly, it says that for an ideal $J \subseteq \CC[x_1,\dots,x_n]$, we have $\mcal{I}(\mcal{Z}(J)) = \sqrt{J}$, where $\sqrt{J}$ is the radical ideal of $J$. This result extends to all algebraically closed fields $k$ as well!

\subsubsection{Affine Algebraic Geometry}
\textcolor{cyan}{Commutative algebra} is a classical subdomain of algebra concerned with algebraic objects (like groups, rings, modules) that obey the law of commutativity $a*b = b*a$ (in a precise mathematical sense). It provides the local algebraic tools required to understand the global geometric objects. One of the most useful tools from commutative algebra is that of \emph{localization} of objects, which, in fact, can be generalized to maps in a given category. This allows one to "formally" invert a (class of) maps in any category to create a new category with extra maps. For example, homotopy theories are constructed in this way. The combined power of the Nullstellensatz and the localization allows one to transport almost all of the machinery from commutative algebra, locally, into the realm of algebraic geometry. Such systematic processes form the base of an important subfield of algebraic geometry called the \textcolor{cyan}{affine algebraic geometry}. Some of the open problems in affine algebraic geometry involve a deep understanding of the ring with the simplest possible description: the polynomial ring in $n$-variables $k[x_1,\dots,x_n]$.

\subsection*{What is Motivic Homotopy Theory?}
\addcontentsline{toc}{section}{What is Motivic Homotopy Theory?}
\textcolor{cyan}{Motivic homotopy theory} is a sophisticated algebro-geometric machinery that utilizes tools from all of the above-mentioned areas. The fundamental idea is to develop a \emph{homotopy theory for algebraic geometry}. In algebraic topology, one studies continuous maps between topological spaces and their homotopy classes. Motivic homotopy theory creates an analogous framework for algebraic varieties and schemes. Unlike algebraic topology, the framework of algebraic geometry is \emph{relative} - rather than studying a scheme by itself, one studies a \emph{scheme over a scheme}, that is, a morphism of schemes.

\subsubsection{The Flashback}
The plot begins with Grothendieck's revolutionary ideas about \textcolor{cyan}{motives} (from the French usage: \emph{motifs}) in the 1960s through the observation that various cohomology theories (de Rham, \'etale, crystalline, $\dots$) for algebraic varieties seemed to capture overlapping but distinct aspects of the same underlying geometric information. He envisioned a universal theory of motives - abstract objects that would encode the essential cohomological data of varieties, with each specific cohomology theory arising as a \emph{realization} of these motives. He further conjectured that there exists a suitable cohomology theory (now called the \tcolor{cyan}{motivic cohomology}) through which all other known theories should factor. So, philosophically speaking, motives are meant to capture the \emph{cohomological essence} of an algebraic variety. His original program was extraordinarily ambitious but proved technically challenging (\cite{grothendieck1969standard}). He envisioned that the category of motives would be Abelian and would provide a universal cohomology theory, but constructing this category rigorously remained elusive for decades. Motivic cohomology was first conjectured to exist by Beilinson and Lichtenbaum in the mid-1980s via certain "motivic complexes" (see \cite{beilinson1985higher, beilinson1987notes} and references therein). Up till this date, several mathematicians continue investing huge efforts in this direction, and a lot of progress has been made so far in assembling pieces of this motivic puzzle.
\medskip

\textbf{Higher Chow groups:} Bloch \cite{bloch1986algebraic} introduced the theory of \tcolor{cyan}{higher Chow groups}, a generalization of the classical Chow groups, which was the first integral (as opposed to rational) definition of Borel-Moore motivic homology for quasi-projective varieties over a field (see also \cite{levine1994bloch}). The higher Chow groups later turned out to be the motivic cohomology $H^{p,q}$ in disguise due to the work of Voevodsky (\cite{voevodsky2002motivic}): for a smooth scheme $X$ over any field and integers $p,q \in \ZZ$, there is a natural isomorphism of groups
        $$H^{p,q}(X,\ZZ) \xrightarrow{\simeq} \CH^q(X, 2q-p).$$

\textbf{The Homotopical Revolution:} A crucial parallel development was Quillen's introduction of \tcolor{cyan}{model categories} \cite{quillen1967homotopical}, which provided an axiomatic framework for doing homotopy theory in contexts far beyond topology. This would later prove crucial for the foundations of motivic homotopy theory. On the other hand, between 1970-80s, \tcolor{cyan}{algebraic $K$-theory} developed rapidly under Atiyah, Bass, Milnor, Serre, Quillen et al and it became clear that it had deep connections to the motivic phenomena, but the precise relationship remained mysterious (see also the beautiful explanation of how Serre's conjecture naturally lead to the birth of $K$-theory \cite{lam2006serre}).
\medskip

\textbf{The Birth of $\AA^1$-Homotopy Theory}: The breakthrough finally came when Morel and Voevodsky realized they could bypass some of the difficulties in motivic theory by working in a homotopical rather than homological context. Their key insight was to construct a homotopy theory where the affine line $\AA^1$ plays the role that the unit $[0,1]$ interval plays in classical homotopy theory. They defined the motivic or the \tcolor{cyan}{$\AA^1$-homotopy category} of schemes \cite{MV99} and established that it has the structure of a closed model category. They also showed how to construct motivic cohomology as a \emph{representable} functor in this setting. Their approach involved several sophisticated techniques involving the choice of a Grothendieck topology (Nisnevich), imbibing the theory of simplicial sheaves (thereby encoding homotopical data), and exploiting localization techniques (Bousfield) to systematically invert $\AA^1$-homotopy equivalences.
\medskip

\textbf{Motivic Spaces and Motivic Spectra:} 
Similar to the classical homotopy theory, the motivic homotopy theory has an unstable and stable version. For a field $k$, the unstable motivic theory $\mcal{H}(k)$ has several realization functors that allow one to transport information between the world of algebraic topology and motivic geometry. Passing onto the $S^1$-spectra and inverting the Tate circle $\PP^1$, one obtains the \tcolor{cyan}{motivic spectra} $\Spt(k)$ and the associated stable $\AA^1$-homotopy theory $\SH(k)$. The inversion of $\PP^1$ signifies that spaces become \emph{weakly equivalent} upon smashing with $\PP^1$s. Moreover, inverting $\PP^1$ is essential to have non-trivial transfers in $\SH(k)$ (see \cite[Part III, \S 2, Remark 2.8]{Nordfjordeid2007motivic}). The luxury of the $\AA^1$-homotopy category is that several of the algebraic theories become representable here:
{\small
\begin{enumerate}
\item The algebraic $K$-theory is representable (see \cite{gepner2009motivic, spitzweck2009bott}):
If $k$ is a Noetherian ring with finite Krull dimension, then there is a natural isomorphism in $\SH(k)$
    $$ \Sigma_{\PP^1}^{\infty} \PP^{\infty}_{+} [\beta^{-1}] \xrightarrow{\simeq} \textbf{KGL}, $$
where $\PP^{\infty}$ is the infinite projective space, $\beta$ is a certain Bott element and \textbf{KGL} is the algebraic $K$-theory of Voevodsky \cite{voevodsky1998A1}.

\item Chow groups are representable: assume $k$ is a perfect field. For every $n\ge 0$, and every smooth $k$-scheme $X$, there is a canonical bijection
    $$ \CH^n(X) \xrightarrow{\simeq} \Hom_{\SH(k)} (X_+, \textbf{K}(\ZZ(n), 2n). $$
Here, $\textbf{K}(\ZZ(n),2n)$ is the \emph{motivic Eilenberg–Mac Lane space} 

\item Motivic cohomology is representable (\cite[Chapter 14]{mazza2006lecture}): For every $X\in \Sm_k$ and integers $p,q \in \ZZ$, there is an isomorphism
    $$ H^{p,q}(X,\ZZ) \xrightarrow{\simeq} \Hom_{\SH(k)}(\Sigma_{\PP^1}^{\infty} X_+,  \Sigma^{p,q} \textbf{HZ}) $$
where $\textbf{HZ}$ is the \emph{motivic Eilenberg-Maclane spectrum}.

\item Vector bundles are representable on smooth affine schemes (\cite{morel2012A1topology, schlichting2017euler, asok2017affine}): If $k$ is smooth over a Dedekind ring with perfect residue fields, then for any smooth \emph{affine} $k$-scheme $X$, there is a pointed bijection
        $$ \mcal{V}_r(X) \xrightarrow{\simeq} [X, \Gr_r]_{\AA^1}.$$
Here, $\mcal{V}_r$ denotes the set of algebraic vector bundles of rank $r$ and $\Gr_r$ is the Grassmannian of rank $r$.
\end{enumerate}}

\subsubsection{Some Phenomenal Applications of Motivic Homotopy Theory}
\textbf{The Milnor Conjecture:} 
The Milnor Conjecture \cite[Question 4.3]{milnor1970algebraic} is a statement that asserts a relationship between the Milnor $K$-theory and the fundamental ideal of the Grothendieck-Witt ring of a field $k$ (see \cref{App;MilnorK-theory}). The full solution is due to Orlov-Vishik-Voevodsky \cite{orlov2007exact, Voevodsky00Milnorconj, Suslin97Voevodsky} using motivic homotopy theory:
    $$\mu: K^{\M}_*(F)/2 \xrightarrow{\simeq} \bigoplus_{n\ge 0} \mcal{I}^n(F)/ \mcal{I}^{n+1}(F)$$
\textbf{The Bloch-Kato Conjecture}: 
The Bloch-Kato conjectures were solved due to the cumulative efforts of several mathematicians between 1980-2000. On one side of the plot, we have a family of related conjectures stated independently by Beilinson and Lichtenbaum, which were later unified into what's now called the \emph{Beilinson-Lichtenbaum Conjecture}. It predicts that the motivic cohomology and \'etale cohomology agree in certain ranges for a smooth scheme $X$ over arithmetic fields:
\begin{align*}
    H^{p,q}(X, \mathbb{Z}/n) \xrightarrow{\simeq} H^{p}_{\acute{e}t}(X, \mu_n^{\otimes q}) \quad \text{for}\quad p\le q.
\end{align*}
Around the 1990s, Voevodsky, with several others, recognized that proving the Beilinson Lichtenbaum conjecture required developing motivic cohomology systematically by homotopical methods. On the other side, Bloch and Kato \cite{bloch1986p-adic} formulated their conjecture connecting Milnor $K$-theory to Galois cohomology via the Bloch-Kato map (a.k.a. the norm residue homomorphism). Again, these were a collection of statements rather than a single statement, proven incrementally, with different cases requiring different levels of the motivic machinery. These conjectures were then further generalized and unified into what's called the \emph{Bloch-Kato Conjecture}. It became reformulated in terms of motivic cohomology as it became representable by motivic Eilenberg-MacLane spaces. Building on the proof of the Milnor conjecture, the Bloch-Kato conjecture was eventually proven in full generality through the combined work of several mathematicians, with the final breakthrough coming from Voevodsky \cite{voevodsky2011motivic}: it describes the relationship between Milnor $K$-theory and Galois cohomology: for any field $k$ and integers $n$ and $l\ne$ char $k$, we have an isomorphism
\begin{align*}
    K^{\M}_r(F)/n \xrightarrow{\simeq} H_{\acute{e}t}^r(F, \mu_n^{\otimes r})\quad \text{for all} \quad n,r \ge 0 
\end{align*}
The case $n =2$ is the Milnor conjecture, and $r=2$ is the Merkurjev–Suslin theorem. The proof of the Bloch-Kato conjecture heavily relies on the motivic homotopy theory infrastructure developed between the 1990s-2000s, particularly: The motivic Adams spectral sequence, motivic Steenrod operations (\cite{voevodsky2003reduced}), and so on. A complete proof of the Bloch-Kato conjecture automatically implies the original Beilinson and Lichtenbaum conjecture \cite{voevodsky2003Z/2} and a related Quillen-Lichtenbaum conjecture. The Bloch-Kato isomorphism is also called the \emph{norm residue isomorphism}. In retrospect, rather than being applications of motivic homotopy theory, these conjectures were the driving factors for the emergence of motivic homotopy theory. These demonstrate that motivic homotopy theory is not just abstract machinery but could solve concrete, longstanding problems in arithmetic and algebraic geometry. The following are some more recent applications.
\begin{itemize}
\item Voevodsky’s slice tower construction in $\SH(k)$ (see \cite{voevodsky2002open} and \cite{voevodsky2002possible}) and the verifications of several related conjectures (see e.g., \cite{voevodsky2003zero, levine2008homotopy, spitzweck2012slices, levine2013convergence, bachmann2024etale} and references therein)
\item The construction of Grothendieck’s Six Functor Formalism in $\SH(k)$ \cite{Ayoub2007six},
\item Just like the complex cobordism spectrum \textbf{MU} is universal among complex-oriented cohomology theories, the existence of $\SH(k)$ gives rise to algebraic cobordism spectrum $\Omega^n$ represented by \textbf{MGL}, which is the universal oriented cohomology theory for a smooth scheme $X$ of codimension $n$ (\cite{levine2007algebraic}):
\begin{align*}
    \Omega^n(X) \xrightarrow{\simeq} [\Sigma^{\infty}_{\PP^1}X_+, \Sigma^{2n,n} \textbf{MGL}]_{\AA^1}
\end{align*}
\item Morel’s computation of motivic stable $\pi_0^{\AA^1}$ of the sphere spectrum \cite{morel2012A1topology}, followed by first computations of motivic unstable $\pi_1^{\AA^1}$ by Asok-Morel (\cite{asokmorel2011}) and higher motivic homotopy groups, with applications to vector bundles of Asok-Fasel (\cite{asok2014cohomological, asok2014algebraicspheres}).
\end{itemize}

\textbf{Vista:} This script emerged as a cross-fertilization of ideas collectively drawn from the traditional (commutative algebra, algebraic topology) to the modern (affine algebraic geometry, motivic homotopy theory) vogue machinery. Several mathematical inceptions are borne out of the culmination of these realms, rendering the reciprocal movement fruitful in the last century. The author's ambitious goal is to convince the reader of one such beautiful truth in the next ten-squared pages or so.

% ----------------------------------------------------------------
{\fontfamily{antt}\selectfont
\section*{Blueprint of the script}
\addcontentsline{toc}{section}{Blueprint of the Script}

This script is primarily structured into five chapters that we shall now outline. 
\medskip

In \tcolor{Blue}{\cref{chp1}}, we begin with a gentle introduction to the topology at infinity, a machinery that is exclusively designed to study open manifolds (non-compact, without boundary) and review the general knowledge and motivation for such a machinery. The ideal overriding goal would be to develop such machinery for smooth affine schemes (cf. \tcolor{Blue}{\cref{app:Mot-top-infty}}). We will touch upon the theory of exotic affine varieties in low dimensions in \tcolor{Blue}{\cref{intro:A1-cont-survey}}. Following this, \tcolor{Blue}{\cref{chp2}} introduces us to the beautiful world of motivic homotopy theory. This chapter provides the foundational scaffolding upon which the remainder of the script is assembled. In particular, we set up the basic language for the $\AA^1$-homotopy theory via simplicial Nisnevich sheaves and construct the unstable $\AA^1$-homotopy category in \tcolor{Blue}{\cref{subsec:unstable-construction}}, where algebro-topological analogs such as (path) connectedness and contractibility are furnished. Among important formulations are that of functoriality (\tcolor{Blue}{\cref{sect:functoriality}}) and the notion of relative $\AA^1$-contractibility (\tcolor{Blue}{\cref{intro:rel-A1-contr}}). Affine algebro-geometric counterpart is discussed in \tcolor{Blue}{\cref{chp3}}, where we study in detail the characterization of the affine $n$-space from various sources of related formulations, giving us a multi-faceted vantage point. The foundations developed in these chapters will allow us to develop the theory of relative $\AA^1$-contractibility of smooth schemes in low dimensions in \tcolor{Blue}{\cref{chp4}}. We proceed dimension-by-dimension. The theory of $\AA^1$-contractibility for \'etale (i.e., zero-dimensional) schemes over fields and its extension over arbitrary base schemes are discussed in \tcolor{Blue}{\cref{sect:Rel-etale schemes}}. Here we obtain the following (\tcolor{Blue}{\cref{0-dim:etale-schemes-iso}}):
\begin{theorem*}
Let $X$ be an \'etale scheme of finite presentation over an arbitrary base scheme $S$. Then the canonical morphism $f: X\to S$ is a relative $\mathbb{A}^1$-weak equivalence in $\Spc_S$ if and only if $f$ is an isomorphism in $\Sm_S$. 
\end{theorem*}
We also show that any separated \'etale scheme is necessarily $\AA^1$-rigid (\tcolor{Blue}{\cref{0dim:rigid}}). In \tcolor{Blue}{\cref{sect:dim=1}}, we move on to dimension 1 where the uniqueness of affine line $\AA^1$ is established over arbitrary fields (\tcolor{Blue}{\cref{1-dim theorem}}) which gives us the leverage to extend this characterization to the relative setting over (normal) base schemes in \tcolor{Blue}{\cref{sect:reldim=1}}. The main theorem here is the following (\tcolor{Blue}{\cref{1-dim DD}})

\begin{theorem*}
For a smooth separated morphism of finite presentation $f:X\to S$ and relative dimension $1$ over a Noetherian normal scheme $S$, the following are equivalent:
\begin{enumerate}
\item $f:X\to S$ is a relative $\AA^1$-weak equivalence in $\mathrm{Spc}_S$,
\item $f:X\to S$ is an $\AA^1$-fiber space,
\item $f:X\to S$ is a Zariski locally trivial $\AA^1$-bundle. 
\end{enumerate} 
\end{theorem*}
This has the following major applications: 
\medskip

a) The characterization of the affine line over (normal) base schemes (\tcolor{Blue}{\cref{A1unique:DD}}):
\begin{corollary*}
A smooth separated scheme $f:X\to S$ of finite presentation and relative dimension $1$ over a Noetherian normal scheme $S$ is isomorphic to $\mathbb{A}^1_S$ if and only if all of the following holds:
\begin{itemize}
\item[1.] $X$ is a smooth $\mathbb{A}^1$-contractible curve in $\mathcal{H}(S)$,
\item[2.] the relative canonical sheaf of differentials $\Omega_{X/S}$ of $f$ is trivial, and
\item[3.] $f$ has a section.
\end{itemize}
\end{corollary*}
b) The relative variant of Zariski Cancellations (\tcolor{Blue}{\cref{Cor:rel-ZCP:dim1}}):
\begin{corollary*}
Let $f: X\to S$ be a smooth separated scheme of finite presentation and relative dimension $1$ over a normal scheme $S$ such that $X\times_S \AA^n_S\cong \AA^{n+1}_S$ for some $n\geq 1$. Then $X$ is $S$-isomorphic to $\AA^1_S$. 
\end{corollary*}
Here, the normality of the base scheme is essential as we provide plenty of counter-examples in \tcolor{Blue}{\cref{non-examples:reldim1}}. Systematically, we proceed to dimension 2, where we discuss the uniqueness of an affine plane over a field of characteristic zero, as was recently proven by Choudhury-Roy (\tcolor{Blue}{\cref{A2isunique})} and extend it over perfect fields in \tcolor{Blue}{\cref{A2unique:perfect}}. This characterization being open for imperfect fields, we provide some speculations involving non-trivial forms of $\AA^2_k$ in \tcolor{Blue}{\cref{sect:non-trivial-forms-A2}}. All this will enable us to prove the following main result (\tcolor{Blue}{\cref{2-dim DD}}):
\begin{theorem*}
Let $X$ be a smooth scheme of relative dimension $2$ over a Dedekind scheme $D$ with characteristic zero residue fields. Furthermore, let the canonical morphism $f: X \to S$ be affine. Then the following are equivalent:
\begin{enumerate}
\item $f$ is a relative $\AA^1$-weak equivalence in $\Spc_S$,
\item $f$ is a Zariski locally trivial $\AA^2$-bundle.
\end{enumerate}
\end{theorem*}
This has the following major applications:
\medskip

a) The characterization of the affine plane in relative dimension 2 (\tcolor{Blue}{\cref{A2unique:DD}}):
\begin{corollary*}
A smooth affine scheme $f:X\to D$ over an \emph{affine} Dedekind scheme $D$ with characteristic zero residue fields is isomorphic to $\mathbb{A}^2_D$ if and only if both of the following holds:
\begin{enumerate}
\item $X$ is a smooth $\mathbb{A}^1$-contractible surface in $\mathcal{H}(D)$, and
\item the relative canonical sheaf of differentials $\omega_f$ is trivial.
\end{enumerate}
\end{corollary*}
b) The relative variant of Zariski Cancellations (\tcolor{Blue}{\cref{Cor:rel-ZCP:dim2}}):
\begin{corollary*}
Let $D$ be an affine Dedekind scheme with characteristic zero residue fields. Then for a smooth affine scheme $f: X\to D$ of relative dimension 2 over a Dedekind scheme $D$, the isomorphism $X\times_D \AA^1_D \cong \AA^3_D$ implies that $X \cong \AA^2_D$.
\end{corollary*}
This gives us the complete picture in relative dimension 2 as there are well-established counter-examples over a field of positive characteristics (\tcolor{Blue}{\cref{eg:Asanuma3F}} and other examples therein). Following these characterizations, we jump on to higher dimensions ($\ge 4$) in \tcolor{Blue}{\cref{sect:higherdimensions}} where our main motivation is to expand on techniques to produce abundant relatively $\AA^1$-contractible smooth affine schemes. The relative $\AA^1$-contractibility of higher-dimensional smooth (affine) schemes has been studied extensively in \tcolor{Blue}{\cref{sect:ADF-family}} where our motivation is to exemplify that a certain family of smooth affine quadrics, which has been well-studied in the literature, gives ample instances of relatively $\AA^1$-contractible smooth schemes. This fact is spelled-out in \tcolor{Blue}{\cref{thm:An-fiber-space-contractible}}:
\begin{theorem*} 
Let $f:X\to S$ be a separated $\mathbb{A}^n$-fiber space over a regular scheme $S$, for all $n\ge 0$. Then $X$ is relatively $\mathbb{A}^1$-contractible over $S$.
\end{theorem*} 
In \tcolor{Blue}{\cref{chp5}}, we study the relative $\AA^1$-contractibility of threefolds, which completes the picture of our quest for relative $\AA^1$-contractibility. We study the family of Koras-Russell threefolds $\KK$ over a perfect field in \tcolor{Blue}{\cref{sect:Koras-Rusell-perfect}} and establish its unstable $\AA^1$-contractibility over perfect fields (\tcolor{Blue}{\cref{A1-cont-over-perfect-fields}}):
\begin{prop*}
For any perfect field $k$, the canonical morphism $f:\KK\to \Spec k$ is an $\AA^1$-weak equivalence in $\Spc_k$.
\end{prop*}
The above proposition, when coupled with Morel-Voevodsky's gluing technique (\cref{MV:gluing theorem}), will allow us to derive the $\AA^1$-contractibility of Koras-Russell threefolds over relative base schemes as well:
\begin{theorem*} 
The following are \tcolor{Blue}{\cref{KR3FoverZ}} and \tcolor{Blue}{\cref{KR3Fmain:Noetherian}}:
\begin{enumerate}
\item The Koras-Russell three fold of first kind $f: \KK\to \Spec \ZZ$ is relatively $\AA^1$-contractible in $\Spc_{\ZZ}$. In particular, $\KK$ is an $\AA^1$-contractible scheme in $\Spc_{\ZZ}$ that is not isomorphic to $\AA^3_{\ZZ}$.
\item Let $S$ be any Noetherian scheme with perfect residue fields. The Koras-Russell threefold of first kind $f: \KK \to S$ is relatively $\AA^1$-contractible in $\Spc_S$.
\end{enumerate}
\end{theorem*}

By observing that all these techniques solely involve fibrewise analysis, we proceed to invoke all these techniques to their higher-dimensional analogs in \tcolor{Blue}{\cref{sect:generalized-KR}}. Here, we study the \emph{generalized Koras-Russell varieties} and enhance its unstable $\AA^1$-contractibility over perfect fields (\tcolor{Blue}{\cref{KR3Fprototypes:field-A1-cont:perfect}})
\begin{theorem*}
Let $k$ be any perfect field. Then the canonical map of the generalized Koras-Russell varieties $\XX_{m}(\ul{n},\psi) \to \Spec k$ is an $\AA^1$-weak equivalence in $\Spc_k$.
\end{theorem*}
This will then allow us to extend this to general base schemes (\tcolor{Blue}{\cref{KR3F:prototypes-A1-cont-base:Noetherian}}):
\begin{corollary*}
Let $S$ be any Noetherian scheme with perfect residue fields. Then the canonical morphism of the generalized Koras-Russell varieties $\varphi: \XX_{m}(\ul{n},\psi) \to S$ is an $\AA^1$-weak equivalence in $\Spc_S$.
\end{corollary*}
All these results will fruitage to a major application: existence of exotic motivic spheres in higher dimensions. We proceed with a survey of exotic motivic spheres in dimensions $\le 3$ in \tcolor{Blue}{\cref{sect:overview-mot-sph}} and establish the existence of exotic spheres in dimensions $\ge 4$ in \tcolor{Blue}{\cref{sect:exist-mot-sph-higherdim}}. In particular, we show that the complement of Koras-Russell prototypes is sufficiently connected (\tcolor{Blue}{\cref{X-p-isA1-connected}}):
\begin{lemma*}
Let $k$ be any infinite perfect field. Then the strictly quasi-affine variety $\XX_m \bs \{\bullet\}$ is $\AA^1$-chain connected.
\end{lemma*}
Primarily using a motivic excision style result, we then show that these varieties in fact provide examples of exotic motivic spheres (\cref{exotic-motspheres-countereg}):
\begin{theorem*}
Let $k$ be an infinite perfect field. Then for every $m\ge 0$, the strictly quasi-affine variety $\XX_m\bs \{\bullet\}$ of dimension $m+3$ is $\AA^1$-homotopic to $\AA^{m+3}_k \bs\{0\}$, but is not isomorphic to $\AA^{m+3}_k\bs \{0\}$ as a $k$-scheme.
\end{theorem*}

This concludes the bulk of the script. The author has taken substantial efforts to make this script as self-contained as possible. For this \emph{motive}, a good collection of background and supporting ideas is provided in \tcolor{Blue}{\cref{app:category}}, \tcolor{Blue}{\cref{app:A1-alg-top}}, and \tcolor{Blue}{\cref{app:Scheme-theory}}.
}
\newpage

% -----------------------------------------------------------
{\fontfamily{antt}\selectfont
\section*{The script in a concise non-expert viewpoint}
\addcontentsline{toc}{section}{For Non-Experts}

Spaces similar to "points" in algebraic geometry are called \tcolor{cyan}{$\AA^1$-contractibles}, and our script revolves around these objects. The affine spaces $\AA^n_k$ are primordial examples of $\AA^1$-contractible schemes, and so one ponders whether a smooth affine scheme $X$ being $\AA^1$-homotopic to $\AA^n_k$ in the motivic homotopy category automatically implies that they are isomorphic as schemes as well?! We will review this situation over base fields and relatively extend it to general base schemes. We will witness that there are instances where this does hold and instances otherwise; spaces of the latter nature will be called \tcolor{cyan}{exotic}. We shall cover this aspect in the first part of our script in \cref{chp4}. In contrast to this aspect, the material further studies the "compact" analog of exoticness. The smooth schemes $\Absn$ are examples of spheres in motivic geometry. This makes one ponder 
if these are the only spheres up to $\AA^1$-homotopy, i.e., whether a smooth scheme $X$ being $\AA^1$-homotopic to $\Absn$ in the motivic homotopy category automatically implies that they are isomorphic as schemes?! This forms the sole essence of \cref{chp5}, where we discover the anti-verity of this expectation.
}

% \afterpage{\blankpage}

\chapter{Acknowledgments}
\markboth{Acknowledgment}{}

{\fontfamily{antt}\selectfont
\textsc{To all the beautiful people who landed me safe on ground...}
\vspace{2mm}

I am in excess debt towards the prudent parenting of Vijayalakshmi B. and Dr. Madhavan N., whose unceasing love and meticulous parenting have helped me surmount snags of life, and Kishore Kumar M.V. for making us proud every moment with his relentless source of creativity and unwavering willpower.
\vspace{1mm}

This piece of research, and thereby, the manuscript would cease to exist if it were not for the ever energetic and expert guidance from my lovely advisors: Paul Arne {\O}stv{\ae}r and Adrien Dubouloz. I am deeply grateful to Paul Arne for introducing me to motivic geometry and for his numerous academic advices throughout this journey - especially for insights on drafting concise mathematical articles and for several sportive gym interactions! If I were to say this joint doctorate had worked out without much hustle, it is due to Adrien's warm hospitality and curation towards administrative conundrums - I thoroughly enjoyed our coffee room math discussions and relaxed conversations, which I'll cherish for years to come!
\vspace{1mm}

I extend my gratitude to Aravind Asok and Marc Hoyois for various insightful discussions at Math events.
\vspace{1mm}

I remain grateful to Frédéric Déglise and Jean Fasel for various academic interactions during conferences and my stay in France.
\vspace{1mm}

I express my gratitude to Matthias Wendt for many fruitful discussions during my visit to the Algebra and Topology Oberseminar at Wuppertal.
\vspace{1mm}

I remain grateful to Tariq Syed for his countless academic advices and cheerful interactions during our countless Math meetings.
\vspace{1mm}
 
I thank Parnashree Ghosh, Anand Sawant, and Biman Roy for their fruitful and friendly discussions at CIRM, Luminy.
\vspace{1mm}

I gratefully acknowledge the research funding provided by Università di Milan and Institut Mathematiques de Bourgogne, and the mobility grant provided by the Université Franco-Italienne / Università Italo-Francese through the \emph{C2-379: Programme Vinci 2024}. I thank the timely interventions and cumulative efforts of various secretaries and professionals, especially Stefania and Vanessa in Milan, and Anissa, Bibiane, Emeline, Francis, Magali, and Laura in Dijon, for pragmatically working along despite linguistic hurdles in making this cotutelle carry out smoothly - I owe you all a ton of reverence!
\vspace{1mm}

                                 % MILAN LAB - split up
Having spent close to 4 years now, Milano has become close to my heart. I am grateful to several versatile mathematicians from Milano, especially Fabrizio Andreatta, Federico Binda, Bert Geemen, Amnon Neemon, and Paolo Stellari, whose interactions have helped me in different ways throughout these years. I am immensely grateful to my cycle-mates: Andrea, Alessandra, Filippo, Kaixing, Mariano, and Silvia for being compassionate during these 3 lovely years. I'd like to thank Martina for her countless assistance and swift responses with the cotutelle formalities - it'd have been difficult without you! My warm thanks to my academic brother, Federico, for many useful discussions and for traveling alongside me for various mathematical events across Europe! My sincere thanks to Giulia for helping me out with a crucial bureaucratic process at a point in time! Among many cheerful people that I've met in Milan, I had the pleasure of meeting Dhrumith, Marco Manoj, and Lorenzo, whose affection and timely assistance carry paramount significance.
\vspace{1mm}

            % DIJON LAB  - split up
My stay at the IMB, Dijon, was such a delight, as it allowed me to meet several interesting mathematicians and dearly people. I remain grateful to Mattia Cavicchi, Daniele Faenzi, Lucy Moser-Jauslin, Jan Nagel, and Julia Schneider for organizing the reading seminars and for countless laid-back interactions. The community experience was exemplified by several kind people that I met, including Arnaud, Deniz, Dimitrios, Edwin, Federica, Felipe, Ivan, Keyao, Leila, Lyalya, Maya, Mehdi, Oscar, Ouneïs, Pietro, Thibaut, and Victor (FR+BR) - cheers to have made IMB a home every single day! I had the pleasure of sharing the apartment space with a caring flatmate and a fellow mathematician, Théo, who helped me stabilize the daily routine and several bureaucratic fickle, along with kind neighbors Helal-Sandhya. Among many cheerful people that I've met outside of Math, I'm happy to have met Jeswin and Sajana, who made life more sociable!
\vspace{1mm}

I'd like to extend my warm thanks to Ritheesh for being out there for me at times of need. This exciting journey would have been impossible if not for the anchoring, consistent support of several people back home during the last 5-year-long tide of my life, especially Asohamithran, Durairaj, and Vassavi. Thank you, guys, we've made it!
\vspace{2mm}

My deepest gratitude to all those who have spent their time underpinning me throughout this journey, who, I assume, are philanthropic enough to envision and acknowledge that their contribution has had a palpable impact on my life so far.
\vspace{1mm}

\textsc{Last but foremost, to nature and science for emphatically providing everyone with a source of rejuvenation and mental agility.}

\begin{flushright}    
\textsc{Krishna Kumar M.V.}
\vspace{1mm}

Dijon \& Milan (2022-2025)
\end{flushright}

}

% \afterpage{\blankpage
% \thispagestyle{empty}}

\chapter{Useful notations}
\markboth{Useful notations}{}

\begin{tabular}{llp{.7\textwidth}}
$R_X$ &--& Representable functor associated to a space $X$ \\
$\textbf{K}(X,n)$ &--& Eilenberg-Maclane space of $X$ in degree $n$\\ 
$\Sm_S$    &--& The category of smooth separated schemes over $S$\\
$\sPsh(\Sm_S)$     &--& The category of simplicial presheaves over $\Sm_S$\\
$\sShv(\Sm_S)$    &--& The category of simplicial sheaves over $\Sm_S$\\
$\Psh_{\AA^1}(\Sm_S)$    &--& The category of $\AA^1$-invariant presheaves over $\Sm_S$\\
$\Shv_{Nis}(\Sm_S)$    &--& The category of Nisnevich sheaves over $\Sm_S$\\
$\Spc_S$     &--& The category of motivic spaces over $S$\\
$\SH(S)$   &--& The stable $\AA^1$-homotopy category over $S$\\
$\mcal{H}(S)$    &--& The unstable $\AA^1$-homotopy category over $S$\\
$\RR^n$   &--& The $n$-dimensional Euclidean space\\ 
$\CC^n$    &--& The $n$-dimensional complex (analytic) space\\
$\SS^n$    &--& The $n$-dimensional sphere \\
$\AA^n_S$   &--& The $n$-dimensional affine space over a base scheme $S$ \\
$S^1$   &--& The simplicial circle\\
$\GG_m$       &--& The Tate circle\\
$\SS^{p,q}$   &--& Mixed motivic spheres in weights $p$ and $q$\\
$\PP^n_S$     &--& The $n$-dimensional projective space over $S$\\
$\ML(R)$      &--& The Makar-Limanov invariant of a commutative ring $R$\\
$\mcal{D}(R)$  &--& The Derksen invariant of $R$\\
$\kappa(s)$      &--& Residue field at a point $s \in S$\\
$\bar{\kappa}(X)$   &--& Logarithmic Kodaira dimension of an algebraic variety $X$\\
$\Pic(X)$        &--& Picard group of $X$\\
$\pi_n^{\AA^1}$  &--& $n$-th $\AA^1$-homotopy sheaf\\
$\Th(E)$        &--& Thom space associated to the vector bundle $\pi: E\to B$\\
$\Sigma X$        &--& Simplicial suspension of a space $X$\\
$\pi_1^{\infty}$   &--& The fundamental group at infinity\\
$\Omega_f$     &--& Sheaf of relative K\"ahler differentials associated to $f$\\
$\omega_f$   &--& Relative canonical sheaf of $f$\\
$\GW(k)$     &--& Grothendieck-Witt ring over a field $k$\\
$\W(k)$     &--& Witt ring over $k$ \\
$K^{\M}$ &--&  The Milnor $K$-theory\\
$K^{\MW}$  &--&  The Milnor-Witt $K$-theory
\end{tabular}

% \clearpage
% \thispagestyle{empty}
% \hfill
% \clearpage
% \vfill
% \begin{center}
% Welcome to the end of the beginning!
% \end{center}

\chapter*{}
\markboth{}{}

\thispagestyle{empty}
\begin{center}

\begin{flushright}
To the challenges that made me stronger...
\end{flushright}

\end{center}

\mainmatter

\begin{savequote}
One can always see the light at the end of the tunnel, if only one can make the ends meet. \qauthor{The author.}
\end{savequote}
\chapter{Overview}\label{chp1}
\markboth{I Overview}{}

\section{Introduction} 
The Euclidean spaces $\RR^n$ are one of the most basic objects of study in manifold theory. They are open (non-compact, without boundary) smooth $n$-manifolds that are contractible, i.e., homotopy equivalent to a point, and so they have trivial homology and homotopy theories. It is therefore natural to characterize this object with respect to this very property, namely, the following:
\begin{question}\label{qstn:OPC}
Is the Euclidean space $\RR^n$ the unique open contractible smooth manifold in dimension $n$?
\end{question}
Phrased differently, the question claims that "every open smooth $n$-manifold that is homotopic to $\RR^n$ is homeomorphic to $\RR^n$". This statement should ring a bell to its closely resembling contrast in topology: the celebrated \emph{Poincaré Conjecture}, whose generalized version asserts that "every closed smooth $n$-manifold homotopic to $\SS^n$ in a specified category $\mathcal{C}$ is also homeomorphic to $\SS^n$ in $\mathcal{C}$". The topological situation $\mathcal{C}=\rm{Top}$ is proven to be true in all dimensions: Smale ($n\ge 5$), Freedman ($ n=4$), and Perelman ($n=3$), and other cases were well-known. On the flip side, the smooth situation $\mathcal{C}=\rm{Diff}$ is not fully solved: well-known for $n\le 2$, the case $n=3$ is peculiar as the topological and the smooth version coincide, the case $n=4$ seems to be notoriously difficult and has been open for decades now (refer to \cite{scorpan2005wild}). The smooth version fails in certain higher dimensions due to the presence of Milnor's \emph{exotic spheres} (\cite{milnor1956spheres}) - a sphere that is homeomorphic to $\SS^n$ but not diffeomorphic to the $\SS^n$ with the standard differential structure. 
\medskip

In the light of this breathtaking program, \cref{qstn:OPC} is famously known as the \emph{Open Poincaré Conjecture}. This question is now completely understood using the machinery of \emph{topology at infinity}, as we shall review in the upcoming sections. Our goal in this chapter is to give the background and solution to this question in the topological setting and motivate its algebro-geometric parallel. In \cref{intro:Top-infty}, we briefly review the formulation in topology describing homology and homotopy at infinity. We then introduce the algebro-geometric analog of \cref{qstn:OPC} in \cref{intro:A1-cont-survey} (\cref{qstn:OPC-Alg-geo}). This program is also one of the key motivating factors for the author to delve deeply into algebraic geometry, and this script could be seen as one of its by-products.

% ------------------------------------------------------------------
\subsection{Topology at Infinity}\label{intro:Top-infty}
Topology at infinity is the study of how spaces behave "at their boundary" or in regions that are very far from any fixed point. It's a way of studying the global structure of non-compact spaces by examining what happens as one moves arbitrarily far out. This is systematically studied by developing the so-called "end" of a space. An \emph{end} of a (simplicial or chain) complex is a sub-complex with a particular type of infinite behavior capturing non-compactness in topology and infinite generation in algebra. The theory of ends dates back as far as the 1930s to Freudenthal \cite{freudenthal1930enden}, who initiated this in connection with topological groups.
\medskip

The theory of end spaces was initially applied to that of $\sigma$-compact spaces - a space $X$ with an increasing sequence $\{K_n\}$ of compact subspaces with each $K_n$ in the interior of $K_{n+1}$ such that
            $$X = \bigcup_{n=0}^\infty K_n$$
and the compact subsets $K_n$ of $X$ are partially ordered by inclusion. The idea was to study the homotopy theory of such spaces by considering its "ends" pointing outward of $X$. The primary example to think of is Euclidean spaces. For instance, $\RR$ has two ends pointing outside of it, the $+\infty$ and $-\infty$, whereas $\RR^2$ only has one end pointing outside, the $\infty$. More precisely, consider the system of spaces
  $$\mathcal{E}(X):= \{\mathrm{Closure}(X\bs K): K\subset X\ \rm{compact} \}.$$
This forms an inverse system (a.k.a. pro-object) in the category of spaces and thus, a profinite space. The \emph{set of ends} of $X$ is a connected component of this system of spaces, i.e.,
        $$e(X) := \lim_{K} \pi_0\ \bigg(\mathcal{E}(X)\bigg) $$
equipped with the inverse limit topology.

\begin{remark}
For those spaces without a boundary, its end can be interpreted as an "ideal boundary" which captures the different possible notions of infinity. Some common instances are one-point compactification in point-set topology, asymptotic behaviors in geometric topology, boundary at infinity in hyperbolic geometry, and so on.
\end{remark}

Here is a more refined definition of end spaces.
\begin{defn}
The end space $e(W)$ of a space W is the space of paths
$$\omega: \big([0,\infty],\{\infty\}\big) \to (W^{\infty}, \{\infty \})$$
with the compact-open topology such that $\omega^{-1}(\infty) = \{\infty\}$. Here, $W^{\infty}:= W\cup \{\infty\}$ is the one-point compactification of $W$.
\end{defn} 
For a topological space $W$, its end space $e(W)$ can be realized as a homotopy theoretic model for the behavior of $W$ at infinity.
\begin{example}
Here are some basic examples:
\begin{itemize}
\item If $K$ is compact, then $e(K) = \emptyset$.
\item The real line has two components at infinity corresponding to the two distinct connected open neighborhoods $(-\infty,-k)$ and $(k,\infty)$, for $k\in \RR$. Thus, $e(\RR) \cong \SS^0 = \{\pm \infty\}$. Therefore, an end space need not be necessarily connected. 
\item For all $n \ge 2$, $\RR^n$ has only one end since $\RR^n\bs K$ has only one unbounded component for any compact subset $K\subset \RR^n$, thus $e(\RR^n) \cong \SS^{n-1}$. 
\end{itemize}
\end{example}

% -------------------------------------------------------------------------
\subsubsection{Homology at infinity}
A suitable framework to study the theory of end spaces is the proper homotopy theory, in the sense that a map $f: V \to W$ is a \emph{proper homotopy equivalence} if $f$ is a proper\footnote{inverse image of compact is compact} map, and it induces homotopy equivalence in the proper category. The properness of the $f$ is equivalent to extending $f$ to infinity via $f^{\infty}: V^{\infty}\to W^{\infty}$. Just as ordinary homology theory is good enough to study compact spaces, \emph{locally finite homology} plays a key role for non-compact spaces. Recall that the \emph{singular homology} groups $H_*(W)$ are defined as the homology theory associated to the singular chain complex $S_*(W)$, i.e., $H_*(W):= H_*(S_*(W))$. By replacing the singular chain complexes with locally finite chain complexes $S_*^{lf}(W)$, one defines the \emph{locally finite homology} groups $H_*^{lf}(W):= H_*(S_*^{lf}(W))$. From the definition of these chain complexes, one immediately verifies that if $K$ is any compact space, then $H_*(K) =  H_*^{lf}(K)$ (\cite[Example 3.4]{hughes1996ends}). These complexes provide us with the following definition:

\begin{defn}\label{defn:sing-hom-infty}
The \emph{singular homology at $\infty$} of a space $W$ is defined as $H_*^{\infty}(W):= H_*(S_*^{\infty}(W)$, where 
 $$S^{\infty}(W):= \textrm{Cone}(i: S_*(W)\to S_*^{lf}(W))_{*+1}$$ 
of the inclusion $i:S_*(W)\to S_*^{lf}(W)$ obtained by viewing a $r$-simplex $\sigma: \Delta^r \to W$ as a locally finite singular chain (\cite[Definition 3.8]{hughes1996ends}).
\end{defn}

These so-defined homology groups are related by the following long exact sequence:
$$\dots \to H^{\infty}_r(W)\to H_r(W)\xrightarrow{i_*} H^{lf}_r(W)\to H_{r-1}^{\infty}(W)\to \dots$$
\begin{example}\label{eg:homol-infty-sphere}
The following provides some first examples:
\begin{itemize}
\item Let $K$ be any compact space. Then $S_*(K) = S_*^{lf}(K)$ and so the group $H^{\infty}_* (K)$ is trivial. 
\item Using Poincaré-Lefschetz type argument, one observes that the homology at infinity of any smooth open contractible $n$-manifold is that of a sphere of dimension $n-1$. For $n\ge 1$, let $W = \RR^n\times K$. Then we have $H_*(W) = H_*(K)$ and $H_*^{lf}(W) = H_{*-n}(K)$. And so,
  $$H_*^{\infty}(W) =  H_*(K\times S^{n-1}) = H_*(K)\oplus H_{*-n+1}(K).$$ 
\end{itemize}
\end{example}
Thus, the homology at infinity $H_*^{\infty}(W)$ of a space $W$ is the proper homotopy invariant given by the difference between homology $H_*(W)$ and locally finite homology $H^{lf}_*(W)$ and the compactness of $W$ is determined by the vanishing of $H_*^{\infty}(W)$.

\subsubsection{Homotopy at infinity}
In a similar vein as above, one has a parallel theory for homotopy groups. As in the case of end spaces, let us assume $X$ is a $\sigma$-compact topological space with the compact subspaces $K_n\subset X$ partially ordered by inclusion, and if $K\subset K'$, we simply consider the inclusion $X\bs K' \subset X\bs K$. A \emph{base point at infinity}, denoted by $\star$, is given by an open subset $U\subset X$ such that for any compact set $K$, there exists a compact set K' with $K\subset K'$ and $U\bs K'$ non-empty and simply connected. 
\begin{defn}\label{defn:fund-grp-at-infty}
Here is a definition of the fundamental group at infinity (\cite[P. 2]{diez2001fundamental}):
\begin{itemize}
\item Let $X$ be any topological space with a base point at infinity $\star$. Then we define the \emph{fundamental group at infinity of $X$ based at $\star$} as the inverse limit of the relative homotopy groups given as follows:
    $$\pi_1^{\infty}(X,\star) = \lim_{\leftarrow} \pi_1(X\bs  K, U\bs K)$$
where the inverse limit is over the cofinal family of compact subsets $K$ of $X$ such that $U\bs K$ is simply connected, partially ordered by inclusion, and using the natural induced maps on the fundamental groups. 
\item We say that $X$ is \emph{simply connected at infinity} if $\pi_1^{\infty}(X,\star)$ is trivial.
\end{itemize}
\end{defn}

Alternatively, the space $(X,\star)$ is said to be \emph{simply connected at infinity} if for every compact $K \subsetneq X$, there exists $K \subsetneq K' \subsetneq X$ such that the induced map on the fundamental groups 
    $$\pi_1(X\bs K') \xrightarrow{0} \pi_1(X\bs K) $$ 
is the trivial map.

\begin{example}\label{egs:homotopy-at-infty}
Here are some basic examples:
\begin{itemize}
\item If $K$ is compact, then the group $\pi_1^{\infty}(K)$ is trivial.
\item For all $n\ge 0$, we have $\pi_1^{\infty}(\RR^n) \cong \pi_1(\SS^{n-1})$. Hence, $\RR^n$ is simply connected at infinity if and only if $n\ge 3$. 
\end{itemize}
\end{example}

We are now in a position to spell out the solution of \cref{qstn:OPC}. The proof is due to the culmination of several of the preceding mathematical works (see \cref{intro:exotic-threefolds}; for historical contributions see \cite[\S 1.1]{asok2021A1}) with the final breakthrough coming from Siebenmann \cite{siebenmann1968detecting}. Combined with the astonishing work of Perelman and Freedman on the low-dimensional Poincaré conjecture (\cite{Freedman1997work, Perelman2014poincare}), Siebenmann's result gives us the following characterization of Euclidean spaces:
\begin{theorem}\label{thm:char-of-Rn}
The following version is stated in \cite[Theorem 2]{asok2021A1}.
\begin{enumerate}
\item For $n\le 2$, $\RR^n$ is the unique open contractible $n$-manifold,
\item For $n\ge 3$, $\RR^n$ is the unique open contractible $n$-manifold that is simply connected at infinity.
\end{enumerate}
\end{theorem}

As witnessed above, the assumption of simply connectedness at infinity is necessary for dimension 3 and above. This is due to the existence of certain \emph{exotic} manifolds, which by definition, are smooth contractible manifolds that are not homeomorphic to $\RR^n$. We will revisit this situation shortly in \ref{intro:exotic-threefolds}. 

\begin{remark}
In terms of end spaces, we say that a space $X$ is simply connected at infinity if $e(X)$ is both simply connected and connected as a space. In fact, rather than just studying the $\pi_1^{\infty}$, one can study the "homotopy type at infinity" in the context of proper homotopy theory, for which we redirect the interested reader to  \cite[Chapter 9]{hughes1996ends}.    
\end{remark}
It is time to transport this story into the realm of algebraic geometry!

%--------------------------------------------------------------------
\section{Smooth Affine Schemes up to $\AA^1$-Contractibility}\label{intro:A1-cont-survey}

In parallel to the topological situation, one observes that affine spaces $\AA^n_k$ are one of the most basic objects of study in algebraic geometry. They are smooth, affine, and $\AA^1$-contractible - equivalent to that of a base field $k$. The latter property implies that many of the known algebro-geometric theories become trivial for $\AA^n_k$. Hence, the following question is inevitable:
\begin{question}\label{qstn:OPC-Alg-geo}
Is $\AA^n$ the unique smooth affine $\AA^1$-contractible scheme of dimension $n$?
\end{question}
Alternatively, for a base field $k$, if a smooth affine scheme of dimension $n \ge 0$ is $\AA^1$-homotopic to $\AA^n_k$, then is it also isomorphic to $\AA^n_k$ as a $k$-scheme? We will begin with a historical outlook of \cref{qstn:OPC-Alg-geo}.

\subsection{Exotic Affine Varieties}
The study of real and complex manifolds up to topological contractibility is a far-reaching concept that has several connections among nearby fields. The sophistication of algebraic geometry is the ability to forget the algebraic structure, when necessary, allowing one to compare it with the underlying topological structure. These are achieved by certain \emph{realization functors}. As a consequence, one might obtain spaces that are topologically contractible but not contractible in the $\AA^1$-homotopy category. As a linkage to topological situation, a scheme that is $\AA^1$-contractible but not isomorphic to $\AA^n_k$ will be called \emph{exotic}. Experience from the topological situations from the preceding sections humbles one to walk through \cref{qstn:OPC-Alg-geo} dimension-by-dimension. Any 1-dimensional smooth complex contractible manifold is isomorphic to $\AA^1_{\CC}$ (use \cref{thm:char-of-Rn} with $\CC\cong \RR^2$). A difficulty presents itself already in the next step.

\subsubsection{Exotic complex surfaces}\label{intro:exoticsurfaces} 
The complex space $\CC^2$ is a contractible complex manifold. But apparently, this is not the only one. The first counter-example was constructed in 1971 in the landmark paper by C.P. Ramanujam \cite{Ram71} in response to a question of M.P. Murthy on the Zariski cancellation in dimension 2. He hoped that the hypotheses of a smooth affine rational complex variety $S$ being topologically contractible automatically imply $S$ is isomorphic to $\CC^2$. But to his (our) much surprise, he constructed a counter-example himself, now called the \emph{Ramanujam surface} in his eponym. The working strategy of his counter-example can be outlined as follows:
\begin{example}\label{eg:Ramanujamsurf}
Take the complex projective plane $\PP^2_{\CC}$. Consider a cubic with a cuspidal (or flex) point $q$ and another non-degenerate conic $Q$ such that their intersection multiplicity at a point $p\in \PP^2$ is 5 ($m_p(C\cap  Q) = 5 $) and another point $r\in \PP^2$ where they intersect transversally ($m_r(C\cap Q)=1$). Such a $Q$ exists either by avoiding the situation where $p=q$ or by deducing the group law on the smooth locus of $C$. Now, consider the blowup $X:= Bl_r(\PP^2)$ at the point $r$ 
\begin{align*}
 & \hspace{15mm} \pi: X\to \PP^2\quad  \textrm{such that} \\
 & \pi^{-1}(C) = \widehat{C}\quad  \text{and} \quad \pi^{-1}(Q) = \widehat{Q}
\end{align*}
where $\widehat{C}$ and $\widehat{Q}$ are the strict transforms of $C$ and $Q$ respectively. The Ramanujam surface is defined by setting
        $$\mathcal{R}:= X \bs \{\widehat{C}\cup \widehat{Q} \}$$
He then shows that $X$ is simply connected with trivial homology class $H_1(X) = (1)$ and that $\pi_1(X)$ is an Abelian group. Hence, by the Hurewicz theorem, one gets that $\pi_1(X)\cong H_1(X) = (1)$, whence $X$ is topologically contractible. However, he \cite[\S 3]{Ram71} showed that $\pi_1^{\infty}(X) \ne 0$, while the complex affine spaces $\CC^n$ are all simply connected at infinity. Since $\pi_1^{\infty}$ is a homeomorphic invariant, this, in particular, implies that $X \ncong \CC^2$.
\end{example}

But in the same paper, he proves that under a stronger assumption, $\CC^2$ can be uniquely identified:
\begin{theorem*}
Let $X$ be a non-singular complex algebraic surface which is contractible AND \emph{simply connected at infinity}. Then $X$ is isomorphic to the affine two-space as an algebraic variety.
\end{theorem*}

Following this attractive construction, several analogs of Ramanujam surfaces have been studied and since then established. In 1988, Gurjar-Miyanishi \cite{Gurjar1988affine} classified all the acyclic affine surfaces with log Kodaira dimension at least 1. Following this, Tom Dieck-Petrie studied a family of these smooth affine surfaces, which is now eponymously called the \emph{tom Dieck-Petrie surface} defined as follows:

\begin{example}\label{eg:tomDieck-petri}
For integers $m > l \ge 2$ with $m,l$ coprime, the \emph{tom Dieck-Petrie surface} $\nu_{m,l}$ is a complex surface defined by the following equation:
\begin{equation*}
    \nu_{m,l} := \biggl\{ \frac{(xz+1)^m - (yz+1)^l -z}{z} =0 \biggr\} \subset \AA^3_{\CC}
\end{equation*}
\end{example}

Let us now spell out some of the interesting properties of $\nu_{m,l}$. The smooth affine surface $\nu_{m,l}$ is topologically contractible \cite[Theorem A]{Tom1990contractible} but is not isomorphic to $\CC^2$. This is because the former has log Kodaira $\bar{\kappa}(\nu_{m,l})= 1$, but the latter has $-\infty$. As a consequence, let us also note that $\nu_{k,l}$ cannot be $\AA^1$-chain connected (\cref{defn:A1-chainconn}). Indeed, due to the classification theory of \cite{asokmorel2011}, any $\AA^1$-chain connected variety is necessarily $\AA^1$-uniruled\footnote{this is also called \emph{log-uniruled} in the literature} (cf. \cref{defn:A1-uniruled}) and hence has log Kodaira dimension $-\infty$.
\medskip

Another elegant tool to prove the topological contractibility is via the machinery of \emph{affine modifications}. In their classical paper, Kaliman-Zaidenberg establish the conditions under which the topology is preserved under these modifications. Due to \cite[Example 3.1]{kalimanZaid1999affine}, we have that $\nu_{m,l}$ is obtained as the affine modification of $\AA^2$ along the divisor $\mathcal{C} := \{x^l - y^m =0\}\subset \AA^2$ with center $(1,1) \subset D$; here $\mathcal{C}$ is the cuspidal curve which is topologically contractible (even $\AA^1$-contractible!). Thus, due to \cite[Corollary 3.1]{kalimanZaid1999affine}, one sees that $\nu_{m,l}$ is topologically contractible. Essentially, exploiting the affine modification techniques and higher Chow groups calculation, \cite[Theorem 3.2]{DPO2019} also shows that $\nu_{k,l}$ is in fact \emph{stably $\AA^1$-contractible} - a phenomenon by which an algebraic variety $(X,x)$ becomes $\AA^1$-contractible after finitely many suspensions 
    $$(X,x) \wedge (\PP^1, \infty)\wedge \dots \wedge(\PP^1, \infty)\sim_{\AA^1} *.$$ 
Since $\nu_{m,l}$ is not $\AA^1$-chain connected from the above discussion, one could also learn from this that the affine modifications do not preserve $\AA^1$-chain connectedness in general (cf. \cite[Example 2.32]{DPO2019}).
\medskip

The characterization of Ramanujam on the affine plane does not generalize to higher dimensions, in the sense that there are smooth complex affine varieties in all dimensions $n \ge 3$ which are topologically contractible, even simply connected at infinity, but not isomorphic to $\CC^n$. In fact, due to a theorem of Ramanujam-Dimca (\cite[Theorem 3.2]{zaidenberg2000exotic}) we have the following:
\begin{theorem*}
Let $X$ be a contractible smooth affine algebraic variety. If $dim_{\CC}X = n\ge 3$, then $X$ is diffeomorphic to $\RR^{2n}$.
\end{theorem*}

The proof involves the usage of $h$-cobordism and Lefschetz hyperplane section theorems. But vaguely put, in dimensions $d\ge 3$, one can obtain such an exotic affine variety by simply taking the product of one such topologically contractible surface with the desired number of affine spaces $\AA^n$.

\begin{remark}
We shall see that due to the $\AA^1$-homotopy theoretical characterization of the $\AA^2_k$ (\cref{A2isunique}), such exotic surfaces are detected by $\AA^1$-contractibility. In particular, this will show that there are topologically contractible surfaces that are not $\AA^1$-contractible, establishing that $\AA^1$-contractibility is a considerably stronger notion than the topological contractibility (at least over fields admitting an embedding into $\CC$).
\end{remark}

% ----------------------------------------------------
\subsubsection{Exotic real threefolds}\label{intro:exotic-threefolds}
Recall from \cref{thm:char-of-Rn} that simply connectedness is necessary from dimension 3. The following pioneering example is one of its kind. J.H.C. Whitehead in 1934 proposed the following argument \cite{whitehead1934certain} in an attempt to prove the Poincaré conjecture for closed 3-manifolds: start with a homotopy equivalence $f: M \to S^3$. Removing a point, we have an open contractible manifold $M\bs \{\bullet\}$ and a continuous map $\widetilde{f}: M \bs\{\bullet\} \to \RR^3$. Now, he essentially argued that "an open contractible 3-manifold that is homotopic to $\RR^3$ is homeomorphic to $\RR^3$" and so that this homeomorphism $\widetilde{f}$ would extend by continuity across infinity, inducing a homeomorphism $f: M \to S^3$. This proof collapsed soon after his own counter-example \cite{whitehead1935prototype}, where he produced an open contractible 3-manifold that is not homeomorphic to $\RR^3$. This is called the \emph{Whitehead manifold} $\mathcal{W}$ (\cref{pic:Whitehead-continuum}) in his eponym.
\begin{figure}[hbt!]
    \centering
    \includegraphics[width=0.8\textwidth,keepaspectratio]{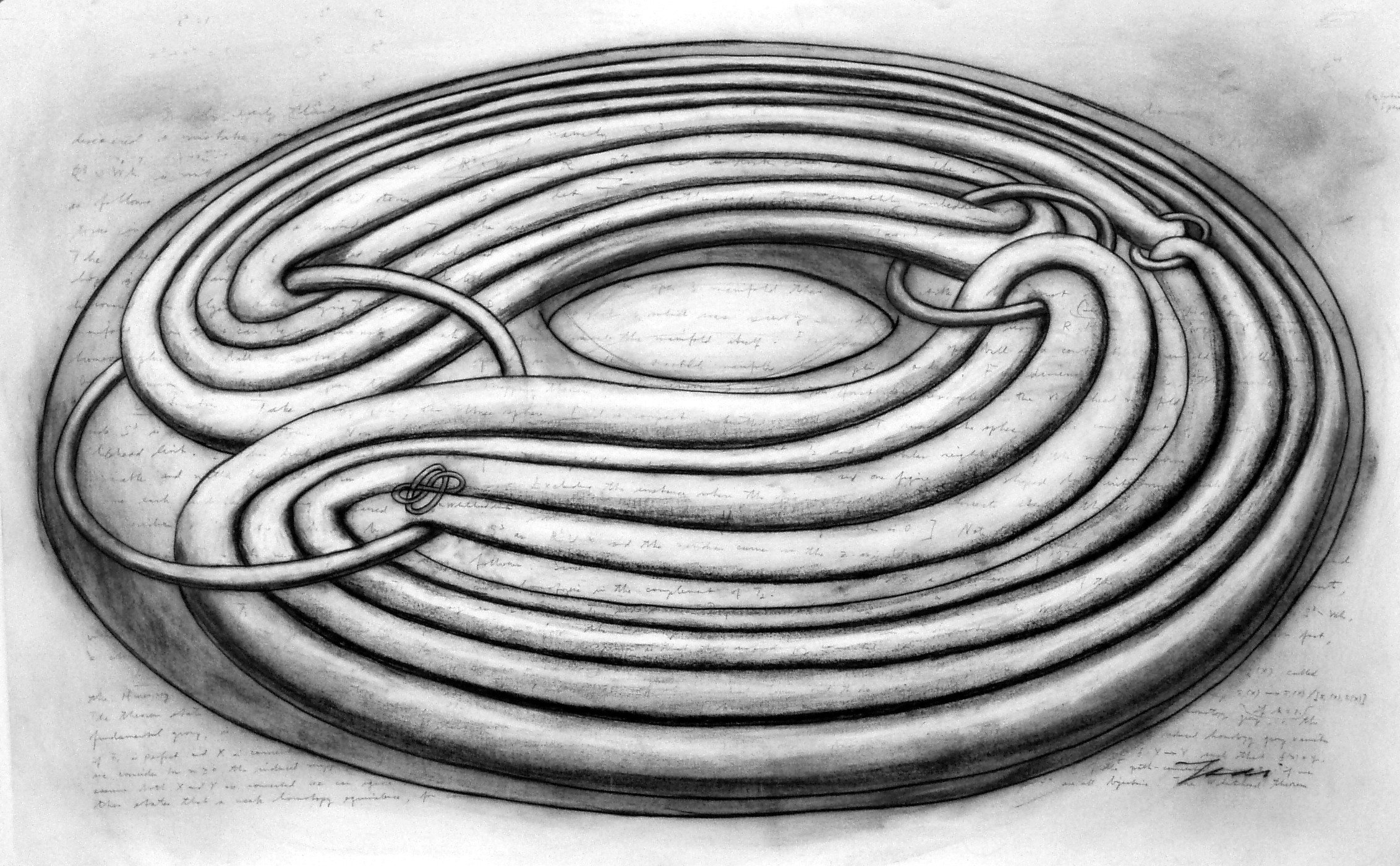}
    \caption{Whitehead Continuum, \cite{Tsai-Whitehead}}
    \label{pic:Whitehead-continuum}
\end{figure}
\vspace{1mm}

Following this exotic threefold, higher-dimensional analogues were established and their associated cancellation questions were studied deeply by several mathematicians exploiting tools available from the solution of the higher-dimensional Poincaré conjecture and the surgery theory (see the amazing cumulative works of \cite{glimm1960two,  mazur1961note, poenaru1960decompositions, mcmillan1962some, stallings1962piecewise, curtis1965infinite}). A crucial property was then observed: the Euclidean space $\RR^3$ stayed simply connected in the complement of any compact subset, but the Whitehead manifold did not, and thus began the utility of topology at infinity!
\medskip

Spaces such as the \emph{Whitehead manifolds} $\mathcal{W}$ constitute the basic motivation to study the obstruction theory related to the characterization of $\AA^n$ and further in constructing motivic invariants at infinity (\cref{app:Mot-top-infty}) to distinguish exotic affine $\AA^1$-contractible varieties. In \cref{chp5}, we will visit yet another \emph{aesthetically notorious} family of exotic affine threefolds $\KK$, which plays the analogous role of the Whitehead manifold in algebraic geometry. 

\subsubsection{\textsc{Moving forward...}}
The \cref{qstn:OPC-Alg-geo} can be considered, in part, as the genesis of this script. In alignment with topology, this desired characterization is known as the \emph{motivic open Poincaré conjecture}. The existence of exotic varieties as noted above restricts one from obtaining a complete solution to \cref{qstn:OPC-Alg-geo} as in \cref{thm:char-of-Rn}. One of the vivid dreams in this direction is to have a motivic homotopy at infinity (\cite[Question 5.4.4]{asok2021A1}, also see \cref{app:Mot-top-infty}). To sum up, our aim in this script can be concisely articulated as follows:

\begin{itemize}
\item Study the \cref{qstn:OPC-Alg-geo} in low dimensions over fields and extend the characterization to a more general base scheme in relative dimensions up to 2. In the sequel, we will also review the situation in dimensions $\ge 4$ and illustrate a key usage of a theorem from affine algebraic geometry in harnessing plenitude of relatively $\AA^1$-contractible schemes  (\cref{sect:higherdimensions}).

\item Study the theory of exotic affine varieties and extend the $\AA^1$-contractibility of Koras-Russell threefolds (\cref{sec:KR-prototypes}) and its generalized version (\cref{sect:generalized-KR}) over perfect fields and Noetherian base schemes. Following this, we study the compact analogs of this problem and prove the existence of exotic motivic spheres (\cref{sect:exotic-motivic-spheres}) in all dimensions $\ge 4$.
\end{itemize}

% \afterpage{\blankpage}

\begin{savequote}
Everything can be taken from a man but one thing: the last of the human freedoms — to choose one’s attitude in any given set of circumstances.
\qauthor{"Man's Search for Meaning" by Viktor Frankl}   
\end{savequote}
\chapter{Motivic Homotopy Theory}\label{chp2}
\markboth{II Motivic Homotopy Theory}{}

{\small
{\fontfamily{bch}\selectfont
\subsubsection{\textsc{Chapter Summary}}
In this chapter, we establish the foundations of (unstable) motivic homotopy theory. In \cref{subsec:unstable-construction}, we construct the unstable $\AA^1$-homotopy category $\mcal{H}(S)$ after providing the necessary motivation and the framework of simplicial Nisnevich sheaves to build such a category. We describe the crucial motivic localization functor $\L_{mot}$ in \cref{subsec:L-mot} and explain its ability to retain both the other localizations, namely $\L_{\AA^1}$ and $\L_{Nis}$. The $\AA^1$-model structure is dealt with in \cref{subsec:A1-model-structure}. In the sequel, we define the central jargon of our script - the $\AA^1$-contractibility in \cref{intro:A1-contr} and various important revolving notions such as $\AA^1$-connectedness (\cref{subsec:A1-connectedness}), $\AA^1$-rigidity (\cref{subsec:A1-rigid-schemes}), and $\AA^1$-chain connectedness (\cref{subsec:A1-chain-connectedness}). This brings us to the fundamental notion of functoriality (\cref{sect:functoriality}), which plays a key role in motivic homotopy theory. We then apply the functoriality in the case when the space is $\AA^1$-contractible and obtain preliminary results for future applications. \cref{intro:rel-A1-contr} and subsequent parts will play a pivotal role in \cref{chp4}, as they provide a framework for working relative to a base scheme. We finally conclude with the necessary background on Milnor-Witt $K$-theory in \cref{sect:Milnor-Witt}, which finds its application in \cref{chp5}.
}}

% --------------------------------------------------------------------
\section{The \texorpdfstring{$\mathbb{A}^{1}$}{A1}-Homotopy Category} 
\label{sect:unstable-A1-category}
In this section, we begin with the fundamentals of the construction of the unstable $\AA^1$-homotopy theory (a.k.a. Motivic Homotopy Theory) associated to the category of smooth, separated schemes of finite type over a base scheme $S$. To describe the homotopy theory of motivic spaces, one takes two approaches: the traditional approach via the language of model categories (cf. \cref{app:model-cats}) as established by Quillen \cite{quillen1967homotopical} or the relatively modern language of $(\infty,1)$-categories (e.g., \cite{lurie2009HTT}). As far as our script is concerned, we will freely use both of these approaches while predominantly using the former approach. The latter approach using the $\infty$-machinery will be exploited at instances where it makes the context more tractable. Nevertheless, the reader should worry less about this aspect as all our formulations can very well be illustrated via both languages using modern bridges that exist between these two worlds (one such bridge is here: \cite[\S 2]{robalo2015Motives}). For an exposition-type background, see \cite{brazelton2018inftycats, deglise2021intoductory}.
\medskip

\textbf{\textsc{Running Convention}} Throughout this chapter, unless mentioned otherwise: $S$ will be a locally Noetherian base scheme of finite Krull dimensions, $k$ will denote a field with its units $k^\times$, $\Sch_S$ = separated schemes of finite type  $S$, $\Sm_S$ = smooth separated schemes of finite type over $S$.

\subsubsection{Why build on simplicial presheaves and not on schemes themselves?}
The $\AA^1$-homotopy category on $\Sm_S$ is built on the category of simplicial presheaves $\sPsh(\Sm_S)$. It is natural to ponder why not define it directly on $\Sm_S$. The reason is that though the category $\Sm_S$ is complete (i.e., it has all limits) but it is very far from being co-complete (i.e., it lacks colimits), and it is common knowledge that colimits are indispensable to build a homotopy theory on any category. A convenient solution is to "formally add" the desired colimits to $\Sm_S$. In the categorical language, this is equivalent to taking the \emph{presheaf category on $\Sm_S$}. This is achieved by the Yoneda embedding defined via the following fully faithful functor,
\begin{align*}
    & Y: \Sm_S \hookrightarrow \Fun((\Sm_S)^{op}, \mathcal{S}et)=: \Psh(\Sm_S) \\
    & \hspace{11mm} X \mapsto \bigg(U \mapsto R_U  :=  \Hom_{\Sm_S}(U, X) \bigg)
\end{align*}
Here, $R_U$ is the representable presheaf functor\footnote{also called the \emph{contravariant $\Hom$} functor} on $U$. The advantage is that we can now exploit the model categorical framework already built for the category of presheaves (explore more in \cref{app:model-on-sSets}). Any presheaf can be viewed as a simplicial presheaf when equipped with the "discrete" simplicial structure. Summing up, we have a series of functors
        $$\Sm_S\xhookrightarrow{Y} \Psh(\Sm_S) \xrightarrow{D} \sPsh(\Sm_S)$$
where $D$ represents the 'discrete' simplicial structure on the category of presheaves. However, this approach has not yet another issue. Let $U, V\subset X$ be an open (Zariski) subscheme of a scheme $X$ such that $X= U\cup V$. Then consider the following squares
\[ \begin{tikzcd}
 U\cap V \arrow[r] \arrow[d] &  U \arrow[d] \\
 V \arrow[r] & X
 \end{tikzcd} 
\hspace{8mm}
\overset{Y}{\longrightarrow}
\hspace{10mm}
\begin{tikzcd}
 R_{U\cap V} \arrow[r] \arrow[d] &  R_U \arrow[d] \\
 R_V \arrow[r] & R_X
\end{tikzcd} 
 \]                    
The first square (on the left) is a pushout in $\Sm_S$, where $U\cap V$ is the fiber product of schemes. However, the corresponding square (on the right) obtained via taking the representable presheaves does not stay as a pushout in $\Psh(\Sm_S)$ unless $X=U$ or $X=V$. In other words, by a priori defining spaces as presheaves, the union as presheaves $U \cup_{\Psh} V$ is not the same as the categorical union $X$. Recall from classical topology that squares that satisfy this analogous property are called \emph{Mayer-Vietoris squares}. Stated in analogy, this issue can be seen as a \emph{failure of Mayer-Vietoris property} for schemes. A suitable way to resolve this is to modify the squares themselves as follows.
\begin{defn}\label{EDS}
An \emph{Elementary Distinguished Square} (EDS, for short) in $\Sm_S$ is a square of the form 
\[\begin{tikzcd}
 p^{-1}(U) \arrow[r] \arrow[d] &  V \arrow[d,"p"] \\
 U \arrow[r,"j"] & X
\end{tikzcd}\]
such that $p$ is an \'etale morphism, $j$ is an Zariski open immersion such that the pullback 
        $$p^{-1}(X\bs U) \xrightarrow{\cong} X\bs U$$ 
is an isomorphism (here $X \bs U$ is equipped with the reduced subscheme structure with support in the closed subset $X \bs U$).
\end{defn}

\begin{example}
An important class of EDS is provided by Zariski open coverings of the form $X= U\cup V$. Here, one takes $ p=j_V$ as an open embedding and the condition that $p^{-1}(X\bs U)\to X\to U$ is an isomorphism is equivalent to the condition that $U\cup V= X$. 
\end{example}

One can easily see that an EDS is a pushforward in $\Sm_S$. In other words, $X$ is recovered as the colimit of the diagram $U\leftarrow p^{-1}(U) \to V$. This guarantees us the Mayer-Vietoris-type property for these kinds of squares. Now, recall that an open covering for schemes generalizes the notion of open coverings via families of maps, which is formulated via the notion of Grothendieck topologies (see \cref{App;Groth-Top}). The strategy to coin a suitable category of spaces that also includes the analogous Mayer-Vietoris style property is to find a Grothendieck topology that is generated by these EDS as its "open coverings". To our fortune, there does exist such a topology and even more, it is super-convenient in motivic homotopy theory (cf. \cref{App:Nisnevich-useful})

\subsubsection{The Nisnevich topology}\label{sect:Nisnevich}
The \emph{Nisnevich topology}, eponym to Yeo A. Nisnevich, appeared in \cite{nisnevich1989top} as the \emph{completely decomposed topology} or the cd-topology, whose initial motivation was to construct certain descent for spectral sequences converging to various algebraic $K$-theory groups. For our purposes, we will only need the definitions and some properties of this topology. The following is one of the several equivalent definitions (see \cref{App:Nisnevich-top}) of Nisnevich topology due to \cite[\S 3.1]{MV99}.

\begin{defn}
A finite family of \'etale morphisms $\{p_i: U_i \to  X\}_{i\in I}$ is a \emph{Nisnevich cover} if and only if for each point $x\in X$ and some index $I$, there is an $i\in I$ and a point $y \in U_i$ with $p_i (y_i)=x_i$ such that the induced map on the corresponding residue fields $\kappa(x)\to \kappa(y)$ is an isomorphism. The \emph{Nisnevich topology} on $\Sm_S$ is the Grothendieck topology generated by these Nisnevich covers.
\end{defn}

Any Zariski cover is a Nisnevich cover. On the other hand, an \'etale cover is Nisnevich only if it is surjective on $k$-points for all fields $k$, as illustrated by the following example (\cite[Example 3.41]{antieau2017primer}).

\begin{example}\label{eg:Nisnevich-cover}
Let $k$ be a field with characteristic unequal to 2 and let $a\in k^{\times}$. Consider a covering of $\AA^1$ by $j: \AA^1\bs \{a\}\to \AA^1$ and an \'etale map $p: \AA^1\bs \{0\}\to \AA^1$ defined by $x\mapsto x^2$. Then this \'etale cover is Nisnevich if and only if $a=b^2$, for some $b\in k$.
\end{example}

An EDS as in \cref{EDS} is closely related to Nisnevich covers as follows (\cite[Example 3.45]{antieau2017primer}).
\begin{example}
The Nisnevich cover from the \cref{eg:Nisnevich-cover} is not an EDS. However, upon removing one of the square roots of $a$ from $\AA^1\bs \{0\}$
\[\begin{tikzcd}
  &  \AA^1\bs \{0,b\} \arrow[d,"p"] \\
 \AA^1\bs\{a\} \arrow[r,"j"] & \AA^1
\end{tikzcd}\]
The square becomes an EDS. For this reason, one sometimes calls an EDS the \emph{Nisnevich distinguished square}.
\end{example}

The following furnishes the definition of a presheaf to be a sheaf in the Nisnevich topology (\cite[Definition 2.2]{voevodsky1998A1}).
\begin{defn}\label{defn:Nis-sheaf}
A contravariant functor $\mcal{F}:\Sm_S^{op} \to \Set$ is called a \emph{Sheaf in Nisnevich topology} if the following two conditions hold:
\begin{enumerate}
\item $\mcal{F}(\emptyset) = *$ (takes the initial to the final object),
\item for any EDS as in \cref{EDS}, the square of $\Set$
      \[\begin{tikzcd}
         \mcal{F}(X) \arrow[r] \arrow[d] &  \mcal{F}(V) \arrow[d,"p"] \\
         \mcal{F}(U) \arrow[r,"j"] & \mcal{F}(p^{-1}(U))
        \end{tikzcd}\]
is cartesian, that is, $\mcal{F}(X) \simeq \mcal{F}(U)\times_{\mcal{F}(p^{-1}(U)} \mcal{F}(V)$.
\end{enumerate}
\end{defn}
Presheaves satisfying the conditions of \cref{defn:Nis-sheaf} are said to satisfy the \emph{Brown-Gersten property} or the \emph{BG property} (\cite[P. 99]{MV99}), which can be seen as the analogous excision property for the Nisnevich topology. The terminology is chosen so that the authors in \cite{brown2006algebraic} studied this in the Zariski site. The following affirms that Nisnevich topology is indeed the right choice if we expect to have a Mayer-Vietoris type result in the $\AA^1$-homotopy category.
\begin{lemma}
Any EDS is a co-cartesian square in the category $\Shv(\Sm_S)$. In particular, the canonical morphism  
        $$V/(U \times_X V) \to X/U$$ 
is an isomorphism of Nisnevich sheaves.
\end{lemma}
\begin{proof}
See \cite[Lemma 1.6]{MV99}.    
\end{proof}

\subsection{The Unstable Construction}\label{subsec:unstable-construction}
\textsc{SLOGAN:} The $\AA^1$-homotopy theory is the homotopy theory of simplicial presheaves on $\Sm_S$ satisfying Nisnevich descent and $\AA^1$-invariance.
\medskip

In this section, we will introduce the fundamental tools that will aid us in constructing the category of motivic spaces $\Spc_S$ and eventually the unstable motivic homotopy category $\mcal{H}(S)$. For the sake of exposition, we will use the language of model categories. Among all the popular model structures on motivic spaces, the \emph{$\AA^1$-local projective model structure} is conveniently suited in terms of functoriality, and so we will make use of this model structure. For various interesting accounts on the available model structures on simplicial presheaves, the reader can refer to \cite{blander2001local, dugger2001universal, dundas2003motivicfunctors, dugger2004hypercovers, isaksen2005flasque}.

\subsubsection{Bousfield Localizations}
The advantage of using the local projective model structure is that every scheme is both a local projective fibrant and cofibrant as a presheaf and sheaf. However, we have that the category $\Sm_S$ embeds fully faithfully in the local projective homotopy category of simplicial presheaves; thus, to have a distinguished identification of smooth schemes in $\sPsh(\Sm_S)$, we would need to further enlarge the class of (local projective) weak equivalences. This is done via the process called \emph{Bousfield localizations}, which can be thought of as an analogous localization of rings from commutative algebra (see \cref{app:Bousfield} for a background). To build the category of spaces, we need to undergo two such localizations: the Nisnevich localization and the $\AA^1$-localization that we describe below.

\subsubsection{Nisnevich descent}
Consider the category of simplicial presheaves $\sPsh(\Sm_S)$ equipped with the Nisnevich topology. Let $U_{\bullet}$ be an object in $\sPsh(\Sm_S)$ such that $U_n$ is a presheaf of sets in $\sPsh(\Sm_S)$, for each $n$. Suppose that the map $U \to V$ is a Nisnevich cover, then one defines the \v{C}ech complex $\check{U}\to V$ associated to $U\to V$ by defining $\check{U}_n:= U\times_V\dots \times_V U$\footnote{the self-product of $U$ taken $(n+1)$-times over $V$}. Then the Bousfield localization of the $\sPsh(\Sm_S)$ with respect to this \v{C}ech covering exists and which is called the \emph{Nisnevich localization}.

\begin{theorem}(\cite[Theorem 3.33]{antieau2017primer})
The Bousfield localization of $\sPsh(\Sm_S)$ with respect to the class of \v{C}ech covers $\check{U}_{\bullet}\to V$ exists. We denote this localization functor by $\L_{Nis} \sPsh(\Sm_S)$.
\end{theorem}

We say that a simplicial presheaf $X\in \sPsh(\Sm_S)$ satisfies \emph{Nisnevich descent} if for every \v{C}ech covers $\check{U}_{\bullet}\to X$, we have that the                     $${\rm{hocolim}_n\ U_n \xrightarrow{\simeq} X} $$ 
is a weak equivalence of simplicial presheaves. More generally, one considers descent with respect to more general covers than the \v{C}ech covers in any Grothendieck topology $\tau$, which are called hypercovers. Recall that a covering $U_{\bullet} \to X$ is called a \emph{hypercover} if each $U_n$ is a coproduct of representables, the induced map $U_0\to X_0$ is a $\tau$-cover and each $U^{\Delta^n}\to U^{\partial\Delta^n}$ is a $\tau$-cover in degree 0.

% -----------------------------------------------
\subsubsection{$\AA^1$-Localization: inverts $\AA^1$}
At this point, it is not clear as how one should define contractible schemes in this category. To construct a homotopy theory on $\sPsh(\Sm_S)$ that captures the underlying structure of $\Sm_S$, we need to enlarge weak equivalences that are more intrinsic and that are parametrized by the objects of $\Sm_S$ itself. To remedy this, we enlarge the category $\sPsh(\Sm_S)$ with respect to another natural class of morphisms. Recall that in topology, homotopies are parametrized with respect to an interval $[0,1]$. As $[0,1]$ is not an algebraic variety, we replace it by the affine line $\AA^1$. Let $I$ denote the class of maps $\{\AA^1\times U\to U \mid U\in \Sm_S\}$ in $\sPsh(\Sm_S)$. Alternatively, one can also choose\footnote{this is possible since $\Sm_S$ is essentially small} a subset $J\subseteq I$ such that $X\in \Sm_S$ varies over a representative of each isomorphism class of $\Sm_S$ and we will need to invert these classes of morphism as well. This process is called the \emph{$\AA^1$-localization}. In other words, all the projection maps 
            $$\{\pr_2: \AA^1\times U\to U\mid U\in \Sm_S\}$$
would have now been inverted in the to-be-constructed $\AA^1$-homotopy category of schemes.

\begin{defn}\label{defn:A1-invar-sheaf}
Let $\mcal{F} \in \sPsh(\Sm_S)$. Then we say that $\mcal{F}$ is said to satisfy \emph{$\AA^1$-invariance} if the induced morphism 
    $$\pr^*: \mcal{F}(X) \to \mcal{F}(X\times\AA^1)$$ 
of the natural projection $X \times \AA^1\to X$ is a bijection. Such presheaves are said to be $\AA^1$-invariant and we denote this subcategory of $\AA^1$-invariant presheaves by $\Psh_{\AA^1}(\Sm_S) \subset \Psh(\Sm_S)$.
\end{defn}

\begin{example}
Interestingly, note that not all (even representables!) presheaves are $\AA^1$-invariant. The affine scheme $\GG_m$ represents the sheaf of units and is $\AA^1$-invariant (cf. \cref{Gm-A1-rigid}). On the other hand, the affine line $\AA^1$ represents the sheaf of global sections, but it is not $\AA^1$-invariant (cf. \cref{eg:nonA1rigid-A1&P1}).
\end{example}

\begin{remark}\label{rem:A1-invar:onto}
Note that since $\pr_2: \AA^1_k \times U \to U$ has a section, it is always split injective. Hence, $\mathcal{F}$ is $\AA^1$-invariant provided $\pr^*$ is surjective. Moreover, observe that $\mathcal{F}$ is $\AA^1$-invariant if and only if the induced sections $i_0^* = i_1^*: \mathcal{F}(\AA^1_k \times U) \to \mcal{F}(U)$ agree, for all $U\in \Sm_k$ (\cite[Lemma 2.16]{mazza2006lecture}). 
\end{remark}

\begin{prop}(\cite[Proposition 3.55]{antieau2017primer})
The Bousfield localization of $\sPsh(\Sm_S)$ with respect to the class of morphisms $J$ exists. We denote this $\AA^1$-localized category by $\L_{\AA^1} \sPsh(\Sm_S)$.
\end{prop}
The $\AA^1$-localization functor $\L_{\AA^1}: \Psh(\Sm_S)\to \Psh_{\AA^1}(\Sm_S)$ is the left adjoint to the inclusion functor $\iota: \Psh_{\AA^1}(\Sm_S)\hookrightarrow \Psh(\Sm_S)$.
\medskip

With these two localizations set in place, we can now finally define the motivic homotopy category.

\begin{defn}\label{defn:MHT}
For a base scheme $S$, the category of \emph{motivic $S$-spaces}, denoted $\Spc_S$, is the category of simplicial presheaves over $\Sm_S$. The \emph{$\AA^1$-homotopy category}, denoted $\HH(S)$, is the homotopy category of $\Spc_S$ obtained by imposing the Nisnevich descent and the $\AA^1$-localization
\begin{align*}
    \HH(S) := \Shv_{Nis}(\Sm_S) \cap \Psh_{\AA^1}(\Sm_S)\subset \sPsh(\Sm_S)=: \Spc_S
\end{align*}
\end{defn}
There are also other variants that one could be interested in. For instance, by only imposing the Nisnevich descent, one forms the Nisnevich local homotopy category, denoted by $\HH(S)_{Nis}$, and similarly, by only imposing the $\AA^1$-invariance, one has the $\AA^1$-local homotopy theory $\HH(S)_{\AA^1}$. In our script, we will only be concerned about the homotopy category with both the localizations and hence only about $\HH(S)$.

\subsubsection{Pointed Motivic Spaces}\label{pointed-spaces}
One similarly has a parallel theory for the pointed $\AA^1$-homotopy category by considering the pointed versions of motivic spaces. For instance,
\begin{itemize}
\item The affine line $\AA^1$ is pointed by 1,
\item The multiplicative scheme $\GG_m$ is pointed by 1,
\item The projective line is pointed by $\infty$ (in fact, here the base point does not matter, as $(\PP^1,\infty) \simeq (\PP^1,x)$, for any $x \in \PP^1$),
\item Any general motivic space $X$ is pointed by the zero object $*$ and so, we have $X_+ := X \sqcup {*}$.
\end{itemize}
The corresponding pointed motivic spaces is denoted by $\Spc_{S,\bullet}$ and the $\AA^1$-homotopy category is denoted by $\mcal{H}(S)_{\bullet}$. There is an adjunction
    $$(-)_+: \mcal{H}(S) \leftrightharpoons \mcal{H}(S)_{\bullet}: \mcal{U}$$
where $\mcal{U}$ is the right adjoint that forgets the base point. One can invoke several operations in $\mcal{H}(S)_{\bullet}$ as one does in topology. For instance, if $Y\to X$ is a map of pointed motivic spaces, then the cofibre in $\mcal{H}(S)_{\bullet}$ can be computed as 
\[\begin{tikzcd}
	X & Y  \\
	* & Y/X
	\arrow[from=1-1, to=1-2]
	\arrow[from=1-1, to=2-1]
	\arrow[from=1-2, to=2-2]
	\arrow[from=2-1, to=2-2]
\end{tikzcd}\]
with the cofiber $Y/X$ pointed by the image of $Y$. Similarly, using the zero object $*$ in $\Spc_S$, we have suspensions in $\mcal{H}(S)_{\bullet}$
\[\begin{tikzcd}
	X & *  \\
	* & \Sigma X
	\arrow[from=1-1, to=1-2]
	\arrow[from=1-1, to=2-1]
	\arrow[from=1-2, to=2-2]
	\arrow[from=2-1, to=2-2]
\end{tikzcd}\]
for any motivic space $X$. Moreover, if $X \vee Y$ denotes the coproduct of motivic spaces, then we have the smash product in $\mcal{H}(S)_{\bullet}$ given by 
        $$X\wedge Y:= \frac{X\times Y}{X\vee Y}. $$
The stable $\AA^1$-homotopy theory ${\SH(S)}$ is then obtained by formally inverting the motivic sphere $\PP^1$ (see \cref{sec:motivic-spheres})
\begin{align*}
    \SH(S) := \mcal{H}(S)_{\bullet}[(\PP^1)^{-1}]\\
\end{align*}
which gives the following functors:
\begin{align*}
\mcal{H}(S)_{\bullet} \xrightleftharpoons[\Omega^{\infty}]{\Sigma^{\infty}} \SH(S)
\end{align*}
For more background on the stable theory, one can refer to \cite{adams1974stable, voevodsky1998A1, jardine2000motivicspectra, hovey2001spectra,dundas2003motivicfunctors, Nordfjordeid2007motivic}.

\subsubsection{$\L_{Nis} \leftrightharpoons \L_{\AA^1}$ destroy each other}
Although the definition of $\HH(S)$ is formally straightforward, there is a practical difficulty when one tries to do some computations from scratch. The localization functors $\L_{Nis}$ might destroy the $\AA^1$-invariance and $\L_{\AA^1}$ need not preserve the (Nisnevich) sheaf nature. Consider the following instance (\cite[\S 3, Example 2.7]{MV99}):

\begin{example}
Let $U_0 = \AA^1\bs \{0\}$ and $U_1= \AA^1\bs \{1\}$, and let $U_{01} = U_0\cap U_1$. Since both $U_0$ and $U_1$ are $\AA^1$-invariant (cf. \cref{Gm-A1-rigid}), it follows that $U_{01}$ is also $\AA^1$-invariant. Now, for a choice of closed embedding of $j: U_{01}\hookrightarrow \AA^n$ (for some $n$), define the coproduct
    $$F := (U_0\times \AA^n) \bigsqcup_{U_{01}}(U_1\times \AA^n)$$
which is a (non-smooth) scheme $F \in \Sch_S$. Here, the morphism $U_{01}\to U_i\times \AA^n$ is the product of $j$ with an open embedding $U_{01}\to U_i$, for $i=0,1$. Now, let $X$ be any connected scheme. Consider the sheaf
$$R_F(X):= \Hom(X,F) = \Hom_{Sch_S}(X,U_0\times \AA^n) \bigsqcup_{\Hom_{\Sm_S}(X, U_{01})} \Hom_{\Sm_S}(X, U_1\times \AA^n).$$
Since $\L_{\AA^1}$ preserves pushouts, we have that 
$$\L_{\AA^1} R_F \simeq  \L_{\AA^1}(U_0\times \AA^n) \sqcup_{\L_{\AA^1}(U_{01})}\L_{\AA^1}(U_1\times \AA^n).$$
The $\AA^1$-localization contracts $\AA^n$s and moreover due to the $\AA^1$-invariance of $U_0, U_1,$ and $U_{01}$, we get 
    $$\L_{\AA^1}(U_i) = U_i\ \ ;   \L_{\AA^1}(U_{01}) = U_{01}$$
for $i= 0,1$ and we finally obtain its (associated) $\AA^1$-invariant sheaf as
        $$\L_{\AA^1} R_F \simeq U_0 \sqcup_{U_{01}} U_1.$$
We claim that $\L_{\AA^1} R_F$ cannot be a sheaf. Suppose that it were a sheaf; then it should agree with its sheafification. But as sheafification is a left adjoint, it preserves pushouts, and hence we have 
    $$\L_{\AA^1}R_F \simeq  U_0\sqcup_{\AA^1} U_1 \simeq  \AA^1$$
which is a contradiction to the fact that $\L_{\AA^1} R_F$ is $\AA^1$-invariant, because $\AA^1$ is not (cf. \cref{eg:nonA1rigid-A1&P1})! Hence, we conclude that $\L_{\AA^1}R_F$ cannot be a sheaf.
\end{example}

\subsubsection{The $\AA^1$-Singular Functor}\label{A1-Sing-functor}
In terms of computation, a practical replacement for $\L_{\AA^1}$ functor is the motivic singular chain functor $\Sing_*^{\AA^1}$. In analogy with topology (cf. \cref{eg:n-simplex}), it is obtained from the following \emph{algebraic $n$-simplex}:
    $$\Delta^n_{\AA^1}:= \Spec \ZZ[t_0,\dots,t_n]/ \biggr(\sum_{i=0}^n t_i-1 \biggl).$$
The space $\Delta^n_{\AA^1}$ exists as a smooth affine $k$-scheme that is (non-canonically) isomorphic to $\AA^n_k$. One of the main utilities of this space is to construct the (Suslin-Voevodsky) $\AA^1$-singular chain functor
    $$\Sing_*^{\AA^1}(X) := \bigg( U \mapsto \Hom_{\Sm_k}(\Delta_{\AA^1}^{*} \times U, X)\bigg).$$
One shows that this gives a cosimplicial scheme $\Delta^{\bullet}\in \Fun(\Delta, \Sm_S)$. The following are some of the evident advantages of this functor:
\begin{lemma}
For any $F\in \Psh(\Sm_S)$, the presheaf $\Sing^{\AA^1}_*(F)$ is $\AA^1$-invariant.
\end{lemma}
\begin{proof}
This is \cite[\S 3, Corollary 3.5]{MV99}.
\end{proof}

\begin{prop}
If $F\in \Psh_{\AA^1}(\Sm_S)$, then $F\to \Sing_*^{\AA^1}(F)$ is an equivalence of presheaves.
\end{prop}
\begin{proof}
See \cite[P. 87]{MV99}.
\end{proof}

\begin{prop}\label{Sing=L_A1:equiv}
The canonical morphism $\L_{\AA^1} \to \Sing_*^{\AA^1}$ is an equivalence of functors.
\end{prop}
\begin{proof}
This is \cite[Corollary 3.8]{MV99}.
\end{proof}

\begin{corollary}\label{L_A1-finite-pdts}
The functor $\L_{\AA^1}$ preserves finite products.    
\end{corollary}
\begin{proof}
This follows from the identification of $\L_{\AA^1} \simeq \Sing_*^{\AA^1}$ in \cref{Sing=L_A1:equiv}. Note that since $\Delta^{op}$ is sifted (for e.g., see \cite[Remark 5.4.4.8]{lurie2009HTT}), colimits indexed over $\Delta^{op}$ commutes with products. As a result, $\Sing_*^{\AA^1}$ preserves finite products and so does $\L_{\AA^1}$.
\end{proof}

The following is an illustration of \cref{L_A1-finite-pdts} in action.
\begin{example}
For all $n$ and for all $X\in \Sm_S$, the canonical projection $X\times \AA^n\to X$ is an isomorphism in $\HH(S)$. This is clear for $n=1$. For $n\ge 2$, this is a straightforward application of \cref{L_A1-finite-pdts}.
\end{example}

\subsection{The Motivic Localization Functor}\label{subsec:L-mot}
Recall that $\L_{Nis}$ and $\L_{\AA^1}$ might cancel out each others' power. A wise way to retain both of its powers while preserving the categorical constraints is to iterate them infinitely many times. The resulting functor is called the \emph{Motivic Localization Functor,} denoted by $\L_{mot}$. In categorical terms, $\L_{mot}$ is obtained as the left adjoint of the inclusion functor $\iota: \HH(S)  \subset \Spc_S$. It is explicitly defined as follows:
$$\L_{mot}:= \textrm{colim} (\L_{Nis}\to \L_{\AA^1}\L_{Nis} \to \L_{\AA^1}\L_{Nis}\L_{\AA^1}\to \dots..) =  (\L_{\AA^1}\circ \L_{Nis})^{\circ \mathbb{N}}$$
and so described in this context, we have 
        $$\HH(S) := \L_{mot}(\Spc_S).$$

\subsubsection{How does $\L_{mot}$ preserve both localizations?}
For any simplicial presheaf $\mcal{F}$, the $\L_{mot}(\mcal{F})$ is a sheaf that satisfies both Nisnevich descent and $\AA^1$-invariance. Indeed, note that the following presentations of $\L_{mot}$ are equivalent as a consequence of the cofinality (cf. \cref{cofinal}):
\begin{align*}
    \textrm{colim} (\L_{Nis}\circ \L_{\AA^1}\to (\L_{Nis} \circ \L_{\AA^1})^2 \to \dots.. )\\
    \textrm{colim} (\L_{\AA^1}\circ \L_{Nis}\to (\L_{\AA^1} \circ \L_{Nis})^2 \to \dots.. )
\end{align*}
The former is a colimit computed in $\Shv_{Nis}(\Sm_S)$, which is closed under filtered colimits and hence the resulting object is again a (Nisnevich) sheaf. The latter is a colimit computed in $\Psh_{\AA^1}(\Sm_S)$, which is also closed under filtered colimits, hence it preserves $\AA^1$-invariance. The property that both of these categories $\Shv_{Nis}(\Sm_S)$ and $\Psh_{\AA^1}(\Sm_S)$ are closed under filtered colimits is because they are \emph{accessible} subcategories of $\Spc_S$ (\cite[Example 5.2.4.7]{lurie2009HTT}).

\begin{prop}\label{Prop:L_mot-preserve}
The following are some of the essential preliminary properties of $\L_{mot}$:
\begin{enumerate}
\item The functor $\L_{mot}$ preserves colimits.
\item The functor $\L_{mot}$ preserves finite products.
\item The functor $\L_{mot}$ is locally cartesian.
\end{enumerate}
\end{prop}
\begin{proof}
(1) follows from the fact that $\L_{mot}$ is a left adjoint. (2) follows from the observation that both $\L_{Nis}$ and $\Sing^{\AA^1}$ preserve finite products. The $\Sing^{\AA^1}$ is given as the sifted colimit of right adjoint functors. Now, from the presentation of $\L_{mot}$, it is constructed as the filtered colimits, which are known to preserve finite products. (3) follows from the facts that both $\L_{\AA^1}$ and $\L_{Nis}$ is locally cartesian and $\L_{Nis}$ is left exact (\cite[Corollaries 3.5 \& 3.9]{hoyois2017six}).
\end{proof}

\begin{remark}
It in important to note that in some literature such as \cite{asok2021A1, bachmann2021norms, rondigs2025grothendieck}, it is common to call the category of simplicial presheaves as motivic spaces and the localized category $\L_{mot}(\sPsh(\Sm_S)) =: \mcal{H}(S)$ as the $\AA^1$-homotopy category. On the other hand, literature such as \cite{antieau2017primer} terms the $\AA^1$-localized category as the category of spaces, which, in comparison, is the same as the $\AA^1$-homotopy category $\HH(S)$. In this script, we have adopted the former approach. 
\end{remark}

\subsubsection{The bird-eye view}\label{the-bird-eye}
All this information can be summarized as follows:

\[\hspace{-7mm}
\begin{tikzcd}
	&&& {\HH(S)_{Nis}} \\
	{} & {\mathrm{\Psh(\Sm_S)}} & {\mathrm{\Spc_S}} &&& {\mathrm{\HH(S)}} & {\mathrm{\HH(S)}_{\bullet}} \\
	& {\mathrm{\Sm_S}} && {\HH(S)_{\AA^1}} &&& {\mathrm{SH(S)}} \\
	{} && {}
	\arrow["{\L_{mot}}"{description}, curve={height=-12pt}, squiggly, from=1-4, to=2-6]
	\arrow["{{\mathrm{D}}}", from=2-2, to=2-3]
	\arrow["{\L_{Nis}}", curve={height=-12pt}, from=2-3, to=1-4]
	\arrow["{{\mathrm{\L_{mot}}}}", squiggly, from=2-3, to=2-6]
	\arrow["{\L_{\AA^1}}"', curve={height=12pt}, from=2-3, to=3-4]
	\arrow["{{\mathrm{\sqcup \{*\}}}}", from=2-6, to=2-7]
	\arrow["{\Omega^{\infty}_{\PP^1}}", shift left=4, tail reversed, no head, from=2-7, to=3-7]
	\arrow["{\sum_{\PP^1}^{\infty}}"', shift right=4, from=2-7, to=3-7]
	\arrow["{{\mathrm{Y}}}"', hook', from=3-2, to=2-2]
	\arrow["{\L_{mot}}"{description}, curve={height=12pt}, squiggly, from=3-4, to=2-6]
\end{tikzcd}\]

\subsection{The \texorpdfstring{$\mathbb{A}^{1}$}{A1}-Model Structure} \label{subsec:A1-model-structure}
We will now brief on the choice of a model structure for $\HH(S)$. For a detailed account of the background on the model categories, one can refer to \cref{app:model-cats}. For the scope of this thesis, we will use the $\AA^1$-local projective model structure (for e.g., see \cite{bousfield1972homotopy, blander2001local, antieau2017primer}). We will use the model structures developed for simplicial presheaves to define a model structure on $\HH(S)$ assuming the material in \cref{App:model:simp-pshv}. 

\begin{defn}
The following definitions forms the fundamental toolkit for identifying an $\AA^1$-local model structure on $\Sm_S$:
\begin{enumerate}
\item A motivic space $Z \in \Spc_S$ is called \emph{$\AA^1$-local} if it is local projective fibrant and the induced map
            $$\pr^* :\Hom(U, Z)\to \Hom (U\times \AA^1, Z) $$
is a weak equivalence of simplicial sets with respect to the canonical projection $\pr: U\times \AA^1\to U$. Hence, a Nisnevich sheaf is $\AA^1$-local if and only if it is $\AA^1$-invariant.

\item A morphism of motivic spaces $f: X\to Y$ is called an \emph{$\AA^1$-local weak equivalence} if 
    $$(Qf)^* : \Hom(QY,Z)\to \Hom(QX,Z)$$
is a weak equivalence of simplicial sets, for any $\AA^1$-local space $Z$. Here $Q$ is the local projective cofibrant replacement functor (cf. \cref{cofib-replacement})
\end{enumerate}
\end{defn}

The following result precisely gives us the $\AA^1$-model structure on $\mcal{H}(S)$. The version stated here is due to Blander (\cite{blander2001local}) as expressed in (\cite[Theorem 6.2.7]{severitt2006motivic}):

\begin{theorem}
The motivic homotopy category $\HH(S)$ is equipped with respect to the following classes of maps
\begin{itemize}
\item $\mcal{W}$eq as $\AA^1$-local weak equivalences,
\item $\mcal{F}$ib as objectwise fibrations
\item $\mcal{C}$of as $\AA^1$-local projective cofibrations
\end{itemize}
forms a simplicial proper cellular model category. This is called the \emph{$\AA^1$-local projective model structure} on $\HH(S)$.
\end{theorem}

The following fact is useful in identifying the $\AA^1$-local spaces (\cite[Remark 3.58]{antieau2017primer}): a motivic space $Z \in \Spc_S$ is  \emph{$\AA^1$-local} if it satisfies the following:
\begin{itemize}
\item takes values in Kan complexes (hence, fibrant in $\Spc_S$)
\item satisfies Nisnevich hyperdescent (hence, fibrant in $\HH(S)_{Nis}$)
\item if $Z(U)\to Z(U\times \AA^1)$ is a weak equivalence of simplicial sets for al $U\in \Sm_S$ (hence, fibrant in $\HH(S)_{\AA^1}$).
\end{itemize}
Thus, it is important to note that the $\AA^1$-local spaces are precisely the fibrant objects in the $\AA^1$-homotopy category $\mcal{H}(S)$.

\subsection{\texorpdfstring{$\mathbb{A}^{1}$}{A1}-Contractibility}\label{intro:A1-contr}
One of the essential notions in the realm of $\AA^1$-homotopy theory is that of schemes that are contractible. This notion is central to formulating our main results in \cref{chp4} and \cref{chp5}. We shall now see the definition and some basic examples. Later in \cref{intro:rel-A1-contr}, we shall be a bit more expansive and connect this notion with the contents of our script.
 
\begin{defn}\label{A1-contractible:defn}
Let $\XX$ be a motivic space over an arbitrary scheme $S$. We say that $\XX$ is \emph{$\AA^1$-contractible} if the structure morphism $f: \XX\to S$ is an $\AA^1$-weak equivalence in the category of motivic spaces $\Spc_S$ (or) equivalently, $f$ is an isomorphism in the $\AA^1$-homotopy category $\HH(S)$, that is, $\L_{mot}(f)$ is an isomorphism in $\mcal{H}(S)$.
\end{defn}

\begin{example}\label{A1-contr:egs}
There are plenitude examples of $\AA^1$-contractible spaces in nature:
\begin{enumerate}
\item The affine line $\AA^1$ is $\AA^1$-contractible, by the very construction. In fact, any affine space $\AA^n$ is $\AA^1$-contractible given by the na\"ive $\AA^1$-homotopy $$\AA^1\times \AA^n \to \AA^n \quad (t,x_1,\dots,x_n)\mapsto (tx_1,\dots,tx_n)$$
\item The singular cuspidal curves $C_{r,s} := \{x^r-y^s=0\}$, for coprime integers $r,s$, are $\AA^1$-contractible \cite[Example 2.1]{asok2007unipotent}. In fact, the normalization map
        $$\AA^1 \to C_{r,s} \quad \text{defined as }\quad w \mapsto (w^s, w^r) $$
is an $\AA^1$-weak equivalence of motivic spaces,
\item Any vector bundle is, in fact, an $\AA^1$-homotopy equivalence \cite[Example 2.2]{MV99}. More generally, any Zariski and Nisnevich locally trivial bundles with $\AA^1$-contractible fibers are seen to be $\AA^1$-contractible. In fact, in \cref{chp4} we will see that the $\AA^1$-contractibility captures all the Zariski locally trivial $\AA^n$-bundles in low relative dimensions up to 2.
\end{enumerate}
We shall encounter several other examples as we progress!
\end{example}
\medskip

We will now recall the analogous notions of connectedness and path-connectedness in algebraic geometry.

\subsection{\texorpdfstring{$\mathbb{A}^{1}$}{A1}-Connectedness}\label{subsec:A1-connectedness}
Recall that in topology, a space is \emph{connected} if it cannot be written as a union of two proper non-empty open subsets. This is captured by the set of path components $\pi_0$ that counts the number of distinct components of a space. Similarly, there is a notion for schemes, which is invariant under connected components. 

\begin{defn}\label{defn:A1-conn}
Let $X$ be a $k$-scheme. The \emph{$\AA^1$-connected component sheaf} of a scheme $X$ is the Nisnevich sheaf associated to the presheaf 
        $$\Sm_k \ni U \mapsto \Hom_{\mcal{H}(k)}(U,X)=: [U,X]_{\AA^1} $$
for all $U\in \Sm_k$ and is denoted by $\pi_0^{\AA^1}(X)$. We say that a scheme $X$ is \emph{$\AA^1$-connected} if $\pi_0^{\AA^1}(X) \to \Spec k$ is isomorphism of sheaves. Similarly, $X$ is \emph{$\AA^1$-disconnected} if it is not $\AA^1$-connected.
\end{defn}

In other words, $\pi_0^{\AA^1}(X)$ can be seen as the sheafification of the presheaf of sections $X(U)$ in the $\AA^1$-homotopy category. Note that for any space $X$, $\pi_0^{\AA^1}(X)$ is only a (Nisnevich) sheaf of sets as in the case of classical topology.

\begin{remark}
For a simplicial presheaf $X$, one can also define the sheaf of \emph{simplicial connected components}, denoted $\pi_0^s(X)$, given by the sheafification of 
        $$U\mapsto [U,X]_{\mcal{H}o_s}$$
where $\mcal{H}o_s$ is the simplicial homotopy category (\cite[Definition 2.1.1]{asokmorel2011}). The two so-mentioned sheaves are related by the motivic localization functor $\L_{mot}$ as follows:
        $$\pi_0^{\AA^1}(X) \simeq \pi_0^s(\L_{\AA^1}(X))$$
This is sometimes used as an alternative definition of the sheaf of $\AA^1$-connected components.
\end{remark}

The following shows the invariant nature of $\pi_0^{\AA^1}$ that will be later useful when dealing with $\AA^1$-contractibility.
\begin{example}\label{S-point}
Every $\AA^1$-connected smooth scheme $X$  over a locally Henselian scheme $S$ (for instance, the spectrum of a field) admits an $S$-point (\cite{MV99}, Remark 2.5). In particular, having an $S$-point is an inherent property of a simplicial sheaf which is $\AA^1$-invariant under $\AA^1$-weak equivalences.
\end{example}

The following result, called the \emph{Unstable 0-connectivity theorem}, is fundamental for these sheaves:
\begin{theorem}\label{unstable-0-connect:thm}
If $\mathcal{F}$ is a space, then the canonical map $\mathcal{F} \to \pi_0^{\AA^1}(\mathcal{F})$ is an epimorphism after Nisnevich sheafification. In particular, if $\mathcal{F}\to \mathcal{G}$ is a $\AA^1$-weak equivalence of spaces and $\mathcal{F}$ is $\AA^1$-connected, then  $\mathcal{G}$ is also $\AA^1$-connected.
\end{theorem}
\begin{proof}
This is \cite[\S 2, Corollary 3.22]{MV99}.
\end{proof}

\subsubsection{$\AA^1$-invariance of $\pi_0^{\AA^1}$}
Recall from \cref{defn:A1-invar-sheaf} that a presheaf is $\AA^1$-invariant if the induced map of the canonical projection is a bijection. By design, we have that the \emph{presheaf} of $\AA^1$-connected components $U \mapsto [U,X]_{\AA^1}$ is $\AA^1$-invariant. This is also true classically. Indeed, we have the set of path components parametrized by the interval $I=[0,1]$ (which can be seen as a presheaf via $\pi_0:=  Y \mapsto [Y, X]$, for locally contractible spaces $X, Y$) can be shown to be $I$-invariant. As the presheaf $\pi_0$ is discrete, it implies that its sheafification stays $I$-invariant (see \cite[P. 87, Example 2]{MV99}). The analogous question in the motivic setting is the famous $\AA^1$-invariance conjecture of Morel:
\begin{conj*}
The sheaf of $\AA^1$-connected components $\pi_0^{\AA^1}$ is $\AA^1$-invariant.
\end{conj*}
The conjecture does hold true in some cases (cf. \cite[Remark 2.5]{choudhury2024}), but it has been proven to fail in general (\cite{ayoub2023counterexamples}, \cite[\S 4]{balwe2015a1}). However, we have the following universal property of $\pi_0^{\AA^1}$.
\begin{lemma}
Let $\XX$ be any space and $\mcal{F}$ be any $\AA^1$-invariant sheaf. Then any morphism $\XX\to \mcal{F}$ uniquely factors via $\XX\to \pi_0^{\AA^1}(\mcal{F})$.
\end{lemma}
\begin{proof}
This is due to \cite[Lemma 2.8]{balwe2015a1}.
\end{proof}

\subsubsection{Higher \texorpdfstring{$\mathbb{A}^{1}$}{A1}-homotopy sheaves}
To conclude, note that subsequent to $\pi_0^{\AA^1}$, one can also define the higher $\AA^1$-homotopy sheaves by taking maps from the simplicial motivic spheres as follows.
\begin{defn}
Let $(\mathfrak{X},x)$ be any pointed space. Then the $i$-th $\AA^1$-homotopy sheaf is the Nisnevich sheaf on $\Sm_k$ associated with the presheaf 
    $$\Sm_k\ni U \mapsto \Hom_{\mcal{H}(k)}(S^i \wedge U_+ , (\mathfrak{X},x))$$
for $U\in \Sm_k$. It is denoted by $\pi^{\AA^1}_i(\mathfrak{X},x)$.
\end{defn}

Like in topology, one shows that $\pi_1^{\AA^1}$ is Nisnevich sheaf of groups and $\pi_i^{\AA^1}$ are all Nisnevich sheaf of Abelian groups for all $i\ge 2$ (\cite{morel2012A1topology}). Recall that the Whitehead theorem states that any weak equivalence between CW-complexes is also a homotopy equivalence. The following is the $\AA^1$-analog of it due to \cite[\S 3, Proposition 2.14]{MV99}:
\begin{prop}\label{Whitehead-thm}
Assume that $\mathfrak{X}$ and $\mathfrak{Y}$ are $\AA^1$-connected pointed spaces. Then a morphism $f:\mathfrak{X}\to \mathfrak{Y}$ is an $\AA^1$-weak equivalence if and only for any choice of base point $x\in \mathfrak{X}$ with $y:= f(x)$ the induced morphism
    $$\pi_i^{\AA^1}(\mathfrak{X},x) \to \pi_i^{\AA^1}(\mathfrak{Y},y) $$
is an isomorphism of Nisnevich sheaves.
\end{prop}

\begin{example}
The higher $\AA^1$-homotopy sheaves $\pi_i^{\AA^1}$ (for $i\ge1$) form primordial examples of $\AA^1$-invariant sheaves (See \cref{thm:A1-htpysheaves-invariant}).
\end{example}

\subsection{\texorpdfstring{$\mathbb{A}^{1}$}{A1}-Rigid Schemes}\label{subsec:A1-rigid-schemes}
In this section, we study certain special kinds of motivic spaces that can be understood as 'discrete' schemes and have trivial higher $\AA^1$-homotopies.

\begin{defn}\label{defn:A1-rigid}
A scheme $f: X\to \Spec k$ is called \emph{$\AA^1$-rigid} if for any smooth scheme $U \in \Sm_k$, the induced morphism 
        $$ \Hom_{\Sm_k}(U, X) \to \Hom_{\Sm_k} (U \times_k \AA^1_k, X)$$
of the canonical projection $U \times \AA^1\to U$ is a bijection.
\end{defn}
Therefore, a scheme is $\AA^1$-rigid if it is $\AA^1$-local when viewed as a space and $\AA^1$-invariant when viewed as a presheaf. The $\AA^1$-rigid schemes stay $\AA^1$-invariant under (reasonable) field extensions (\cite[Lemma 2.1.11]{asokmorel2011}):

\begin{lemma}\label{A1-chain-conne-sep-ext:lemma}
A smooth $k$-scheme $X$ is \emph{$\AA^1$-rigid} if and only if for every finitely generated separable extension $L/k$ the map
        $$\Hom(\Spec L, X)\to \Hom(\Spec L \times_k \AA^1_k, X)$$
induced by the canonical projection $\AA^1_k \times \Spec L \to \Spec L$ is a bijection.
\end{lemma}

\subsubsection{Examples of $\AA^1$-rigid schemes}
Any 0-dimensional scheme over a field $k$ is $\AA^1$-rigid (for a generalization of this result, see \cref{0dim:rigid}). Abelian varieties, Smooth projective curves of positive genus, and smooth subvarieties and products of an $\AA^1$-rigid variety all constitute examples of $\AA^1$-rigid schemes (see also \cite{asokmorel2011, asok2021A1, choudhury2024}). The following portray some common occurrences.
\begin{example}\label{Gm-A1-rigid}
The affine scheme $\GG_m$ is $\AA^1$-rigid over any reduced commutative ring $R$.
\begin{proof}
For simplicity, let us prove this over fields $R = k$, noting that the proof is essentially the same over any such $R$. Note that as a presheaf  $\GG_m = k[Y,Y^{-1}]$ induces a functor $\GG_m(U) = \mathcal{O}_U^{\times}$, the sheaf of units. Being a Zariski presheaf, it suffices to check the property (\cref{defn:A1-rigid}) locally on affine schemes. Indeed, let $X = \Spec A$ be any affine scheme where $A$ is any finitely generated $k$-algebra, then we have
    $$ \Hom_{\Sm_k}(X, \GG_m) \to \Hom_{\Sm_k} (X \times_k \AA^1_k, \GG_m).$$
Since $A$ is reduced, we have that $\GG_m(X) = \mathcal{O}_X^{\times} \cong A^{\times}$. On the other hand, since $X \times_k\AA^1_k \cong  \Spec A\times_k \Spec(k[x]) \cong \Spec (A\otimes_k k[x]) \cong \Spec A[x]$, we have that $\GG_m(A[x])\cong A^{\times}$, which completes the proof.
\end{proof}
\end{example}
\begin{example}\label{eg:nonA1rigid-A1&P1}
The projective line $\PP^1$ is not $\AA^1$-rigid. This can be tested even on the simplest space $X = \Spec k$. This is because every embedding $\Spec k \hookrightarrow \PP^1_k$ is isomorphic to a constant embedding, but there are clearly non-constant embeddings $\AA^1 \hookrightarrow \PP^1$, and as a result
    $$ \Hom_{\Sm_k}(\Spec k ,\PP^1_k) \to \Hom_{\Sm_k}(\Spec k \times_k \AA^1_k, \PP^1_k)$$
is not surjective (\cref{rem:A1-invar:onto}). For the same reason, $\AA^1$ also cannot be $\AA^1$-rigid.
\end{example}

Consider the functor $\rm{h}: \Sm_k \to \HH(k)$ that is obtained by composing the corresponding functors in the diagram \ref{the-bird-eye}. Then we have that $\rm{h}$ embeds the full subcategory of $\AA^1$-rigid schemes fully faithfully into the $\AA^1$-homotopy category (\cite[Lemma 2.1.9]{asokmorel2011}). This has the following consequences:

\begin{lemma}\label{A1-rigid:mainresult}
Let $X,Y \in \Sm_k$ be two $\AA^1$-rigid schemes, then we have the following:
\begin{enumerate}
\item for any $U\in \Sm_k$, the canonical map 
\begin{align*}
    & \Hom_{\Sm_k}(U,X) \to \Hom_{\mathcal{H}(k)}(U,X):= [U,X]_{\AA^1} \\
    & \hspace{24mm} f \mapsto [f]
\end{align*}
is a bijection,
\item the canonical map $X \to \pi_0^{\AA^1}(X)$ is an isomorphism of (Nisnevich) sheaves. As a result, $X$ is $\AA^1$-connected if and only if $X \cong \Spec k$,
\item $\rm{h}(X) \simeq_{\AA^1} \rm{h}(Y) \iff X\cong Y$, that is, two $\AA^1$-rigid schemes are $\AA^1$-weakly equivalent if and only if they are isomorphic as $k$-schemes.
\end{enumerate}
\end{lemma}

\subsubsection*{Local nature} 
$\AA^1$-rigid schemes demonstrate one of the several stark dichotomies that exist between classical topology and algebraic geometry. Classically, every smooth manifold admits a local base that consists of smooth contractible neighborhoods, and thereby, the local system is a collection of connected spaces. In contrast, every scheme that is locally of finite type has a base that built out of products of $\GG_m$s (\cite[Lemma 2.3]{choudhury2024}), in particular, they are product of $\AA^1$-rigid varieties, whence in general, is not $\AA^1$-connected unless it is trivial (\cref{A1-rigid:mainresult} (2)). Nevertheless, they are important as they form the building blocks of the $\AA^1$-model structure:

\begin{lemma}
A smooth scheme of finite type is $\AA^1$-local fibrant if and only if it is $\AA^1$-rigid.
\end{lemma}
\begin{proof}
\cite[Lemma 4.1.8]{asok2021A1}.
\end{proof}

\subsection{\texorpdfstring{$\mathbb{A}^{1}$}{A1}-Chain Connectedness}\label{subsec:A1-chain-connectedness}
We now focus on the analogous notion of path connectedness in algebraic geometry. Recall from topology that a space $X$ is \emph{path connected} if any two given points $x,y$ can be joined by a map $\gamma:[0,1]\to X$ such that $\gamma(0)= x$, $\gamma(1)= y$. The analogous notion in $\AA^1$-homotopy, established by Asok-Morel (\cite{asokmorel2011}), is that of being "connected by the images (chains) of the affine line".

\begin{defn}
Given $X\in \Sm_k$ and a finitely generated separable extension of $L/k$, and points $x_0, x_1 \in X(L)$ an \emph{elementary $\AA^1$ equivalence} between $x_0$ and $x_1$ is a morphism $f : \AA^1 \to X$ such that $f(0) = x_0$ and $f(1)= x_1$. Two points $x,y\in X(L)$ are $\AA^1$-equivalent if they are equivalent with respect to the equivalence relation generated by elementary $\AA^1$-equivalence.
\end{defn}

\begin{defn}\label{defn:A1-chainconn}
A $k$-scheme $X$ is \emph{$\AA^1$-chain connected} if for any finitely generated separable field extension $L/k$, $X(L)$ is non-empty and for all rational points $x, y \in X(L)$, there exists an integer $N$ and a sequence of points $x = x_0,\dots, x_{N-1},x_N =y\in X(L)$ and an elementary $\AA^1$-weak equivalences $f_1,\dots,f_N: \AA^1_L \to X$ such that $f_i(0)= x_{i-1}$ and $f_i(1)= x_i$ for $1\le i\le N$: any two points can be connected by the images of a chain of maps from $\AA^1$.
\end{defn}

Alternatively, one says that a variety $X$ is $\AA^1$-chain connected if the set of $\AA^1$-equivalence classes of $L$-points $X(L)/\sim$ consists of exactly one element.

\begin{defn}
Let $X\in \Sm_k$. Then the \emph{sheaf of $\AA^1$-chain connected component of $X$} is defined as 
        $$\pi_0^{ch}(X) := \pi_0^s(\Sing_*^{\AA^1}(X)).$$
We say that $X$ is $\AA^1$-chain connected if $\pi_0^{ch}(X)\simeq \Spec k$.
\end{defn}

Various relations have been studied between $\pi_0^{\AA^1}$ and $\pi_0^{ch}$ (for example, see \cite{asokmorel2011, balwe2015a1}). We will list a few that lie within the scope of our script.
\begin{lemma}
For $X\in \Sm_k$, the map $\phi: \pi_0^{ch}(X)\to \pi_0^{\AA^1}(X)$ is an epimorphism of Nisnevich sheaves.
\end{lemma}

In topology, any path-connected space is also connected. But the converse is not true (e.g., topologist's sine curve). The following facts show us the striking affinity for schemes as well:
\begin{prop}
If $X\in \Sm_k$ is $\AA^1$-chain connected, then it is $\AA^1$-connected.
\end{prop}
\begin{proof}
\cite[Proposition 4.2.4]{asok2021A1}.
\end{proof}

Asok and Morel also conjectured (\cite[Conjecture 2.2.8]{asokmorel2011}) that for smooth schemes over $k$, the above map $\phi$ is an isomorphism, that is, it is also a monomorphism. They also observed that this should hold provided $\Sing^{\AA^1}_*(X)$ is $\AA^1$-local for a space $X$ and showed that this is indeed true for smooth proper varieties (\cite[Corollary 2.4.4]{asokmorel2011}). However, this has now proven to be false in generality due to the work of \cite[\S 4.1]{balwe2015a1}. This tells us that being $\AA^1$-chain connected is indeed stronger than being $\AA^1$-connected.
\medskip

\begin{defn}\label{defn;covered-by-affine}
We say that an $n$-dimensional smooth $k$-variety $X$ is \emph{covered by affine spaces} if $X$ admits an open affine cover by finitely many copies of $\AA^n_k$ such that the intersection\footnote{this is automatic if $k$ is an infinite field} of any two copies of $\AA^n_k$ has a $k$-point.
\end{defn}

\begin{example}
The following are some (non-)examples of $\AA^1$-chain connected varieties:
\begin{itemize}
\item Any variety that is covered by affine spaces in the sense of \cref{defn;covered-by-affine} is $\AA^1$-chain connected. These include rational smooth proper varieties of dimension $\le 2$, smooth proper toric varieties, and generalized flag varieties for connected reductive groups over $k$.
\item The Koras-Russell threefolds (\cite[Example 2.28]{DPO2019}) and Generalized Koras-Russell varieties are $\AA^1$-chain connected (cf. \cref{X-p-isA1-connected}).
\item The Ramanujam surface \cref{eg:Ramanujamsurf} and the tom Dieck-Petrie surfaces \cref{eg:tomDieck-petri} are not $\AA^1$-chain connected.
\end{itemize}
\end{example}

Generalizing the sheaf of $\AA^1$-chain connected components, the authors in \cite{balwe2015a1} construct the \emph{universal $\AA^1$-invariant sheaf} in the following sense: Set $\mcal{S}(-):= \pi_0^{ch}(-)$, then for any Nisnevich sheaf $\mcal{F}$, the canonical morphism $\mcal{F}\to \mcal{S}(\mcal{F})$ is surjective and so consider 
    $$\mcal{F}\to \mcal{S}(\mcal{F}) \to \mcal{S}^2 (\mcal{F}) \to \dots \to \mcal{S}^n(\mcal{F}) \to \dots$$
where $\mcal{S}^{n+1} (\mcal{F})$ is inductively defined as $\mcal{S}(\mcal{S}^n(\mcal{F}))$, for all $n\in \mathbb{N}$.
The \emph{universal $\AA^1$-invariant sheaf is defined as}
        $$\mcal{L}(\mcal{F}) := \lim_{\substack{n}} \mcal{S}^n (\mcal{F})$$
Then they show \cite[\S 2]{balwe2015a1} that for any sheaf $\mcal{F}$, $\mcal{L}(\mcal{F})$ is $\AA^1$-invariant and for any $\AA^1$-invariant sheaf $\mcal{G}$, any morphism $\mcal{F} \to \mcal{G}$ uniquely factors via $\mcal{F}\to \mcal{L}(\mcal{F})$. Moreover, if $\pi_0^{\AA^1}(\mcal{F})$ is $\AA^1$-invariant, then the canonical map $\mcal{L}(\mcal{F}) \to \pi_0^{\AA^1}(\mcal{F})$ is an isomorphism.

\subsection{Functoriality}\label{sect:functoriality}
Given any category, a vital tool to establish is its functoriality with respect to its morphisms. The functoriality in motivic homotopy theory is well-established and facilitates handling certain problems from a categorical viewpoint. For a comprehensive overview of functoriality, we recommend the following resources: \cite[\S 3]{MV99}, \cite{Ayoub2007six, CD2019}. For other recent accounts, one can refer to \cite{hoyois2017six, ELSO22, bachmann2024stronglya1, rondigs2025grothendieck}. In alignment with the material presented in our script, we will restrict ourselves to the base change functoriality, which we shall apply to the special case when the spaces are $\AA^1$-contractible.

\subsubsection{The Base Change Functor}
Let $f:T \to S$ be a morphism of base schemes\footnote{By a base scheme, we mean a Noetherian separated scheme of finite Krull dimension}. Then the pullback along $f$ defines a functor 
\begin{align*}
    &f^{\bullet}: \Sm_S\longrightarrow \Sm_T\\
    & \hspace{10mm} U \hspace{2mm} \longmapsto U_T:= U \times_S T.
\end{align*}
Precomposition with $f^{\bullet}$, we obtain another colimit preserving functor at the level of motivic spaces defined as
\begin{align*}
    & f_*: \Spc_T \longrightarrow \Spc_S\\
    & \hspace{7mm} \mcal{F} \hspace{5mm} \longmapsto \mcal{F} \circ f^{\bullet}:= U\mapsto \mcal{F}(U_T). 
\end{align*}
By left Kan extension (cf. \cref{app:Kan-ext}), we see that $f_*$ admits a left adjoint $f^*: \Spc_S \to \Spc_T$ on the level of motivic spaces
        $$f^* \dashv f_* $$
which is a strict symmetric monoidal category and preserves limits. Since every (pointed) motivic space is a colimit of representable motivic spaces (cf. \cref{prop:Rep-psh-colimits}), we can characterize $f^*$ by the following formula
\begin{align}\label{pshv-pullback}
    f^*(X_+) = (X \times_S T)_+
\end{align}
for all $X\in \Sm_S$. 

\subsubsection{Smooth Base Change}
If furthermore, $f: T \to S$ is a smooth morphism of base schemes, then composition with $f$ defines a functor $f_{\bullet}: \Sm_T\to \Sm_S$. Precomposition with this functor defines another functor $f^{\bigstar}: \Spc_S \to \Spc_T$ which itself admits a left adjoint $f_{\sharp}: \Spc_T \to \Spc_S$ via (enriched) Kan extension, that is 
    $$ f_{\sharp} \dashv f ^{\bigstar} \dashv f_{*}$$
Again, using the fact that any motivic space is a colimit of representable ones, $f_\sharp$ can be characterized by the formula
    $$f_\sharp(X\xrightarrow{g} T)_+ = (X \xrightarrow{g} T \xrightarrow{f} S) $$
for every $X\in \Spc_T$. In other words, $f_\sharp$ is simply induced by the forgetful functor $\mcal{U}: \Sm_S\to \Sm_T$. On the other hand, if $g : Z \to S$ is a morphism, then the canonical $S$-morphism $Z \times_S T$ defines a map $B(Z) \to f_*f^{\bigstar} B(Z)$ which is natural in $Z$ and $B \in \Spc_T$, and consequently, a natural transformation $\Id_{\Spc_S}\to f_*\circ f^\bigstar$. The following lemma says these two functors agree whenever $f$ is a smooth morphism:
\begin{lemma}
If $f: X \to Y$ is a smooth morphism of base schemes, the adjoint 
         $$f^* \to f^\bigstar $$
of the natural transformation $\Id_{\Spc_S}\to f_*\circ f^\bigstar$ is a natural transformation.
\end{lemma}

\begin{lemma}
If $f: X\to Y$ is a morphism of base schemes. Then $f^*$ preserves finite objects, dualizable objects. Furthermore, if $f$ is smooth, then $f_{\sharp}$ preserves finite objects.
\end{lemma}
\begin{proof}
\cite[Lemma 3.11]{rondigs2025grothendieck}.
\end{proof}

\textsc{Convention recap:} For a morphism of $S$-schemes $f:X \to Y$, some authors (by abusive of terminology) use that $f$ is an $\AA^1$-weak equivalence in $\HH(S)$ or that $f$ is an isomorphism of motivic spaces. We have made a clear distinction in the following sense: we will adopt to say either that $f$ is an $\AA^1$-weak equivalence in $\Spc_S$ or equivalently, that $f$ is an isomorphism in $\HH(S)$.
\medskip

The following crucial fact gives us the concise base change functoriality relating the category of motivic spaces and the associated $\AA^1$-homotopy category. We will denote by $Lf^*$ (resp. $RF_*$) the left (resp. right) derived Quillen functor on the level of $\AA^1$-homotopy categories (cf. \cref{app:Quillen-func}).

\begin{prop}\cite{DMO25}
Let $f: T\to S$ be a morphism of schemes, let $\Spc_S$ and $\Spc_T$ be the category of motivic spaces endowed with their respective local projective model structures, and let $\HH(S)$ and $\HH(T)$ be their respective localizations with their $\AA^1$-local projective model structure. Then the following holds:
\begin{enumerate}
  \item The adjoint pair $(f^*,f_*)$ induces Quillen adjunctions $$\Spc_S \stackrel{\overset{f_*}{\longleftarrow}}{\underset{f^*}{\longrightarrow}} \Spc_T  \quad \textrm{and}  
  \quad \HH(S) \stackrel{\overset{Rf_*}{\longleftarrow}}{\underset{Lf^*}{\longrightarrow}}\HH(T).$$
  \item If $f :T\to S$ is smooth, then the adjoint pair $(f_\sharp,f^{\bigstar})$ induces Quillen adjunctions 
  $$\Spc_T \stackrel{\overset{f^\bigstar}{\longleftarrow}}{\underset{f_\sharp}{\longrightarrow}}\Spc_S \quad  \textrm{and} 
  \quad \HH(T)\stackrel{\overset{Lf^\bigstar}{\longleftarrow}}{\underset{Rf_\sharp}{\longrightarrow}}\HH(S).$$
\end{enumerate}
\end{prop}
\begin{proof} 
The fact that $(f^*, f_*)$ (resp. $(f_\sharp, f^{\bigstar})$ when $f$ is smooth) induces a Quillen adjunction between the corresponding model categories is a consequence of using the projective model structure, and this follows from \cite[Proposition A.2.8.7]{lurie2009HTT}. Indeed, the first assertion follows from the observation that if $U_\bullet \to X$ is a Nisnevich hypercover of a smooth $S$-scheme $X$, then $(U_T)_{\bullet}\to X_T$ is a Nisnevich hypercover of $X_T$ and that for every smooth $S$-scheme $X$, $(X\times_S \AA^1_S)\cong X_T\times_T \AA^1_T$ which implies by \cite[Proposition A.3.7.9]{lurie2009HTT} that the Quillen adjunction 
$$\Spc_S \stackrel{\overset{f_* \circ \mathrm{id}}{\longleftarrow}}{\underset{\mathrm{id} \circ f^*}{\longrightarrow}} \Spc_T $$ 
induces a Quillen adjunction 
$$\HH(S) \stackrel{\overset{\mathrm{id}\circ f^* \circ \mathrm{id}}{\longleftarrow}}{\underset{\mathrm{id} \circ f_*\circ \mathrm{id}}{\longrightarrow}}\HH(T)$$ 
The second assertion follows for the same reasons after observing that when $f: T\to S$ is smooth, a Nisnevich hypercover $U_\bullet \to X$ of a smooth $T$-scheme $X$ is again a Nisnevich hypercover of $X$ considered as a smooth $S$-scheme via the composition with $u$ and that for every smooth $T$-scheme $X$, $X\times_T \AA^1_T\cong X\times_S \AA^1_S$ as schemes over $S$.
\end{proof}

We now present several consequences of this base change functoriality, which will play a crucial role in \cref{chp4}. 

\begin{corollary}\label{cor:pushpull-A1-weak} \cite{DMO25}
Let $f: T\to S$ be a morphism of schemes. Then the following holds:
\begin{enumerate}
\item For every morphism $u:Y\to X$ between smooth $S$-schemes which is an $\AA^1$-weak equivalence in $\Spc_S$, the morphism $u_T: Y_T\to X_T$ is an $\AA^1$-weak equivalence in $\Spc_T$.  
\item If $f: T\to S$ is smooth, then every morphism $u: Y\to X$ between smooth $T$-schemes which an $\AA^1$-weak equivalence in $\Spc_T$ is also an $\AA^1$-weak equivalence in $\Spc_S$.
\end{enumerate}
\end{corollary}
\begin{proof}
For the first assertion, we consider the following commutative diagram 
\[\begin{tikzcd}
    \Spc_S  & \Spc_T  \\ 
    \HH(S) & \HH(T)
    \arrow["f^*", from=1-1, to=1-2]
    \arrow["Lf^*", from=2-1, to=2-2]
    \arrow["\mathrm{id}", from=1-1, to=2-1] 
   \arrow["\mathrm{id}",from=1-2, to=2-2]
\end{tikzcd}\]
The functor $ f^*$: $ \Spc_S \to \Spc_T$ maps the presheaves associated to the smooth $S$-schemes $X$ and $Y$ to the presheaves associated to the smooth $T$-schemes $X_T$ and $Y_T$ respectively and maps the natural transformation associated to $u: Y\to X$ onto that associated to $u_T: Y_T\to X_T$. On the other hand, by construction of the projective model structure on $\Spc_S$, the presheaves associated to $X$ and $Y$ are cofibrant objects of $\Spc_S$, whence of $\HH(S)$ by definition of left Bousfield localization. Thus, $u: Y\to X$ is an $\AA^1$-weak equivalences between cofibrant objects of $\Spc_S$, and since $f^*$ is a left Quillen functor (that is, using the adjunction $f^*\dashv f_*$), it follows from Brown's lemma (see \cref{ken-brown-lemma}) that the image  $u_T: Y_T\to X_T$ of $u: Y\to X$ is $\AA^1$-weak equivalence in $\Spc_T$. The second assertion follows from the same reasoning using the commutative diagram 
\[\begin{tikzcd}
	 \Spc_T  & \Spc_S  \\
	 \HH(T) & \HH(S) 
	\arrow["f_\sharp", from=1-1, to=1-2]
    \arrow["Lf_\sharp", from=2-1, to=2-2]
    \arrow["\mathrm{id}", from=1-1, to=2-1] 
   \arrow["\mathrm{id}",from=1-2, to=2-2]
\end{tikzcd}\]
and the fact that $f_\sharp$ maps the presheaves associated to the smooth $T$-schemes $X$ and $Y$ to those associated to $X$ and $Y$ considered as smooth $S$-schemes via the composition with $f: T\to S$. 
\end{proof}

The following portrays that relative $\AA^1$-weak equivalences behave well under smooth base change.
\begin{corollary}\label{lem:3-from-2}\cite{DMO25}
Let $f: Y\to X$ and $g: Z\to Y$ be smooth morphisms between smooth schemes over a scheme $S$. If $g:Z\to Y$ is an $\AA^1$-weak equivalence in $\Spc_Y$ and $f:Y\to X$ is an $\AA^1$-weak equivalence in $\Spc_X$, then $h= f\circ g:Z\to X$ is an $\AA^1$-weak equivalence in $\Spc_X$.
\end{corollary}
\begin{proof} Since $f: Y\to X$ is a smooth morphism, \Cref{cor:pushpull-A1-weak} (2) implies that $g: Z \to Y$ is an $\AA^1$-weak equivalence in $\Spc_X$ and the assertion then follows from the "two out of three" axiom for weak equivalence in a model category (cf. \cref{defn:model-cats} [2.]). 
\end{proof}

\begin{remark} In the setting of \Cref{lem:3-from-2}, the implication "\emph{$f$ and $g$ are $\AA^1$-weak equivalence in $\Spc_X$ $\Rightarrow$ $g$ is an $\AA^1$-weak equivalence in $\Spc_Y$}" does not hold in general, see \Cref{ex:A1-cont-not-factorize} and \Cref{exa:form-generic-fiber}.
\end{remark}

\subsection{Base Change under \texorpdfstring{$\mathbb{A}^{1}$}{A1}-Contractibility}\label{subsec:Base-change}
We now apply these results to the case when the space is $\AA^1$-contractible, which we will need in both \cref{chp4} and \cref{chp5}.
\begin{prop} \label{pullbackfunctor} \cite{DMO25}
Let $u: X \to S$ be a smooth $\AA^1$-contractible $S$-scheme. Then for any arbitrary morphism of schemes $f: T \to S$, the induced map $u_T: X_T\to T$ is a smooth $\AA^1$-contractible $T$-scheme, where we have denoted the scheme-theoretic pullback by $X_T:= X\times_S T$.
\end{prop} 
\begin{proof}
This is a straightforward application of \Cref{cor:pushpull-A1-weak} (1).
\end{proof}

This proposition has several interesting implications, as presented in \cite{DMO25}: 

\begin{corollary}\label{basechange:imperfect} 
Let $X$ be a smooth $\AA^1$-contractible scheme over a field $k$. Then for every field extension $k\subset L$ (separable or not), $X_L:=X\times_{\Spec(k)} \Spec(L)$ is an $\AA^1$-contractible $L$-scheme.
\end{corollary}
\begin{proof}
A proof for the case when $L/k$ is a finite separable field extension is given in \cite[Proposition 6.2.15]{severitt2006motivic}. But this in fact extends to any arbitrary field extensions due to \cref{pullbackfunctor}.
\end{proof}

\begin{corollary} \label{cor:A1-cont-fibers} 
Let $f: X \to S$ be a smooth $\AA^1$-contractible $S$-scheme. Then for every point $s$ of $S$ (closed or not) with residue field $\kappa(s)$, the scheme theoretic fiber $X_s:=X\times_{S} \Spec \kappa(s)$ of $f$ over the point $s$ is a smooth $\mathbb{A}^1$-contractible $\kappa(s)$-scheme.
\end{corollary}

\begin{corollary}\label{basechange:closedimm}
Let $i: Z \hookrightarrow S$ be a closed immersion of schemes. If $f:X\to S$ is a smooth $\AA^1$-contractible $S$-scheme, then $\mathrm{pr}_2: X\times_S Z \to Z$ is an $\AA^1$-contractible $Z$-scheme.
\end{corollary}

The following example complementing \Cref{basechange:imperfect} illustrates the fact that the $\mathbb{A}^1$-contractibility property does not necessarily descend under arbitrary morphisms $f: T\to S$: 

\begin{example} \label{ex:forms-A1}
Let $k$ be an imperfect field of characteristic $p>0$ with perfection $k^{\mathrm{perf}}$ and let $g:\Spec k^{\mathrm{perf}} \to \Spec k$ be the associated base change morphism. Let $f: X\to \Spec k$ be a non-trivial $k$-form of the affine line $\mathbb{A}^1_k$. Since the affine line does not admit non-trivial separable form  \cite{russell1970forms}, it follows that the base change $X_{k^{\mathrm{perf}}}$ of $X$ by $g$  is isomorphic to $\mathbb{A}^1_{k^{\mathrm{perf}}}$, whence that $X_{k^{\mathrm{perf}}}$ is an $\mathbb{A}^1$-contractible $k^{\mathrm{perf}}$-scheme. On the other hand, $X$ is not an $\mathbb{A}^1$-contractible $k$-scheme. Indeed, $X$ being a non-trivial $k$-form of $\mathbb{A}^1_k$, it follows from \cite[Lemma 1.1]{russell1970forms} that either $X(k)=\emptyset$ or $X$ is isomorphic to the complement in the projective line $\mathbb{P}^1_k$ of a closed point $x$ whose residue field $\kappa(x)$ is a purely inseparable extension of $k$. In the first case, the non-$\mathbb{A}^1$-contractibility of $X$ over $k$ follows from \Cref{S-point}, and in the second case, it follows from \cite[Theorem 4.1]{achet2017picard} which asserts that such a non-trivial form $X$ has a non-trivial Picard group (cf. \cref{A1-cont-trivial-Pic}).
\end{example}

\begin{remark}\cite[Remark 3.17]{Mad}
If $u: X\to S$ is a smooth $\AA^1$-contractible scheme in $\Spc_S$, then its base change with respect to an arbitrary morphism $f: T\to S$ is recovered as that of the (usual) scheme-theoretic pullback. As a consequence, this implies crucially that the following diagram 
\[\begin{tikzcd}
	{\Sm_S} &&& {\Sm_T} \\
	{\Spc_S} &&& {\Spc_T} \\
	{\HH(S)} &&& {\HH(T)}
	\arrow["{f^{\bullet}}", from=1-1, to=1-4]
	\arrow[hook,"Y", from=1-1, to=2-1]
	\arrow[hook,"Y", from=1-4, to=2-4]
	\arrow["{f^*}", from=2-1, to=2-4]
	\arrow["{\L_{mot,S}}"', from=2-1, to=3-1]
	\arrow["{\L_{mot,T}}", from=2-4, to=3-4]
	\arrow["{Lf^*}"', from=3-1, to=3-4]
\end{tikzcd}\]
commutes. Recall the notations from \cref{sect:functoriality}. Here, $Lf^*: \Spc_S \to \Spc_T$ is the induced derived pullback functor on the respective $\AA^1$-homotopy categories and $\L_{mot,T}$ (resp. $\L_{mot,S}$) is the corresponding motivic localization functor defined on $\Spc_T$ (resp. $\Spc_S$). In other words, for a smooth $\AA^1$-contractible scheme $u: X\to S$, the derived pullback can be retrieved as the scheme-theoretic pullback and vice-versa:
\begin{equation}
\begin{split}
Lf^*(\L_{mot,S}(X)) 
        & \simeq Lf^*(\Id_S) \quad (X\ \text{is}\ \AA^1 \text{-contractible}) \\
        & \simeq \Id_T \quad (\text{functoriality of}\ Lf^*)  \\
        & \simeq \L_{mot,T}(T)  \\
        & \simeq \L_{mot,T}(X_T) \quad (\cref{pullbackfunctor}) \\
Lf^*(\L_{mot,S}((X))
        & \simeq \L_{mot,T}(f^*(X)) \quad (\text{\cref{pshv-pullback}})\\
\end{split}
\end{equation}
\end{remark}

\subsection{Relative \texorpdfstring{$\mathbb{A}^{1}$}{A1}-Contractibility of Smooth Schemes}\label{intro:rel-A1-contr}
We will now definitively clarify the meaning of the phrase \emph{relative $\AA^1$-contractibility} that we extensively use in the chapters ahead.
\begin{defn}
For any $k$-schemes $X,Y\in \Sm_k$, we say that a morphism $f:X\to Y$ is \emph{relatively $\AA^1$-contractible} if $f$ is an $\AA^1$-weak equivalence in the category of motivic $Y$-spaces $\Spc_Y$ or equivalently, $X$ is an $\AA^1$-contractible scheme in the $\AA^1$-homotopy category of $Y$, denoted $\HH(Y)$.
\end{defn}
To emphasize this relativeness, we shall call the morphism $f$ a \emph{relative $\AA^1$-weak equivalence} of motivic spaces. Being "relatively" $\AA^1$-contractible is a subtle notion, and the need for such a distinction in a general setting is demonstrated by the following example. Observe also that this example is a non-pathological and, in fact, a fundamental instance:

\begin{example}\label{ex:A1-cont-not-factorize} \cite[Example 3.15]{Mad}
Consider the smooth morphism of $k$-schemes 
\begin{equation}
\begin{split}
   & g: Z:= \mathbb{A}^2_k\to \mathbb{A}^1_k =: Y \\
   & (x,y) \mapsto x^2 y^2 + y 
\end{split}
   \hspace{12mm}
   \begin{tikzcd}
	{\AA^2_k} && {\AA^1_k} \\
	& {\Spec(k)}
	\arrow["g", from=1-1, to=1-3]
	\arrow[from=1-1, to=2-2]
	\arrow[from=1-3, to=2-2]
  \end{tikzcd}
\end{equation}
We have that both $Y$ and $Z$ is $\AA^1$-contractible over $X$. Therefore, $g$ is an $\mathbb{A}^1$-weak equivalence in $\Spc_k$. But $g$ cannot be an $\AA^1$-weak equivalence in $\Spc_{\AA^1_k}$. Suppose it was so, then by  \Cref{cor:A1-cont-fibers}, every fiber of $f$ over a point of $\mathbb{A}^1_k$ would be an $\mathbb{A}^1$-contractible scheme over the corresponding residue field. But observe that the fiber of $f$ over $0$ is the disjoint union of a copy of $\mathbb{A}^1_k$ and a copy of $\GG_{m,k}$. The latter component is not even $\mathbb{A}^1$-connected (\cref{Gm-A1-rigid}) in $\Spc_k$ and so the fiber cannot be $\AA^1$-connected, whence cannot be $\AA^1$-contractible (see \cref{eg:S-point}). Therefore, we conclude that $g$ cannot be a relative $\AA^1$-weak equivalence.
\end{example}

Retreating to \Cref{cor:pushpull-A1-weak} (1), one is intrigued to ask the following: for a morphism $f: Y\to X$ between smooth $S$-schemes, under what assumptions is there a (suitable) morphism $u: T\to S$ such that whenever the base change $f_T: Y_T\to X_T$ of $f$ is an $\AA^1$-weak equivalence in $\Spc_T$, then $f: Y\to X$ is an $\AA^1$-weak equivalence in $\Spc_S$. The \Cref{ex:forms-A1} illustrates that faithful flatness of $g$ is, for instance, not sufficient. Alternatively, we can ask the following question:

\begin{question}\label{fiber:pullback}
For a smooth morphism of schemes $f: X \to S$, does the property that all fibers of $f$ over points of $s$ are $\AA^1$-contractible over the corresponding residue fields imply that $X$ is an $\AA^1$-contractible $S$-scheme?
\end{question}

We note that such a desired property (\cref{fiber:pullback}) holds true in the stable $\AA^1$-homotopy category $\mathrm{SH}(S)$ owing to the fact that the family of functors $i_s^{!}:\mathrm{SH}(S)\to \mathrm{SH}(\kappa(s))$, $s\in S$, is conservative (see e.g. \cite[B20]{DJK21}). A recent remark due to \cite[Proposition 2.1]{realetale2025} points out that this property is, to our much surprise, also true unstably)! In fact, the solution to \cref{fiber:pullback} follows as a consequence of the \emph{gluing theorem} (\cite[Theorem 2.21]{MV99}), and this remains true for any Grothendieck topology that is at least as fine as the Nisnevich topology.

\begin{theorem}\label{MV:gluing theorem}
Let $S$ be any locally Noetherian scheme of finite Krull dimension. Then for any motivic space $\mcal{F}$, the square
\[\begin{tikzcd}
	{j_{\#} j^* \mcal{F}} & {\mcal{F}} \\
	U & {i_*i^*(\mcal{F})}
	\arrow[from=1-1, to=1-2]
	\arrow[from=1-1, to=2-1]
	\arrow[from=1-2, to=2-2]
	\arrow[from=2-1, to=2-2]
\end{tikzcd}\]
is a homotopy cocartesian in $\mcal{H}(S)$. Here, $j: U \hookrightarrow S$ is associated with the open immersion of an open subscheme $U\subset S$.
\end{theorem}
\begin{proof}
This is \cite[Theorem 2.21]{MV99}.
\end{proof}

The following is a powerful consequence of \cref{MV:gluing theorem}.

\begin{prop}\cite{DMO25} \label{fiberwise=relative} \label{pointwise:phenomenon}
For a scheme $f: X \to S$ smooth over a Noetherian scheme $S$ of finite Krull dimension, the following properties are equivalent:
\begin{enumerate}
\item The morphism $f:X\to S$ is an $\AA^1$-weak-equivalence in $\Spc_S$.
\item For every point $s$ of $S$ with residue field $\kappa(s)$, the scheme theoretic fiber of $f$ over $s$ is a smooth $\mathbb{A}^1$-contractible $\kappa(s)$-scheme.
\end{enumerate}
\end{prop}
\begin{proof} 
Since our hypotheses differ slightly from those in \cite[Proposition~2.1]{realetale2025}, we will now sketch the argument. The implication $(1) \implies (2)$ follows from \cref{cor:A1-cont-fibers}. Let us now prove the converse: given a smooth scheme $f: X\to S$ over a scheme $S$ of finite Krull dimension, we want to show that if $f_s: X_s \to \Spec \kappa(s)$ is $\AA^1$-weak equivalence in $\Spc_{\kappa(s)}$ for every point $s \in S$, then $f: X\to S$ is a relative $\AA^1$-weak equivalence in $\Spc_S$. We argue by induction on $\dim S$. The case $\dim S=0$ reduces to the observation that a smooth morphism 
\[ 
X\longrightarrow S, \qquad S= \Spec R 
\]
with $R$ Artinian local, is an $\AA^1$-weak equivalence if and only if 
\[
X\times_{S}S_{\mathrm{red}} \longrightarrow  S_{\mathrm{red}} 
\]
is an $\AA^1$-weak equivalence (see \Cref{remark:nilimmersion}). Assume the claim holds for all schemes of dimension $\le d$ and let $\dim S= d+1$. Since the property for $f:X \to S$ to be an $\AA^1$-weak equivalence in $\Spc_S$ is local on the base $S$ with respect to  Nisnevich hypercovers, it suffices to verify the assertion in the case where $S$ is the spectrum of an local ring $R$ of dimension $d$, which we can further assume to be Henselian (this is due to \cite[Proposition A.3]{bachmann2021norms}).
\medskip

Let $j: U \to S$ be the open immersion and $i:\{s\} \to S$ be the complementary closed immersion. The open immersion $j$ provides us an adjoint pair $(j^*\dashv j_*)$ given by $j^*:\HH(S) \to \HH(U)$, its right adjoint $j_*:\HH(U) \to \HH(S)$. Furthermore, since $j$ is an open immersion (whence a smooth morphism), it provides another adjoint pair $(j_\sharp\dashv j^*)$ given by $j^*$ and its left adjoint $j_\sharp:\HH(U) \to  \HH(S)$. Similarly, the closed immersion $i$ provides us an adjoint pair $(i^*\dashv i_*)$ given by $i^*: \HH(S) \to \HH(\kappa(s))$ and its right adjoint $i_*:\HH(\kappa(s)) \to\HH(S)$.
\medskip

Recall from \cref{sect:functoriality} that we have $j_{\sharp} j^* X = j_\sharp (X_U)$, where $X\times_S U \xrightarrow{\pr_2} U\xrightarrow{j} S$ is viewed as a smooth $S$-scheme via the composition $\mathrm{pr}_2:X\times_S U\to U$ with the open immersion $j:U\to S$. We have a canonical map $\widetilde{\pr_1} :j_\sharp j^* X \to X$ that corresponds to the immersion $\mathrm{pr}_1: X\times_S U \to X$ and another canonical map $\widetilde{\pr_2}: j_\sharp j^* X\to j_\sharp U$ that corresponds to the projection $\mathrm{pr}_2:X\times_S U\to U$. On the other hand, the base change by $i: \{s\}\to S$ induces a canonical morphism $\psi: X \to i_*i^* X$ in $\HH(S)$ and we trivially have a canonical map $\widetilde{\psi}: j_\sharp U \to i_*i^* X$. These maps assemble into a commutative square 
\begin{equation}\label{diag:MV-cocartesian}
\begin{tikzcd}
   j_{\sharp} j^*X \arrow [r,"\widetilde{\pr_1}"] \arrow[d,"\widetilde{\pr_2}",swap] & X \arrow[d,"\psi"] \\
   j_\sharp U \arrow[r,"\widetilde{\psi}",swap] & i_*i^*X
\end{tikzcd}
\end{equation}
in $\HH(S)$. Since $U:= X\bs \{s\}$ has dimension strictly smaller than $d$, the restriction $f_U: X_U \to U$ is an $\AA^1$-weak equivalence in $\HH(U)$ by the induction hypothesis and so
\[\begin{tikzcd}
    {X_s} & X & {X_U} \\ {\Spec \kappa(s)} & S & U
   \arrow["{\mathrm{pr}_1}", from=1-1, to=1-2]
   \arrow["{f_s}"', from=1-1, to=2-1]
   \arrow["f"', from=1-2, to=2-2]
   \arrow["{\mathrm{pr}_1}"', from=1-3, to=1-2]
   \arrow["{f_U}", from=1-3, to=2-3]
   \arrow[from=2-1, to=2-2] 
   \arrow["j", hook', from=2-3, to=2-2]
\end{tikzcd}\]
we have that $j_\sharp (f_U):j_\sharp (j^*X) = j_\sharp X_U\to j_\sharp U$ is an $\AA^1$-weak equivalence in $\HH(S)$. It follows that the canonically induced map of presheaves 
\begin{align}\label{map1}
            X \xrightarrow{\simeq} j_{\sharp}U \cup_{j_{\sharp}X_U} X
\end{align}
is an $\AA^1$-weak equivalence in $\Spc_S$. On the other hand, the proof of \cite[Theorem 2.21]{MV99} shows that \cref{diag:MV-cocartesian} is a homotopy co-cartesian and so the canonically induced map of presheaves 
\begin{align}\label{map2}
    j_{\sharp}U\cup_{j_\sharp X_U} X \xrightarrow{\simeq} i_*i^* X
\end{align}
is an $\AA^1$-weak equivalence in $\Spc_S$. Combining \cref{map1} and \cref{map2}, we see that the induced map 
\begin{align}\label{map3}
        X \xrightarrow{\simeq} i_*i^*(X)
\end{align}
is an $\AA^1$-weak equivalence in $\Spc_S$. But then we have $i_*i^*(X) = i_* (X_s)$ and by assumption $f_s:X_s\to \Spec(\kappa(s))$ is an $\AA^1$-weak equivalence in $\Spc_{\kappa(s)}$ which then implies that the motivic space $i_*(X_s)$ is $\AA^1$-contractible in $\Spc_S$ due to functoriality. The conclusion that $f:X\to S$ is an $\AA^1$-weak equivalence in $\Spc_S$ follows from the $\AA^1$-weak equivalence of \cref{map3}.
\end{proof}

\begin{remark} \label{remark:nilimmersion}
Suppose $S$ is a base scheme and let $i\colon S_{\mathrm{red}} \hookrightarrow S$ be the reduction (a closed immersion defined by a nilpotent ideal, whose open complement is empty). Then the pullback $i^\ast\colon H(S)\longrightarrow  H(S_{\mathrm{red}})$ along $i$ is an equivalence of unstable motivic homotopy categories on account of the Morel-Voevodsky localization theorem \cite[Theorem 3.2.21]{MV99}.
\end{remark}

\begin{remark}
The cautious reader might have noticed a subtlety: our model structure (projective version) of $\Spc_S$ is different from that of the model structure used in \cite{MV99} (injective version). While it is a good catch, also note that both of these model structures are give rise to the same $\AA^1$-homotopy category, essentially because the class of $\AA^1$-weak equivalences are the same in both these structures (one may also wish to refer to \cite[Theorem 2.2]{isaksen2005flasque} or \cite[Proposition 8.1]{dugger2001universal}).
\end{remark}

\subsubsection{Formal consequences of $\AA^1$-contractibility}
\begin{example}
If the base field $k$ admits an embedding into $\CC$, then any $\AA^1$-contractible scheme is topologically contractible (\cite[Lemma 2.5]{asok2007unipotent}).
\end{example}

\begin{example}\label{A1-cont-trivial-Pic}
An $\AA^1$-contractible scheme has a trivial Picard group. In fact, every vector bundle on an affine $\AA^1$-contractible scheme has to be trivial (\cite[Corollary 6.2]{asok2007unipotent}).
\end{example}

\begin{example}
Any non-trivial $\AA^1$-rigid scheme (e.g., $\GG_m$) cannot be $\AA^1$-contractible, a fact attributed to the non-triviality of the $\AA^1$-connected component sheaf.
\end{example}

\begin{example}\label{eg:S-point}
Any smooth $k$-variety $X$ of dimension $n$ is $\AA^1$-contractible if and only if $X$ is $\AA^1$-connected and $\pi_i^{\AA^1}(X,x)$ is trivial, for all $0\le i\le n$. This is a natural consequence of \cref{Whitehead-thm}. Hence, such $X$ admits an $S$-point (cf. \cref{S-point}).
\end{example}

\begin{example}
Any $\AA^1$-contractible smooth affine scheme $X$ has the integral motivic cohomology of the ground field, i.e., $\mathrm{H}^{*,*}(X,\ZZ)\simeq \rm{H}^{*,*}(\Spec k, \ZZ)$. In particular, in relation to the higher Chow groups of Bloch, all the Chow groups vanish for these schemes $\rm{CH}^{*}(X) \simeq \Spec k$ (\cite[Corollary 6.1]{asok2007unipotent}).
\end{example}

\subsubsection{Homotopy purity}
Purity isomorphism is an indispensable tool in the light of computation in $\AA^1$-homotopy theory and connects it to algebraic $K$-theory and motivic cohomology. It is an analog of the tubular neighborhood theorem in differential topology and was established by Morel-Voevodsky \cite[Theorem 2.23]{MV99}.

\begin{theorem}\label{purity-theorem}
Let $S$ be any Noetherian scheme of finite Krull dimension. Let $i: Z \hookrightarrow X$ be a closed embedding of smooth schemes over $S$. Then there is a canonical isomorphism in $\mcal{H}_{\bullet}(k)$ of the form
                $$\frac{X}{X \bs i(Z)} \simeq \Th(N_Z X) $$
where $\rm{Th(-)}$ is the Thom space associated to the normal bundle $N_Z X$ of $Z$ in $X$.
\end{theorem}
A quick illustration: suppose a scheme $X$ has an open cover $\{U\}$ for which the normal bundle trivializes, and then we have that 
    $$U/(U \bs (U \cap Z)) \simeq (U \cap Z)_+ \wedge (\AA^n / \Absn) $$
The following lemma will be useful for us later in \cref{chp5}.

\begin{lemma}\label{purity:cor-VoeZ/2}
Let $Z\subset X$ be any closed smooth subscheme of a smooth scheme $X$. Let $f: X\to X'$ be a morphism such that $X'$ and $Z':= f^{-1}(Z)$ are smooth schemes and the induced map $\phi: N_{Z'}X'\to f^*N_ZX$ is an isomorphism. Let $\phi'$ be the Thom space defined by $\phi$. Then the square
\[\begin{tikzcd}
    X'/(X'\bs Z') \arrow [r] \arrow [d] & X/(X\bs Z) \arrow[d] \\ 
    \Th(N_{Z'}X') \arrow [r,"\phi'"] & \Th(N_Z X)
\end{tikzcd} \]
commutes.
\end{lemma}
\begin{proof}
This is due to \cite[Lemma 2.1]{voevodsky2003Z/2}.
\end{proof}

\subsubsection{Motivic excision}
The following theorem portrays an analog to the \emph{excision style} result for $\AA^1$-homotopy groups for smooth schemes, which we shall need in \cref{chp5}.
\begin{theorem}
Let $k$ be an infinite field. Suppose $X$ is $\AA^1$-connected, and $j: U \to X$ is an open immersion of an $\AA^1$-connected scheme whose closed complement is everywhere of codimension $d\ge 2$. Fix a base point $x\in U(k)$. If furthermore $X$ is $\AA^1$-$m$-connected, for $m \ge d-3$, then the canonical morphism
    $$j_*: \pi_i^{\AA^1}(U,x) \to \pi_i^{\AA^1}(X,x)  $$
is an isomorphism for $0\le i \le d-2$ and an epimorphism for $i=d-1$.
\end{theorem}
\begin{proof}
This is due to \cite[Theorem 4.1]{asok2009A1-excision}.
\end{proof}

%-------------------------------------------------------------------------
\section{Milnor-Witt \texorpdfstring{$K$}{K}-Theory}\label{sect:Milnor-Witt}
The objective of this short section is to introduce the fundamentals of the Milnor-Witt $K$-theory, only highlighting the basic features. We redirect the inquisitive readers to \cref{app:MWK} for an expansive background. Recall that the \emph{Milnor Witt $K$-theory} is essentially an extension of Milnor's $K$-theory $K^{\M}_*(F)$, which wisely mixes the generators and relations of Milnor $K$-theory and the Grothendieck Witt ring. It is isomorphic to the Grothendieck ring (\cref{app:Gw-ring}) in degree 0 and the Witt ring (\cref{app:Witt-ring}) in negative degrees.

\begin{defn}
Let $F$ be any arbitrary field with its units $F^{\times}$. The \emph{Milnor-Witt $K$-theory} or the \emph{Milnor-Witt ring}, denoted as $K^{\MW}_*(F)$, is the $\ZZ$-graded associative unital ring freely generated by the formal symbols $[a]$ of degree +1, for $a\in F^{\times}$ and a symbol $\eta$ of degree -1 (called the \emph{Hopf element}) subjected to the following relations: 

\begin{enumerate}\label{MW:relations}
\item $[a][1-a] = 0$, for any $a\ne 0,1$ (Steinberg relation)
\item $[ab] = [a]+[b]+\eta[a][b]$, for any $a,b\in F^{\times}$
\item $\eta[a]=[a]\eta$, for any $a\in F^{\times}$
\item $\eta(\eta[-1]+2) = 0$
\end{enumerate}
\end{defn}

For units $a_i\in F^\times$, the symbols $[a_1,\cdots,a_n]$ usually means the product of $[a_1]\cdot \dots \cdot[a_n]$ (for $i = 1,\dots,n$) and $\eta[a]$ means the product $\eta \cdot [a]$. For $a \in F^\times$, the class $\Lin a\Rin:= 1+\eta[a]$. By construction, the association 
                $$ F \mapsto K^{\MW}_*(F) $$
is functorial. Recall from \cref{app:MWK} that the Milnor Witt $K$-theory is related to the Milnor $K$-theory (\cite{milnor1970algebraic}) by killing $\eta$, that is, $K^{\MW}_*/(\eta) = K^{\M}_*(F)$. By convention, the image of the $n$-tuple product $[a_1,\dots,a_n]$ under this map is denoted by $\{a_1,\dots, a_n\}$. We will need the description of the following element of $GW(F)$, which in particular, states that the Milnor-Witt $K$-theory is $\epsilon$-commutative for $\epsilon:= - \Lin -1 \Rin$. For all $n\ge 0$ (see also \cref{MW:epsilon-defn}), we have that 
    $$n_{\epsilon}:= \sum_{i=1}^{n} \Lin (-1)^{i-1} \Rin \quad \text{and}\quad  (-n)_{\epsilon}:= \epsilon n_{\epsilon}.$$

\subsubsection{Residue homomorphism}
Briefly speaking, there is a homomorphism of (cochain) complex called the \emph{residue homomorphism} with coefficients in the Milnor-Witt sheaf. Let $F$ be a field and $v:F\to \ZZ\cup\{-\infty\}$ be a discrete valuation with residue field $\kappa(v)$ and valuation ring $\mathcal{O}_v$ and choose a uniformizing parameter $\pi = \pi_v$ of $v$ (see \cref{app:residue-homom} for more background).
\begin{theorem}
There is a unique homomorphism of graded Abelian groups 
        $$\partial^{\pi}_v : K^{\MW}_*(F)\to K^{\MW}_{*-1}(\kappa(v))$$
of degree -1 such that $\partial_v^\pi$ commutes wit the multiplication by $\eta$ and 
\begin{enumerate}
\item $\partial_v^\pi([\pi, u_1,\dots,u_n]) = [\Bar{u_1},\dots,\Bar{u_n}]$
\item $\partial_v^\pi([u_1,\dots,u_n])=0$
\end{enumerate}
\end{theorem}
We call such a $\partial_v^\pi$ the \emph{residue homomorphism}. It turns out that the homomorphism $\partial_v^{\pi}$ depends on the chosen valuation $\pi$. This is the reason for \cite{morel2012A1topology} to introduce the theory of \emph{twisted Milnor Witt $K$-theory}. After undergoing certain modifications as explained in \cref{app:twists}, we can now conveniently set the \emph{Milnor-Witt $K$-theory twisted by $L$} as        
    $$K^{\MW}_n(F,L) := K^{\MW}_n(F)\otimes_{\ZZ[{F^\times}]} \ZZ[L^\times]$$
The twisted homomorphism
    $$\partial_v: K^{\MW}_*(F,L)\to K^{\MW}_{*-1}(\kappa(v), (\mathfrak{m}_v/\mathfrak{m}^2_v)^*\otimes L)$$
is given by $\partial_v(\alpha \otimes l) = \partial_v^\pi(\alpha)\otimes \overline{\pi}^* \otimes l$, for $l\in L$ well-defined and is independent of the uniformizer $\pi$ (cf. \cref{app:lemma-residue-free}). Here, $L$ is any line bundle and $(\mathfrak{m}_v/\mathfrak{m}^2_v)^*$ represents the relative tangent space, which, by definition, is the dual of the relative cotangent space $\mathfrak{m}_v/\mathfrak{m}^2_v$ (seen as $\kappa(v)$-vector spaces).

\subsubsection{Motivic Brouwer degree}
An important invariant from the Milnor-Witt $K$-theory is the motivic version of the Hopf-Brouwer map. It gives a description of the non-trivial $\AA^1$-homotopy sheaf of the motivic sphere $\AA^n \bs \{0\}$ (for $n\ge 2$) as a consequence of $\AA^1$-Hurewicz theorem.

\begin{theorem}\cite[Theorem 6.40]{morel2012A1topology}
For $n\ge 2$, we have a canonical isomorphism of strictly $\AA^1$-invariant sheaves 
            $$\pi^{\AA^1}_{n-1}(\AA^n\bs\{0\}) \cong \pi_n^{\AA^1}((\PP^1)^{\wedge n}) \cong K^{\MW}_n $$
\end{theorem}
The advantage of such a theorem is that it helps us to understand the endomorphism ring $[\AA^n\bs\{0\}, \AA^n\bs \{0\}]_{\AA^1}$, for any $n\ge 2$. 

\begin{corollary}\cite[Corollary 6.43]{morel2012A1topology}\label{A1-Brouwer-degree}
Let $(n,i) \in \mathbb{N}^2$ and $(m,j) \in \mathbb{N}^2$ be pairs of integers. For $n\ge 2$, we have a canonical isomorphism:
\[ 
Hom_{\mathcal{H}_{\bullet}(k)}(S^m\wedge (\GG_m)^{\wedge j}, S^n \wedge (\GG_m)^{\wedge i}) = 
\begin{cases*}
0                & if $m < n$.  \\
K^{\MW}_{i-j}(k)   & if $m=n$ and $i>0$  \\
0               & if $m=n$, $j>0$ and $i=0$\\
\ZZ             & if $m=n$ and $j=i=0$
\end{cases*} 
\]
\end{corollary}

Coupling the above results, one obtains that 
        $$[\AA^n\bs\{0\}, \AA^n\bs\{0\}]_{\AA^1} = K^{\MW}_0(k).$$ 
If $f:\Absn \to \Absn $ is morphism in $\Spc_k$, then the class $[f]$ in $K_0^{\MW}(k)\cong \GW(k)$ is called the \emph{motivic Brouwer degree of $f$}. To compute this degree, we can use the following cohomological technique.
\begin{lemma}\label{A1degreemap}
Let $f:\Absn\to \Absn$ be a morphism in $\mathcal{H}(k)$. Then $f$ is an isomorphism if and only if 
        $$f^*: H^{n-1}(\Absn, K^{\MW}_n)\to H^{n-1}(\Absn, K^{\MW}_n)$$
is an isomorphism of Milnor-Witt $K$-theory groups. In other words, the degree of $f$ is 1.
\end{lemma}
\begin{proof}
The reader is redirected to \cite[Lemma 2.2]{DF18}.
\end{proof}

% \afterpage{\blankpage
% \thispagestyle{empty}}

\begin{savequote}
Whoever fights monsters should see to it that in the process he does not become a monster. And if you gaze long enough into an abyss, the abyss will gaze back into you.
\qauthor{"Beyond Good and Evil" by Friedrich Nietzsche}
\end{savequote}
\chapter{Affine Algebraic Geometry}\label{chp3}
\markboth{III Affine Algebraic Geometry}{}

{\small
{\fontfamily{bch}\selectfont
\subsubsection{\textsc{Chapter Summary}}
We dedicate this chapter to appreciating the main results of our script from the perspective of affine algebraic geometry. We will present a survey of some of the fundamental open problems that we will later connect with the context of motivic homotopy theory. This chapter begins with an overview of problems in \cref{AAG:overview} and discusses the quest for the characterization of the polynomial ring in $n$-variables $k[x_1,\dots,x_n]$ or equivalently, the affine $n$-space $\AA^n_k:= \Spec k[x_1,\dots, x_n]$ in \cref{sect:char-affinespaces}. Of particular importance is the Linearization Problem (\cref{linearize-conj}) and its connection with motivic homotopy theory, which is illustrated via an illuminative baby example (\cref{C*-actions:KR3F}). Its grown-up version would be rigorously discussed in \cref{chp5}. We end the section with the famous "Zariski Cancellations" in \cref{sec:Zariskicancellations} in an algebro-geometric context. This chapter solely focuses on the background knowledge from the affine algebraic geometry, with intermediate guiding comments to the chapters further ahead. The reader is also advised to skim the chapter, if need be, and return when the relevant backward references arise.
}}

\subsubsection{Affine Fiber Spaces and Local Triviality}
Flat morphisms of finite presentation with fibers isomorphic to affine spaces have been intensively studied in Affine Algebraic Geometry ("AAG", in short), in connection to the Zariski Cancellation Problem and and the Dolga\v{c}ev-Ve\v{i}sfe\v{i}ler Problem (see e.g. \cite{Kra96} and the references therein for an overview). The main utility of the following definitions is that, from the motivic perspective, such morphisms provide an abundant source of relatively $\AA^1$-contractible schemes.

\begin{defn}\label{def:A1-fib-space}
An \emph{$\AA^n$-fiber space} over a scheme $S$ is a smooth morphism of finite presentation $f:X\to S$ whose fibers over all points $s\in S$
    $$X_s := X\times_S \mathrm{Spec}(\kappa(s)) \xrightarrow{\simeq } \AA^n_{\kappa(s)}$$ 
are isomorphic to the affine space $\AA^n$ over the corresponding residue field $\kappa(s)$.
\end{defn}

\begin{defn}
An $\AA^n$-fiber space $f:X\to S$ as above is called a \emph{locally trivial $\AA^n$-bundle} in the Zariski (resp. Nisnevich, resp. \'etale) topology if every $s\in S$ admits a Zariski (resp. Nisnevich, resp. \'etale)  neighborhood $U$ in $S$ such that 
    $U\times _S X\cong U\times \AA^n$ as schemes over $U$. 
\end{defn}

\begin{remark}
The above notion of $\AA^n$-fiber space corresponds to that of "$\AA^n$-fibration" sometimes called "affine fibration" classically used by some authors in AAG, (see e.g. \cite{asanuma1987polynomial, sathaye1983polynomial,BhD92,DF10}). Some other authors also define $\AA^n$-fibrations (\cite{KM78, KW85}) in this context by requiring only that the general closed fibers are isomorphic to $\AA^n$ over the corresponding residue fields and other fibers are geometrically integral. In the language of Asanuma, morphisms with all closed fibers isomorphisms to $\AA^n_R$ are called pseudo-polynomial $R$-algebras over $n$-variables. Our proposed terminology is intended to clarify the notion and to make a clear distinction from the topological or model categorical notion of fibration. 
\end{remark}

%-------------------------------------------------------------------------------
\section{A Quick Overview}\label{AAG:overview}
\begin{enumerate}
\item \textbf{The Dolga\v{c}ev-Ve\v{i}sfe\v{i}ler Conjecture:} A general form of the conjecture can be stated as follows:
\begin{qstn*}
Let $S$ be the spectrum of any regular local ring and let $f: X\to S$ be an $\AA^n$-fiber space. Then does $f$ necessarily admit a locally trivial $\AA^n$-bundle structure in the Zariski topology?
\end{qstn*}
Anticipating the affirmation of this question, one can show that the following special case admits a positive solution:
\begin{qstn*}
Is every $\AA^n$-fibre space $X\to \AA^m$ isomorphic to the trivial fibre space $\AA^m\times \AA^n\to \AA^m$?
\end{qstn*}
Indeed, due to the result of Bass-Connell-Wright (\cite{BCW76}), we have that every $\AA^n$-bundle over an \emph{affine} variety is isomorphic to a vector bundle. Since the base is an affine space, invoking the Quillen-Suslin theorem \cite{Qui76}, we get that any vector bundle over the affine space $\AA^m$ is trivial.
\medskip

This conjecture was originally stated in \cite[Conjecture 1]{Veisfeiler1974unipotent}, where one can additionally find several other implications and connections with the theory of algebraic groups and $k$-forms of affine spaces. The case when $n= m= 1$ was proven to be true in the same paper (see [\textit{ibid}, Proposition 3.7]). The case when $n=1, m\ge 0$ was answered positively in \cite{KW85} and the case when $n=2,m=1$ was answered positively in \cite{kaliman2001families}. The case when $n=m=2$ was studied in \cite{DF10, blanc2022bivariables}. For more background on this conjecture, we redirect the reader to \cite[\S 4]{gupta2022Zariski}. In the context of our script, we will formulate an answer to this question in relative settings in low dimensions in the language of motivic homotopy theory in \cref{sect:reldim=1} and \cref{sect:reldim=2}.

\item \textbf{The Automorphism Problem:} This problem asks for a complete description of the automorphism ring of the polynomial ring $k^{[n]}$ or, equivalently, to determine the group, $\Aut(\AA^n_k)$, called the \emph{affine Cremona group}. We have a complete understanding in dimension 1: $\Aut(\AA^1_k) \cong (k,+) \rtimes (k^\times,\times)$ and dimension 2: given by the amalgamated product, see \cref{Aut(A2):amalgamated}. The problem seems to be notoriously difficult in full generality with only partial results known in dimensions $\ge 3$. A related question of interest is to study under what conditions for a variety $X$, does the $\Aut(X) \cong \Aut(\AA^n)$ imply that $X \cong \AA^n$. For example, under the condition that $X$ is affine, the group $\Aut(X)$ obtains the structure of an \emph{ind-group} and this enables the characterization of $\Aut(\AA^n)$ as follows \cite[Theorem 1.1]{kraft2017automorphism}:
\begin{theorem*}
If $X$ is a connected affine variety, then every isomorphism of ind-groups between $\Aut(X)$ and $\Aut(\AA^n)$ is induced by an isomorphism $X\xrightarrow{\cong} \AA^n$ of varieties.    
\end{theorem*}
For a generalization and more interesting characterizations, \cite{kraft2021affine} and references therein.

\item \textbf{The Embedding Problem:} For a field $k$ and integers $n\ge m\ge 0$, is every embedding of $\AA^m_k \hookrightarrow \AA^n_k$ necessarily equivalent to the standard embedding of the coordinates? We present a comprehensive picture of this problem in \cref{sect:embedding problem}.

\item \textbf{Linearization Conjecture:} This concerns studying if every $\GG_m$-action on $\AA^n$ equivalent to the linear action. The precise statement and background are given in \cref{linearize-conj}.

\item \textbf{Zariski Cancellation Problem:} ZCP asks if $\AA^n$ is "cancellative", which we shall discuss more in \cref{sec:Zariskicancellations}. 

\item \textbf{Jacobian Conjecture:} For $n\geq 0$, if $f: \CC^n\to \CC^n$ be a polynomial map with $\det f\in \CC^{\times}$, then is $f$ invertible? Stated in colloquial terms,
\begin{center}
If $f: \CC^n \to \CC^n$ is a polynomial map such that its derivative is non-singular at every point, then $f$ is bijective. 
\end{center}
This simple-looking statement has paramount consequences and relations with most problems in AAG. It is listed as one of the foundational problems of this century (\cite[Problem 16]{smale1998mathematical}). The conjecture originates due to Keller in dimension 2 \cite{keller1939ganze} and has been widely open ever since. For more background and references, see \cite{jelonek1992jacobian, rudin1995injective, van1997polynomial}.
\end{enumerate}

%--------------------------------------------------------------------
\section{Characterization of the Affine Space}\label{sect:char-affinespaces}
Fix a field $k$. Recall that by an \emph{affine 
$n$-space over $k$}, we mean the $n$-dimensional affine variety given by the Zariski spectrum of the $n$-dimensional polynomial ring over $k$, i.e.,             
                $$\AA^n_k:= \Spec k[x_1,\dots,x_n].$$
As discussed above, several attempts were and are being made to characterize this object in an algebro-geometric way. This section exhibits some of the key directions and results obtained as a result of the tremendous amount of work from the past decades. This will also be an opportunity for us to provide a taste of where our structure results fit into the overall framework in cohesion with motivic homotopy theory.

\subsection{The AMS Theorem and the Embedding Problem} \label{sect:embedding problem}
The story in this direction begins with a purported proof of the following lemma due to Segre \cite{segre1956corrispondenze} in an attempt to prove the 2-dimensional Jacobian conjecture (See \cite[\S 1]{van2004aroundA-Mtheorem}). The lemma was then independently proven by \cite{moh1975embeddings} and \cite{suzuki1974proprietes}, which is now famously called the \emph{Abhyankar-Moh-Suzuki theorem}, or simply the AMS theorem:

\begin{lemma}\label{lemma-segre}
Let $f(t), g(t)\in \CC[t]$ be non-constant polynomials such that $\CC[f,g] = \CC[t]$, then either $\text{deg} \ f \mid \text{deg}\ g$ (or) $\text{deg}\ g \mid \text{deg}\ f$. 
\end{lemma}

This innocent statement has far-reaching geometric consequences in the world of AAG, including the embedding problem for surfaces \cite[\S 1.6]{moh1975embeddings}. In view of the application to the latter problem, one unanimously refers to the following as the AMS theorem:

\begin{theorem}\label{A-M-S theorem}
For any algebraically closed field $k$ of characteristic zero, every embedding $f: \AA^1_k \hookrightarrow \AA^2_k$ is equivalent\footnote{Two embeddings $f,g:\CC\to \CC^n$ are equivalent if there exists a polynomial automorphism $F:\CC^n\to \CC^n$ such that $F \circ g = f$} to the standard embedding. In other words, $f$ is a coordinate in $\AA^2_k$.
\end{theorem}

Recall that by the standard embedding, we mean $t \mapsto (t,0)$. In the situations as above, we say that the embedding $f$ is \emph{rectifiable} and otherwise, call it \emph{non-rectifiable}. A simple example of a rectifiable embedding is as follows:

\begin{example}
The map $f:\CC\to \CC^2$ given by $t\mapsto (t,t^2)$ is rectifiable. It is rectified by the automorphism $F\in \Aut(\CC^2)$ given by $F(X,Y)\mapsto (X, Y-X^2)$. Note that $F$ is indeed an automorphism with the inverse given by $(X, Y) \mapsto (X, Y+X^2)$.
\end{example}

A zoo of non-rectifiable embeddings into surfaces exists over a field of positive characteristics. In literature, these are also called the \emph{non-trivial lines} and their existence was first demonstrated as early as the 1950s by \cite{segre1956corrispondenze} and later extensively by \cite{nagata1972automorphism}. The one given below is a simple example of a non-rectifiable embedding:
\begin{example}
Let $k$ be any field of characteristic $p>0$. Then consider two polynomials 
\begin{align*}
 f(t)= t^{p^2} \quad \quad g(t)= t^{p(p+1)}+t.
\end{align*}
Observe that we have $k[f(t), g(t)]= k[t]$, but neither of the degrees of $f$ or $g$ divides the other. In the view of equivalence between \cref{lemma-segre} and \cref{A-M-S theorem}, it amounts to saying that the embedding 
\begin{align*}
\AA^1_k \to \AA^2_k, \quad \text{defined as} \quad  t\mapsto (t^{p^2}, t^{p(p+1)}+t)
\end{align*}
is non-rectifiable.
\end{example}

\begin{remark}
The embedding problem is also referred to as the \emph{Epimorphism Problem} in the view of correspondence between affine varieties and their coordinate rings. In fact, the epimorphism problem asks the following: Fix a field $k$ of characteristic zero, let $\phi:k[X, Y]\twoheadrightarrow k[X]$ be an epimorphism of $k$-algebras with kernel $F=ker(\phi)$, then is $k[F]^{[1]} \cong k[X, Y]$?
\end{remark}

\subsubsection{Generalizations of the Embedding Problem}
There are essentially two ways to generalize \cref{A-M-S theorem}. 
\begin{enumerate}
\item For any field $k$ of characteristic zero, is every embedding $\AA^1_k \hookrightarrow \AA^n_k$ rectifiable? When $k=\CC$, due to \cite{jelonek1987embedding, crachiola2008algebraic}, we now have that it is indeed true in dimensions $n\geq 4$. However, over reals $k=\RR$, Shastri \cite{shastri1992knots} constructed an embedding, now called the \emph{Shastri embedding}, of the trefoil knot (which is the simplest non-trivial knot) into $\AA^3_{\RR}$ and proved that it is non-rectifiable (see \cite[P. 1582]{gupta2022Zariski}).

\item Another type of generalization to this problem is now called the \emph{Abhyankar-Sathaye Conjecture}, which states algebraically asks: for $f\in \CC[X_1,\dots, X_n]$ such that $\CC[X]/(f)\simeq \CC^{[n-1]}$, is $f$ a coordinate? i.e., $\CC[X] = \CC[f, f_2,\dots,f_n]$, for some $f_i\in \CC[X]$. In geometric terms, this refers to the rectifiability of $\AA^{n-1}_k \hookrightarrow \AA^n_k$. This conjecture is open for all $n \geq 3$. For a recent survey of known special cases and a general overview, we redirect the readers to \cite[\S 5]{gupta2022Zariski} and the references therein.
\end{enumerate}
In a vast generality, one is tempted to ask the following question, which a priori seems heavily challenging:
\begin{qstn*}
For $2 \leq m \leq n-1$, is every embedding $\AA^m\hookrightarrow \AA^n$ rectifiable?
\end{qstn*}
This question, viewed in correspondence with the epimorphism statement, is solved in various cases. Interested readers can find relevant references in \cite[\S 2]{duttagupta2015epi}.

\subsection{Linearization Problem}\label{linearize-conj} 
The Linearization Problem (LP) aims to understand $\AA^n$ via certain algebraic (reductive) group actions. From time immemorial, studying a space (algebraic variety) by means of (an algebraic) group action has proven to be a highly efficient tool. The main class of groups that one usually exploits is the unipotent groups (e.g., vector groups $\GG_a^n$, for $n\geq 1$), the reductive groups (e.g., algebraic tori $\GG_m^n$, for $n\geq 1$), and the subgroups of the automorphism group of affine spaces $\Aut(\AA^n)$. The LP is of paramount interest in AAG as a part of the quest to characterize $\AA^n_k$ among various topologically contractible varieties. It can be stated as follows:

\begin{qstn*}
Fix a reductive (algebraic) group $G$. For every action $G \curvearrowright \AA^n$, can we find a suitable polynomial change of coordinates such that it becomes equivalent to a linear action?
\end{qstn*}

Recall that for any multiplicative algebraic group $G$ with its group of units, $G^{\times}$ and a $G$-variety $X$, a group action $h: G\times X\to X$ is said to be \emph{linear} if it is conjugate to an action of $GL_n$ on $X$.  Put alternatively, there exists an automorphism $\phi: \AA^n\to \AA^n$ such that for any $\lambda \in G^{\times}$, we have that $\phi \circ h(\lambda) \circ \phi^{-1}$ is linear. 

\subsubsection{In dimensions 1 and 2}
The LP has been settled positively in dimensions 1 and 2, essentially due to the complete understanding we have of the automorphism group $\Aut(\AA^n)$ for $n\le 2$. Recall that a polynomial map $\psi:\AA^n\to \AA^n$ is an isomorphism if and only if it is bijective. In other words, if $\psi= (\psi_1,\dots, \psi_n)$ is given by a set of regular functions, then the above amounts to saying that the set of polynomials $\{\psi_j\}$ generates the polynomial ring $k[x_1,\dots,x_n]$. In particular, the LP takes its simplest form in dimension 1, since the group $\Aut(\AA^1)$ is completely determined by the affine group via  $x \mapsto ax+b$, for $a\in k^\times$ and consequently $\Aut(\AA^1) \cong \GG_a\rtimes \GG_m$, for the affine linear group $\GG_a$. Thus, in particular, every $\GG_m$-action on $\AA^1$ is linearizable. The situation has already become complicated for surfaces: it holds in dimension 2 essentially due to the elegant characterization of $\Aut(\AA^2)$ due to the work of \cite{Kulk1953polynomial, gizatullin1975automorphisms} that we describe below.

\begin{theorem}\label{Aut(A2):amalgamated}
For any base field $k$, the group $\Aut(\AA^2)$ is the amalgamated product $\mathcal{A}_2 *_{\mathcal{B}_2} \mathcal{J}_2$. Here, $\mathcal{A}_2$ and $\mathcal{J}_2$ are the subgroups of $\Aut(\AA^2)$ representing affine transformations and triangular transformations and $\mathcal{B}_2$ denotes their intersection.
\end{theorem}

In particular, this answers the LP positively in dimension 2.
\begin{corollary*}
Every algebraic subgroup of $\Aut(\AA^2)$ is either conjugate to a subgroup of $\mathcal{A}_2$ or $\mathcal{J}_2$. In particular, every reductive subgroup of $\Aut(\AA^2)$ is conjugate to a subgroup of $GL_2$.
\end{corollary*}

\begin{corollary*}
Every locally nilpotent derivation on $\AA^2$ is equivalent to the one given by $f(x)\frac{\partial}{\partial y}$, for a polynomial $f(x)$. In other words, every $\GG_a$-action on $\AA^2$ is equivalent to one of the form $t\cdot (x,y)\mapsto (x,y+ tf(x))$.
\end{corollary*}
\begin{remark}

The subgroups of $\Aut(\AA^n)$ that are exclusively generated by $\mathcal{A}_n$ and $\mathcal{J}_n$ are called \emph{tame} automorphism. The difficulty in comprehending the $\Aut(\AA^n)$ is that we do not know if all of its subgroups are tame. Moreover, this characterization using the amalgamated product does not generalize to higher dimensions. For $ n\geq 4$, there exists a non-linearizable $\GG_m$-action on $\AA^n$ over fields of positive characteristics due to Asanuma (\cite{asanuma1994nonlinearizable}).
\end{remark}
Interestingly enough, the LP threw a curveball in dimension 3.

\subsection{\texorpdfstring{$\CC^{\times}$}{C}-actions on the Koras-Russell Threefolds}\label{C*-actions:KR3F}

In dimension $n=3$, \cite{koras1989linearization} studied a certain smooth affine variety defined by the explicit polynomial equation 
    $$\mathcal{R} := \{x^2 z + x+y^3+t^2=0\} \subset \AA^4_{\CC}$$
then called the \emph{Russell cubic}, eponymous to Russell who observed its similarities with $\CC^3$. There is a non-linearizable $\CC^{\times}$-action on $\mathcal{R}$ given by 
\begin{equation}
\sigma: \CC^{*}\times \mathcal{R} \to \mathcal{R}, \quad (\lambda, (x,y,z,t)) \mapsto (\lambda^6 x, \lambda^{-6}y, \lambda^{3}z,\lambda^2 t)
\end{equation}

\subsubsection{Why is this action non-linearizable?} 
It is a well-known fact that the existence of a fixed point set is a necessary condition for an action to be linearizable. This follows as a fact that the fixed point set is a subset of $\AA^m$ (for some $m$) and whose embedding into $\AA^n$ has to be equivalent to the linear one. Moreover, such a fixed point set also has to be connected. But if the action $\sigma$ defined above is linearizable, then so is its restriction to the sixth roots of unity $\mu_6 \subset \CC^{\times}$. But we notice that the corresponding fixed point is given by $\{x(1+xz)=0\}\subset \AA^2$, which is disconnected! It is important to note that $\mathcal{R}$ is topologically contractible (cf. \cite{zaidenberg2000exotic}) and hence topologically indistinguishable from $\AA^3_{\CC}$. To make things spicier, it was also shown \cite[Lemma 5.1]{kaliman2005actionsofC*C+} that $\mathcal{R}$ is diffeomorphic to $\RR^6$ as a real manifold, whence $\mathcal{R}$ has an incredible property of having the logarithmic Kodaira dimension equal to $\kappa(\mathcal{R})= -\infty$, which is that of $\AA^n_k$. 
\medskip 

So, the curveball here are the following: a) if $\mathcal{R}\cong \CC^3$, then this produces a counter-example to the LP in dimension 3 as $\sigma$ gives rise to a non-linearizable $\CC^{\times}$-action on $\CC^3$, b) if $\mathcal{R} \ncong \CC^3$, then this produces a counter-example to the characterization of $\AA^3$. Hence, the study of these affine varieties played a pivotal role in AAG. This motivated the authors in \cite{KR97} to completely classify all the prototypes similar to that of $\mathcal{R}$ that admit hyperbolic\footnote{those that have a unique fixed point and the weights of the induced linear actions on the tangent space at this point are all non-zero, and their product is negative.} $\CC^{\times}$-actions and to construct a family of threefolds that allegedly would produce a counter-example to this conjecture. These are called the \emph{Koras-Russell threefolds} and are broadly divided into three kinds depending on the number of $\GG_a$-actions that these threefolds admit. The Koras-Russell threefolds of \emph{first kind} admit more than one $\GG_a$-action and are given by the following equation:
\begin{equation}\label{KR3FIkind:equation}
\KK := \{x^m z + x+y^r+t^s=0 \}
\end{equation}
where $m,r,s \geq 2$ integers with gcd $(r,s)=1$. Note that Koras-Russell cubic threefolds $\mathcal{R}$ are a particular example of $\KK$ with $r=3$ and $s=2$. Those that are form the Koras-Russell threefolds of \emph{second kind} admit essentially only one $\GG_a$-action (\cite[\S 2]{petitjean2016autKR3FIIkind}) and are given by the following equation:
\begin{equation}\label{KR3FIIkind:equation}
   \Tilde{\KK} :=  \{x+ (x^n+z^k)^l y +t^s = 0\} \subset \AA^4
\end{equation}
where $n,k,s\geq 2,\ s> k,\ l \geq 1,\ \text{gcd}\ (k, s)=1 = \text{gcd}\ (k,n)$. Finally, there are other kinds of Koras-Russell threefolds (sometimes called the \emph{third kind} in the literature) that only admit trivial $\GG_a$-action (for more general types, we refer to \cite{KKMLR97}). It is important to note that all of these three kinds admit a hyperbolic $\GG_m$-action with a unique fixed point, by design! As far as this script is concerned, we shall only focus on the novel properties of the first kind \cref{KR3FIkind:equation}. The LP stood open in dimension 3 for some time, as it was difficult to conclude if the Koras-Russell threefolds were actually isomorphic to $\AA^3$ in the interest of settling the LP. This was finally settled by the work of Makar-Limanov \cite{makar1996hypersurface} by inventing a purely algebraic invariant (which he called \emph{AK-invariant} but later became an eponym) that emerged solely for this purpose!

\subsubsection{The Makar-Limanov and Derksen invariants} 
For any commutative ring $R$, the \emph{Makar-Limanov invariant} is defined as follows:
\begin{equation}
    \ML(R):= \bigg\{\bigcap \Ker(\partial)\mid \partial \in \LND(R) \bigg\} \subseteq R
\end{equation}
Here, $\partial$ is the locally nilpotent derivation which corresponds to free $\GG_a$-actions. In particular, for an affine variety $(X,\mathcal{O}_X)$, the $\ML(X)\subset \Gamma(X,\mathcal{O}_X)$ corresponds to all those regular functions on $X$ that are $\GG_a$-invariant. As a first instance, we have that $\ML(\AA^3_{\CC})= \CC$ as there are no non-constant regular functions on affine spaces. On the other hand, he proved that $\ML(\mathcal{K}) = \CC[x]$ which distinguished it from $\AA^3_{\CC}$. Thus, the LP was finally settled positively for threefold in 1997 (\cite{KKMLR97}). Following this, another similar invariant, called the \emph{Derksen invariant}, denoted by $\mathcal{D}(R)$, emerged (\cite{derksen1997constructive}) which for a ring $R$. It is defined as the sub-algebra generated by 
\begin{align}
    \mathcal{D}(R):= \bigg\{\bigcap \Ker \partial \mid \partial \in \LND(R)\bs \{0\} \bigg\} \subseteq R
\end{align} 
Alternatively, $\mathcal{D}(R)$ is the sub-ring of $R$ generated by the rings of invariants of all the non-trivial exponential maps of $R$. Though they seem closely related, neither of them depends on the other (cf. \cite{crachiola2003derksen}). The following calculations shows that $\KK$ is not isomorphic to $\AA^3_k$:

\begin{example}\label{KR3F-not-A3:eg}
Let $k$ be a field and let $R = k[x_1,\dots,x_n]$. Then $\ML(R) = k$ for all $n$. When $n \ge 2$, we have $\mathcal{D}(R) = R$\footnote{as an exception when $n=1$, $\mathcal{D}(R)=k$}. For the Koras-Russell threefolds $\KK$ with its coordinate ring $\Gamma(\KK)$, we have that $\ML(\KK) = k^{[1]}\ne k$ (\cite[Corollary 9.7]{freudenburg2006algebraic}) and $\mathcal{D}(\KK) = k^{[3]} \neq  \Gamma(\KK)$ which distinguishes it from $\AA^3_k$ ( \cite[Theorem 9.6]{freudenburg2006algebraic}).
\end{example}

Later, Kaliman \cite{kaliman2002polynomials} provided the following elegant argument to show that $\AA^3_{\CC}$ is non-isomorphic to $\mathcal{K}$:
\begin{theorem*}
For any field $k$ of characteristic zero, if the general fibers of a regular function $\AA^3_k\to \AA^1_k$ are isomorphic to $\AA^2_k$, then so is every other fiber.
\end{theorem*}
This was later extended to all fields of characteristic zero by \cite[Theorem 4.2]{Daigle2009note} via Lefschetz principle arguments. For instance, note that the generic fiber $\pr_x^{-1}(0)$ of $\KK$ with respect to the projection $\pr_x: \KK\to \AA^1$ is the cylinder on a cuspidal curve $\AA^1\times \{y^2+t^3\}$ and is singular. In particular, it cannot be isomorphic to $\AA^2$. A similar characterization was given by Miyanishi, which was later shown to fail in higher dimensions (see \cite[Theorem 1]{kaliman2000miyanishi} and references therein).

\begin{remark}
In the context of the embedding problem, the question of whether threefolds like $\KK$ were isomorphic to $\AA^3$ was then considered a potential threat to Abhyankar-Sathaye conjecture (see \cite[\S 2.2]{duttagupta2015epi}).
\end{remark}

\subsubsection{Where does motivic homotopy theory fit in here?} 
From the motivic viewpoint, the affine $n$-space $\AA^n$ forms primordial examples of contractible schemes (cf. \cref{A1-contr:egs}). Two varieties that are not isomorphic in $\mathcal{H}(k)$ (or non-$\AA^1$-weakly equivalent in $\Spc_k$) as spaces cannot be isomorphic as $k$-varieties. In particular, if $\KK$ was not $\AA^1$-contractible, then it cannot be isomorphic to the affine space $\AA^3$. Towards this pursuit, Asok \cite{asok2011a1} initiated a program to show that $\KK$ is not $\AA^1$-contractible and hence non-isomorphic to $\AA^3_{\CC}$. The idea was to produce possible obstructions in the $G$-equivariant homotopy theory. But as a key step in the overall program failed (cf. \cite[\S 1]{HKO16}), it was sadly unable to show that $\KK$ is not $\AA^1$-contractible. In the same paper (\textit{ibid}) Hoyois-Krishna-{\O}stv{\ae}r, in fact, proved the diametrically opposite result that $\KK$ is $\AA^1$-contractible in the stable homotopy category $\SH(\CC)$. Finally, Dubouloz-Fasel managed to prove the unstable $\AA^1$-contractibility of $\KK$ over fields of characteristic zero \cite{DF18}. We will give a comprehensive overview of these results in \cref{sec:KR-prototypes} and prove that these results can further be generalized from fields to a general base scheme. In all, it is an astonishing fact that this family of threefolds is completely invisible to the lens of powerful algebraic tools such as motivic cohomology (or equivalently, higher Chow groups) and carries trivial vector bundles, just like $\AA^3$ (\cite[\S 3]{HKO16}).

\begin{remark}
In relation to the Embedding Problem, it has been shown by \cite{dubouloz2010inequivalent} that there are inequivalent embeddings of Koras-Russell threefolds into $\AA^4$.
\end{remark}

\subsection{Cancellations in Algebraic Geometry}\label{sec:Zariskicancellations}
The Cancellation Problem, in a broad sense, is a statement about a certain property of affine $k$-algebras in the following sense: Let $A$ be any (say, finitely generated) affine algebra over a field $k$. Then for any $k$-algebra $B$,
 $$\text{does}\quad   A[x]\cong_k B[x]\quad  \text{imply}\quad  A\cong_k B\ ?$$
If this holds, we say that such a $k$-algebra $A$ is \emph{cancellative}. A particularly important version of this problem is when $A$ is the $n$-dimensional polynomial ring $A = k[x_1,\dots,x_n]$ and its cancellation is of paramount importance in AAG. This is because, in the light of Hilbert's Nullstellensatz, one can then transport this problem to the realm of algebraic geometry. In the ring theoretical context, it was posed by Zariski - hence, also referred to as the Zariski Cancellation Problem (ZCP) - at the Paris colloquium in 1949 (\cite{nagata1967theorem}, \S 5) stated as follows:

\begin{question}
Let $k$ be a (resp. an algebraically closed) field. Then for an $k$-algebra $B$ (resp. for an affine $k$-scheme $Y$), does $B\otimes k^{[1]}\cong k^{[n+1]}$ (resp. $Y \times \AA^1_k \cong \AA^{n+1}_k$) imply $B \cong k^{[n]}$ (resp. $Y \cong \AA^n_k$)?
\end{question}

Here, by $k^{[n]}$, we mean the polynomial ring in $n$-variables. An interesting counter-example when $k= \RR$ was constructed by \cite{hochster1972nonuniqueness} on the projective module defined by the tangent bundle over the real sphere is stably free but not free. The ZCP has been proven true over an arbitrary fields in dimension 1 and 2: For affine curves, it is due to \cite{abhyankar1972uniqueness} and for affine surfaces over a field of characteristic zero, it is due to \cite{miyanishi1980affine, fujita1979zariski} and over perfect fields due to \cite{russell1981affine} and more generally, due to \cite{bhatwadekar2015cancellation}. An alternative simplified proof over an arbitrary field was given by \cite{crachiola2008algebraic}. The problem now remains widely open in dimensions $n\geq 3$ over fields of characteristic zero. In dimension 3, the existence of the family of Koras-Russell threefolds of first kind $\KK$ as discussed above (\cref{KR3FIkind:equation}) is seen as a potential counter-example to ZCP. Dubouloz \cite{dubouloz2009cylinder} proved that all the cylinders over the Koras-Russell threefolds have trivial Makar-Limanov invariants,
    $$\ML(\KK \times \AA^1)\cong k \cong \ML(\AA^3) \cong \ML(\AA^n)$$ 
If this were proven otherwise, it would have answered that $\KK\times \AA^1$ cannot be isomorphic to $\AA^4$ and thus, closing the conjecture; but from the current conclusion, it remains to prove or disprove the isomorphism $\KK\times \AA^1 \cong \AA^4$.

\subsubsection{ZCP in Positive Characteristics}\label{ZCP:positive-char}
Due to the works of Gupta \cite{gupta2014ZCPpaper-1}, ZCP admits a negative solution in every dimension $n \geq 3$ over fields of positive characteristics. This counterexample essentially comes from another counterexample that was constructed in the context of attacking the Linearization Problem in positive characteristics. In particular, \cite[Theorem 2.2]{asanuma1994nonlinearizable} produced a $k$-algebra $A$ that is stably isomorphic $\AA^n$ and carrying a \emph{non-linearizable} $\GG_m$-action on $A$, for all $n \geq 4$ over an infinite field $k$ of characteristic $p>0$. The main ingredient towards the construction of such an $k$-algebra $A$ emerges from his following theorem:
\begin{theorem*}
Let $k$ be a field of characteristic $p>0$ and
    $$A = k[X,Y,Z,T]/\langle X^m Y + Z^{p^e}+T+T^{sp}\rangle$$
where $m,e,s$ are positive integers such that $p^e\nmid sp$ and $sp\nmid p^e$. Let $x$ denote the image of $X$ in $A$. Then $A$ satisfies the following two properties:
\begin{enumerate}
\item $A^{[1]}\cong_{k[x]} k[x]^{[3]} \cong_k k^{[4]}$
\item $A\ncong_{k[x]} k[x]^{[2]}$.
\end{enumerate}
\end{theorem*}
Using the techniques developed by Makar-Limanov and Crachiola, \cite{gupta2014ZCPpaper-1} proved that such an $k$-algebra $A$ indeed produces a counter-example to ZCP in dimension 3, and soon after also extended this to all dimensions $n \geq 4$ \cite{gupta2015survey} and thus completely settled the ZCP over fields of positive characteristics. In eponym, we call these affine varieties the \emph{Asanuma-Gupta varieties}. It is crucial to note that the major tool that was exploited in all these instances is the \emph{existence of the non-trivial line} and its generalization over fields of positive characteristics. Let us recall that a \emph{non-trivial line} is polynomial $f\in k^{[2]}$ such that 
 $$k^{[2]}/(f) \cong k^{[1]}\quad \text{but}\quad  k[f]^{[1]} \ncong k^{[2]}$$
(that is, $f$ cannot be made as a coordinate). For the Gupta's counter-example, the polynomial equation 
            $$f(Z,T):= Z^{p^e}+T+T^{sp}$$
with $sp\nmid p^e, \ p^e\nmid sp$ served as a non-trivial line. Such a line is an example of an embedding of $\AA^1 \hookrightarrow \AA^2$ that is not the standard inclusion of coordinates.

\begin{remark}
As a result of the structure theorem obtained in \cite{DMO25}, we will see in \cref{sect:reldim=1} and \cref{sect:reldim=2} that the variant of the Zariski cancellation works over a more general base scheme in relative dimensions up to 2.
\end{remark}

\subsection{Asanuma's Stable Structure Theorem}\label{Asanuma-stable-structure}
Recall that the \emph{Dolga\v{c}ev-Ve\v{i}sfe\v{i}ler Conjecture} asks if over a regular local ring $R$, an $\AA^n$-fiber space can be equipped with a Zariski locally trivial $\AA^n$-bundle structure. As noted previously, no complete cases have been known from dimensions $n \geq 3$. However, from dimensions $n \ge 4$, the conjecture fails over fields of positive characteristic, which can be seen as a consequence of either of the following facts:
\begin{itemize}
\item presence of non-trivial lines (equivalently, non-rectifiable embeddings of $\AA^1$),
\item presence of non-linearizable $\GG_m$-actions on $\AA^n_k$,
\end{itemize}
We will state and review this aspect from the motivic viewpoint in \cref{sect:reldim=2}. We shall conclude this chapter by mentioning a crucial output of Asanuma, called the \emph{Stable Structure theorem} (\cite[Theorem 3.4]{asanuma1987polynomial}) stated in the AAG language. For this, we shall quickly recall the following definition:

\begin{defn}
For smooth $k$-schemes $X$ and $Y$, we say that $X$ is \emph{stably $k$-isomorphic} to $Y$, if there exists $i,j \geq 0$ such that 
        $$X\times_k \AA^i \cong Y \times_k \AA^j. $$ 
In ring-theoretic terms, we say that a $k$-algebra $A$ is \emph{stably polynomial} if for some $l,m \geq 0$, $A\otimes_k k^{[i]} \cong k^{[j]}$.
\end{defn}

When translated to motivic homotopy theory, it provides us with an elegant technique to produce a plenitude of $\AA^1$-contractible schemes relative over a base scheme. The following simple-looking commutative algebra fact expresses a potential limitation of the motivic homotopical techniques towards the classification of problems such as Dolga\v{c}ev-Ve\v{i}sfe\v{i}ler Conjecture or equivalently, the triviality of (Zariski) locally trivial $\AA^n$-bundles over affine spaces:

\begin{theorem}\label{Asanuma-Stab-Strc:thm}
Let $R$ be a Noetherian ring and let $A$ be a finitely generated flat $R$-algebra such that $A\otimes \kappa(p)\cong\kappa(p)^{[n]}$ for any prime ideal $p$ of $R$. Then $A$ is an $R$-subalgebra (up to an isomorphism) of a polynomial ring  $R^{[m]}$ (for some $m$) such that 
$$A^{[m]}\cong Sym_{R^{[m]}}(\mathfrak{P}\otimes R^{[m]}),$$
where $\mathfrak{P}$ is a projective $A$-module.
\end{theorem}

\cref{Asanuma-Stab-Strc:thm} asserts that over a Noetherian ring $R$, any finitely generated flat $R$-algebra $A$ with all fibers isomorphic to a polynomial ring in $m$-variables is a stably polynomial ring, that is, $A^{[k]}\cong R^{[m]}$, for some integers $k,m\ge 0$. Translated to algebro-geometric language, it asserts the following: 

\begin{theorem*}
Over a locally Noetherian scheme $S$, any \emph{affine} morphism $f: X\to S$ that is an $\AA^n$-fiber space is Zariski locally \emph{stably} isomorphic to a locally trivial $\mathbb{A}^m$-bundle (for some positive integer $m\ge 0$).
\end{theorem*}
This result is critical in the motivic homotopy theory illustrated by the following fact \cite[Lemma 3.3.3]{asok2021A1}:
\begin{lemma}\label{stableiso:A1-weq}
Stably isomorphic smooth varieties are $\AA^1$-weakly equivalent.
\end{lemma}

\begin{corollary*}
If $f: X\to S$ is any affine morphism that is also an $\AA^n$-fiber space over a Noetherian scheme $S$, then $X$ is stably isomorphic to a Zariski locally trivial $\AA^{n+m}$-bundle (for some $m$). In particular, in the view of \Cref{stableiso:A1-weq} $X$ is a $\AA^1$-contractible scheme over $S$.
\end{corollary*}

Thus, \cref{Asanuma-Stab-Strc:thm} provides us a fruitful way to harness several families of $\AA^1$-contractible schemes. Crucially, observe that from the motivic standpoint, this theorem precisely portrays that the $\AA^1$-homotopy theoretical techniques alone cannot distinguish $\AA^n$-fiber spaces from the Zariski locally trivial ones.

% \afterpage{\blankpage}

\begin{savequote}
Have you also learned that secret from the river that there is no such thing as time? That the river is everywhere at the same time, at the source and at the mouth, at the waterfall, at the ferry, at the current, in the ocean and in the mountains, everywhere, and that the present only exists for it, not the shadow of the past nor the shadow of the future.
\qauthor{"Siddhartha" by Hermann Hesse.}
\end{savequote}
\chapter{Relative $\AA^1$-Contractibility of Smooth Schemes}\label{chp4}
\markboth{IV Relative $\AA^1$-Contractibility}{}

The results presented in this chapter are sourced from the article \cite{DMO25}.

{\small
{\fontfamily{bch}\selectfont
\subsubsection{\textsc{Chapter Summary}}
Recall that a smooth $k$-scheme $X$ is said to be \emph{$\AA^1$-contractible} if the structure map $X \to \Spec k$ is an $\AA^1$-weak equivalence in $\Spc_k$. It turns out that in relative dimensions $n<3$, the notion of $\AA^1$-contractibility is strong enough to characterize $\AA^n$. In this chapter, we provide a comprehensive explanation of this statement and demonstrate that the $\AA^1$-homotopy theory detects Zariski-local trivial $\AA^n$-bundles in relative dimensions $n < 3$. We begin with the case of \'etale schemes over an arbitrary base in \cref{sect:Rel-etale schemes}. In \cref{sect:dim=1}, we study the relative $\AA^1$-contractibility of smooth curves over a normal scheme and provide a characterization of the affine 1-space relative to this setting. We systematically study the relative $\AA^1$-contractibility of smooth surfaces in \cref{sect:reldim=2} and provide a characterization of the affine 2-space over certain Dedekind schemes. In both these settings, we obtain a positive solution to the relative variant of the Zariski Cancellation Problem as a consequence. Following this, we analyze and review the common knowledge for higher relative dimensions $n\geq 4$ in \cref{sect:higherdimensions} and prove that $\AA^n$-fiber spaces provide abundant examples of relatively $\AA^1$-contractible spaces.
}}

%------------------------------------------------------------------------
\section{Relative \texorpdfstring{$\AA^1$}{A1}-Contractibility of \'Etale Schemes}\label{sect:Rel-etale schemes}

We begin by establishing the structure theorem in relative dimension 0, i.e., for \'etale schemes, in terms of $\AA^1$-contractibility and enhance this to arbitrary base schemes. An important consequence here is the identification of \'etale schemes with that of $\AA^1$-rigid schemes. The case of zero-dimensional smooth $k$-schemes is equivalent to that of the characterization of \'etale $k$-algebras. This gives us the following straightforward statement.

\begin{lemma}\label{0-dim prop}\cite{DMO25}
Let $X$ be an \'etale scheme over an arbitrary field $k$. Then the canonical map $f: X \to \Spec k$ is an $\AA^1$-weak equivalence in $\Spc_k$ if and only if $f$ is an isomorphism in $\Sm_k$.
\end{lemma}
\begin{proof}
One implication is obvious. Assume that $X$ is $\AA^1$-contractible. Since $f: X\to  \Spec k$ is \'etale, $X$ is a disjoint union of spectra of finite separable field extensions of $k$. As every $\AA^1$-contractible space is $\AA^1$-connected (in particular, connected as a topological space), we have that $X \cong  \Spec k'$, for some finite separable extension $k'$ of $k$. But by \Cref{S-point}, any $\AA^1$-contractible scheme over $k$ admits a section, and so $k'\cong k$.
\end{proof}

\begin{theorem}\label{0-dim:etale-schemes-iso}\cite{DMO25}
Let $X$ be an \'etale scheme of finite presentation over an arbitrary base scheme $S$. Then the canonical morphism $f: X\to S$ is a relative $\mathbb{A}^1$-weak equivalence in $\Spc_S$ if and only if $f$ is an isomorphism in $\Sm_S$.
\end{theorem}
\begin{proof}
One direction is obvious. Assume that $f:X\to S$ is a relative $\AA^1$-weak equivalence in $\Spc_S$. For every point $s\in S$ with residue field $\kappa(s)$, it follows from \Cref{cor:A1-cont-fibers} that $f_s := \pr_2: X_s =X\times_{S} \Spec \kappa(s)\to \Spec \kappa(s)$ is a relative $\mathbb{A}^1$-weak equivalence in $\Spc_{\kappa(s)}$, whence, by \Cref{0-dim prop}, that the morphism $f_s$ is an isomorphism of schemes. Thus, $f: X\to S$ is a bijective morphism with trivial, hence radicial, residual extensions from which it follows that $f: X \to S$ is universally injective. Since $f$ is on the other hand \'etale, it is an open immersion (see \cite[\href{https://stacks.math.columbia.edu/tag/02LC}{Tag 02LC}]{stacks-project}). But as $f$ is also surjective\footnote{the surjectivity follows from the fact that the fibers cannot be empty for $\AA^1$-contractible schemes}, the conclusion follows.
\end{proof}

Recall from \cref{defn:A1-rigid} that $\AA^1$-rigid schemes are the analogs of discrete objects in the motivic homotopy theory. The authors in \cite[Example 2.1.10]{asokmorel2011} observe that any 0-dimensional scheme over a field is $\AA^1$-rigid. This observation admits a generalization over a base scheme as follows:

\begin{prop}\label{0dim:rigid}\cite{DMO25}
An \'etale scheme of finite presentation $f: X\to S$ over an integral scheme $S$ is $\AA^1$-rigid over $S$.
\end{prop}
\begin{proof}
Let $h:U\to S$ be any smooth $S$-scheme and let $\psi:\AA^1_U\to X$ be an $S$-morphism. We can assume without loss of generality that $U$ is connected. Let $\pi:\AA^1_U\times_U \AA^1_U\to U$ be the flat morphism defined as the composition of the  projections $\mathrm{pr}_1:\AA^1_U\times_U \AA^1_U\to \AA^1_U$ and $\AA^1_U\to U$, and consider the following cartesian diagram 
\[
\begin{tikzcd}
 Z=X \times_{X\times_S X} (\AA^1_U\times_U \AA^1_U) 
  \arrow[r,""] \arrow[d,swap,"j"] 
&  X \arrow[d,"\Delta"] \\
 \AA^1_U\times_U \AA^1_U \arrow[r,"\psi\times \psi"]
 & X\times_S X
 \end{tikzcd}
\]
where $\Delta:X\to X\times_S X$ is the diagonal morphism. Since $f: X\to S$ is separated, $\Delta$ is a closed subscheme of $X\times_S X$ and hence $j: Z\to \AA^1_U\times_U \AA^1_U$ is a closed immersion as well. Since $\pi$ is flat, whence an open morphism, the image by $\pi$ of the Zariski open subset $\AA^1_U\times_U \AA^1_U\setminus j(Z)$ is a Zariski open subset $U'$ of $U$ consisting by construction of all points $u$ for which the morphism $\psi|_{\AA^1_u}:\AA^1_u\to X$ is not constant. The fiber $X_\eta$ of $f: X\to S$ over the generic point $\eta$ of $S$ being an \'etale scheme over the field of functions $K$ of $S$, it is $\AA^1$-rigid by the observation made above. It follows that  there exists a non-empty Zariski open subset $S_0\subset S$ such that letting $U_0=h^{-1}(S_0)$, the restriction $\psi|_{\AA^1_{U_0}}:\AA^1_{U_0}\to f^{-1}(S_0)$ of $\psi$ factors through $U_0$. This implies that $U'\subset U\setminus U_0$, whence that $U'=\emptyset$ (because $U'$ is open in $U$, $U\setminus U_0$ is closed in $U$ and $U$ is connected). Thus, $\psi:\AA^1_U\to X$ factors through the projection $\AA^1_U\to U$, which shows that $X$ is $\AA^1$-rigid. 
\end{proof}

\begin{question}
Does the \cref{0dim:rigid} still hold if $X$ is not separated?
\end{question}

%-----------------------------------------------------------------
\section{Relative \texorpdfstring{$\AA^1$}{A1}-Contractibility of Smooth Curves} \label{sect:dim=1}
In this section, we show that being affine is a consequence of $\AA^1$-contractibility for smooth curves (\cref{A1contracurves:affine}). We then present an extensive proof of the fact that $\AA^1$ is the only smooth $\AA^1$-contractible curve over any field (\cref{1-dim theorem}). Finally, we extend this result over a reasonably arbitrary base scheme in relative dimension 1 (\cref{1-dim DD}). Let us begin by recalling the following definition:

\begin{defn}\label{defn:k-form}
Let $Y$ be a $k$-variety. We say that $Y$ is a \emph{$k$-form} of $\AA^n_k$ if for some (say, finitely generated) field extension $L/k$ the scheme-theoretic base change            
$$Y_L:= Y \times_k \Spec L \cong \AA^n_L \cong \Spec L[x_1,\dots,x_n]$$ 
We then say that this $k$-form $Y$ is trivial if $L= \Spec k$ and non-trivial otherwise.
\end{defn}

\subsection{$\AA^1$-contractibility over fields}
Let us fix a field $k$. We begin with the following lemma that says that the $\AA^1$-contractibility of a curve $C$ necessarily implies that $C$ is affine.

\begin{lemma}\label{A1contracurves:affine}\cite{DMO25}
Any smooth $\AA^1$-contractible curve $C$ over a field is affine.
\end{lemma}
\begin{proof}
Since $C$ is a smooth curve, it is either affine or projective. On the contrary, suppose that $C$ is projective. By the classification result of smooth proper (equivalently, projective in dimension 1) curves over a field $k$ up to $\AA^1$-weak equivalence \cite[Proposition 2.1.12]{asokmorel2011}, we have that the only $\AA^1$-connected curve is the projective line $\mathbb{P}^1_k$. But clearly $\PP^1_k$ cannot be $\AA^1$-contractible (for e.g., $\Pic(\PP^1) \cong \ZZ \ncong 0$; cf. \cref{A1-cont-trivial-Pic}). This uniqueness of $\PP^1_k$ among $\AA^1$-connected proper curves follows from the observation that other candidates, such as proper curves of positive genus or non-trivial $k$-forms of $\mathbb{P}^1_k$ or $\GG_{m,k}$, are all $\AA^1$-rigid varieties, hence cannot be $\AA^1$-contractible. Thus, $C$ has to be affine.
\end{proof}

As a consequence of this, one can uniquely characterize the affine line among $\AA^1$-contractible curves over arbitrary fields (see also \cite[Claim 5.7]{asok2007unipotent}).

\begin{theorem}\label{1-dim theorem}
A smooth curve $C$ over an arbitrary field $k$ is $\AA^1$-contractible in $\mathcal{H}(k)$ if and only if $C$ is isomorphic to the affine line $\AA^1_k$.
\end{theorem}
\begin{proof}
One implication is obvious: the affine line $\AA^1_k$ is $\AA^1$-contractible over any field. Suppose that $C$ is a smooth $\AA^1$-contractible curve. To begin with, let us assume $k$ is algebraically closed. Then one knows that the (regular) compactification $P$ of $C \hookrightarrow P$ over perfect fields is, in fact, a smooth curve. Moreover, by the classification result and the genus argument mentioned in \cref{A1contracurves:affine}, one again concludes that $P \cong \mathbb{P}^1_k$. By the definition of the compactification, we have that $C = \PP^1 \backslash Z$, where $Z \hookrightarrow P$ is the closed complement $Z:= \PP^1 \bs C$. As $C$ is $\AA^1$-contractible and in particular, $\AA^1$-connected (\cref{eg:S-point}), $Z$ cannot have more than one closed point. On contrary, suppose that $Z= \{p_1,p_2 \}$, then we see that 
        $$C\cong \PP^1\bs Z \cong \PP^1\bs \{p_1,p_2\}\cong \GG_{m,k}.$$
But from the $\AA^1$-rigidity of $\GG_m$ (\cref{Gm-A1-rigid}), we have that $\GG_{m,k}$ cannot be $\AA^1$-connected; let alone $\GG_{m,k} \bs \{p_1,\dots,p_n\}$. Hence, $Z = \{p_1\}$. Thus, one concludes that $C \cong \PP^1_k \backslash \{pt\} \cong \AA^1_k$. Now, let $k$ be arbitrary. Let $\widetilde{k}$ be an algebraic closure of $k$. Then, by \Cref{basechange:imperfect},
    $$C_{\widetilde{k}}:=C\times_{\Spec k} \Spec \widetilde{k} \to \Spec \widetilde{k}$$
is a smooth $\AA^1$-contractible curve over $\widetilde{k}$, which, by the previous case, is isomorphic to $\AA^1_{\widetilde{k}}$. Thus, $C$ is a $k$-form of $\AA^1_k$, which automatically, due to \Cref{ex:forms-A1}, implies that it has to be the trivial $k$-form of $\AA^1$. This completes the proof.
\end{proof}

\begin{remark}
For an alternate proof over algebraic closed fields of characteristic zero using $\GG_a$-actions, see \cite[Corollary 1.29]{freudenburg2006algebraic}.
\end{remark}

\begin{corollary}\label{A1-fiberspace}\cite{DMO25}
For any base scheme $S$, if $f: X\to S$ is a relative $\AA^1$-weak equivalence in $\Spc_S$ of relative dimension 1, then $X$ is an $\AA^1$-fiber space.
\end{corollary}
\begin{proof}
For any point $s\in S$, consider the base change of $f$ along $\Spec\ \kappa(s)\to S$. By \Cref{pullbackfunctor}, each fiber $X_s:= X\times_S \Spec \kappa(s)$ is a smooth $\AA^1$-contractible curve in $\mathcal{H}(\kappa(s))$. By \Cref{1-dim theorem}, we have that $X_s\cong \AA^1_{\kappa(s)}$.
\end{proof}

\subsection{Relative dimension 1}\label{sect:reldim=1}
By exploiting the \Cref{def:A1-fib-space} and \cref{1-dim theorem}, we furthermore extend this characterization over a base scheme via the following structure result.

\begin{theorem}\label{1-dim DD}\cite{DMO25}
For a smooth separated morphism of finite presentation $f:X\to S$ and relative dimension $1$ over a Noetherian normal scheme $S$, the following are equivalent:
\begin{itemize}
\item[1.] $f: X\to S$ is a relative $\AA^1$-weak equivalence in $\mathrm{Spc}_S$,
\item[2.] $f: X\to S$ is an $\AA^1$-fiber space,
\item[3.] $f: X\to S$ is a Zariski locally trivial $\AA^1$-bundle. 
\end{itemize}
\end{theorem}
\begin{proof}
The implication $(1) \iff (2)$ is a consequence of \cref{fiberwise=relative} and \cref{1-dim theorem}. Indeed assuming (1), we have that for any $s\in S$, the base change morphism $f_s: X_s \to \Spec \kappa(s)$ is an $\AA^1$-weak equivalence in $\Spc_{\kappa(s)}$, equivalently, $X_s$ is a smooth $\AA^1$-contractible curve in $\mathcal{H}(\kappa(s))$, whence by \cref{1-dim theorem} imply that $X_s \cong \AA^1_{\kappa(s)}$. The \cref{fiberwise=relative} now glues this fibrewise data to produce that $X$ is an $f: X\to S$ is an $\AA^1$-weak equivalence in $\Spc_S$. The other implication follows directly from \cref{fiberwise=relative}. The implication $(3) \implies (1)$ follows from \cref{A1-contr:egs}. The implication $(2) \Rightarrow (3)$ is due to the characterization of Zariski-locally trivial $\AA^1$-bundles by Kambayashi-Wright \cite{KW85}. For the convenience of the reader, we now briefly recall the scheme of the proof given in \cite{KW85}.
\medskip

We begin with the case where $S$ is the spectrum of a discrete valuation ring $R$, with uniformizing parameter $\pi$, we denote by $\eta$ the generic point of $S$ with corresponding residue $K=\mathrm{Frac}(R)$ and by $s$ the closed point of $S$, with residue field $\kappa$. Since $f:X\to S$ is an $\AA^1$-fiber space, the fiber $X_\eta$ is isomorphic to $\AA^1_K=\mathrm{Spec}(K[x_0])$. Viewing $x_0$ as a rational function $X$, we can assume up to changing $x_0$ by $\pi^n x_0$ for suitable $n\in \mathbb{Z}$ if necessary, that $x_0$ is actually a regular function on $X$ and that its restriction to the closed fiber $X_s\cong \AA^1_\kappa$ of $f: X\to S$ is non-zero. The inclusion $R[x_0]\hookrightarrow  \Gamma(X,\mathcal{O}_X)$ then corresponds to an $S$-morphism $h:X\to \AA^1_S=\mathrm{Spec}(R[x_0])$ which restricts to an isomorphism over the generic point of $S$. If the restriction $h_{s}: X_s \to \AA^1_s$ of $h_0$ over the closed point $s$ of $S$ is constant, then arguing as in the proof of \cite[Proposition 1.4]{KW85}, there exists an $S$-morphism $\tau:\AA^1_S\to \AA^1_S$ and a  quasi-finite $S$-morphism $h':X\to \AA^1_S$ both restricting to isomorphisms over the generic point of $S$ such that $h=h'\circ \tau$, whence an open immersion by Zariski Main Theorem \cite[Corollaire 8.12.10]{EGAIV-3-Grothendieck1966}, from which it follows that $h':X\to S$ is an \'etale morphism restricting to an isomorphism over the generic point of $S$. On the other hand, since the restriction $h'_s: X_s= \AA^1_\kappa \to \AA^1_\kappa=\AA^1_s$ is \'etale, it is an isomorphism. This implies that $h'$ is an \'etale bijective morphism with trivial residual extensions, whence an isomorphism. 
\medskip

The case where $\dim S=1$ being settled, the proof then proceeds by an induction on the dimension of the base scheme $S$: letting $S$ be the spectrum of Noetherian normal local ring of dimension $n\geq 2$ and letting $s$ be its closed point, we can assume by induction that the restriction $f':X'\to S'$ of $f: X\to S$ over $S':=S\setminus\{s\}$ is a Zariski locally trivial $\AA^1$-bundle. The argument in \cite[$\S$ 2.2-2.3]{KW85} applies verbatim to give the conclusion that $f':X'\to S'$ is isomorphic to the  trivial $\AA^1$-bundle $\AA^1_{S'}\to S'$. The diagram in which all the horizontal maps are open immersions determines a rational map of $S$-schemes $h: X \dashrightarrow \AA^1_S$
\[\begin{tikzcd}
   X \arrow[d,"f"] & X'\cong \AA^1_{S'} \arrow[d] \arrow[l] \arrow[r] & \AA^1_S \arrow[d] \\ S & S' \arrow[l] \arrow[r] & S 
\end{tikzcd} \]  
which restricts to an isomorphism over $S'$. Since $\AA^1_S$ is affine, $h$ is uniquely determined by rational function on $X$, which is regular outside $X_s$, and since $X$ is normal because $f: X\to S$ is flat and $S$ is normal and $X_s$ has codimension $\geq 2$ in $X$, it follows that the rational map $h$ is in fact an everywhere defined $S$- morphism $h: X\to \AA^1_S$. Moreover, since $h$ restricts to an isomorphism over $S'$, the canonical homomorphism $h^*\Omega_{\AA^1_S/S}\to \Omega_{X/S}$ is an isomorphism outside $X_s$, whence is an isomorphism since $h^*\Omega_{\AA^1_S/S}$ and $\Omega_{X/S}$ are invertible sheaves and $X_s$ has codimension $\geq 2$ in $X$ (cf. \cite[Theorem 5.10.5]{EGAIV-2-Grothendieck1965elements}). In sum, $h: X\to \AA^1_S$ is an \'etale morphism of $S$-schemes restricting to an isomorphism over $S'$. Since $X_s\cong \AA^1_{\kappa(s)}$, we conclude as in the previous case that the restriction $h_s:X_s\to \AA^1_{\kappa(s)}$ is an isomorphism as well and  finally that $h:X\to \AA^1_S$ is an isomorphism of $S$-schemes.
\end{proof}

\cref{1-dim DD} provides us with an elegant characterization to identify the affine line relative to a base scheme. 

\begin{corollary}\label{A1unique:DD}\cite{DMO25}
A smooth separated scheme $f:X\to S$ of finite presentation and relative dimension $1$ over a Noetherian normal scheme $S$ is isomorphic to $\mathbb{A}^1_S$ if and only if all of the following holds:
\begin{itemize}
\item[1.] $X$ is a smooth $\mathbb{A}^1$-contractible curve in $\mathcal{H}(S)$,
\item[2.] the relative canonical sheaf of differentials $\Omega_f$ of $f$ is trivial, and
\item[3.] $f$ has a section.
\end{itemize}
\end{corollary}

\begin{proof} 
One direction is clear: the affine line $\AA^1_S$ is $\AA^1$-contractible over $S$, has trivial sheaf of relative differential $\Omega_{\AA^1_S/S}$ and admits a section given by $0_S: S\to \AA^1_S$. 
\medskip

Now suppose the converse: since $f: X\to S$ is $\AA^1$-contractible, it is a Zariski locally trivial bundle by \cref{1-dim DD}. Since the automorphism group functor $\mathrm{Aut}(\AA^1)$ of the affine line is the affine group $\mathbb{G}_a\rtimes \mathbb{G}_m$, every Zariski locally trivial $\AA^1$-bundle over $S$ is an affine-linear bundle, that is, is isomorphic over $S$ to the total space of a torsor $\rho: P\to S$ under a certain line bundle $p: L= \mathbb{V}(\mathcal{L})\to S$ on $S$ for some invertible sheaf $\mathcal{L}$ on $S$. The latter invertible sheaf $\mathcal{L}$ is uniquely determined up to isomorphism by the property that under the isomorphism $f^*:\Pic(S)\to \Pic(X)$ induced by $f$ (see e.g.  \cite[Theorem 5]{magid1975picard}) the class of $f^*\mathcal{L}$ equals that of $\omega_f$. The assumption that $\omega_f$ is isomorphic to $\mathcal{O}_X$ implies that $\mathcal{L}\cong \mathcal{O}_S$, whence that $f:X\to S$ is isomorphic over $S$ to the total space of a $\mathbb{G}_{a,S}$-torsor $P\to S$, and the assumption that $f$ has a section implies in turn that $P\to S$ is the trivial $\mathbb{G}_{a,S}$-torsor $\mathbb{G}_{a,S}\to S$, whence that $X$ is isomorphic to $\AA^1_S$. 
\end{proof} 

\begin{remark}
As a consequence of the fact that every vector bundle torsor over an affine scheme is trivial, it follows from the proof of \Cref{A1unique:DD} that a smooth $\AA^1$-contractible scheme $f: X\to S$ of finite presentation and relative dimension $1$ over an affine Noetherian normal scheme $S$ always has a section. So, for a Noetherian normal \emph{affine} scheme $S$, relative $\AA^1$-contractibility and triviality of the relative canonical sheaf are sufficient to characterize the affine line $\AA^1_S$ among smooth $S$-schemes of finite presentation and relative dimension $1$ over $S$.
\end{remark}

On the other hand, if $S$ is not affine, the existence of a section in \cref{A1unique:DD} is a necessary condition, as illustrated by the following example.
\begin{example}
Consider the quasi-projective scheme $Y := (\PP^1\times\PP^1) \bs \Delta(\PP^1)$ obtained by deleting the diagonal $\Delta (\PP^1) := \text{Im} (\Delta:\PP^1 \hookrightarrow \PP^1\times_{\Spec k} \PP^1)$. Let $\pr_1: Y\to \PP^1$ be a morphism defined by the projection onto the first coordinate. Observe that $\pr_1$ is a Zariski locally trivial $\AA^1$-bundle and it is an $\AA^1$-weak equivalence in $\Spc_{\PP^1}$. Moreover, the scheme $Y$ is affine as it is given by the complement of an ample divisor in $\PP^1\times\PP^1$. However, the morphism $\pr_1$ cannot be a trivial torsor over $\PP^1$ as $Y$ does not admit any section, since any section necessarily intersects the diagonal $\Delta(\PP^1)$.
\end{example}

\subsubsection{Zariski Cancellation for Relative Curves}
As a consequence of \Cref{A1unique:DD}, we obtain the affirmation of the following relative variant of the Zariski Cancellation Problem for curves.

\begin{corollary}\label{Cor:rel-ZCP:dim1}\cite{DMO25}
Let $f: X\to S$ be a smooth separated scheme of finite presentation and relative dimension $1$ over a normal scheme $S$ such that $X\times_S \AA^n_S\cong \AA^{n+1}_S$ for some $n\geq 1$. Then $X$ is $S$-isomorphic to $\AA^1_S$. 
\end{corollary}
\begin{proof} 
Since the relative canonical sheaf $\omega_{X \times_{S} \AA^n_S/S}$ is isomorphic to 
    $$\pr_1^*\omega_f\otimes \pr_2^*\omega_{\AA^{n+1}_S /S} \cong \pr_1^*\omega_f$$ 
we infer from the isomorphism $X\times_S \AA^n_S\cong \AA^{n+1}_S$ that $\pr_1^*\omega_f$ is trivial. However, since $\pr_1^*:\mathrm{Pic}(X)\to \mathrm{Pic}(X\times_S \AA^n_S)$ is an isomorphism, we have that $\omega_f$ is the trivial sheaf. On the other hand, $f$ has a section given by the composition of the zero section $s_0: S\to \AA^{n+1}_S$ with the projection $\pr_1: \AA^{n+1}_S \cong X\times_S \AA^n_S\to X$. Finally, since $\pr_1: X \times_S \AA^n_S\to X$ and $\AA^{n+1}_S\to S$ are $\AA^1$-weak equivalences in $\mathrm{Spc}_S$ so is $f: X\to S$ and the conclusion then follows from \Cref{A1unique:DD}.
\end{proof}

\subsection{Some Counterexamples over Non-normal Base} \label{non-examples:reldim1}
We will now present some counterexamples to the \cref{1-dim DD}, especially when the base scheme is not normal. In particular, these provide examples of relatively $\AA^1$-contractible schemes which are not Zariski locally trivial $\AA^1$-bundles. The first example is due to Madhav Nori (\cite[\S 3]{KW85}) that gives us a geometric instance of an $\AA^1$-fiber space that cannot be a Zariski locally trivial $\AA^1$-bundle over a singular (even $\AA^1$-contractible!) scheme.

\begin{example}
Let $k$ be algebraically closed field and let $C:=\{V^2-W^3 =0\}\subset \AA^2_k$ be the cuspidal curve which is singular at the point $p= (0,0)$. Consider the map 
\begin{align*}
    & F: \AA^1_k \to C\times \mathbb{P}^1_k\\
    & \hspace{10mm} t \mapsto \bigg((t^2,t^3),[t:1]\bigg)  
\end{align*}
The map $F$ is a closed immersion. Let $X = C\times \PP^1_k \bs Z$ be the open subscheme of $C\times \PP^1_k$, where $Z:= \textrm{Im}(F)$. Now the projection $\pr_1: C\times \PP^1_k\to C$ is a (trivial) $\PP^1_k$-fibre space which restricts to the map $\phi: X\to C$. The map $\phi$ is clearly an $\AA^1$-fiber space since for all points $s\in X$, the fiber $\phi^{-1}(s)\cong (\PP^1_k\bs [t:1])\cong \AA^1_k$.
\medskip

We now claim that $\phi$ is \emph{not} a Zariski locally trivial $\AA^1$-bundle in any open set around the point $p$. Suppose by contrary, $\phi:X \to C$ has a locally trivial $\AA^1$-structure around $p$, then by shrinking $C$ about the point $p$, we have an isomorphism $f: C\times \AA^1_k\to X$ over $C$.
\[ \begin{tikzcd}[row sep=tiny]
& X \arrow[dd,"\phi"] \\
C\times \AA^1_k \arrow[ur,"\cong"] \arrow[dr] & \\
& C
\end{tikzcd} \]
Now, considering the composite $C\times \AA^1_k \to X \hookrightarrow C\times \PP^1_k$ restricted over the fiber over any point $x\in C$ gives us an open immersion $\AA^1_k\hookrightarrow \PP^1_k$. This extends uniquely to an isomorphism (say) $\sigma_x:\PP^1_k \to \PP^1_k$. To conclude, we observe that the map 
\begin{align*}
        & C\times \PP^1_k \to C\times \PP^1_k   \\
        & \hspace{3mm} (x,y)\mapsto \bigg(x, \sigma_x(y)\bigg)
\end{align*} 
carries the singular curve $C\times \{\infty\}$ birationally onto the smooth curve $Z$, which is not possible.
\end{example}

The second example is derived from \cite{Ha75}. The base scheme $S$ is not regular, and the $\AA^1$-fiber space over it we consider cannot be stably Zariski locally trivial (for the same reason it is not Zariski locally trivial). Nevertheless, due to \cref{1-dim DD}, we now know that the $\AA^1$-fiber space considered is indeed relatively $\AA^1$-contractible!

\begin{example}\label{ex:Hamann-example}Let $k$ be a field of characteristic $p>0$, let $S\subset \AA^2_k= \Spec k[x,y]$ be singular, non semi-normal, curve with equation $$x^{p+1}-y^p=0$$ 
and let $X\subset S\times_k \mathrm{Spec}(k[u,v])$ be the affine surface with equation 
            $$u^p+xv^p-v = 0.$$
Denote by $A$ and $B$ the coordinate rings of $S$ and $X$ respectively. The restriction to $X$ of the projection $\pr_S$ is an affine $\AA^1$-fiber space $f:X \to S$ whose restriction over the complement $S^*\cong \Spec k[t^{\pm 1}]$, where $t= x^{-1}y$, of the unique singular point $o=(0,0)$ of $S$ is isomorphic to  trivial $\AA^1$-bundle 
        $$S^*\times _k \AA^1_k=\Spec  k[t^{\pm 1}][\tilde{u}]$$
over $S^*$, where $\tilde{u} = u+tv$. It is known in contrast that there is no Zariski open neighborhood of $o$ in $S$ over which $f: X\to S$ restricts to a trivial $\AA^1$-bundle \cite[Lemma 1.6]{Ha75}, in particular, $f: X\to S$ is not a Zariski locally trivial $\AA^1$-bundle. Nevertheless, it can be verified by direct computation that the morphism of $S^*$-schemes 
\begin{align*}
    & X\times_{S^*} \AA^1_{S^*}\to \AA^2_{S^*}\\
    & \hspace{8mm} (\tilde{u},w)\mapsto (z_1,z_2)= \bigg(\tilde{u}-t(\tilde{u}+tw)^p,w+(\tilde{u}+tw)^p\bigg)
\end{align*}
is an isomorphism (with inverse given by 
    $$(z_1,z_2)\mapsto \bigg (z_1+t(z_1+tz_2)^p,z_2-(z_1+tz_2)^p) \bigg)$$ 
which further extends over $o$ to an isomorphism of $S$-schemes $X\times_S \AA^1_S\to \AA^2_S$. It then follows due to \cref{1-dim DD} that $f:X\to S$ is an $\AA^1$-weak equivalence in $\Spc_S$. 
\end{example}

Some other examples of $\AA^1$-fiber spaces $f:X\to S$ over non-normal curves $S$ which are not a Zariski locally trivial $\AA^1$-bundle are given \cite[$\S$ 3.4]{KW85}.

\begin{example}\label{Kambayashi:example:non-normalbase}
As a third example, consider $S$ to be any rational affine plane cuspidal curve over a field $k$ of characteristic zero with equation 
            $$\{x^p-y^q=0\},$$
where $p,q\geq 2$ are relatively prime integers. Let $f:X\to S$ be the restriction of the projection $\mathrm{p}_S:\mathbb{P}^1_S \to S$ to the complement $X$ of the Zariski closure $Z \subset \mathbb{P}^1_S$ of the graph $\Gamma\subset \AA^1_S$ of the normalization morphism $\nu:\AA^1_k\to S$ of $S$. By construction, every such $f: X\to S$ is an $\AA^1$-fiber space restricting to a trivial $\AA^1$-bundle over the complement of the unique singular point $c= (0,0)$ of $C$, but whose restriction over any Zariski open neighborhood $U$ of $c$ in $C$ cannot be trivial. Indeed, otherwise if such a trivialization $X|_U\cong \AA^1_U$ existed then the inclusion of $U$-schemes 
        $$\AA^1_U\cong X|_U\hookrightarrow \mathbb{P}^1_U$$ 
would extend to a morphism $\mathbb{P}^1_U\to \mathbb{P}^1_U$ mapping the singular curve $U=\mathbb{P}^1_U \bs \AA^1_U$ birationally onto the smooth curve $Z|_U\cong \AA^1_U$, which is not possible. Nevertheless, we see that all of these $\AA^1$-fiber spaces $f:X \to S$ are $\AA^1$-weak equivalences in $\Spc_S$.
\end{example}

The final example is of a different flavor: the base is a normal scheme, but this example illustrates the importance in \Cref{1-dim DD} of considering \emph{all} fibers of $f: X\to S$ and not only those over closed points of $S$, even when $S$ is defined over an algebraically closed field.

\begin{example} \label{exa:form-generic-fiber} 
Let $k$ be an algebraically closed field of characteristic $p>0$. Consider the morphism 
\begin{align*}
    & f: \Spec k[x,y] \cong \AA^2_k := X \longrightarrow S := \Spec k[t] \cong \AA^1_k \\
    & \hspace{34mm} (x,y)\longmapsto x^{p^2}-y-y^{p^2+p}.
\end{align*}
Then $f$ is a smooth morphism between $\AA^1$-contractible $k$-schemes, which has all its closed fibers isomorphic to $\AA^1_k$. Indeed, given $\lambda\in k$, as $k$ is algebraically closed we can write $\lambda=-\mu^{2p}$ for some $\mu \in k$ so that $f^{-1}(\lambda)$ equals the smooth curve in $\AA^2_k$ with equation 
        $$(x+\mu)^{p^2}-y-y^{p^2+p}=0, $$ 
which is the image of $\mathbb{A}^1_k$ by the closed immersion $\mathbb{A}^1_k \to \mathbb{A}^2_k$ defined by 
        $$ s \mapsto (s+s^{p^2+p}-\mu,s^{p^2}).$$
On the other hand, the fiber of $f$ over the generic point of $\AA^1_k$ is the purely inseparable $k(t)$-form $Y$ of $\mathbb{A}^1_{k(t)}$ with equation                 
        $$x^{p^2}-y-y^{p^2+p}-t=0$$
in $\AA^2_{k(t)}$. So, $f: X\to S$ is not an $\AA^1$-fiber space. However, on the other hand, since $Y$ is not $\AA^1$-contractible by \Cref{1-dim theorem}, it follows from \Cref{pullbackfunctor} that $f: \AA^2_k\to \AA^1_k$ is not a relative $\AA^1$-weak equivalence in $\Spc_{\AA^1}$.
\end{example}

%---------------------------------------------------------------------
\section{Relative $\AA^1$-Contractibility of Smooth Surfaces}\label{sect:reldim=2}
In this section, we continue to investigate the characterization in dimension 2. We begin with a recent result of Choudhury-Roy on the uniqueness of the affine plane (\cref{A2isunique}). As in the case for curves, we systematically establish that being affine is a consequence of $\AA^1$-contractibility for smooth surfaces (\cref{A1contrasurface:affine}). This seemingly simple fact allows us to, in fact, extend the aforementioned uniqueness of the affine plane over perfect fields (\cref{A2unique:perfect}), leaving the story open for imperfect fields. We will digress a bit over imperfect fields using the theory of non-trivial $k$-forms of $\AA^2_k$. Finally, all these will allow us to state an elegant characterization of Dedekind schemes $D$ in relative dimension 2 (\cref{2-dim DD}). As a consequence, we all characterize the affine plane over $D$ (\cref{A2unique:DD}) and present the Zariski Cancellation in this context (\cref{Cor:rel-ZCP:dim2}).

\subsection{\texorpdfstring{$\AA^1$}{A1}-Contractibility over Fields}
Recall from \cref{intro:exoticsurfaces} that the topological contractibility alone is not sufficient to characterize the affine 2-space. On the other hand, tools coming from birational geometry, such as the logarithmic Kodaira dimension (affine version of the usual Kodaira dimension), proved to be very useful for studying the exotic complex structures arising from surfaces such as \cref{eg:Ramanujamsurf}, \cref{eg:tomDieck-petri}, and their generalized products (see \cite[\S 1.5]{zaidenberg2000exotic}). The seed for characterizing the affine plane primarily stems from the work of Miyanishi via $\GG_a$-actions \cite{Miyanishi75algebraic}. However, it was recently proved by Choudhury-Roy (\cite[Theorem 1.1]{choudhury2024}) that the $\AA^1$-contractibility is a much stronger invariant compared to topological contractibility and the log Kodaira invariant when it comes to the characterization of $\AA^2$. We will now state their theorem and give an overview of the machinery involved in their proof.
\begin{theorem}\label{A2isunique}
A smooth affine surface $X$ over a field $k$ of characteristic zero is $\AA^1$-contractible in $\mathcal{H}(k)$ if and only if it is isomorphic to $\AA^2_{k}$.
\end{theorem}

\subsubsection{Strategy of the proof} 
The work of Miyanshi-Sugie (\cite[\S 4.1]{miya1981noncomplete},\cite{miyanishi1980affine}) provides an algebraic characterization for a non-singular affine surface $X = \Spec (A)$ over an algebraically closed field $k$ of characteristic zero. They proved that $X \cong \AA^2$ if and only if 
\begin{itemize}
\item $A$ is a unique factorization domain with $A^\times = k^\times$,
\item $X$ has the logarithmic Kodaira dimension $\bar{\kappa}(X)= -\infty$.
\end{itemize}
Assuming the $\AA^1$-contractibility of $X$, we get that the Picard group is trivial and the group of units $\GG_{m,k}(X)\cong k^\times$. It remains to show that $\bar{\kappa}(X)=-\infty$. To achieve this, the authors in \cite[\S 4]{choudhury2024} study the notion of a variety being \emph{dominated by the images of $\AA^1$}, which in simple terms translates to "being able to cover up the given variety by affine lines". 

\begin{defn}\label{defn:A1-uniruled}
A $k$-variety $X$ over a field of characteristic zero is said to be    
\begin{enumerate}
\item[(a)] \emph{dominated by images of $\AA^1$} if there is an open dense subset $U\subset X$ such that, for every $p\in U(k)$, there is an $\AA^1$ in $X$ through $p$, that is, there exists a non-constant morphism from $\AA^1\to X$ whose image contains $p$)
\item[(b)] \emph{$\AA^1$-uniruled} if there is a dominant generically finite morphism $H: \AA^1_k\times_k Y\to X$ for some $k$-variety $Y$
\end{enumerate}
\end{defn}

In general, we have $(b)\implies (a)$; but if the base field $k$ is uncountable, then we have $(a)\iff (b)$. Exploiting such $\AA^1$-connectedness-type results, Choudhury-Roy showed a crucial fact that for any $\AA^1$-connected surface over an algebraically closed field $k$ of characteristic zero, one of the following occurs (\cite[Theorem 4.9]{choudhury2024}):
\begin{itemize}
\item there is a non-constant $\AA^1$ through all $x \in X(k)$,
\item there is a non-constant morphism $H: \AA^1_Y\to X$ for some irreducible $Y\in \Sm_k$ such that the dimension of the closure of the image of $H$ is at least 2.
\end{itemize}
Thus, when $k$ is uncountable, the above says that $X$ has to be $\AA^1$-uniruled, whence by \cite{iitaka1977logarithmic} has negative log Kodaira dimension. Now, given any field $k$ of characteristic zero, the idea is to embed $k$ as a subfield of an uncountable algebraically closed field $L$ and realize the extension $L/k$ as the filtered colimit of finitely generated sub-extensions of $k$. Translated to algebro-geometric terms, these extensions correspond to morphisms of (affine) smooth schemes with affine transition maps, whence by \cite[Corollary 1.24]{MV99} preserves $\AA^1$-contractibility under pullbacks, i.e., $X_L:= X\times_{\Spec k} \Spec L \to \Spec L$ is an $\AA^1$-contractible variety. Supported by the previous justifications, we have that $X_L \cong \AA^2_L$. In other words, $X$ is a $k$-form of $\AA^2$. To conclude, observe that due to \cite[Theorem 3]{kambayashi1975absence} this has to be a trivial $k$-form and so $X\cong \AA^2_k$.
\medskip

As a result, we have that a complex affine surface $X$ is $\AA^2_{\CC}$ if and only if it is topologically contractible and $\AA^1$-connected. Hence, connecting back to our story from \cref{intro:exoticsurfaces}, this result provides us with a proof that the Ramanujam surface and the tom Dieck-Petrie surfaces cannot be $\AA^1$-connected, which independently implies that they can neither be $\AA^1$-chain connected nor be $\AA^1$-contractible. Moreover, \cref{A2isunique} also answered the conjecture of \cite[Question 2.34]{DPO2019} that any $\AA^1$-contractible surface indeed has negative log Kodaira dimension (as they are $\AA^1$-connected!). To this end, for a smooth affine surface $X$, we have the following implications:
\begin{align*} 
 \AA^1\text{-contractible} \implies  \AA^1\text{-connected} \implies \AA^1\text{-uniruled} \implies \bar{\kappa}(X) = -\infty 
\end{align*}
Having said this, it is natural to ask the following extension.
\begin{question}\label{uniqueA2:perfect}
Can we extend \cref{A2isunique} over arbitrary fields?
\end{question}

\subsection{Extension to Perfect Fields}
We will first show that the affineness is a consequence of $\AA^1$-contractibility for smooth surfaces (\cref{A1contrasurface:affine}) and extend the \cref{A2isunique} over perfect fields (\cref{A2unique:perfect}).

\begin{lemma}(\cite{DMO25}) \label{A1contrasurface:affine}
Any smooth $\AA^1$-contractible surface is necessarily affine.
\end{lemma}
\begin{proof}
For this proof, we will use arguments from the stable homotopy type of punctured tubular neighborhood (cf. \cref{app:Mot-top-infty}). To show that a smooth $\AA^1$-contractible surface $X$ is affine, it suffices to show that $X$ is \emph{connected at infinity}, in the sense that for every smooth proper completion $X \hookrightarrow \bar{X}$ of $X$, the boundary $\partial{X} = \bar{X}\setminus X$ is a connected scheme. For now, let us assume that $X$ is indeed connected at infinity. By the existence of resolution of singularities for surfaces, we can assume that $\partial{X}$ is a simple normal crossing (SNC) divisor on $\bar{X}$. Since $X$ is $\AA^1$-contractible,  $\mathcal{O}_X(X)^* = k^*$ and $\mathrm{Pic}(X)$ is trivial, from which it follows that $\mathrm{Pic}(\bar{X}) \cong \mathrm{Cl}(\bar{X})$ equals the free abelian group generated by the divisor classes in $\mathrm{Cl}(\bar{X})$ of the irreducible components of $\partial{X}$. If $\partial{X}$ is in addition connected, then by applying the proof due to \cite[P. 512, Theorem, item 3]{fujita1982topology}, one concludes that $\partial{X}$ is the support of an effective ample divisor on $\bar{X}$, which implies that $X$ is affine. 
\medskip

Now it remains to argue that a smooth $\AA^1$-contractible surface is connected at infinity. For this we use a key property of the stable $\AA^1$-homotopy type at infinity: the stable $\AA^1$-homotopy type of $X$ at infinity is isomorphic (in the stable $\AA^1$-homotopy category $\SH(k)$) to that of $\AA^2_k$ (cf. \cref{eg:SH-infty-is-A^d}), in particular, to the $\PP^1$-suspension spectrum $\Sigma_{\PP^1}^{\infty}(\AA^2_k\bs \{0\})_+$ of the smooth connected scheme $\AA^2_k\setminus\{0\}$. This identification combined with the description in \cite{DDO2022punctured} of the the stable $\AA^1$-homotopy type of $X$ at infinity in terms of the boundary divisor $\partial{X} = \bar{X}\setminus X$ in a smooth SNC completion $\bar{X}$ of $X$ implies that $\partial{X}$ must be connected.
\end{proof}

Essentially, using this powerful fact that any smooth $\AA^1$-contractible surface is automatically affine, we can promote \cref{A2isunique} to perfect fields.

\begin{theorem}\label{A2unique:perfect}\cite{DMO25}
A smooth surface $X$ over a perfect field $k$ is $\AA^1$-contractible in $\mathcal{H}(k)$ if and only if it is isomorphic to $\AA^2_k$.
\end{theorem}
\begin{proof}
One direction is obvious: $\AA^2_k$ is an $\AA^1$-contractible surface over any field $k$. Suppose that $X$ is $\AA^1$-contractible, then we have that $X$ is affine due to \cref{A1contrasurface:affine}. In essence, once we have that $X$ is an affine $\AA^1$-contractible smooth surface, a proof can be obtained by manipulating the arguments from \cite{russell1981affine, Choudhury2024A1type}. To begin with, recall that due to \cite{kambayashi1975absence}, $\AA^2$ does not admit non-trivial forms over a perfect field. Hence, it suffices to show that the base change of $X$ to an algebraic closure of $k$ is isomorphic to $\AA^2$. So, we can henceforth assume, without loss of generality, that $k$ is an algebraically closed field. Since $X$ is $\AA^1$-contractible, it is in particular $\AA^1$-connected. So, \cite[Theorem 2.3]{Choudhury2024A1type}, implies that there exists an irreducible smooth $k$-scheme $W$ and a non-constant homotopy $H : W \times_k \AA^1_k\to X$ that $H(0)$ is dominant. Now, \cite[Corollary 2.4]{Choudhury2024A1type} provides us that any $\AA^1$-connected $k$-variety over an algebraically closed field of characteristic zero has negative logarithmic Kodaira dimension. However, in view of the existence of resolution of singularities for surfaces, this fact can be extended to perfect fields as well, where the appropriate extension of the theory of logarithmic Kodaira dimension is available (for a detailed account on this, see \cite{russell1981affine}). Moreover, the $\AA^1$-contractibility of $X$ implies that $\mathcal{O}_X(X)^*=k^*$ and that $\mathrm{Pic}(X)$ is trivial and in addition, we have that $X$ is also affine (\cref{A1contrasurface:affine}), the conclusion is immediate by appealing to \cite[Theorem 2]{russell1981affine}.
\end{proof}

\begin{corollary}\label{A2-fiberspace}\cite{DMO25}
Let $X$ be any smooth separated scheme over a base scheme $S$ with perfect residue fields. Then the canonical morphism $f: X\to S$ is a relative $\AA^1$-weak equivalence in $\Spc_S$ of relative dimension 2, then $X$ is an $\AA^2$-fiber space.
\end{corollary}
\begin{proof}
For any point $s\in S$, consider the base change of $f$ along $\Spec\ \kappa(s)\to S$. By \Cref{pullbackfunctor}, each fiber $X_s \to \Spec\ \kappa(s)$ is a relative $\AA^1$-weak equivalence. By \Cref{A2unique:perfect}, we have that $X_s \cong \AA^2_{\kappa(s)}$.
\end{proof}
Having said this, we mention that such a characterization is fairly open over fields of positive characteristics, which heavily depends on establishing a suitable notion of being (separably) $\AA^1$-uniruledness for $k$-varieties.

\subsection{Non-trivial $k$-forms over \texorpdfstring{$\AA^2_k$}{A2}} \label{sect:non-trivial-forms-A2}
One approach to analyze the proof of \cref{A2unique:perfect} over imperfect fields is to study and classify the possible $k$-forms of the affine plane $\AA^2_k$. The following example shows that a non-trivial $k$-form of $\AA^2$ need not necessarily be $\AA^1$-contractible. 

\begin{example}
Let $k$ be a field of characteristic $p>0$ and let $k'$ be a purely inseparable extension of $k$ of degree $p$, so $k'= k(t^{\frac{1}{p}})$ for some (in fact any) element $t$ of $k^\times$ which is not a $p$-th power. Then the quotient 
 $$ U=(\mathfrak{R}_{k'/k} \mathbb{G}_{m,k'}) / \mathbb{G}_{m,k}$$
where $\mathfrak{R}_{k'/k}$ denote the Weil restriction (cf. \cref{app:weil-restriction}), is a smooth connected unipotent $k$-group of dimension $p-1$. Explicitly, by Osterl\'e (\cite{oesterle1984nombres}), $U$ is isomorphic to the subgroup of $\mathbb{G}_{a,k}^p$ defined by the equation 
    $$ x_{0}^p+tx_1^{p}+\cdots +t^{p-1}x_{p-1}^p-x_{p-1}=0.$$
By \cite[\S 4.1]{achet2019picard}, we have $\Pic(U)\cong \mathbb{Z}/p\mathbb{Z}$. This, in particular, implies that $U$ cannot be an $\mathbb{A}^1$-contractible $k$-variety.
\end{example}

\subsubsection{Why is $\Pic(U)$ non-trivial?}
To see that $\Pic(U)$ is non-trivial, we use that $\mathrm{R}_{k'/k} \mathbb{G}_{m,k'}$ is isomorphic to the principal open subset $N_{k'/k}\neq 0$ of $\mathbb{A}_k^p$, where $N_{k'/k}$ is the norm of the field extension, which implies that 
\begin{align*}
 &\hspace{10mm} H^0(\mathrm{R}_{k'/k} \mathbb{G}_{m,k'}, \mathcal{O}_{\mathrm{R}_{k'/k} \mathbb{G}_{m,k'}}^*)/k^*=N_{k'/k}\\ 
 & \textrm{and}\\
 & \hspace{10mm}\Pic(\mathrm{R}_{k'/k} \mathbb{G}_{m,k'})=H^1(\mathrm{R}_{k'/k} \mathbb{G}_{m,k'}, \mathcal{O}_{\mathrm{R}_{k'/k} \mathbb{G}_{m,k'}}^*)=0.
\end{align*}
From the exact sequence of $k$-group schemes $$0\to \mathbb{G}_{m,k}\to \mathrm{R}_{k'/k} \mathbb{G}_{m,k'}\to U\to 0,$$ 
and $H^0(U,\mathcal{O}_U^*)=k^*$, we get the short exact sequence
{\small
\begin{align*}
1\to H^0(\mathrm{R}_{k'/k} \mathbb{G}_{m,k'}, \mathcal{O}_{\mathrm{R}_{k'/k} \mathbb{G}_{m,k'}}^*)/k^*=\mathbb{Z}\langle N_{k'/k}\rangle \to H^0(\mathbb{G}_{m,k}, \mathcal{O}_{\mathbb{G}_{m,k}}^*)\cong \mathbb{Z}\to \Pic(U)\to 0
\end{align*}}
in which the image of $N_{k'/k}$ in $\mathbb{Z}$ is simply $p$, which gives $\Pic(U)\cong \mathbb{Z}/p\mathbb{Z}$. Moreover, by Osterl\'e (\cite{oesterle1984nombres}), the group $U$ is known to be \emph{$k$-wound}: every $k$-morphism $\mathbb{A}^1_k\to U$ is constant with image a $k$-rational point of $U$. The conventional definition of $k$-wound due to Tits is formulated by asking that there is no non-constant $k$-group scheme homomorphism $\mathbb{G}_{a,k}\to U$, but this property is equivalent to the apriori stronger requirement that there is no non-constant morphism $\mathbb{A}^1_k\to U$  (see \cite[Proposition B.3.2]{conrad2015pseudo} for a justification).

\begin{remark}
We know that $\AA^2$ does not admit a non-trivial $k$-form over a separably closed extension \cite[Theorem 3]{kambayashi1975absence}. By observing more closely to the properties of the construction due to Osterl\'e (\cite{oesterle1984nombres}), one see as a natural connection to the approach of \cite{Choudhury2024A1type} in proving that when $k$ is an algebraically closed field of characteristic zero a smooth $\mathbb{A}^1$-connected $k$-variety admits non-constant morphism from $\mathbb{A}^1_k$. 
\end{remark}

In light of the above example, we pose the following natural question. 
\begin{question}\cite{DMO25}
Does there exist a non-trivial $k$-form of $\AA^2_k$ over an imperfect field $k$ which is $\AA^1$-contractible?
\end{question}
Any candidate satisfying the question aforementioned would produce a counterexample to the characterization of $\AA^2$ over imperfect fields.

\subsection{Relative Dimension 2}\label{sect:reldim=2}
Recall that a \emph{Dedekind scheme} is a separated integral locally Noetherian scheme $D$ whose local rings at its closed points are discrete valuation rings. Notable examples are the Zariski spectrum of fields, the Zariski spectrum of Dedekind domains, such as rings of integers of number fields, and smooth curves over a field.
\medskip

We begin with a theorem of A. Sathaye that, in algebro-geometric language, conveys that any $\AA^2$-fiber space over a discrete valuation ring of characteristic zero is necessarily a Zariski locally trivial $\AA^2$-bundle.

\begin{theorem}\label{sathayethm} 
Let $R$ be an equicharacteristic zero rank one discrete valuation ring. Let the unique maximal ideal of $R$ be $tR$. Let $K= \text{Quot}(tR)$ and $k = R/tR$ and let $A= R[X_1,\dots,x_n]$ be an affine domain over $R$ such that 
$$A \otimes_R K\cong K^{[2]} \quad \text{and} \quad A\otimes_R k\cong k^{[2]}.$$
Then $A\cong R^{[2]}$.
\end{theorem}
\begin{proof}
This is due to \cite[Theorem 1]{sathaye1983polynomial}.
\end{proof}
\begin{remark}
In relation to the embedding problem \cref{sect:embedding problem}, this theorem provides the following solution for affine three spaces: 
Let $k$ be a field of characteristic zero. Let $f\in k^{[3]}$ be an element in the polynomial ring $k[X,Y,Z]$. Further, assume that 
$$k[X,Y,Z]\otimes_{k[f]} k[f]\cong k(f)^{[2]}$$ and all the partial derivatives have no common zeroes in $k^3$, i.e.,
$$\frac{\partial f}{\partial X}=\frac{\partial f}{\partial Y}= \frac{\partial f}{\partial Z}=0.$$ 
Then $k[X,Y,Z]/(f-\lambda)\cong k^{[2]}$, for every $\lambda\in k$. In other words, it implies that 
        $$k[X,Y,Z]\cong k[f]^{[2]}.$$ 
For more background, we redirect the reader to \cite{sathaye1983polynomial, miyanishi1984algebrotopA3}.
\end{remark}

For our purposes, the remarkable result (\cref{sathayethm}) coupled with functoriality in the motivic homotopy theory gives us the following structure theorem (\cref{2-dim DD}) and consequently, an elegant characterization of $\AA^2$ in relative dimension 2 (\cref{A2unique:DD}):

\begin{theorem}\label{2-dim DD} \cite{DMO25}
Let $X$ be a smooth scheme of relative dimension $2$ over a Dede-kind scheme $D$ with characteristic zero residue fields. Furthermore, let the canonical morphism $f: X \to S$ be affine. Then the following are equivalent:
\begin{enumerate}
\item $f$ is a relative $\AA^1$-weak equivalence in $\Spc_S$,
\item $f$ is a Zariski locally trivial $\AA^2$-bundle.
\end{enumerate}
\end{theorem}
\begin{proof}
The implication $(2)\implies (1)$ is clear: a Zariski locally trivial $\AA^2$-bundle is a relative $\AA^1$-weak equivalence. For the converse, suppose that $X$ is relatively $\AA^1$-contractible of dimension 2 over $D$. Then, by \Cref{A2-fiberspace}, $f: X\to D$ is an $\AA^2$-fiber space over a discrete valuation ring of characteristic zero. The conclusion is evident from \cref{sathayethm}.
\end{proof}

The hypothesis that $f: X\to D$ is an affine morphism is indispensable, as in comparison to the one-dimensional case (where this hypothesis was not needed), which raises the following question:

\begin{question}\cite{DMO25}
Is there a relatively $\mathbb{A}^1$-contractible  $\mathbb{A}^2$-fiber space $f: X\to D$ for which $f$ is not an affine morphism?
\end{question}

As a consequence of \cref{2-dim DD}, we have the unique characterization of the affine plane over a Dedekind scheme.
\begin{corollary}\label{A2unique:DD}\cite{DMO25}
A smooth affine scheme $f:X\to D$ over an \emph{affine} Dedekind scheme $D$ with characteristic zero residue fields is isomorphic to $\mathbb{A}^2_D$ if and only if both of the following holds:
\begin{enumerate}
\item $X$ is a smooth $\mathbb{A}^1$-contractible surface in $\mathcal{H}(D)$, and
\item the relative canonical sheaf of differentials $\omega_f$ is trivial.
\end{enumerate}
\end{corollary}
\begin{proof} 
One direction is obvious: the affine space $\AA^2_D$ is smooth, $\AA^1$-contractible, and the relative canonical sheaf $\omega_{\AA^2_D/D}$ is trivial.
\medskip

Suppose the converse: by \Cref{2-dim DD}, we have that $f: X\to D$ is a Zariski locally trivial $\AA^2$-bundle. Since $D$ is affine, by \cite[Theorem 4.9]{BCW76}, the map $f:X\to D$ is in fact isomorphic over $D$ to a vector bundle $p: E= \Spec_D(\mathrm{Sym}^{\bullet} \mathcal{E})\to D$ for some locally free sheaf $\mathcal{E}$ of rank $2$ on $D$. Taking relative differentials, we have that $p^{*}\mathcal{E}\cong \Omega_f$ and consequently, applying the $\det$ functor, that $p^{*}(\det \mathcal{E})\cong \omega_f$. On the other hand, since $D$ is $1$-dimensional, it follows from \cite[Theorem 1]{SerreModulesprojectifs} that $\mathcal{E}$ is isomorphic to the direct sum of $\mathcal{O}_D$ and an invertible sheaf $\mathcal{L}$ such that $f^*\mathcal{L}\cong \omega_f$. Since $f^*: \Pic(D)\to \Pic(X)$ is an isomorphism, the triviality of $\omega_f$ is equivalent to that of $\mathcal{E}$ and the assertion follows.
\end{proof}

\subsubsection{Zariski Cancellation for Relative Surfaces}
An important consequence is the relative variant of Zariski Cancellations over certain Dedekind schemes.

\begin{corollary}\label{Cor:rel-ZCP:dim2}\cite{DMO25}
Let $D$ be an affine Dedekind scheme with characteristic zero residue fields. Then for a smooth affine scheme $f: X\to D$ of relative dimension 2 over a Dedekind scheme $D$, the isomorphism $X\times_D \AA^1_D \cong \AA^3_D$ implies that $X \cong \AA^2_D$.
\end{corollary}
\begin{proof}
Again, the given isomorphism $X\times_D \AA^1_D\cong \AA^3_D$ implies that $\omega_{\AA^3_D/D}=p_1^*\omega_f\otimes p_2^*\omega_{\AA^2_D/D}\cong p_1^*\omega_f$ and the induced isomorphism of Picard groups $p_1^*:\mathrm{Pic}(X)\to \mathrm{Pic}(X\times_D \AA^2_D)$ implies that $\omega_f$ has to be trivial. But $X$ is stably isomorphic to an affine space, which implies that $X$ is eventually $\AA^1$-contractible in $\mathcal{H}(D)$ (see references from \cref{Asanuma-stable-structure}). The desired conclusion now follows from \Cref{A2unique:DD}.
\end{proof}

\subsubsection{Counter-examples to \cref{2-dim DD} when char $k = p> 0$}
The following examples show that the analog of \Cref{2-dim DD} does not hold in positive characteristics.
\begin{example}\label{eg:Asanuma3F}
Let $k$ be a field of characteristics $p>0$. Consider the affine three-fold $V = \Spec R$ where 
\begin{align*}
    R = \frac{k[X,Y,Z,T]}{\Lin X^m Y - Z^{p^e}-T-T^{sp} \Rin} \quad \text{where}\quad m\geq 2,\quad p^e \nmid sp\ \text{and}\ sp\nmid p^e.
\end{align*}    
Define a map $\pi_x: V\to \AA^1_k$ given by the projection onto the $X$-coordinate. This map $\pi_X$ is an $\AA^2$-fiber space. From \cite[Theorem 3.4]{asanuma1987polynomial}, it follows that the ring $R$ is stably a polynomial ring over $k[X]$. In particular, this shows that the affine variety $V$ is relatively $\AA^1$-contractible in $\mathcal{H}(\AA^1_k)$ and consequently also $\AA^1$-contractible in $\mathcal{H}(k)$. However, $\pi_X$ is \emph{not} a Zariski locally trivial $\AA^2$-bundle (\textit{ibid}, Theorem 5.1).
\end{example}

The following example provides us with an $\AA^2$-fiber space over a discrete valuation ring with residue fields of positive characteristics, which portrays the obstruction to extending \Cref{Cor:rel-ZCP:dim2}.

\begin{example}
Let $R$ be a Noetherian integral domain and let $p$ be any non-invertible prime in $R$. Choose $\pi$ to be any non-invertible, non-zero divisor in $R$. Define an $R$-algebra 
        $$A = R[X,Y,Z]/\Lin -X^{p^e}+Y+Y^{sp}+\pi Z \Rin,$$
where $e,s$ are positive integers such that $sp \nmid p^e$ and $p^e\nmid sp$. Then for an affine variety $f: X:= \Spec A \to \Spec R$, we have the following due to \cite[Corollary 5.3]{asanuma1987polynomial}: 
\begin{enumerate}
\item The map $f: X \to \Spec R$ is an $\AA^2$-fiber space (\cref{def:A1-fib-space}),
\item The affine variety $X$ is stably $\AA^1$-contractible, that is, $X\times_{\Spec R} \AA^1\cong \AA^3$
\item The affine variety is not the affine 2-space, $X \ncong \AA^2_R$.
\end{enumerate}
The property (2) and (3) gives the negative solution to the Zariski cancellation problem over $A$.
\end{example}

\subsubsection{Vignette}
Let $X$ be a smooth scheme of dimension $n<3$ over a field $k$. Then, due to the characterization results that we have presented in this chapter, we obtain the following beautiful characterization:
\begin{itemize}
\item $X$ is $\AA^1$-contractible $\iff$ $X\cong \AA^n$ and so its logarithmic Kodaira is negative, $\bar{\kappa}(X) = -\infty$. When $n=2$ with $k$ uncountable, $X$ being $\AA^1$-connected already implies that $\bar{\kappa}(X)= -\infty$ (however, the converse is false: take $X= \GG_m \times \AA^1_k$).
\end{itemize}
This makes one ponder the following questions.
\begin{question}
To what extent does $\AA^1$-contractibility imply $\AA^1$-chain connectedness from dimensions $\ge 3$ and vice-versa?
\end{question}
\begin{question}
Do $\AA^1$-contractible varieties have negative logarithmic Kodaira and vice-versa from dimensions $\ge 3$?
\end{question}

%--------------------------------------------------------------
\section{Relative \texorpdfstring{$\AA^1$}{A1}-Contractibility in Higher Dimensions} \label{sect:higherdimensions}
The characterization of Zariski locally trivial $\AA^n$-bundles up to  $\AA^1$-contractibility does not hold from relative dimensions $n \geq 3$. In this chapter, we will explain this perspective and present a geometric proof of the fact that $\AA^n$-fiber bundles produce an abundance of examples of relatively $\AA^1$-contractible schemes (\cref{thm:An-fiber-space-contractible}). This will be explained via a concrete illustration (\cref{Winkelmann:example}).

\subsubsection{The characterization of $\AA^n$ in higher dimensions}
The characterization of affine spaces up to $\AA^1$-contractibility is no longer true starting from dimensions $n \geq 3$, even over the "kindest" base fields. For fields of positive characteristics, this can already be deduced by combining the negative solution of Zariski Cancellations as discussed in \cref{ZCP:positive-char} and \cref{stableiso:A1-weq}. Over a field of characteristic zero, the Koras-Russell threefolds give the first examples of a smooth $\AA^1$-contractible variety that is not isomorphic to $\AA^3_k$ as introduced in \cref{C*-actions:KR3F} (we will study this family in detail in \cref{chp5}).
\medskip

In higher dimensions, the beautiful work of Asok-Doran \cite{asok2007unipotent} established that from dimensions $n \geq 4$, there exist abundant examples of smooth $n$-schemes that are $\AA^1$-contractible but are not isomorphic to $\AA^n$. They demonstrated this by constructing everywhere stable $\GG_a$-actions on affine spaces with strictly quasi-affine quotients. Towards this, they exploited an \emph{algebraic characterization} (\textit{ibid, Theorem 4.20}) that, in essence, characterizes when a given $\GG_a$-action is \emph{everywhere stable}. Their main result is the following (\cite[Theorem 5.1]{asok2007unipotent}):
\begin{theorem*}
For every $m\geq 4$, there exists a denumerably infinite collection of pairwise non-isomorphic $m$-dimensional exotic $\AA^1$-contractible varieties, each admitting an embedding into a smooth affine variety with pure codimension 2 smooth boundary.
\end{theorem*}

In fact, starting from dimension $\ge 6$, there exists a moduli of exotic smooth varieties arising from the quotient of an action of an unipotent group on an affine space. We include below the statement for the reader's clarity (\cite[Theorem 5.3]{asok2007unipotent}).
\begin{theorem*}
For every $ m\geq 6$ and every $n\geq 0$, there exists a connected $n$-dimensional scheme $S$ and a smooth morphism $f: X\to S$ of relative dimension $m$, whose fibers over $k$-points are $\AA^1$-contractible and strictly quasi-affine and non-isomorphic. 
\end{theorem*}

It is important to note that all the varieties obtained as a result of this construction are strictly\footnote{quasi-affine but not affine, e.g., $\AA^n \backslash \{0\}$, for $n \geq2$} quasi-affine smooth $\AA^1$-contractible schemes which are not isomorphic to $\AA^n$. In fact, a simplest prototype of such an occurrence, a smooth (exotic) fourfold, was already studied by Winkelmann in the context of complex geometry.
\begin{theorem}
The following result has been abridged from several results sourced from \cite[\S 2]{winkelmann1990free}:
\begin{enumerate}
\item There exists a unipotent group $U \cong (\CC^1,+)$-action on $\CC^5$ given by a certain quadratic transformation such that the quotient $X = \CC^5/U$ is diffeomorphic to $\CC^4$, but not biholomorphic. In particular, $X$ is a strictly quasi-affine variety that is a non-Stein manifold (hence $X\ncong \CC^4$). 
\item There exists a unipotent group $U\cong (\CC^2,+)$-action on $\CC^6$ given by certain affine linear transformation such that the quotient $X =\CC^6/U$ is diffeomorphic to $\CC^4$, but not biholomorphic. In particular, $X$ is a strictly quasi-affine variety that is a non-Stein manifold (hence $X\ncong \CC^4$). 
\end{enumerate}
\end{theorem}

We will revisit this prototype shortly (cf. \cref{Winkelmann:example}). Intrigued by these results, one can ponder on the following question (\cite[Question 5.10]{asok2007unipotent}):
\begin{question}
Does there exist a 3-dimensional quasi-affine quotient of $\AA^4$ by $\GG_a$?.
\end{question}
As a special case, it has been shown that any everywhere stable $\GG_a$ action on $\AA^4$ associated to the triangular locally nilpotent derivation always produces an affine space quotient, that is, $\AA^4/\GG_a \cong \AA^3$ (\cite[Theorem 2.1]{deveney2004triangular}).
\medskip

Therefore, in coherence with our analysis on low relative dimensions, the program of characterizing the affine $n$-space via $\AA^n$-fiber spaces fails in higher dimensions. In other words, for every $ n\geq 3$, one can now produce relatively $\AA^1$-contractible schemes that are no longer $\AA^n$-fiber spaces. We give several examples of this fact in this current section (note that this would already fail in low relative dimensions under certain assumptions: see \cref{ex:Hamann-example} and \cref{Kambayashi:example:non-normalbase}). 

\subsection{Families of \texorpdfstring{$\AA^1$}{A1}-Contractible Schemes via \texorpdfstring{$\AA^n$}{An}-fiber Spaces: A Geometric Proof}\label{sect:ADF-family}
In \cite{ADF2017smooth}, the authors study the representability of motivic spheres by smooth schemes over a commutative unital ring $k$ and they show that a certain family of (split) smooth affine quadric hypersurfaces $Q_{2n}$ in $\AA^{2n+1}$ (defined below) have the $\AA^1$-homotopy type of a motivic sphere (see \cref{sec:motivic-spheres}), hence solving a representability problem for these classes of motivic spheres.
\begin{equation}\label{ADF:Q2n-defn}
\begin{aligned}
    & Q_{2n-1}:= \Spec k[x_1,\dots,x_n,y_1,\dots,y_n]/\langle \sum_i x_iy_i-1  \rangle\\
    & Q_{2n} := \Spec k[x_1,\dots,x_n,y_1,\dots,y_n,z]/\langle \sum_ix_iy_i-z(1+z) \rangle
\end{aligned}
\end{equation}
In particular, \cite[Theorem 2.2.5]{ADF2017smooth} shows that for any 
commutative ring $k$, 
\begin{align}\label{eqn:ADF-Q2m}
Q_{2n-1} \simeq (S^1)^{n-1}\wedge (\GG_m)^{n}  (\PP^1)^{\wedge n}\quad   \text{and} & \quad   Q_{2n}\simeq (S^1)^{n}\wedge (\GG_m)^n \quad \text{for all}\ n \ge 0. 
\end{align}
These results were already known to be true in the $S^1$-stable $\AA^1$-homotopy category due to Morel (See \textit{ibid}, P. 2), and the \cref{eqn:ADF-Q2m} shows that it is even true unstably! The proof essentially takes into account that $\Sm_k$ is equipped with the Nisnevich topology, which brings the motivic purity (\cref{purity-theorem}) into the act. In the same paper, they also construct the following example: let 
    $$E_n := \{x_1 = \dots =x_n=0; z = -1\} \subset Q_{2n}$$
be the closed subscheme of $\AA^{2n+1}$. Then by setting, 
    $$X_{2n} := Q_{2n}\bs E_n \subset \AA^{2n+1}$$
they show that each of these subvarieties $X_{2n}$ are strictly quasi-affine varieties that are $\AA^1$-contractible over $\Spec k$. Furthermore, they also proved that $X_{2n}$'s cannot be obtained as the quotient of an unipotent group action on affine spaces. This strikes a stark dichotomous behavior with topology, where every such contractible space is indeed a free quotient of an affine space. Aligned with our context, we show that all these $X_{2n}$'s produce abundant examples of $\AA^1$-contractible schemes \emph{relative to an affine space} base. To make this more comprehensible, let us now prove a fact that the $\AA^n$-fiber spaces \emph{always} form a class of relatively $\AA^1$-contractible spaces (\cref{thm:An-fiber-space-contractible}). One major consequence of this will be that the family of smooth quasi-affine subvarieties $X_{2n}$'s constructed above even becomes relatively $\AA^1$-contractible. The following theorem is, in essence, a consequence of Asanuma's Stable Structure theorem from \cref{Asanuma-stable-structure}.

\begin{theorem} \label{thm:An-fiber-space-contractible} \cite{DMO25}
Let $f: X\to S$ be a separated $\mathbb{A}^n$-fiber space over a regular scheme $S$. Then $X$ is a relatively $\mathbb{A}^1$-contractible scheme in $\mathcal{H}(S)$.
\end{theorem} 
\begin{proof} 
Observe that the assumptions made on $f: X\to S$ are stable under base change to a cover of $S$ by Zariski affine open subsets. In this view, we can assume, without loss of generality, that $S$ is affine. Indeed, this follows from the local nature of $\AA^1$-contractibility - that being a Nisnevich locally trivial $\AA^1$-bundle can be tested on affine open neighborhoods. Since $S$ is regular and $f:X \to S$ is a smooth separated morphism, $X$ is a separated regular scheme. So, by \emph{Jouanolou–Thomason homotopy lemma} (see e.g. \cite[Lemma 3.1.4]{asok2021A1}) there exists a Zariski locally trivial torsor  $\rho: V\to X$  under a vector bundle $p: E\to X$, say of rank $r$, whose total space $V$ is an affine scheme. For every point $s\in S$, the induced $E_s$-torsor 
    $$\rho_s:V_s\to X_s\cong \AA^n_{\kappa(s)}$$
is isomorphic to the trivial one $p_s: E_s\to X_s$ because $X_s$ is affine. Since on the other hand $E_s$ is isomorphic to the trivial vector bundle on $X_s\cong \AA^n_{\kappa(s)}$ by Quillen-Suslin theorem \cite{Qui76}, we conclude that $\rho_s: V_s\to X_s$ is isomorphic to the trivial $\AA^r$-bundle
 $$\mathrm{pr}_2:\AA^r_{\kappa(s)}\times_{\kappa(s)}\AA^n_{\kappa(s)}\to \AA^n_{\kappa(s)}.$$ 
Thus, $g= f \circ \rho: V\to S$ is an affine $\AA^{n+r}$-fiber space over the regular affine scheme $S$. Now it follows from \cite[Corollary 3.5]{asanuma1987polynomial} that there exists $m\geq 0$ such that $g\circ \mathrm{pr}_1:V\times_S \mathbb{A}_S^m\to S$ is isomorphic to a vector bundle $q:F\to S$ of rank $n+r+m$ over $S$.
\[ \begin{tikzcd}
	 & V  & & V \times_S \AA^m_S=F \\
  X \\
  S & & & \AA^m_S 
  \arrow[swap, "f", from=2-1, to=3-1]
  \arrow[swap, "\rho", from=1-2, to=2-1]
  \arrow["g", from=1-2, to=3-1]
  \arrow["\pr_1", from= 1-4, to=1-2]
  \arrow["\pr_2", from= 1-4, to=3-4]
  \arrow[from= 3-4, to=3-1]
   \arrow[swap, "q", from= 1-4, to=3-1]
 \end{tikzcd} \]
This implies in particular that 
            $$g\circ \pr_1:V\times_S \AA_S^m\to S$$ 
is $\AA^1$-weak equivalence in $\Spc_S$. Since $\pr_1:V\times_S \AA^m_S\to V$ is an $\AA^1$-weak equivalence in $\mathrm{Spc}_V$ whence in $\Spc_S$ by \Cref{cor:pushpull-A1-weak} (2), it follows from \Cref{lem:3-from-2} that $g:V\to S$ is an $\AA^1$-weak equivalence in $\Spc_S$. Since on the other hand $\rho:V\to X$ is an $\AA^1$-weak equivalence in $\Spc_X$ whence in $\Spc_S$ by \Cref{cor:pushpull-A1-weak} (2), we conclude again from \Cref{lem:3-from-2} that $f:X \to S$ in $\AA^1$-weak equivalence in $\mathrm{Spc}_S$.
\end{proof}

\begin{remark}    
Let us also note that one can now prove \cref{thm:An-fiber-space-contractible} using the highly sophisticated algebraic techniques coming from motivic homotopy theory (e.g., via \cref{fiberwise=relative}). Nevertheless, we would like to stress that the proof presented here is useful, as it establishes the geometric mechanism of the aforementioned algebraic fact!
\end{remark}

\subsubsection{Illustrative Example: Family of $\AA^1$-Contractible Smooth Quadrics}
We show that the family of smooth affine quadrics as in \cref{ADF:Q2n-defn} appears as a torsor under a vector bundle projected to an affine space. In particular, these quadrics illustrate \Cref{thm:An-fiber-space-contractible} for a separated $\mathbb{A}^2$-fiber space $f:X\to \mathbb{A}^2_{\mathbb{Z}}$ with strictly quasi-affine total space $X$. The following example also shows that subvariety considered by \cite{winkelmann1990free} gives us a simplest illustration of the \cref{thm:An-fiber-space-contractible} in relative dimension 4:

\begin{example}\label{eg:notZLT}\label{Winkelmann:example}\cite{DMO25}
Let $S=\Spec \mathbb{Z}[x,y] =\AA^2_{\mathbb{Z}}$ and consider the smooth affine fourfold 
$$V:=\{xv-yu+z(z-1)=0\} \subset S\times_{\ZZ}\AA^3_{\ZZ} = \Spec \ZZ[x,y][z,u,v].$$
The projection $\pr_1|_V: V\to S$ is a smooth morphism whose restriction over $S\bs \{(0,0)\}$ is a Zariski locally trivial $\mathbb{A}^2$-bundle. On the other hand, $(\pr_{1}|_V)^{-1}(0,0)$ is the disjoint union of two copies 
\begin{align*}
   F_0 = \{x=y=z=0\}\quad  &\text{and}\quad  F_1=\{x=y=0, z= 1\}  
\end{align*}
of $\AA^2_{\ZZ}$. The restriction of $\pr_{1}|_V$ to $X:= V \bs F_1$ is then an $\AA^2$-fiber space $f:X \to S$ whose total space $X$ is a strictly quasi-affine $k$-variety. It follows in particular that $f$ cannot be a locally trivial $\mathbb{A}^2$-bundle, neither in Zariski topology nor in any of the finer topologies such as Nisnevich, \'etale, fppf, or fpqc. Nevertheless, as first noticed in \cite{asok2007unipotent}, $X$ is $\AA^1$-contractible over $\Spec \ZZ$. We now observe that for essentially the same reasons, $X$ has the stronger property to be relatively $\AA^1$-contractible over $S$. Namely, let 
  $$h = \pr_1: W: =S \times_{\ZZ} \Spec \ZZ[t_1,t_2,t_3] \cong \AA^5_{\ZZ} \to S$$ 
be the trivial $\mathbb{A}^3$-bundle over $S$ and consider the $S$-morphism $g:W\to X$ defined by 
$$(t_1,t_2,t_3)\mapsto (u,v,z)=(yt_3-(xt_2-yt_1-1)t_2,xt_3-(xt_2-yt_1-1)t_1,xt_2-yt_1)$$
so that we have the following commutative diagram 
\[\begin{tikzcd}[sep = large]
	{W=\mathbb{A}^5_{\mathbb{Z}}} \\
	& X \\
	& {S=\mathbb{A}^2_{\mathbb{Z}}} \\
	& {\mathrm{Spec}(\mathbb{Z})}
	\arrow["g", from=1-1, to=2-2]
	\arrow["h", from=1-1, to=3-2]
	\arrow[shift right=2, from=1-1, to=4-2]
	\arrow["{f}", from=2-2, to=3-2]
	\arrow[from=3-2, to=4-2]
\end{tikzcd}\]
The morphism $g:W \to X$ is a $\mathbb{G}_{a,X}$-torsor \cite[\S 3]{winkelmann1990free}, whence is a relative $\AA^1$-weak equivalence in $\Spc_X$ and in $\Spc_S$ by \Cref{cor:pushpull-A1-weak} (2). On the other hand, $h: W\to S$ being a trivial $\mathbb{A}^3$-bundle, it is a relative $\AA^1$-weak equivalence in $\Spc_S$. These two properties imply, in turn, by \Cref{lem:3-from-2} that $f:X \to S$ is a relative $\AA^1$-weak equivalence in $\mathrm{Spc}_S$.
\end{example}

\begin{remark} 
Several higher dimensional analogues of the setup in \Cref{eg:notZLT} have been successively constructed and studied in \cite{asok2007unipotent, ADF2017smooth, ADO23} leading to many new examples of smooth strictly quasi-affine $\AA^1$-contractible schemes $X$ over an affine base scheme $B$ having the structure of $\AA^n$-fiber spaces $f:X \to \AA^n_B$ over the affine space $\AA^n_B$. This holds, for instance, for the family of smooth quasi-affine schemes 
{\small
$$f_{n,\mathbf{m}}:=\pr_{x_1,\ldots, x_n}: X_{n,\mathbf{m}}=:\big\{\sum_{i=1}^n x_i^{m_i}y_i+z(z-1)=0\}\bs \{x_1=\ldots =x_n=z-1=0\big\} \to \AA^n_B,$$ }
where $\mathbf{m}=(m_1,\ldots, m_n)$ is multi-index of positive integers. In the case where $B$ is regular, \Cref{thm:An-fiber-space-contractible} implies that for such schemes, the $\AA^n$-fiber space $f_{n,\mathbf{m}}:X_{n,\mathbf{m}}\to \AA^n_B$ is an $\AA^1$-weak equivalence in $\Spc_{\AA^n_B}$, providing in turn, since $\AA^n_B$  is $\AA^1$-contractible over $B$, a different proof that all these schemes are $\AA^1$-contractible over $B$. 
\end{remark}

To conclude, recall that the Dolga\v{c}ev-Ve\v{i}sfe\v{i}ler Conjecture admits a negative solution over fields of positive characteristics due to the counter-example of Asanuma (\cref{eg:Asanuma3F}) and thus, leaving it wide open over fields of characteristic zero. The only known cases are due to \cite{KM78, KW85} in dimension 1 and \cite{sathaye1983polynomial} in dimension 2, whenever the base schemes contain a field of characteristic zero. On the other hand, starting from dimension $n\ge 3$, we have seen in this chapter that there are numerous sources of relatively $\AA^1$-contractible schemes non-isomorphic to $\AA^n$. Hence, the relative $\AA^1$-weak equivalence of the morphism $f:X \to S$ does not, in general, allow us to conclude that $X$ is an $\AA^n$-fiber space whenever $n\ge 3$. Moreover, even if we do know that $X$ is an $\AA^n$-fiber space, one has no clue whether it will be a locally trivial $\AA^n$-bundle in Zariski (or even in finer topologies)!

% \afterpage{\blankpage
% \thispagestyle{empty}}

\begin{savequote}
Never outshine the master.
\qauthor{"The 48 Laws of Power" by Robert Greene}   
\end{savequote}
\chapter{Koras-Russell Prototypes and Exotic Motivic Spheres}\label{chp5}
\markboth{V Koras-Russell Prototypes and Exotic Motivic Spheres}{}

The results presented in this chapter are sourced from the article \cite{Mad}.

{\small
{\fontfamily{bch}\selectfont
\subsubsection{\textsc{Chapter Summary}}
We establish the $\AA^1$-contractibility of the classical Koras-Russell threefolds over perfect fields (\cref{A1-cont-over-perfect-fields}). This will allow us to extend its relative $\AA^1$-contractibility over integers (\cref{KR3FoverZ}) and, further, over Noetherian schemes (\cref{KR3Fmain:Noetherian}) by exploiting the intricate relationship between the relative and fiberwise data (\cref{pointwise:phenomenon}) in the unstable setting. Following this, we study the generalized prototype of Koras-Russell threefolds in higher dimensions and extend their $\AA^1$-contractibility over perfect fields (\cref{KR3Fprototypes:field-A1-cont:perfect}), and systematically relatively over Noetherian schemes (\cref{KR3F:prototypes-A1-cont-base:Noetherian}). As a major application of these results, we will prove the existence of exotic spheres in motivic homotopy theory in every dimension $\ge 4$ (\cref{exotic-motspheres-countereg}) over infinite perfect fields. 
}}

%--------------------------------------------------------------------
\section{The Family of Koras-Russell Threefolds}\label{sec:KR-prototypes}
Recall that a smooth affine scheme $X\to S$ of dimension $n$ is said to be \emph{exotic} if it is relatively $\AA^1$-contractible but not isomorphic, as $S$-schemes, to the affine $n$-space $\AA^n_S$. Moreover, we have seen that even under the assumption that $X$ is $\AA^1$-contractible over any "nicest" base fields, the characterization of $\AA^n_k$ breaks away in higher dimensions. The same is true in dimension 3, and this is due to the \emph{Koras-Russell threefolds}, whose sibling was already introduced in \cref{C*-actions:KR3F}. In what follows, we will discuss the family of Koras-Russell threefolds of the first kind in detail, which by now has become an independent and vibrant area of study in its own right, interpolating between motivic homotopy theory and affine algebraic geometry. Building on this, we study the higher-dimensional prototypes of these Koras-Russell threefolds, called \emph{generalized Koras-Russell varieties} in arbitrary dimensions $n \ge 4$.

\subsection{Extension to Perfect Fields}\label{sect:Koras-Rusell-perfect}
We begin with the definition of the classical Koras-Russell threefolds.

\begin{defn}\label{defn:KR3F}
Let $k$ be any perfect field. Then we define the \emph{Koras-Russell threefold of first kind} $\KK$ as the smooth affine variety described by the explicit polynomial equation 
\begin{equation}\label{KR3F;equation}
\KK := \{x^mz = x+ y^r+t^s\}\subset \AA^4_k
\end{equation}
where the integers $m\ge 2$ and $r,s\ge 1$ such that $\text{gcd}\ (r,s) =1$.
\end{defn}
\begin{figure}[h]
    \centering
    \includegraphics[width=8cm, height=6cm]{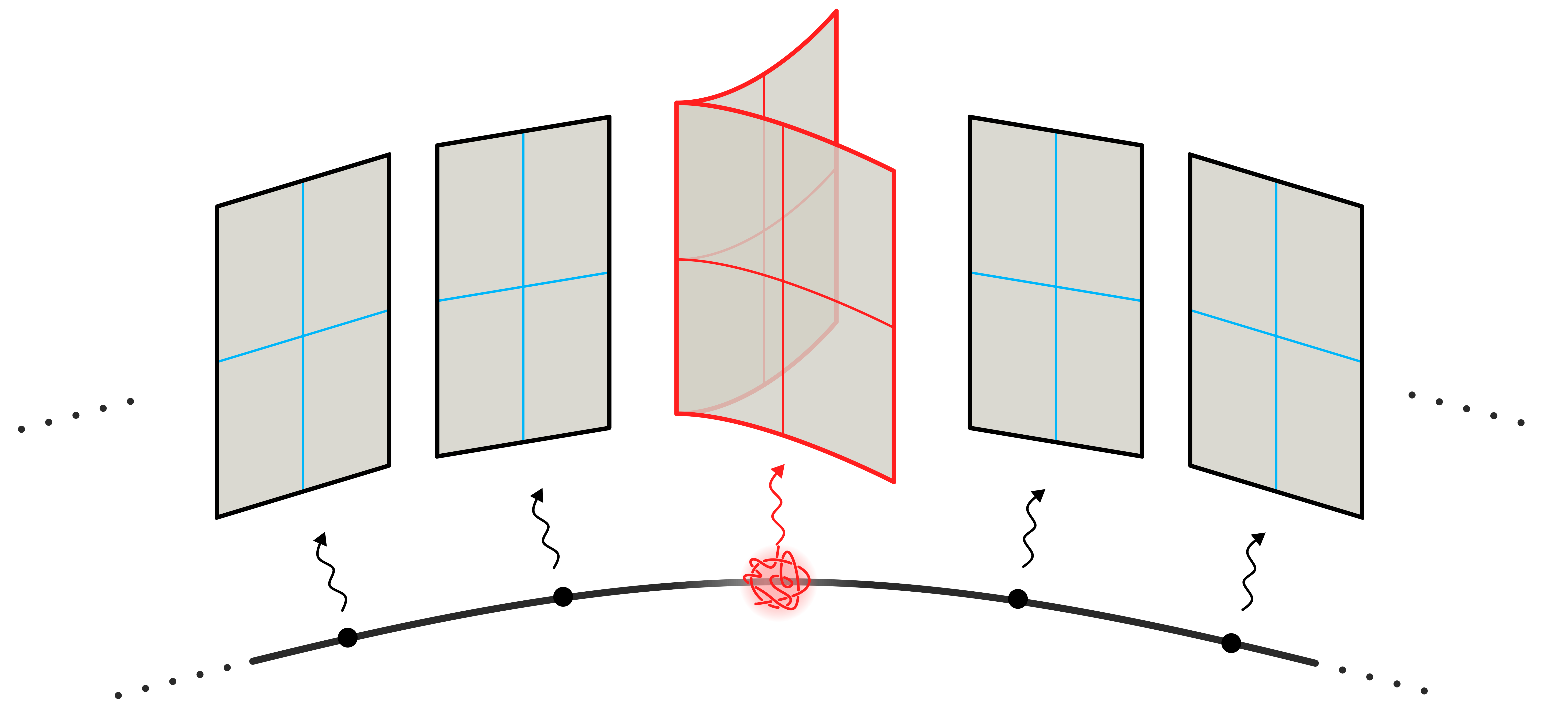}
    \caption{Fibers of $\KK$ when projected to $x$-axis; the generic fiber is shown in red}
    \label{Book-surface}
\end{figure}
Traditionally, Koras-Russell threefolds have been studied extensively only over a field of characteristic zero. But in the light of coprimality of the integers $r,s$ and the usefulness of the coordinate "$x$" to make the Jacobian criterion for smoothness work, it is also well-defined over any perfect field. In particular, this means $\KK$ is defined over the finite fields $\FF_p$ and over arithmetic curves (e.g., the integers $\Spec \ZZ$). The affine variety $\KK$ exhibits the following incredible properties:
\begin{itemize}
\item $\KK$ is topologically contractible (\cite{KR97}), whenever $k$ admits an embedding into $\CC$, 
\item The cylinders $\KK\times \AA^1$ have trivial Makar-Limanov value over char $k=0$, $\ML(\KK\times \AA^1) \cong \ML(\Spec k)$ (\cite{dubouloz2009cylinder}) 
\item  $\KK$ is $\AA^1$-contractible both in the stable $\AA^1$-homotopy category \cite{HKO16} and in the unstable $\AA^1$-homotopy category \cite{DF18} when the base field is of characteristic zero,
\item $\KK$ is not isomorphic to $\AA^3_{\CC}$ as witnessed by the Makar-Limanov invariant: $\ML(\KK)\ncong \ML(\AA^3_{\CC})$.
\end{itemize}
All these properties make $\KK$ an algebro-geometric analog of the Whitehead manifold $\mcal{W}$ from \cref{intro:exotic-threefolds}, and hence make it difficult to construct algebraic invariants to distinguish it among $\AA^1$-contractible smooth affine threefolds. Closely associated with this first kind, there are certain threefolds of the second kind $\widetilde{\KK}$ as illustrated in \cref{KR3FIIkind:equation}, whose stable $\AA^1$-contractibility has been proven in \cite{HKO16}. As we will shortly see that both of these threefolds contain a special affine line $L:=\{x=y=t=0\}$. But the crucial step is to conclude that $\KK \bs L \simeq \AA^2\bs\{0\}$, which, as the authors in \cite{DF18} remark, is not apparent for the second kind $\Tilde{\KK}$. The main tool that makes this possible for the first kind is \cref{exist:algspace} whose analog remains open for the second kind!
\medskip

In \cite[Theorem 1.1]{DF18}, the authors prove the $\AA^1$-contractibility of the Koras-Russell threefold of first kind over a field of characteristic zero. In the following sections, we prove that this, in fact, can be extended over any perfect field and over certain base schemes. To give the context of the strategy, we shall now begin by explaining the elements from \cite{DF18}.

\subsection{The Unstable \texorpdfstring{$\AA^1$}{A1}-Contractibility of $\KK$}
The basic ingredient towards the proof is the categorical tool known as the \emph{very weak five lemma} (cf. \cref{weak5lemma}). To adapt this to our situation, we first note that the threefold $\KK$ contains a hypersurface given by $\{z=0\}:= P\subset \KK$. It can be shown that it is isomorphic to $\AA^2$ with coordinates $y$ and $t$, that is, $P\cong \Spec k[y,t]$. More crucially, $\KK$ also contains an affine line given by 
    $$L:=\{x=y=t=0\} \cong \AA^1 \quad \text{such that}\quad  L \cong \Spec(k[z]).$$ 
Now, considering the complement, we have that $P\backslash L \simeq \AA^2\backslash \{0\}$ in $\Spc_k$. By the very definitions, the varieties $L$ and $P$ intersect transversally in $\AA^4$ at the point $(0,0,0,0)$ which gives us a pullback diagram as follows:
\begin{equation}
\begin{tikzcd}[sep= small]
    &&&&& {P\backslash L \simeq \mathbb{A}^2\backslash \{0\}} & {} & {P\simeq \mathbb{A}^2} \\ \\{} 
    &&&&& {\mathcal{K}\backslash L} && {\mathcal{K}}
    \arrow[from=1-6, to=1-8]
    \arrow["i",from=1-6, to=3-6]
    \arrow["\phi",from=1-8, to=3-8]
    \arrow[from=3-6, to=3-8]
\end{tikzcd}
\end{equation}
Consider the associated cofiber sequence given by the following commutative diagram
\begin{equation}\label{setup:diagram1}
\begin{tikzcd}[sep= small]
    &&&&& {\mathbb{A}^2\backslash \{0\}} & {} & {\mathbb{A}^2} && {\mathbb{A}^2/\mathbb{A}^2\backslash\{0\}} \\  \\
    {} &&&&& {\mathcal{K}\backslash L} && {\mathcal{K}} && {\mathcal{K}/\mathcal{K}\backslash{L}}
	\arrow[from=1-6, to=1-8]
	\arrow["i"', from=1-6, to=3-6]
	\arrow[from=1-8, to=1-10]
	\arrow["\phi", from=1-8, to=3-8]
	\arrow["j", from=1-10, to=3-10]
	\arrow[from=3-6, to=3-8]
	\arrow[from=3-8, to=3-10]
\end{tikzcd}
\end{equation}
In the light of \cref{weak5lemma}, we only need to show that the morphisms $i$ and $j$ are $\AA^1$-weak equivalences in $\Spc_k$.

\subsection{\texorpdfstring{$\AA^1$}{A1}-weak Equivalence of \texorpdfstring{$j$}{j}} \label{A1-weakeq-of-j}
The $\AA^1$-weak equivalence of the morphism of schemes $j:\AA^2/\AA^2\backslash\{0\}\to \KK/\KK\backslash L$ follows from the purity isomorphism in motivic homotopy theory. The defining equation of $\KK$ gives us a description of the normal cone of $L$ inside $\KK$, which can be seen by rewriting
        $$\KK:= \{x(x^{m-1 }z-1) = y^r+t^s\}$$ 
and the fact that the regular function ($x^{m-1}z-1$) is invertible in $L$ (in fact, it is equal to -1 on $L)$, we have that the normal cone $N_{\KK}(L)$ is generated by $y$ and $t$. Coupling this data with the homotopy purity (\cref{purity-theorem}), we have that 
    $$\KK/(\KK\backslash{L}) \simeq \Th(N_{\KK}(L))\simeq L_+\wedge (\PP^1)^{\wedge2}.$$ 
Again via the purity isomorphism, the quotient space $\AA^2/\AA^2\backslash \{0\}$ can be identified with $(\PP^1)^{\wedge2}$ (see for e.g., \cite[\S 4.6]{antieau2017primer} or \cite[\S 2.1]{ADF2017smooth}). The inclusion $\Spec k \simeq \{0\}\hookrightarrow L\simeq \AA^1$ induces a the composite 
$$\Spec k \wedge (\PP^1)^{\wedge 2}\simeq \AA^2/\AA^2\backslash \{0\} \xrightarrow{j} \KK/\KK\backslash L \simeq L_+\wedge (\PP^1)^{\wedge 2}$$
which is an $\AA^1$-weak equivalence (cf. \cite[Lemma 2.1]{voevodsky2003Z/2}) and so is $j$.

\subsection{$\AA^1$-weak Equivalence of $i$}\label{A1-weakeq-of-i}
To show that the morphism $i:\AA^2\backslash \{0\} \to \KK\backslash L$ is an $\AA^1$-weak equivalence, we first produce an explicit $\AA^1$-weak equivalence $g: \KK\backslash L\to \AA^2\backslash \{0\}$ in \cref{explicit-weq}. Then using the functoriality of $\AA^1$-Brouwer degree (\cref{A1degreemap}), we show that the composite $g \circ i: \AA^2\bs \{0\}\to \AA^2\bs \{0\}$ is an $\AA^1$-weak equivalence in $\Spc_k$.

\subsubsection{The explicit $\AA^1$-weak equivalence via Zariski local triviality}\label{explicit-weq}
The upshot is to produce an $\AA^1$-weak equivalence $g: \KK\bs L\to \AA^2\bs \{0\}$ by constructing a Zariski locally trivial $\AA^1$-bundle over $\AA^2\bs\{0\}$. Consider the affine scheme $\KK(s):=\{x^mz=x+y^r+t^s\}$ and the associated strictly quasi-affine scheme $\KK(s)\bs \{0\}$. Observe that when $s= 1$, we have that $\KK(1):= \{x^mz = x+y^r+t\}\cong \{x^m z-x-y^r=t\}$ which by the change of coordinates is isomorphic to $\AA^3_k$. As a result, we have
        $$\KK(1) \bs L \cong \AA^3\bs L\simeq \AA^2\bs \{0\} \times \AA^1\simeq \AA^2\bs \{0\}.$$ 
The following result is proven in \cite[Proposition 3.1]{DF18} for fields of characteristic zero. Here, we show that this trick can be extended to perfect fields as well. In particular, we prove the existence of a quasi-affine fourfold $W$ that is the total space of a Zariski locally trivial $\AA^1$-bundle over $\mcal{K}(s)\bs L$, for all $s\geq 1$. This fourfold $W$ is produced by exploiting the Danielewski fiber product trick (\cite{Danielewski1989cancellation}) that was initially constructed to produce a counterexample to the Zariski Cancellation Problem for surfaces (for more background, see \cite{dubouloz2007addDanielewski}).

\begin{prop}\label{Daniel:trick 4fold}\cite{Mad}
Let $k$ be any perfect field. Then for every $s\geq 2$, there exists a smooth strictly quasi-affine fourfold $W$ which is simultaneously the total space of a Zariski locally trivial $\AA^1$-bundles over $\KK(s)\bs L$, for all $s\ge 1$.
\end{prop}
\begin{proof}
The desired fourfold $W$ is constructed as the fiber product of strictly quasi-affine schemes over an algebraic space $\mathfrak{S}$ whose existence we prove in \cref{exist:algspace}.
\[\begin{tikzcd}
	W & {\KK(s)\backslash L} \\
	{\KK(1)\backslash L} & {\mathfrak{S}}
	\arrow["{\pr_s}", from=1-1, to=1-2]
	\arrow["{\pr_1}"', from=1-1, to=2-1]
	\arrow[from=1-2, to=2-2]
	\arrow[from=2-1, to=2-2]
\end{tikzcd}\]
\cref{exist:algspace} shows that for all $s\geq 1$, the quasi-affine threefolds $\mcal{K}(s)\bs L$ all have a structure of an \'etale locally trivial $\AA^1$-bundle $\rho:\KK(s) \to \mathfrak{S}$ fibered over the same algebraic space $\mathfrak{S}$. A priori, by taking the fiber product 
    $$W:= (\KK(1)\bs L) \times_{\mathfrak{S}} (\KK(s)\bs L),$$
We see that $W$ only has the structure of an algebraic space. But notice that the natural projections $\pr_s: W\to \mcal{K}(s)\bs L$ (for all $s\ge 1$) are an affine morphism and $\mcal{K}\bs L$ is a strictly quasi-affine scheme, which consequently implies that $W$ also has to be a strictly quasi-affine scheme. The fact that this \'etale equivalence descends also to the Zariski topology is because the transition group, which is $\Aut(\AA^1) = \GG_m \ltimes \GG_a$, is special \cite[\S 5]{Grothendieck1958torsion}.
\end{proof}

As an immediate consequence of \cref{Daniel:trick 4fold}, we get our desired result.
\begin{corollary}\label{g-is-weq}
The map $g: \KK(s)\bs L \to \AA^2\bs \{0\}$ is an $\AA^1$-weak equivalence for all $s\geq 1$.
\end{corollary}
\begin{proof}
This follows from the observation that Zariski locally trivial $\AA^n$-bundles preserve $\AA^1$-weak equivalence (see e.g., \cite[Lemma 3.1.3]{asok2021A1}).
\end{proof}

\subsection{Strategy to Construct \texorpdfstring{$\mathfrak{S}$}{S} over Perfect Fields}
Briefly speaking, the existence of the claimed algebraic space follows due to a fundamental fact in moduli theory that if an algebraic group acts freely on a scheme, then the quotient space exists and is an algebraic space (see \cite{knutson2006algebraic} or \cite[\href{https://stacks.math.columbia.edu/tag/071R}{Tag 071R}]{stacks-project}). When the char $k=0$, producing such a $\GG_a$-action on the affine variety $\KK$ is equivalent to giving a locally nilpotent derivation $D$ on its coordinate ring $\Gamma(\KK)$ defined via the exponential map (cf. \cref{G_a-LND}). However, if the char $k= p >0$, this simple correspondence does not a priori apply, for example, $1/p!$ in the exponential function is not well-defined. Nevertheless, one still gets a correspondence by considering a locally finite iterative higher derivation instead of a single locally nilpotent derivation. The former is obtained by considering the truncated version of $\exp(tD)$ (see for e.g., \cite[Part I, \S 1]{miyanishi1978lectures}, \cite[\S 5]{okuda2004kernels}). Therefore, assuming the existence of such a family of iterative higher derivations, we have that for a $\GG_a$-action on $\KK(s)$ its algebraic quotient coincides with the projection $\pr_{x,t}:\AA^4\to \AA^2$ restricted to $q: \KK(s) \to \AA^2$ with the fixed point locus $L$. Hence, we obtain an induced fixed point free $\GG_a$-action on $\KK(s) \bs L$ such that 
\[\begin{tikzcd}
	{\mathcal{K}(s)\backslash L} && {\mathbb{A}^2\backslash \{0\}} \\
	& {(\KK(s)\bs L)/\GG_a}
	\arrow["q",from=1-1, to=1-3]
	\arrow[from=1-1, to=2-2]
	\arrow[from=2-2, to=1-3]
\end{tikzcd}\]
it factors through an \'etale $\GG_a$-torsor over its geometric quotient $(\KK(s)\bs L)/ \GG_a$ as above. The desired algebraic space $\mathfrak{S}$ can be constructed by an \'etale equivalence relation. Let us consider the fibers of the morphism $q:\mcal{K}\bs L\to \AA^2\bs \{0\}$. Observe that for the fibers over points $x\neq 0$, the Zariski locally trivial $\AA^1$-bundle is actually a trivial $\AA^1$-bundle because if $x\neq 0$, one has an element of the form $a = f(y)/x^m$, for some suitable polynomial $f(y)\in k[y]$ such that $D(a)= 1$, (i.e., $a$ is a slice). In other words, for a principal affine open set $U_x:=\{x\neq 0\}$ and a point $b \in U_x$, we have that $q^{-1}(b)\cong \AA^2\bs \{0\} \times \AA^1:= \Spec k[x^{\pm 1},t] \times \Spec k[y]$. 
\medskip

On the other hand, the fiber over the punctured affine line $C_0:= \{x=0\} = \Spec k[t^{\pm 1}]$ is given by $C\times \Spec k[z]$ where               
$$C:=\Spec k[t^{\pm 1},y]/\langle y^r +t^s\rangle.$$
The curve $C \to C_0$ can be identified with a finite \'etale cover of degree $r$ with respect to the projection $(y,t) \mapsto t$. This is the exact crossroad of the fields of positive characteristics and that of characteristic zero. If char $ k= p\mid r$, then one cannot form a finite \'etale cover as claimed above. But since we have that gcd $(r,s)= 1$, one can rework the proof as follows. If $p\mid r$, then one can then consider the finite \'etale cover 
    \begin{align*}
      &\Spec k[t^{\pm 1},y]/\langle y^r +t^s\rangle=: C \to C_0:=\Spec(k[t^{\pm 1}]) \\
      &\hspace{40mm} (y,t)\mapsto y    
    \end{align*}
of degree $s$ with respect to the other coordinate, where we can still extract the roots (i.e., $(y^r)^{1/s}$ exists in $k$) since in this case $p\nmid s$ by the coprimality.

\begin{lemma}\label{exist:algspace}\cite{Mad}
For any perfect field $k$, there exists a smooth algebraic space $\delta: \mathfrak{S}\to \AA^2_k\bs \{0\}$ such that for every $s\geq 1$, the map $q: \KK(s)\bs L\to \AA^2_k\bs \{0\}$ factors through an \'etale-locally trivial $\AA^1$-bundle $\rho: \KK(s) \bs L \to \mathfrak{S}$.
\end{lemma}

The working mechanism of the proof follows from a more general principle involved in constructing an algebraic space from a surface by replacing a desired curve by its finite \'etale covering (\cite[\S 1.0.1]{duboulozfinston2014}). For our context, we will now explain the strategy and the earlier-promised \'etale equivalence following \cite{DF18}.

\begin{proof}
Without loss of generality, let us assume that char $k= p \nmid r$. Observe that the strictly quasi-affine threefolds $\KK\bs L$ are covered by two principal affine opens, (say) $V_x:=\{x\ne 0\}$ where it is already a trivial $\AA^1$-bundle as observed above, and $V_t:= \{t\neq 0\}$ where the suitable modification has to be implemented. It is now enough to build an algebraic space $\mathfrak{S}_t$ locally on $V_t$ and prove that it compatibly extends to the desired algebraic space globally. 
\medskip

We first prove that there (locally) exists an algebraic space $\delta_t: \mathfrak{S}_t \to U_t$ such that $q|_{{V_t}}: V_t \to U_t$ factors through an \'etale locally trivial $\AA^1$-bundle $\rho_t: V_t\to \mathfrak{S}_t$. Moreover, such a factorization via $\delta_t$ should be an isomorphism when restricted to the overlap $U_{xt}:= U_x\cap U_t$. Here, $U_x= \Spec k[x^{\pm1},t]$ and $U_t = \Spec k[x,t^{\pm1}]$ are corresponding principal affine open subsets in $\AA^2\bs\{0\}$. The desired (global) algebraic space $\delta: \mathfrak{S}\to \AA^2\bs \{0\}$ is then obtained by gluing $U_x$ and $\mathfrak{S}_t$ by the identity along $U_{xt}$ and such that $\delta^{-1}_t(U_{xt}) \cong U_{xt}$.
    \[ \begin{tikzcd}
	{\mathcal{K}(s)\backslash L } && {\mathbb{A}^2\backslash\{0\}} \\
	\\  {\mathfrak{S}}
	\arrow["q", from=1-1, to=1-3]
	\arrow["\rho", from=1-1, to=3-1]
	\arrow["\delta", from=3-1, to=1-3]
         \end{tikzcd}
        \hspace{8mm}
        \begin{tikzcd}
	{V_t} && {U_t} \\
	\\	{\mathfrak{S}_t}
	\arrow["q_{\vert_{V_t}}", from=1-1, to=1-3]
	\arrow["\rho_t", swap, from=1-1, to=3-1]
	\arrow["\delta_t", swap, from=3-1, to=1-3]
        \end{tikzcd} \]
Let $h:\Spec R =: C\to C_0$ be the Galois closure of the finite \'etale map $h_0:C_1 \to C_0$ with 
    $$C_1:= \Spec \big(k[t^{\pm1}][y]\big) / \big(y^r+t^s\big) .$$ 
In other words, $C$ gives the normalization of $C_1$ in the Galois closure $\kappa$ with respect to the field extension 
     $$k(t)\hookrightarrow k(t)[y] / \big(y^r+t^s \big).$$ 
The above-mentioned Galois closure $\kappa$ is precisely obtained by adjoining the $r$th-roots to $t^s$, which is valid since $ p\nmid r$ by assumption. By the coprimality of $r$ and $s$, this is equivalent to just adjoining the $r$th-roots to $t$ along with all the $r$th-roots of $-1$. Thus, also note that neither $\kappa$ nor $C_1$ or $C$ depend on $s$. By construction, this gives us an \'etale $G$-torsor $h: C\to C_0$ with the Galois group $G = \text{Gal} (\kappa/k(t))$ and a factorization 
        $$h:C\xrightarrow{h_1} C_1\xrightarrow{h_0} C_0$$ 
where $h_1:C\to C_1$ is an \'etale torsor under some subgroup $H\le G$ of index $r$ equal to the degree of $h_1$. This corresponds to the fact that we have made an $r$-fold Galois cover of $C_0$. To complete the proof, we show that there exists a scheme $S$ obtained by gluing $r$ copies of $S_{\barg}$, where $\barg \in G/H$, by the identity along curves $C_0\times C$ such that the group $G$ acts freely on $S_{\barg}$ giving rise to a geometric quotient $\mathfrak{S}_t:= S/G$ in the category of algebraic spaces in the form of an \'etale $G$-torsor over $\mathfrak{S}$. Moreover, this local morphism glues to a global $G$-invariant morphism which then descends to the corresponding (global) quotient. By definition, the polynomial 
        $$ R[y] \ni y^r+t^s = \prod_{\Bar{g} \in G/H} \big(y-\lambda_{\Bar{g}}\big),$$
splits over $\kappa$ for some elements $\lambda_{\Bar{g}}\in R$ and $\Bar{g}\in G/H$. The Galois group $G$ then acts transitively on the quotient $G/H$ via 
         $$g' \cdot \lambda_{\Bar{g}} = \lambda_{\overline{{(g')^{-1}}\cdot g}}.$$
Since $h:C_1 \to C_0$ is \'etale, we have that for distinct elements $\barg,\barg'\in G/H$, the element $\lambda_{\barg}-\lambda_{\barg'}$ is an invertible regular function on $C$. Consequently, we can find a collection of elements $\sigma_{\barg(x)}\in B =: R[x]$ with corresponding residue classes $\lambda_{\barg}\in R= B/xB$ modulo $x$ on which $G$ acts via 
    $$ g'\cdot \sigma_{\barg}(x) \mapsto \sigma_{\overline{(g')^{-1}\cdot g}}(x) $$
along with a $G$-invariant polynomial $s(x,y)\in B[y]$ such that in the ring $B[y]$, we have 
\begin{equation}\label{termsofroots}
    y^r+t^s+x = \prod_{\barg\in G/G} (y-\sigma_{\barg}(x))+x^m s(x,y).
\end{equation}
Now observe that the principal open subset $V_t$ under the base change along $B$, given by $\hat{V}_t:=  V_t\times_{U_t} \Spec(B)$, is isomorphic to the closed subvariety of $\Spec(B[y,z_1])$ defined by the equation
\begin{equation}
   \hat{V}_t \cong \bigg\{ x^m z_1 = \prod_{\barg\in G/H} (y-\sigma_{\barg}(x))\bigg\} \subset \Spec(B[y,z_1])
\end{equation}
where $z_1= z-s(x,y)$. As mentioned previously, the regular function $\lambda_{\barg}-\lambda_{\barg'}$ is invertible whenever $g\neq g'$, the closed subscheme 
$\{x=0\} \subset \hat{V}_t$ is the disjoint union of $r$-closed subschemes $D_{\barg} \cong \Spec R[z_1]$ with corresponding defining ideals $(x,y-\sigma_{\barg}(x))\in \Gamma(\hat{V}_t,\mcal{O}_{\hat{V}_t})$ on which $G$ acts by permutation. Furthermore, the affine variety $\hat{V}_t$ is covered by affine open subsets of the form $\hat{V}_{t,\barg}$, where 
    $$\hat{V}_{t,\barg} = \hat{V}_t\bs \bigcup_{\barg'\in (G/H)\bs\{\barg\} }D_{\barg}$$
for $\barg \in G/H$. Now, due to the splitting of roots as in \cref{termsofroots}, we have a well-defined rational map
\begin{equation*}
    \hat{V}_{t,\barg} \dashrightarrow \Spec(B[u_{\barg}]),\quad 
    \big(x,y,z_1\big)\mapsto \left (x, \frac{y-\sigma_{\barg}(x)}{x^m}= \frac{z}{\prod_{\barg'\in (G/H)\bs \{\barg'\}}(y- \sigma_{\barg'}(x))} \right)
\end{equation*}
which is in fact an isomorphism of schemes over $\Spec B$. As a consequence, one verifies that the morphism $\hat{q}: \hat{V}_t=V_t\times_{U_t} \Spec B \to \Spec B$ is faithfully flat and that $\hat{q}$ factors through a Zariski locally trivial $\AA^1$-bundle $\hat{\rho}:\hat{V}_t\to S$ over the scheme $\hat{\delta}: S \to \Spec B$. The latter scheme $S$ is obtained by gluing $r$ copies of $S_{\barg}$, for $\barg\in G/H$, of $\Spec B \cong C\times \AA^1$ by identity along the principal open subsets 
$$\Spec B_x\cong C\times \Spec k[x^{\pm1}] \subset S_{\barg}.$$
Put differently, the morphism $\hat{\rho}_t: \hat{V}_t\to S$ is a Zariski-locally trivial $\AA^1$-bundle with local trivializations $\hat{V}_t\ \vert_{S_{\barg}}\cong \Spec B[u_{\barg}]$ and transition functions over $S_{\barg}\cap S_{\bar{g}'}\cong \Spec B_x$ of the form 
        $$u_{\barg} \mapsto u_{\barg'}+ x^{-m}( \sigma_{\barg}(x) -\sigma_{\barg'}(x)).$$
The action of $G$ on $\hat{V}_t$ descends to a fixed point-free action on $S$ defined locally by 
    $$S_{\barg}\ni (c,x)\mapsto (g'\cdot c, x) \in S_{\overline{(g')^{-1}\cdot g}}.$$
This $G$-action produces a quotient which is a scheme only if $h_0:C_1 \to C_0$ is an isomorphism; nevertheless always exists as an algebraic space in the form of an \'etale $G$-torsor $S \to \mathfrak{S}_t:= S/G$. By construction of $\mathfrak{S}_t$, we get the following cartesian square
\[\begin{tikzcd}
        \hat{V}_t && V_t \cong {\hat{V}_t}/G \\ \\
	S && \mathfrak{S}_t = S/G
	\arrow[from=1-1, to=1-3]
	\arrow["\hat{\rho}_t"', from=1-1, to=3-1]
	\arrow["\rho_t", from=1-3, to=3-3]
	\arrow[from=3-1, to=3-3]
\end{tikzcd}\]
with the horizontal maps being \'etale $G$-torsors. Consequently, the induced morphism $\rho_t: V_t\to \mathfrak{S}_t$ is an \'etale locally trivial $\AA^1$-bundle. It remains to observe that the $G$-invariant morphism 
    $$\pr_1 \circ \hat{\delta}: S\to \Spec B \cong U_t \times_{C_0} C \to  U_t$$ 
descends to a morphism $\delta_t: \mathfrak{S}_t \to U_t$ and $\hat{\delta}$ restricts to an isomorphism outside $\{x=0\}\subset U_t$. In contrast, over the punctured fiber, we have that the $\delta^{-1}(\{x=0\})$ is isomorphic to the quotient of $C \times G/H$ by the diagonal action of $G$, and hence to $C/H\cong C_1$. Thus, the algebraic space $\mathfrak{S}_t$ extends to globally $\delta: \mathfrak{S} \to \AA^2\bs\{0\}$.
\end{proof}

\begin{remark}
The perfectness assumption is not strict in \cref{exist:algspace}. In fact, any base field of char $k= p \mid r$ (resp. $p \mid s$) where one can extract the $s$th root of $y$ (resp. $r$th root of $t$) has the desired properties furnished in \cref{exist:algspace}.
\end{remark}

\subsection{\texorpdfstring{$\AA^1$}{A1}-Brouwer Degree}
Let $g:\mcal{K}\bs L\to \AA^2\bs \{0\}$ be the $\AA^1$-weak equivalence from \cref{g-is-weq}. Recall that our aim is to show that the morphism $i: \AA^2\bs \{0\}\to \KK \bs L$ is an $\AA^1$-weak equivalence in $\Spc_k$. To this end, it is enough to show that the composite $i\circ g: \Abs \to \Abs$
$$\Abs\xrightarrow{i} \mcal{K}\bs L \xrightarrow{g}\Abs $$
is an $\AA^1$-weak equivalence in $\Spc_k$. This follows due to the two-out-of-three property of model categories (\cref{defn:model-cats}). To show that $i\circ g$ is an $\AA^1$-weak equivalence, we use the associated $\AA^1$-Brouwer degree map from the motivic homotopy theory (\cref{A1-Brouwer-degree}). The morphism $g$ being an $\AA^1$-weak equivalence, we can identify $\mcal{K}\bs L$ in the Milnor-Witt $K$-theory as $H^1(\mcal{K}\bs L, K_2^{\MW}) = K_0^{\MW}(k)\cdot \mu$, where $\mu = g^{*}(\xi)$ is the pullback of the explicit generator. In the view of \cref{A1degreemap}, it is enough to show that the induced map on the Milnor-Witt $K$-theory sheaves by $i$
$$i^*: H^1(\mcal{K}\bs L, K^{\MW}_2) \to  H^1(\Abs,K^{\MW}_2)$$ 
is an isomorphism of groups. The diagram \cref{setup:diagram1} induces the following commutative diagram in Milnor-Witt $K$-theory:
\begin{equation}
\begin{tikzcd}[sep= tiny, every matrix/.append style={nodes={font=\scriptsize}}, every label/.append style={font=\small}]
        {H^1(\mathcal{K}\backslash L, K_2^{\MW})} && {H^2(\mathcal{K}/(\mathcal{K}\backslash L), K_2^{\MW})} && {H^2(\mathcal{K},K_2^{\MW})} && {H^2(\mathcal{K}\backslash L, K_2^{\MW})} \\
	\\
        {H^1(\mathbb{A}^2\backslash \{0\},K_2^{\MW})} && {H^2((\mathbb{P}^1)^{\wedge 2}, K_2^{\MW})} && {H^2(\mathbb{A}^2, K_2^{\MW})} && {H^2(\mathbb{A}^2\backslash \{0\}, K_2^{\MW})}
	\arrow["\partial", from=1-1, to=1-3]
	\arrow["{i^*}"', from=1-1, to=3-1]
	\arrow[from=1-3, to=1-5]
	\arrow["{j^*}"', from=1-3, to=3-3]
	\arrow[from=1-5, to=1-7]
	\arrow[from=1-5, to=3-5]
	\arrow["{i^*}"', from=1-7, to=3-7]
	\arrow["{\partial'}"', from=3-1, to=3-3]
	\arrow[from=3-3, to=3-5]
	\arrow[from=3-5, to=3-7]
\end{tikzcd}
\end{equation}
Since $j$ is an $\AA^1$-weak equivalence, this implies that $j^*$ is an isomorphism. By $\AA^1$-homotopy invariance of Milnor-Witt $K$-theory, we have that it vanishes for affine spaces, in particular, $H^*(\AA^2, K_2^{\MW})=0$, which, due to the exactness, renders that $\partial'$ is also an isomorphism. Observe that the groups $H^1(\mcal{K}\bs L, K^{\MW}_2)$ and $H^2(\mcal{K}/(\mcal{K}\bs L), K^{\MW}_2)$ are actually free $K^{\MW}_0$-modules of rank 1 (\cite[Lemma 4.5]{asok2014algebraicspheres}). This implies that $\partial$ is $K^{\MW}_0(k)$-linear and, in particular, an injective group homomorphism. Thus, it remains to only show that $\partial$ is surjective. This is proven below in \cref{cocyle-is-surjective} by showing that every cocycle of $H^2(\mcal{K}/(\mcal{K}\bs L),K^{\MW}_2)$ is a boundary of some cocycle of $H^1(\mcal{K}\bs L, K^{\MW}_2)$. In other words, for every such $\beta \in H^2(\mcal{K}/(\mcal{K}\bs L), K^{\MW}_2)$, there exists a cocycle $\alpha\in H^1(\mcal{K}\bs L, K^{\MW}_2)$ such that $\partial(\alpha)= \beta$.

\begin{prop}\label{cocyle-is-surjective}\cite{Mad}
For any perfect field $k$, the connecting homomorphism 
 $$\partial: H^1(\mcal{K}\bs L, K^{\MW}_2)\to H^2(\mcal{K}/(\mcal{K}\bs L),K^{\MW}_2)$$ 
is surjective. 
\end{prop}
Except for our style of presentation, the contents of the proof are essentially the same as in \cite[Proposition 3.3]{DF18}. We will include it here for concreteness.
\begin{proof}
Due to \cite{fasel2020Chow-Witt}, we have the relation (cf. \cref{app:Chow-Witt-reduced-coh})
    $$H^n(\KK /(\KK\bs L), K^{\MW}_*) \cong H_L^n(\KK,K^{\MW}_*)$$
for all $\forall\ n\in \ZZ$. In particular, we have that $H^2(\KK/(\KK\bs L), K^{\MW}_2) \cong H_L^2(\KK,K^{\MW}_2)$. Recall from \cref{A1-weakeq-of-j} that the normal bundle of $L$ in $\KK$ is explicitly generated by the sections $y$ and $t$. Using the localization exact sequence of Chow-Witt theory (\cite[\S 10.4]{fasel2008groupeC-W}), this provides us with a long exact sequence and, as a result, also the generators for $H^2_L(\KK, K^{\MW}_2)$ explicitly as the following class of cocycle:
\begin{equation}
    \beta:= \langle1\rangle\otimes \bar{t} \otimes \bar{y} \in K^{\MW}_0(\kappa(L), \wedge^2 \mathfrak m_L/ \frak m^2_L). 
\end{equation}
Now, to prove our claim, it is enough to show that this cocycle is indeed a boundary of a cocycle $\alpha \in H^1(\KK\bs L, K^{\MW}_2)$. The proof exploits several machinery developed in \cite[\S 5.1]{morel2012A1topology}. Consider the integral subvarieties of $\KK$ as follows: $M:= \{y=0\}\subset \KK$ and $N:=\{t=0\}\subset \KK$ and $L':=\{y=t=0,\ x^{m-1}z=1 \}\subset \KK$. By the construction, we have $M\cap N = L\bigsqcup L'$. By the definition of $M$ and $N$, we have that the element $\gamma:= [y]\otimes \bar{t}\in K^{\MW}_1(\kappa(N), \frak{m}_N/\frak{m}^2_N)$ has a non-trivial boundary only on $L$ and $L'$ with same values of the boundaries in respective residue fields, that is, the boundary of $\gamma$ on $L$ is given by $\langle1\rangle \otimes \bar{t} \wedge \bar{y}\in K^{\MW}_0(\kappa(L),\wedge^2 (\frak{m}_L/ \frak{m}^2_L))$ and on $L'$ given by $\langle1\rangle \otimes \bar{t} \wedge \bar{y}\in K^{\MW}_0(\kappa(L'),\wedge^2 (\frak{m}_{L'}/ \frak{m}^2_{L'}))$. Eventually, it turns out that $\gamma \in  K^{\MW}_1(\kappa(N),\frak{m}_N/\frak{m}_N^2)$ is not a cocycle on $\KK \bs L$, but we can produce one by doing the following modification.
\medskip

Note that the symbol $\alpha':= [x^{m-1}z-1]\otimes \bar{y}\in K^{\MW}_1(\kappa(M), \frak{m}_M/\frak{m}^2_M)$ has non-trivial boundary only on $L'$. Indeed, from the equation of $\KK$, we have that $x(x^{m-1}z-1) = y^r+t^s$ and again since $r,s$ are coprime, we have that any prime containing $x^{m-1}z-1$ and $y$ must also contain the $t$ and in turn also the prime ideal corresponding to $L'$ by the definition of $L'$. We now compute the boundary of $\alpha'$ on $L'$, taking advantage of the above-mentioned primality condition of the ideals. Since $x$ is invertible in $L'$, we can write $(x^{m-1}z -1) = (y^r+t^s)/x = x^{-1} y^r+ x^{-1} t^s \in \mcal{O}_{\frak{m}_{L'}}$. Hence, on $L'$, the boundary of $\alpha'$ is the same as the boundary of $[x^{-1}t^s]\otimes \bar{y}$. The latter can be computed using the explicit relations given in \cite[Lemma 3.5]{morel2012A1topology}. In particular, we have that 
    $$[x^{-1}t^s] = [x^{-1}] + \langle x^{-1} \rangle [t^s] = [x^{-1}] + \langle x \rangle [t^s].$$
Now due to \cref{MW:epsilon-defn}, we have that $[t^s] = s_{\epsilon}[t]$. Finally, using the boundary formulas as given in \cite[Proposition 3.17]{morel2012A1topology}, we obtain the boundary of $\alpha'$ as 
$$d (\alpha') =  \langle x \rangle s_{\epsilon}\otimes \bar{y}\wedge \bar{t} = \langle -x\rangle s_{\epsilon}\otimes \bar{t}\wedge \bar{y}.$$
Set $S:= \{x^{m-1} z=1 \}\subset \KK$ as a codimension 1 subvariety of $\KK$. Since $M$ and $S$ are different codimension 1 subvarieties, we compute the boundary as follows:
$$d([y, x^{m-1}z-1]) = [x^{m-1}z-1]\otimes \bar{y} +\epsilon [y]\otimes \overline{x^{m-1}z-1} $$
As $d^2=0$ and $\epsilon = -\Lin -1\Rin$, we have
$$ d([y] \otimes \overline{x^{m-1}z-1}) = \langle -1 \rangle\ d( [x^{m-1}z-1]\otimes \bar{y}) = \langle x \rangle s_{\epsilon}\otimes \bar{t}\wedge \bar{y}.$$ Now $x$ is an unit in $S$, we have that $d(\Lin x\Rin [y]\otimes \overline{x^{m-1}z-1}) = s_{\epsilon} \otimes \bar{t}\wedge \bar{y}$. Similarly, $d([x^{m-1}z-1] \otimes \bar{t}) = \Lin x\Rin r_{\epsilon} \otimes \bar{t}\wedge \bar{y}$ and as a consequence, we have 
$$d([t]\otimes \overline{x^{m-1}z-1}) = \Lin -1\Rin d([x^{m-1}z-1]\otimes \bar{t})= \Lin -x \Rin r_{\epsilon}\otimes \bar{t}\otimes \bar{y}. $$ 
Thus, $d(\Lin-x\Rin[t]\otimes x^{m-1}z-1) =r_{\epsilon} \otimes \bar{t}\wedge \bar{y}$. Now, choose $g,h\in \mathbb{N}$ such that $gr-hs=1$, which is possible as $(r,s)=1$. For any integers $p,q\in \ZZ$, we have $(pq)_{\epsilon} = p_{\epsilon} q_{\epsilon}$, and so we have $g_{\epsilon} r_{\epsilon} - h_{\epsilon} s_{\epsilon} = \Lin \pm 1 \Rin$ \footnote{it is odd if $gr$ is odd and $\Lin-1\Rin$ otherwise}. In the end, we have obtained that $d(g_{\epsilon}\Lin-x\Rin [t]  \otimes \overline{x^{m-1}z-1}) - h_{\epsilon}\Lin x\Rin [y] \otimes (\overline{x^{m-1}z-1})) = \Lin \pm1 \Rin \otimes \bar{t}\wedge \bar{y}$. This gives us the desired cocycle obtained by modifying $\alpha'$ by setting 
$$\alpha:=  [y]\otimes \bar{t} - \Lin\pm 1\Rin (g_{\epsilon} \Lin -x \Rin[t]\otimes \overline{(x^{m-1}z-1)} - h_{\epsilon} \Lin x\Rin[y]\otimes \overline{(x^{m-1}z-1)})$$
that maps to the generator of $H^2_L(X,K^{\MW}_2)$ under  $\partial$.
\end{proof}

\subsection*{A note on the \emph{lazy} proof}
In the case when char $k=0$, the authors in \cite{DF18} present another proof, which they call the \emph{lazy proof}, exploiting the stable $\AA^1$-contractibility of $\KK$ (\cite[Theorem 4.2]{HKO16}) over a field of characteristic zero. The proof strongly in the stable context relies on the deep fact that every smooth quasi-projective scheme is dualizable in the stable homotopy category $\SH(k)$. This is, in turn, proven by exploiting the existence of resolution of singularities in characteristic zero. Hence, assuming the existence of resolution of singularities over perfect fields, the lazy proof can very well be adapted to our situation.
\medskip

\begin{remark}\cite{Mad}
The family of smooth affine threefolds $\KK$ produces the first example of exotic varieties in mixed characteristics. For instance, this has the consequence that the affine variety $\KK\to \Spec\FF_p$ is a smooth $\AA^1$-contractible threefold which is not isomorphic to the Asanuma-Gupta threefolds (\cref{ZCP:positive-char}). Thereby, all these threefold add to the list of "potential" counter-examples to the Zariski Cancellation in positive characteristics. In view of this distinction, we can propose the following action plans:
\begin{enumerate}\item Construct an invariant to distinguish between a given smooth $\AA^1$-contractible affine threefold and Koras-Russell threefolds.
\end{enumerate}
and in analogy with \cref{A1contracurves:affine} and  \cref{A1contrasurface:affine}:
\begin{enumerate}
\item[2.] Prove or disprove the fact that every smooth $\AA^1$-contractible threefold is necessarily affine. To disprove, it suffices to find a strictly quasi-affine smooth $\AA^1$-contractible threefold.
\end{enumerate}
\end{remark}

\subsection{Relative \texorpdfstring{$\AA^1$}{A1}-Contractibility of \texorpdfstring{$\KK$}{K}}
We can now extend the unstable $\AA^1$-contractibility of $\KK$ over integers and over an arbitrary Noetherian scheme with perfect residue fields.
\begin{prop}\label{A1-cont-over-perfect-fields}\cite{Mad}
For any perfect field $k$, the canonical morphism $f:\KK\to \Spec k$ is an $\AA^1$-weak equivalence in $\Spc_k$.
\end{prop}
\begin{proof}
From \cref{A1-weakeq-of-j} and \cref{A1-weakeq-of-i}, we have that the morphisms $i$ and $j$ are $\AA^1$-weak equivalences. 
Consequently, $\phi:\AA^2 \to \KK$ is an $\AA^1$-weak equivalence in \cref{setup:diagram1} and $\AA^2$ is clearly $\AA^1$-contractible over $\Spec k$.
\end{proof}

\begin{theorem}\label{KR3FoverZ}\cite{Mad}
The Koras-Russell three fold of first kind $f: \KK\to \Spec \ZZ$ is relatively $\AA^1$-contractible in $\Spc_{\ZZ}$. In particular, $\KK$ is an $\AA^1$-contractible scheme in $\Spc_{\ZZ}$ that is not isomorphic to $\AA^3_{\ZZ}$.
\end{theorem}
\begin{proof}
The fibers of $f$ over the generic point $0 \in \Spec \ZZ$ are $\AA^1$-contractible, as the quotient field $\Spec \QQ$ has characteristic zero, for which the proof follows due to \cite{DF18}. Now, for all closed points $q \in \Spec\ZZ$, we have that the corresponding residue field is $\FF_q$ (which in particular is perfect), the $\AA^1$-contractibility over $\Spec\FF_q$ now follows from \cref{A1-cont-over-perfect-fields}. 
\[\begin{tikzcd}
	& {\KK_o} & \KK & {\KK_{q}} \\
	\\
	{\Spec \QQ \cong \Spec \kappa_{o}} && {\Spec \ZZ} && {\Spec \kappa_q \cong \Spec \FF_q}
	\arrow[from=1-2, to=1-3]
	\arrow["{{f_o}}"', from=1-2, to=3-1]
	\arrow["f", from=1-3, to=3-3]
	\arrow[from=1-4, to=1-3]
	\arrow["{{f_q}}", from=1-4, to=3-5]
	\arrow[from=3-1, to=3-3]
	\arrow[from=3-5, to=3-3]
\end{tikzcd}\]
Now, $f:\KK\to \Spec \ZZ$ is a morphism of motivic spaces whose fibers are all $\AA^1$-contractible over the corresponding residue fields. The proof now follows as a consequence of \cref{pointwise:phenomenon}.
\end{proof}

In fact, the above theorem can be generalized as follows.
\begin{theorem}\label{KR3Fmain:Noetherian}\cite{Mad}
Let $S$ be any Noetherian scheme with perfect residue fields. The Koras-Russell threefold of first kind $f: \KK \to S$ is relatively $\AA^1$-contractible in $\Spc_S$.
\end{theorem} 
\begin{proof}
Let us first note that the canonical map $\KK \to S$ is well-defined and smooth due to \cref{smooth:over-arbitrary-base} as all the fibres of $f$ over either the quotient field or corresponding residue fields are smooth schemes. Now, choose a point $s \in S$ with residue fields $\kappa_s$. Now, if $\kappa_s$ has characteristic zero, then the $\AA^1$-contractibility of $\KK \to \Spec \kappa_s$ follows from \cite{DF18}. If char $\kappa_s = p >0$, it follows from \cref{KR3FoverZ} that $\KK \to \Spec \kappa_s$ is $\AA^1$-contractible since by assumption we have that $\kappa_s$ is a perfect field. Now, we have that all the fibres of the morphism $f:\KK\to S$ are $\AA^1$-contractible, and the proof follows due to \cref{pointwise:phenomenon}.
\end{proof}

\subsubsection{A Note on the Deformed Version of $\KK$}
Let us also mention a closely related family of affine threefolds called the \emph{deformed Koras-Russell threefolds} of the first kind studied by the authors in \cite{DPO2019}. They are defined by the equation:
\begin{align*}
   \KK(m,r,s,\phi) =: \KK(\phi) = \{x^m z = y^r + t^s+ x \phi(x,y,t) \}
\end{align*}
where $m,r,s\ge 2$ are integers with $r$ and $s$ are coprime and $\phi(x,y,t)\in k[x,y,t]$ such that $\phi(0,0,0)\in k^\times$. Observe that when $\phi(x,y,t)= a \in k^\times$, then $\KK(\phi)= \KK$ (\cref{defn:KR3F}). And when $\phi(x,y,t)= q(x)\in k[x]$ with $q(0)\in k^\times$, then the corresponding affine threefolds $\KK(q)$ has been shown to be $\AA^1$-contractible over a field of characteristic zero (\cite[Theorem 1.2]{DF18}). Now, for a general polynomial $\phi(x,y,t)$ satisfying $\phi(0,0,0)\in k^\times$, it is known that the corresponding affine varieties $\KK(\phi)$ form examples of the affine modification, in the sense of Kaliman-Zaidenberg (\cite{kalimanZaid1999affine}). In this direction, it has been shown that $\KK(\phi)$ is $\AA^1$-contractible in the stable $\AA^1$-homotopy category $\SH^{S^1}(k)$ (\cite[Theorem 4.7]{DPO2019}) after simplicial suspension over an algebraically closed field of characteristic zero. In addition, the authors in [\textit{ibid}] also study the notion of Koras-Russell fiber bundles, whose definition we include below.
\begin{defn}
Suppose $s(x) \in k[x]$ has positive degree and let $R(x,y,t) \in k[x,y,t]$. Define the closed subscheme $\XX(s,R)$ of $\AA^1_x \times \AA^3 = \Spec(k[x][y,z,t])$ by the equation:
        $$\{s(x)z = R(x,y,t) \}. $$
We say that the projection map $\pr_x: \XX(s,R)\to \AA^1_x$ defines a Koras-Russell fiber bundle if
\begin{enumerate}
\item $\XX(s,R)$ is a smooth scheme, and
\item For every zero $x_0$ of $s(x)$, the zero locus in $\AA^2 = \Spec(k[y,t])$ of the polynomial $R(x_0,y,t)$ is an integral rational plane curve with a unique place at infinity and at most unibranch singularities.
\end{enumerate}
\end{defn}  
The simplicial stable $\AA^1$-contractibility of $\XX(s, R)$ has been established over an algebraically closed field of characteristic zero (\cite[Theorem 4.13]{DPO2019}). In the light of \cref{A1-cont-over-perfect-fields}, we can now essentially lift these results over perfect fields as well\footnote{this is a work in progress}.

%--------------------------------------------------------------------------------------
\section{Generalized Koras-Russell Prototypes}\label{sect:generalized-KR}
In this section, we turn towards studying the higher-dimensional analog of the Koras-Russell threefolds, named as the \emph{generalized Koras-Russell varieties}. The authors in \cite{dubouloz2025algebraicfamilies} recently proved its unstable $\AA^1$-contractibility over a field of characteristic zero. In this spirit, we will now lift the $\AA^1$-contractibility from characteristic zero to arbitrary perfect fields and consequently over some Noetherian scheme. As a sequel, we shall see in \cref{sect:exotic-motivic-spheres} that these varieties are the crucial ingredients to produce exotic motivic spheres.
\medskip

Let us begin with the definition of these varieties. As for the notation, $k$ will denote a perfect field and $k^{[n]}$ will denote the polynomial ring in $n$-variables.

\begin{defn}\label{defn:KR3F-generalized}
Let $r,s \geq 2$ be any coprime integers. Consider the triple $(m,\ul{n},\psi)$ with integers $m\ge 0$ and a $(m+1)$-tuple $\ul{n} = (n_0,n_1,\dots,n_m)$ of multi-index with $n_i>1$ and a polynomial $\psi \in k^{[1]}$ such that $\psi(0)\neq 0$. Then consider the coordinate ring 
    $$R_m(\ul{n},\psi):= \frac{k[x_0, x_1\dots,x_m,y,z,t]}
    {\ul{x}^{\ul{n}}z + y^r+t^s+x_0\ \psi(\ul{x})}$$
where $\ul{x}^{\ul{n}}= \prod_{i=0}^{m} x_i^{n_i}$ and $\ul{x}= \prod_{i=0}^{m}x_i$. The \emph{generalized Koras-Russell varieties} are defined as the affine variety associated with this coordinate ring
\begin{equation}\label{fulleqn:GK-R3F}
    \XX_m(\ul{n},\psi):= \Spec R_m(\ul{n},\psi)
\end{equation}        
\end{defn}
For every $m\ge 0$, the so-obtained affine variety $\XX_{m}(\ul{n},\psi)$ is of dimension $(m+3)$ and is smooth by the Jacobian criterion, by the assumption that $\psi
(0)\neq 0$. Let us now look at a concise example of one such variety in relative dimension $m+3$ constructed over the affine space $\AA^{n-3}_k$. Putting $\psi(\ul{x}) = 1+\ul{x}+ \sum_{i=2}^{n-2}a_i \ul{x}^i$ which is clearly non-vanishing at the origin (as $\psi(0)=1 \ne 0$), we get that
\begin{equation}\label{deformed-example}
\XX_m(n,\psi) := \bigg\{\ul{x}^nz = y^r+t^s+x_0 (1+\ul{x}+\sum_{i=2}^{n-2}a_i \ul{x}^i)\bigg\}\subseteq \AA^{n-3}_k\times \AA^{m+3}_k
\end{equation}
with coordinates given by $\Spec (k[a_2,\dots,a_{n-2}][x_0,\dots,x_m,y,z,t])$. In particular, $\XX_m(n,\psi)$ is a smooth affine scheme under the projection map $\XX_m(n,\psi) \to \AA^{n-3}_k$ given by the projection onto the $a_i$'s.

\begin{theorem}\label{KR3Fprototypes:field-A1-cont:perfect}\cite{Mad}
Let $k$ be any perfect field. Then under the notations and assumptions as in \cref{defn:KR3F-generalized}, the canonical morphism $\XX_{m}(\ul{n},\psi) \to \Spec k$ is an $\AA^1$-weak equivalence in $\Spc_k$.
\end{theorem}
\begin{proof}
The authors in \cite{dubouloz2025algebraicfamilies} have proven the $\AA^1$-contractibility of these varieties over fields of characteristic zero. We can now promote this result over perfect fields in the light of \cref{A1-cont-over-perfect-fields} following the original proof. Let us first note that all these varieties are all stably $k$-isomorphic (\cite[Proposition 3]{dubouloz2025algebraicfamilies}), that is, $\XX_{m}(\ul{n},\psi)\times \AA^1_k \cong \XX_m(\ul{n},1)\times \AA^1_k$, for any polynomial $\psi\in k^{[1]}$ as above. Hence, in the light of $\AA^1$-homotopy invariance (\cref{defn:A1-invar-sheaf}), it is enough to show that $\XX_m:= \XX_m(\ul{n},1)$ is $\AA^1$-contractible. Also, note that when $m=0$ and $\psi(\ul{x})=1$, $\XX_{0}(n_0,1)$ is precisely the family of Koras-Russell threefolds studied in the previous section, which we proved to be $\AA^1$-contractible (\cref{A1-cont-over-perfect-fields}). 
\medskip

Fix $A_{m-1} = k[x_1,\dots,x_{m-1},y,t]$, $R_m = A_{m-1}[x_0,x_m,z]/\big( \prod_{i=0}^{m}x_i^{n_i}z + y^r +t^s +x_0\big)$. Then let us consider the following closed subschemes in $\XX_m$
\begin{align}
& W_m = \{x_m=0\}\cong \Spec (R_m/x_m R_m)\cong \Spec(A_{m-1}[z]) \cong \AA^{m+2}_k \label{W_m} \\
& H_m = \{z=0\}\cong \Spec (R_m/z R_m)\cong \Spec(A_{m-1}[x_m])\cong \AA^{m+2}_k \label{H_m}\\
& P_m = W_m\cap H_m \cong \Spec (R_m/(x_m, z) R_m)\cong \Spec(A_{m-1}) \cong \AA_k^{m+1} \label{P_m}
\end{align}
Now consider the following canonical closed immersions $j_m: H_m\to \XX_m$ and $i_m: H_m\bs P_m\to \XX_m \bs W_m$. By setting $Z:= x_m z$, we have an isomorphism 
$$\XX_m\bs W_m \cong \Spec(A_{m-2}[x_0,x_{m-1},Z][x^{-1}_m]/(x_0^{n_0}\prod_{i\neq m}x_i^{n_i}Z +y^r+t^s+x_0)) \cong \XX_{m-1}\times \GG_{m,k} $$
for which $i_m$ equals the product 
$$j_{m-1}\times id_{\GG_{m.k}}: H_m\bs P_m \cong H_{m-1}\times_k \GG_{m.k} = \Spec(A_{m-2}[x_{m-1}][x_m^{\pm1}])\to \XX_{m-1}\times_k \GG_{m.k}$$
Now, by setting up the cofiber sequence as in \cref{setup:diagram1}, we get the following commutative square
\begin{equation}
\begin{tikzcd}[column sep=13pt, row sep=20pt]
   &&&&& {H_m\bs P_m}   && {H_m}   && {H_m/ (H_m\bs P_m)} \\\\ 
   &&&&& {\XX_m\bs W_m} && {\XX_m} && {\XX_m/(\XX_m\bs W_m)}	
        \arrow[from=1-6, to=1-8]
	\arrow["i_m"', from=1-6, to=3-6]
	\arrow[from=1-8, to=1-10]
	\arrow["j_m", from=1-8, to=3-8]
	\arrow["l_m", from=1-10, to=3-10]
	\arrow[from=3-6, to=3-8]
	\arrow[from=3-8, to=3-10]
\end{tikzcd}
\end{equation}
associated to the open immersions $H_m\bs P_m\to H_m$ and $\XX_m\bs W_m \to W_m$. Observe that $H_m$ is $\AA^1$-contractible by \cref{H_m} and so $\XX_m$ is $\AA^1$-contractible provided $j_m$ is an $\AA^1$-weak equivalence. For $m=0$, the morphism $j_0:H_0\to \XX_0$ is an $\AA^1$-weak equivalence due to \cref{A1-cont-over-perfect-fields}. We now establish this property for $m\ge 1$ by induction on $m$. 
\medskip

To show that $j_m$ is an $\AA^1$-weak equivalence, it is enough to show that $i_m$ and $l_m$ are $\AA^1$-weak equivalences in $\Spc_k$ by the \cref{weak5lemma}. The $\AA^1$-weak equivalence of $i_m$ follows from the fact that $i_m\simeq j_{m-1}\times id_{\GG_{m,k}}$ and  by the induction hypothesis, $j_{m-1}$ is an $\AA^1$-weak equivalence. On the other hand, $l_m$ is also an $\AA^1$-weak equivalence. Indeed, as $P_m$ is obtained as the transversal intersection of two smooth subvarieties inside $\XX_m$, the normal cone $N_{P_m}H_m$ of $P_m$ in $H_m$ is in fact the restriction of the normal cone $N_{W_m}\XX_m$ of $W_m$ in $\XX_m$ to $P_m$. But the normal cone of $N_{W_m} X_m$ is trivial, attributed to the analogous reason as in \cref{A1-weakeq-of-j}, whence the triviality of $N_{P_m}H_m$. The proof now follows from the homotopy purity \cref{purity-theorem} as all these varieties are smooth. In particular, 
$$H_m/(H_m\bs P_m)\simeq \Th(N_{P_m}H_m) \simeq  P_m\wedge \PP^1$$ 
and similarly, 
$$\XX_m/(\XX_m\bs W_m) \simeq \Th(N_{W_m}\XX_m) \simeq W_m\wedge \PP^1.$$
The map $l_m$ coincides with the following map observed under the isomorphisms mentioned above
$$\Th(N_{P_m}H_m)\simeq P_m\wedge \PP^1 \xrightarrow{\iota\wedge \PP^1} W_m\wedge \PP^1\simeq \Th(N_{W_m}\XX_m).$$
obtained as the $\PP^1$-suspension of the closed immersion $\iota: P_m\hookrightarrow W_m$. Hence, the $\AA^1$-weak equivalence of $\iota$ implies that of $l_m$.
\end{proof}

\begin{corollary}\label{KR3F:prototypes-A1-cont-base:Noetherian}\cite{Mad}
Let $S$ be any Noetherian scheme with perfect residue fields. Then under the assumptions as in \cref{fulleqn:GK-R3F}, the canonical morphism $\varphi: \XX_{m}(\ul{n},\psi) \to S$ is an $\AA^1$-weak equivalence in $\Spc_S$.
\end{corollary}
\begin{proof}
The strategy of proof is precisely the same as that of \cref{KR3Fmain:Noetherian}. Choose any point $s\in S$ and consider the induced morphism on the corresponding fiber $\varphi_s:\XX_m(\ul{n},\psi) \times_{\kappa_s} \Spec\kappa_s =: (\XX_m(\ul{n},\psi))_s \to \Spec\kappa_s$. 
\[\begin{tikzcd}
	{\XX_m(\ul{n},\psi)_s} && {\XX_m(\ul{n},\psi)} \\ \\
	{\Spec \kappa_s} && S
	\arrow[from=1-1, to=1-3]
	\arrow["{{\varphi_s} }"', from=1-1, to=3-1]
	\arrow["\varphi", from=1-3, to=3-3]
	\arrow[from=3-1, to=3-3]
\end{tikzcd}\]
Since by the hypothesis such an $\kappa_s$ is perfect, $\varphi^{-1}(s)$ is an $\AA^1$-weak equivalence in $\mcal{H}(\kappa_s)$ owing to \cref{KR3Fprototypes:field-A1-cont:perfect}. The proof now follows from \cref{pointwise:phenomenon}.
\end{proof}

\begin{remark}
The authors in \cite{dubouloz2025algebraicfamilies} remark that $\XX_m$ are "potential" counterexamples to the Zariski Cancellation over fields of characteristic zero in all dimensions $\ge 3$. Conforming to their fact, our extension over perfect fields \cref{KR3Fprototypes:field-A1-cont:perfect} opens up a path to view these varieties $\XX_m$ as "potential" counterexamples to the Zariski cancellations even over fields of positive characteristics. Moreover, none of these varieties can be isomorphic to Asanuma-Gupta varieties (the latter have smooth fibers while the generic fiber of $\XX_m$ is singular), making it a familiar, yet novel candidate in disguise.
\end{remark}

\subsubsection*{Moduli of Arbitrary Smooth Exotic Affine Varieties}
Recall from \cref{sect:higherdimensions} that the authors in \cite[Theorem 5.1]{asok2007unipotent} constructed an infinite collection of pairwise non-isomorphic exotic $\AA^1$-contractible strictly quasi-affine varieties of relative dimensions $\ge 4$. The picture in dimension 1 (\cref{1-dim theorem}) and dimension 2 (\cref{A2isunique}) being filled, the missing gaps were 
\begin{enumerate}
\item[1.] Does there exist an analogous picture for threefolds?
\end{enumerate}
This gap was answered by the authors in \cite{DF18}, exploiting the $\AA^1$-contractibility of Koras-Russell threefolds. In particular, they \cite[Corollary 1.3]{DF18} devised a general method to produce a moduli of arbitrary positive dimension of pairwise non-isomorphic, stably isomorphic, $\AA^1$-contractible smooth affine threefolds using techniques from \cite{dubouloz2011noncancellation} that were initially established over $\CC$.

\begin{enumerate}
\item[2.] If such families consistently arose out of only strictly quasi-affine schemes in dimension $\ge 4$, is there a (purely) affine family with such a property in dimensions $\ge 4$?
\end{enumerate}
The generalized Koras-Russell varieties $\XX_m(n,\psi)$ answer the question above.
\medskip

As a favorable addition, our results (\cref{KR3Fprototypes:field-A1-cont:perfect}) and \cref{KR3F:prototypes-A1-cont-base:Noetherian}) garnish this overall picture by producing a moduli of arbitrary positive dimension of pairwise non-isomorphic, stably isomorphic, $\AA^1$-contractible smooth affine schemes over a generalized base scheme in relative dimensions $\ge 3$.

%--------------------------------------------------------------
\section{Exotic Motivic Spheres}\label{sect:exotic-motivic-spheres}
{\small
In this section, we introduce the \emph{motivic spheres}, one of the most important objects of study in the motivic homotopy theory, and study the existence of exotic motivic spheres, a study that can be seen in contrast with that of exotic affine varieties from \cref{chp4}. To establish this, we first prove that the affine variety $\XX_m\bs \{\bullet\}$ are $\AA^1$-chain connected (\cref{X-p-isA1-connected}) and the quasi-affine varieties $\XX_m \bs \{\bullet\}$ provides a model for an exotic motivic spheres in all higher dimensions $\ge 4$ (\cref{exotic-motspheres-countereg}).}
\medskip

Recall from \cref{pointed-spaces}, the category of pointed $\AA^1$-homotopy category $\mcal{H}(S)_{\bullet}$ is symmetric monoidal with respect to the smash product. This can be exploited to obtain many more objects in $\mathcal{H}(S)_{\bullet}$.
\begin{prop}
For $X\in \Spc_{S\bullet}$, we have a canonical $\AA^1$-weak equivalence
        $$\Sigma X \simeq S^1 \wedge X$$
\end{prop}
\begin{proof}
Observe that the category of presheaves $\Psh(\Sm_S)$ is cocomplete and $\L_{mot}$ preserves coproducts. So, it suffices to compute the pushout at the level of simplicial presheaves, where the coproduct is computed levelwise. In the category of simplicial presheaves, it is clear that $\Sigma X(U) = S^1\wedge X(U)$, for any $U\in \Sm_S$. Hence, as pointed presheaves, we have that $\Sigma X = S^1 \wedge X$. Now to see this in fact holds in $\mcal{H}(S)_{\bullet}$, note that $\L_{mot}$ preserves the smash product operation as it preserves finite products, as well as coproducts and cofibers (\cref{Prop:L_mot-preserve}).
\end{proof}

\subsection{Spheres in Motivic Homotopy Theory}\label{sec:motivic-spheres}
An important application of the smash product structure is to obtain sphere objects in $\mcal{H}(S)_{\bullet}$. In motivic homotopy theory, we have two types of spheres:
\begin{enumerate}
\item The \emph{simplicial circle} $S^1$  comes from topology, and as a simplex it can be seen as  $S^1 = \Delta^1/\partial \Delta$ pointed by the image $\partial \Delta \hookrightarrow \Delta$. In analogy with topology, if $S^0:= \{0\} \sqcup \{1\}$, then $S^1$ is the pushout of the following diagram:
\[\begin{tikzcd}
	{S^0} & {\AA^1} \\
	{*} & {S^1}
	\arrow[from=1-1, to=1-2]
	\arrow[from=1-1, to=2-1]
	\arrow[from=1-2, to=2-2]
	\arrow[from=2-1, to=2-2]
\end{tikzcd}\]
In stark contrast with algebraic topology, $S^1$ exists as a nodal scheme. In particular, $S^1$ is not smooth as an algebraic sphere! The algebraic model is given by $S^1:= \AA^1/\{0\sim 1\}$.

\item The \emph{Tate circle} $\GG_{m,k} := \GG_m = \AA^1_k \backslash \{0\}$ pointed by 1 is an algebro-geometric sphere. It exists as a smooth scheme. Over a base scheme $S$, we define $\GG_{m,S} := \GG_m \times_{\Spec \ZZ} S$. 
\end{enumerate}

One forms a bi-graded (mixed) motivic sphere by smashing copies of both of these spheres (cf. \cref{pointed-spaces}). There are several varying notations for these motivic spheres. We will adopt the following convention that aligns with the grading of motivic cohomology (which is also widely in use):
    $$\SS^{p,q}:= (S^1)^{\wedge (p-q)} \wedge \GG_m^{\wedge q}$$ 
for all $p\ge q \geq 0$. For notational convenience, we will adapt to write $S^{p-q}$ for $(S^1)^{\wedge (p-q)}$ and $\GG_m^q$ for $\GG_m^{\wedge q}$.

\begin{example}
The following are some typical examples of motivic spheres. 
\begin{enumerate}
\item By definition, we have that $\SS^{1,0}= S^1$ and $\SS^{1,1} = \GG_m$. More generally, for every $n\geq 1$, we have $\SS^{2n-1,n} \simeq \AA^n\bs\{0\}$ (\cite[Proposition 4.40]{antieau2017primer}),
\item For every $n\geq 1$, we have a pushout diagram (\cite[Corollary 4.41]{antieau2017primer})
 \[\begin{tikzcd}
	{\Absn} & {\AA^n} \\
	{*} & {\AA^n / \Absn = \SS^{2n,n}}
	\arrow[from=1-1, to=1-2]
	\arrow[from=1-1, to=2-1]
	\arrow[from=1-2, to=2-2]
	\arrow[from=2-1, to=2-2]
\end{tikzcd}\]
which gives us the identifications
$$\SS^{2n,n} :=\Sigma_{S^1} \SS^{2n-1,n} \simeq \AA^n/({\AA^n\bs \{0\}}) \simeq \PP^n/(\PP^n\bs \{0\}) \simeq \PP^n/\PP^{n-1}.$$
In particular, for $n=1$, we get 
\[\begin{tikzcd}
	{\GG_m} & {\AA^1} \\
	{*} & {\PP^1}
	\arrow[from=1-1, to=1-2]
	\arrow[from=1-1, to=2-1]
	\arrow[from=1-2, to=2-2]
	\arrow[from=2-1, to=2-2]
\end{tikzcd}\]
and so $\SS^{2,1} \simeq \Sigma_s \GG_m \simeq  \PP^1$,
\item The smooth affine quadrics $Q_{2n}$'s discussed in \cref{sect:ADF-family} are motivic spheres. We have that 
\begin{align*}
Q_{2n-1}\simeq \SS^{2n-1,n}\quad \text{and}\quad  Q_{2n}\simeq \SS^{2n,n} \simeq (\PP^1)^{\wedge n}. 
\end{align*}
In combination with the (2) above, we obtain:
$$\Sigma_{S^1} \AA^n\bs \{0\} \simeq  (\PP^1)^{\wedge n}$$
for all $n \ge 1$.
\item Whenever $p > 2q$, it is known that the motivic spheres $\SS^{p,q}$ cannot be represented by a smooth scheme \cite[Proposition 2.3.1]{ADF2017smooth}.
\item The special linear group $SL_2$ is a motivic sphere $\SS^{3,2}$ of dimension 3. Indeed, the projection $\pi: SL_2 \to \AA^2\bs \{0\}$ is a Zariski locally trivial $\AA^1$-bundle, whence an $\AA^1$-weak equivalence,
\end{enumerate}
\end{example}

Following the study of exotic affine varieties in \cref{sec:KR-prototypes}, we are now interested in the compact\footnote{though the terminology might be misleading as a motivic sphere need not always be proper: e.g., $\GG_m$} analog.

\begin{defn}\label{exoticspheres:defn}
A (smooth) scheme $X \to S$ of finite type and of dimension $n$ will be called an \emph{exotic motivic sphere} if it is $\AA^1$-homotopic to the smooth scheme $\AA^n_S\bs\{0\}$ without being isomorphic as $S$-schemes, i.e.,
    $$X \simeq \AA^n_S \bs \{0\}\quad  \text{and}\quad  X\ncong \AA^n_S \bs\{0\}.$$ 
\end{defn}

\begin{remark}
In differential geometry, one talks about two kinds of spheres: the topological sphere $\SS^m$, which inherits the canonical subspace topology from $\RR^{m+1}$, and the smooth sphere $\SS^m$ with the standard differential structure. The distinction became evident as Milnor \cite{milnor1956spheres} discovered that one can equip $\SS^7$ with more than one smooth structure (in fact, 28 of them!), leading to the discovery of \emph{smooth exotic spheres} in differential geometry. He further proved that such a phenomenon consistently occurred in higher dimensions as well. Though our study of \emph{exotic motivic spheres} is purely algebraic, it would be worthwhile to explore whether there are plausible, intricate connections when one equips a smooth structure on these varieties.
\end{remark}

A prompt question to ponder is the existence of such exotic schemes over fields and subsequently over a base scheme. In what follows, we shall first give an overview of this problem and establish a counterexample for the following question:
\begin{question}\label{existence:exotic-spheres}
Let $S$ be a reasonably arbitrary base scheme. Then for a smooth scheme $X\to S$ of dimension $n$, does $X \simeq_{\AA^1} \AA^n_S \backslash \{0\}$ imply the isomorphism $X\cong  \AA^n_S \backslash \{0\}$ of $S$-schemes?
\end{question}

\subsection{Low Dimensional Motivic Spheres} \label{sect:overview-mot-sph}
In dimension 0, $\SS^{1,0}:= S^0 \cong \{0\}\cup \{1\}$ is the unique (singular) motivic sphere. In dimension 1, $\SS^{0,1} \cong \GG_m$ is the unique smooth motivic sphere up to isomorphism. This follows as a consequence of $\AA^1$-rigidity of $\GG_m$ as described below.
\begin{prop}\label{exotic-reducedbase-Gm:prop}\cite{Mad}
Over a reduced scheme $S$, if a smooth scheme $X\to S$ is $\AA^1$-homotopic to $\GG_{m,S}$, then it is isomorphic to $\GG_{m,S}$.
\end{prop}
\begin{proof}
We shall prove this over fields \footnote{The author thanks Biman Roy for his generous explanation of the proof over fields.} with a note that the same scheme of proof essentially would work over any reduced scheme due to \cref{Gm-A1-rigid}. Let $X$ be a smooth curve that is $\AA^1$-weakly equivalent to $\GG_m$. Then due to \cref{Gm-A1-rigid}, we have that $\pi_0^{\AA^1}(X)\simeq \pi_0^{\AA^1}(\GG_m) \simeq \GG_m$. Now we want to show that $X$ is also $\AA^1$-rigid. Suppose the contrary, then recall that we have a canonical surjection of Nisnevich sheaves $X\to \pi_0^{\AA^1}(X)$ (\cref{unstable-0-connect:thm}) which gives us a dominant morphism $\varphi: X\to \GG_m$. Since by supposition $X$ is not $\AA^1$-rigid, then there is a finite separable extension $L/k$ and a $\AA^1$-homotopy $H: \AA^1_L \to X$ such that the induced maps on the 0-section and 1-section of $\Spec L\to X$ do not agree $H(0) \ne H(1)$ (see \cref{A1-chain-conne-sep-ext:lemma}). Hence, $H$ is a dominant map which makes the composition $\varphi\circ H: \AA^1_L\to X\to  \GG_m$ constant (since $\GG_m$ is $\AA^1$-rigid). In particular, $\varphi$ is constant, which contradicts the fact that $\varphi$ is dominant. So, $X \cong \GG_m$.
\end{proof}

In dimension 2 over fields of characteristic zero, it has been recently proven (\cite[Theorem 3.1]{Choudhury2024A1type}) that if an open subscheme $U$ of an affine surface $X$ is $\AA^1$-weakly equivalent to $\AA^2 \bs \{0\}$, then $U \cong \AA^2\bs\{0\}$ as $k$-schemes. In view of the extension of the characterization of $\AA^2$ over perfect fields (\cref{A2unique:perfect}), it is likely that this result can be lifted to perfect fields as well. In dimension 3, the family of Koras-Russell threefolds of the first kind $\KK$ provides the necessary obstruction to this question (cf. \cite[\S 3]{Choudhury2024A1type}). However, the question remains widely open in dimensions $\ge 4$ over fields. We close this question in \cite{Mad} in all higher dimensions by providing an obstruction using the results from \cref{sec:KR-prototypes}, which we shall now discuss.

\subsection{Existence of Exotic Spheres in Higher Dimensions}\label{sect:exist-mot-sph-higherdim}
Any member of the family of \emph{generalized Koras-Russell varieties} provides us a way to create a counter-example to the \cref{existence:exotic-spheres} in every dimension $\ge 3$. However, accounting for the fact that all of these varieties $\XX_{m}(\ul{n},\psi)$ are stably isomorphic, it is enough to produce a counter-example using $\XX_m:= \XX_m(n,1)$. The proof that we present below is inspired by the case of threefolds (\cite[\S 3]{Choudhury2024A1type}).
\medskip

For simplicity, let us consider the family of smooth affine varieties defined as in \cref{deformed-example}:
\begin{equation}\label{counter-eg-GKR3F}
\XX_m := \{\ul{x}^n z = y^r+t^s + x_0\} \subset \AA^{m+4}_k \cong  \Spec(k[x_0,\dots,x_m,y,z,t]).
\end{equation}
The canonical morphism $\XX_m\to \Spec k$ makes it into a smooth affine scheme of dimension $m+3$ for all $m\ge 0$ and $n\ge 2$ and $\ul{x}:= \prod_{i=0}^{m} x_i$. This variety is obtained by taking $\psi(\ul{x})=1$ and a single index $\ul{n}=n$ in the original equation \cref{fulleqn:GK-R3F}. Observe that $p =(1,\dots,1,0,1,0)\in \XX_m$ is a rational point. Define the map 
\begin{align*}
        & \hspace{21mm} \phi: \XX_m \to \AA^{m+3}_k \\
        &(x_0,\dots,x_m,y,z,t)\mapsto (x_0,\dots,x_m,y,t)
\end{align*} 
(forgetting '$z$'). Then denote the image $\phi(p)= (1,\dots,1,0,0)=: q \in \AA^{m+3}_k$.

\begin{lemma}\label{tgtspaces-isomorphic}\cite{Mad}
Let $k$ be a field. For all $m\ge 0$, the induced map $d\phi_p: T_p \XX_m \to T_q\AA^{m+3}_k$ is an isomorphism of tangent spaces. As a consequence, we have that the normal bundle of $\XX_m$ at the point $p$ is trivial.
\end{lemma}
\begin{proof}
Set the equation $f(x_0,\dots,x_m,y,z,t):= \{\ul{x}^n z = y^r +t^s+x_0 \}$. We compute the partial derivatives of $f$
\begin{align*}
\frac{\partial f}{\partial x_0}= nx_0^{n-1}x_1^n\dots x_m^n z-1 
\end{align*}
\begin{align*}
\frac{\partial f}{\partial x_j}= nx_0^n \dots x_{j-1}^n x_j^{n-1} x_{j+1}^n \dots x_m^n z, \quad \text{for all}\ 0<j\le m,
\end{align*}

\begin{align*} 
\frac{\partial f}{\partial y}= ry^{r-1}  &&
\frac{\partial f}{\partial z}= \ul{x}^n  &&
\frac{\partial f}{\partial t}=  s t^{s-1} 
\end{align*}
The total derivative $\nabla(f)$ of $f$ can be expressed as
$$\nabla f = \bigg ( (nx_0^{n-1}x_1^n\dots x_m^n z -1), \dots, (nx_0^n x_1^n\dots x_j^{n-1}, \dots, x_m^n z), (ry^{r-1}), (\ul{x}^n), (st^{s-1})  \bigg)$$
which at the point $p$ is given by $\nabla f|_{p} = (n-1, n,\dots,n,0,1,0)$. Recall that the tangent space of $\XX_m$ at $p$ is the kernel of the total derivative at $p$. And hence, depending on the parity of $m$, the tangent space at $p$ can be given as follows: If $m$ is even, we have 
$$T_p\XX_m = \{(a,b,-b,\dots,b,-b,c, -a(n-1),d) \mid a,b,c,d\in k \}\subset \AA^{m+4}_k$$
If $m$ is odd, we have
$$T_p \XX_m = \{(a,b,-b,\dots,0,c, -a(n-1),d) \mid a,b,c,d\in k \}\subset \AA^{m+4}_k$$
In the case when $m$ is odd (respectively even), the induced map $d\phi_p$ is given by the 
$$(a,b,-b,\dots,b,-b,c, -a(n-1),d)\mapsto (a,b,-b,\dots,b,-b,c, d)$$ and (respectively)
$$(a,b,-b,\dots,0,c, -a(n-1),d)\mapsto (a,b,-b,\dots,0,c, d)$$  which are clearly isomorphic. This, in particular, implies that the normal bundle of $\XX_m$ at $p$ is isomorphic to that of $N_{q}\AA^{m+3}_k$ and hence trivial.
\end{proof}

We next show that for the chosen rational point $p\in \XX_m$, the open subscheme $\XX_m \bs \{p\}$ is $\AA^1$-chain connected. In fact to prove our main result (\cref{exotic-motspheres-countereg}), we will only need that $\XX_m\bs \{p\}$ is $\AA^1$-connected. In general, such a question that for a smooth $\AA^1$-connected scheme $X$ if a certain open subscheme $U\subset X$ such that $\text{codim} X\bs U\ge 2$ is again $\AA^1$-connected, is open in general (\cite[Conjecture 2.18]{asok2009A1-excision}). However, this has been proven to be true for coordinate subspaces, i.e., when $X=\AA^n_k$, $k$ is any infinite field (\cite[Lemma 2.15]{asok2009A1-excision}).

\begin{lemma}\label{X-p-isA1-connected}\cite{Mad}
Let $k$ be any infinite perfect field. Then the strictly quasi-affine variety $\XX_m \bs \{p\}$ is $\AA^1$-chain connected.
\end{lemma}

\begin{proof}
The strategy of the proof takes into account the fact that the classical Koras-Russell threefolds $\KK$ are $\AA^1$-chain connected (\cite[Example 2.28]{DPO2019}). Let $L/k$ be any finitely generated field extension. Recall from \cref{defn:A1-chainconn} that we need to show that all the fibers and any points between fibers can be connected by elementary $\AA^1$-homotopies. Consider the projection 
\begin{align*}
    & \hspace{10mm} \pr:\XX_m\to \AA^1_k \quad  \text{given by} \\
    & (x_0,\dots,x_m,y,z,t) \mapsto x_0.
\end{align*}
First, let us observe that the fibers over points of $\GG_m$ are $\AA^1$-chain connected. Indeed, choose a point $\alpha\in \GG_m$. If $\alpha\ne 1$, then $\pr^{-1}(\alpha)\simeq \AA_k^{m+2}$ which is clearly $\AA^1$-chain connected and if $\alpha= 1$, we have $\pr^{-1}(\alpha) \simeq \AA^{m+2}_k\bs\{0\}$ which is again $\AA^1$-chain connected. In all, for every point $\alpha\in \GG_m$, any two points $L$-points in $\pr^{-1}(\alpha)$ can be connected via chains of $\AA^1_L$'s. The fiber over the point $\{x_0=0\}$ is given by $\pr^{-1}(0)\simeq  \AA_k^{m+1}\times_k \Gamma_{r,s}$, where $\Gamma_{r,s}:=\{y^r+t^s=0\}$ is the cuspidal curve. The cuspidal curve is $\AA^1$-chain connected as witnessed by \cref{A1-contr:egs} (2). Hence, the fiber $\pr^{-1}(0)$ is $\AA^1$-chain connected given by the na\"ive $\AA^1$-homotopy 
\begin{align*}
    & \AA^1_L \to \AA^{m+1}\times \Gamma_{r,s}\\
    & \gamma \mapsto (a_0 \gamma,\dots,a_m\gamma, b\gamma^s, c\gamma^r)
\end{align*} 
joining $(0,\dots,0)$ with the point $(a_0,\dots,a_m,b,c)$. Now, to conclude, we only have to show that points between different fibers can be joined via chains of $\AA^1_L$'s. For this, we can cook up an $\AA^1$-homotopy as follows: let $f(w)$ and $g(w)$ be two polynomials in $k[w]$ such that $w^n$ divides $f(w)^r+g(w)^s+w$. The existence of such polynomials can be demonstrated by lifting the polynomials $\text{mod}\ w^n$ by inducting on $n\ge 0$. The case $n=0$ is trivial and for the case $n=1$, choose $f(\alpha w) = 1$ and $g(\alpha w) = -1$. For the case $n=2$, choose $f(\alpha w) = 1+w$ and $g(\alpha w) = w-1$. By proceeding in the same fashion, we can obtain a general strategy by setting:
\begin{align*}
    &\hspace{10mm} f(\alpha w):= 1+a_0 w+ a_1 w^2 +\dots + a_{n-2} w^{n-1} \\
    & \text{and} \\
    &\hspace{10mm} g(\alpha w):= 1-w -\dots - w^{n-1}
\end{align*}
for some coefficients $a_i\in k$. One then has to choose $\{a_i\}$s such that the coefficients of $w_i$ vanish for all $i\le n-1$. Now, define a map $\theta: \AA^1_L \to \XX_m\bs \{p\}$ by
$$w \mapsto \left( \alpha w, 1,\dots,1, \frac{f(\alpha w)^r+ g(\alpha w)^s+ \alpha w-1}{(\alpha w)^n} ,f(\alpha w), g(\alpha w) \right).$$
This connects points from the fibre $\pr^{-1}(0)$ with points of the fiber $\pr^{-1}(\alpha)$, for $\alpha \in \GG_m$. Crucially, note that our chosen point $p=(1,\dots,1,0,1,0)$ does not lie in the image of $\theta$. Suppose on contrary, if $p\in \text{Im}(\theta)$, then we have that $\theta(a)= p= (1,\dots,1,0,1,0)$, for some $a\in \AA^1_L$. This gives us $\alpha w=1$, $f(\alpha w)=1$, and $g(\alpha w)=0$. Upon directly substituting these values and comparing with $p=(1,\dots,1,0,1,0)$, we see that
    $$\frac{f(\alpha w)^r+ g(\alpha w)^s+ \alpha w}{(\alpha w)^n}= 1 \ne 0.$$
This show that the quasi-affine variety $\XX_m\bs \{p\}$ is $\AA^1$-chain connected.
\end{proof}

We now have all the tools to prove our key result that answers \cref{existence:exotic-spheres}.

\begin{theorem}\label{exotic-motspheres-countereg} \cite{Mad}
Let $k$ be an infinite perfect field. Then for every $m\ge 0$, the strictly quasi-affine variety $\XX_m\bs \{p\}$ of dimension $m+3$ is $\AA^1$-homotopic to $\AA^{m+3}_k \bs\{0\}$ but is not isomorphic to $\AA^{m+3}_k\bs \{q\}$ as a $k$-scheme.
\end{theorem}

\begin{proof}
Let $\XX_m$ be the smooth affine variety as described in \cref{counter-eg-GKR3F} with the rational point $p = (1,\dots,1,0,1,0)\in \XX_m$ and with its image $q = (1,\dots,1,0,0)\in \AA^{m+3}_k$. We first show that $\XX_m \bs \{p\}$ is $\AA^1$-homotopic to $\AA^{m+3}\bs\{q\}$. For this, let us set up the following commutative diagram whose rows are cofiber sequences:
\begin{equation}\label{setup:diagram}
\begin{tikzcd}[column sep=10pt, row sep=20pt]
&&&&& {\XX_m\bs \{p\}} & {} & {\XX_m} && {\XX_m/(\XX_m\bs \{p\})} \\
\\
{} &&&&& {\AA^{m+3}_k \bs \{q\}} && {\AA^{m+3}_k} && {\AA^{m+3}_k/ (\AA^{m+3}_k\bs \{q\})}
\arrow[from=1-6, to=1-8]
\arrow["\phi", from=1-6, to=3-6]
\arrow[from=1-8, to=1-10]
\arrow["\psi", from=1-8, to=3-8]
\arrow["\rho", from=1-10, to=3-10]
\arrow[from=3-6, to=3-8]	
\arrow[from=3-8, to=3-10]
\end{tikzcd}
\end{equation}
Due to \cref{KR3Fprototypes:field-A1-cont:perfect}, we have that $\psi: \XX_m \to \AA^{m+3}_k$ is an $\AA^1$-weak equivalence in $\Spc_k$. Next, we observe that $\rho$ is also an $\AA^1$-weak equivalence. Indeed, since the normal bundle associated to the closed immersion $\{q\}\hookrightarrow \AA^{m+3}_k$ is trivial, by the $\AA^1$-homotopy purity (\cref{purity-theorem}), we have that                
    $$\AA^{m+3}_k/ (\AA^{m+3}_k\bs \{q\})\simeq \Spec k\ \wedge (\PP^1)^{\wedge ({m+3})} \simeq (\PP^1)^{\wedge ({m+3})}.$$ 
Since $\{p\} \in \XX_m$ and $\XX_m$ are smooth subvarieties,  due to the aforementioned purity, we have 
    $$\XX_m/(\XX_m\bs \{p\}) \simeq \Th(N_{p}\XX_m). $$ 
By \cref{tgtspaces-isomorphic}, we have that the normal bundle $N_p \XX_m$ is isomorphic to a trivial bundle and hence that 
    $$\Th(N_{p} \XX_m)\simeq \{p\}\wedge (\PP^1)^{\wedge {(m+3)}} \simeq (\PP^1)^{\wedge {(m+3)}}.$$
Hence, we have that 
$$ (\PP^1)^{\wedge {(m+3)}} \simeq \XX_m/(\XX_m \bs \{p\})\xrightarrow{\rho} \AA_k^{m+3}/(\AA_k^{m+3}\bs \{q\}) \simeq (\PP^1)^{\wedge ({m+3})}$$
which due to the fact that the map $\rho$ is induced from the map $d\phi_p$ becomes an $\AA^1$-weak equivalence in the light of \cref{purity:cor-VoeZ/2}. The $\AA^1$-weak equivalence of $\rho$ implies, by definition, that the simplicial suspension 
$$\Sigma_{S^1} \phi: \Sigma_{S^1} \XX_m \bs\{p\} \to \Sigma_{S^1} \AA^{m+3}_k\bs\{q\}$$ 
is an $\AA^1$-weak equivalence of motivic spaces. Now observe that $\XX_m \bs \{p\}$ is $\AA^1$-chain connected from \cref{X-p-isA1-connected} and whence by \cite[Proposition 2.2.7]{asokmorel2011} is thereby $\AA^1$-connected as well. Moreover, it is of codimension $d=m+3$ in $\XX_m$, for $m\ge 0$. Thus, due to \cite[Theorem 4.1]{asok2009A1-excision}, this implies that the open subschemes $\XX_m \bs \{p\} \hookrightarrow \XX_m$ have the same $\AA^1$-homotopy sheaves in a range:
    $$\pi_i^{\AA^1}(\XX_m\bs\{p\}) \xrightarrow{\sim} \pi_i^{\AA^1}(\XX_m) \quad \text{for all}\quad  0\le i \le d-2 $$
and in particular, $\pi^{\AA^1}_1(\XX_m\bs \{p\})$ is trivial (since $\XX_m$ is $\AA^1$-contractible - see \cref{eg:S-point}), which by the $\AA^1$-Hurewicz theorem \cite[Theorem 6.35, Theorem 6.37]{morel2012A1topology} imply that $\phi$ is an $\AA^1$-homology equivalence. To conclude, observe that due to \cite[Theorem 1.1]{shimizu2022relative} any $\AA^1$-homology equivalence is an $\AA^1$-weak equivalence and that $\XX_m \bs \{p\} \simeq \AA_k^{m+3}\bs \{q\}$.
\medskip

Now, we show that these quasi-affine varieties cannot be isomorphic as schemes over $k$. Suppose, on the contrary, that there is an isomorphism 
    $$\sigma: \XX_m \bs\{p\}\to \AA_k^{m+3}\bs\{q\}$$ 
with an inverse 
    $$\tau:\AA_k^{m+3}\bs\{q\}\to \XX_m\bs\{p\}. $$   
Now since the point $\{p\}$ (respectively, $\{q\}$) is of codimesion at least 2 in $\XX_m$ (respectively, in $\AA^{m+3}_k$), the morphism $\sigma$ and $\tau$ can be continuously extended to a smooth morphism ${\sigma'}:\XX_m \to \AA_k^{m+3}$ and $\tau':\AA^{m+3}_k\to \XX_m$. Both the maps $\sigma'$ and $\tau'$ take the same values with the identity maps away from the complement of a rational point, which implies that both $\sigma'$ and $\tau'$ are isomorphisms. But due to the fact that $\ML(\XX_m)\ncong \ML(\AA^{m+3}_k)$ (cf. \cite[Proposition 3.4]{ghosh2023triviality}), this is absurd. Thus, the quasi-affine varieties $\XX_m\bs\{p\}$ and $\AA^{m+3}_k\bs\{q\}$ cannot be isomorphic as $k$-schemes. 
\end{proof}

\begin{remark}[Update!]
The assumption on the base field to be infinite in \cref{exotic-motspheres-countereg} is due to the limitation as expressed in \cite[Remark 2.16]{asok2009A1-excision}. In essence, the limitation is as follows: for a smooth scheme $X$ with an open immersion $U\hookrightarrow X$ and its closed complement $Z:= X\bs U$, if $X(k)\subseteq Z$, then clearly $U$ has no rational point, whence it cannot be $\AA^1$-connected. However, in the light of $\AA^1$-simply connectedness of $\XX_m\bs \{p\}$, the infinite assumption in \cref{exotic-motspheres-countereg} can be lifted due to \cite[Corollary 3.6]{shimizu2022relative}.
\end{remark}

\subsection{Base Change of Motivic Spheres}
In this short section, we will comment on the base change properties of motivic spheres. In anticipation of the existence of exotic motivic spheres over fields (\cref{exotic-motspheres-countereg}), it is natural to systematically ask for the verity over a base scheme $S$. The question of interest can be framed as follows:

\begin{question}\label{qstn:exot-mot-base-scheme}
Let $S$ be any "reasonable" base scheme. Then for every $m\ge 0$, is $\XX_m\bs \{p\}$ $\AA^1$-homotopic to $\AA^{m+3}_S\bs\{0\}$ and if so, are they isomorphic as $S$-schemes?
\end{question}

We anticipate that they should not be isomorphic as $S$-schemes, intuitively, as a result of the failure of relative Makar-Limanov invariants as in the case of fields. To show that they are $\AA^1$-homotopic, we propose to study the following approaches:

\subsubsection{Gluing motivic spheres}
One may choose to work fiberwise as in the instance of $\AA^1$-contractibility. For this, let us fix the base scheme $S$ to be a Noetherian scheme with infinite perfect residue fields. Let us now decipher the obstruction to taking up this approach. We know that $\XX_m$ is $\AA^1$-contractible over $S$ due to \cref{KR3F:prototypes-A1-cont-base:Noetherian}, hence we have that $f:\XX_m \to S$ is an $\AA^1$-weak equivalence of motivic $S$-spaces. For every point $s\in S$, let us consider the base change $f_s: (\XX_m)_s \to \Spec (\kappa_s)$, where $\kappa_s$ is the corresponding residue field at $s\in S$, which by hypothesis is (infinite) perfect. Thus, by \cref{exotic-motspheres-countereg}, we have that for every $s\in S$, the quasi-affine schemes $(\XX_m \bs \{p\})_s$ are $\AA^1$-homotopic to $\AA^{m+3}_{\kappa_s}\bs \{q\}$ in $\Spc_{\kappa_s}$. 
\[\begin{tikzcd}
	{\AA_{\kappa_s}^{m+3}\bs \{0\} \simeq (\XX_m\bs \{p\})_s} && {\XX_m\bs \{p\}} \\
	{\Spec \kappa_s} && S
	\arrow[from=1-1, to=1-3]
	\arrow["{f_s}"', from=1-1, to=2-1]
	\arrow["f", from=1-3, to=2-3]
	\arrow[from=2-1, to=2-3]
\end{tikzcd}\]
But since, in general, the $\AA^1$-homotopy classes of motivic spheres are not trivial (see \cref{A1-Brouwer-degree}). And so, knowing that all the fibers $(\XX_m\bs \{p\})_s$ are $\AA^1$-homotopic to $\AA^{m+3}_{\kappa_s}\bs \{q\}$ does not allow us conclude that the original space $\XX_m\bs \{p\}$ is $\AA^1$-homotopic to $\AA^{m+3}_S\bs \{q\}$. This can be explained by the lack of an analogous \emph{gluing lemma} for motivic spheres as we had for contractible objects (\cref{pointwise:phenomenon}). Anticipating such a gluing lemma for motivic spheres, it is straightforward to verify that there are exotic motivic spheres in relative dimensions $\ge 3$ over a such base scheme $S$.

\subsubsection{The relative approach}
Another approach to show that the varieties in \cref{qstn:exot-mot-base-scheme} are $\AA^1$-homotopic is to directly work over a base scheme $S$ and establish similar results following the footpaths of \cref{exotic-motspheres-countereg}. To begin with, the purity isomorphism (\cref{purity-theorem}) is available over any Noetherian scheme $S$. We believe that by suitably choosing base scheme $S$, one should retrieve back that $\XX_m\bs \{p\}$ is $\AA^1$-chain connected over $S$ by cooking up an explicit $\AA^1$-chain homotopy with the coefficients $\{a_i\}\in S$, as we did in \cref{X-p-isA1-connected}. It only remains to check to what extent we can extend $S$ to retrieve the following facts:
\begin{itemize}
\item "relative" $\AA^1$-excision style result over $S$ (cf. \cite[Theorem 4.1]{asok2009A1-excision}),
\item $\AA^1$-homology is the same as that of $\AA^1$-homotopy in a range (cf. \cite[Theorem 1.1]{shimizu2022relative}). 
\end{itemize}
 
\subsubsection{Uniqueness of motivic spheres in relative dimension 2}
In anticipation of the analogous gluing lemma, one could attempt to establish that there are no exotic motivic spheres in relative dimensions 2. 

\begin{conj}
Let $S$ be any Noetherian scheme with characteristic zero residue fields. Let $f: X\to S$ be a smooth scheme of relative dimension 2. Then if $X$ is $\AA^1$-homotopic to  $\AA^2_S\bs \{0\}$, then $X$ is isomorphic to $\AA^2_S\bs \{0\}$ as a $S$-scheme.
\end{conj}

The proposed strategy is as follows: first, due \cite[Theorem 3.1]{Choudhury2024A1type}, we have that over a field of characteristic zero, a smooth scheme $X\simeq \AA^2_k\bs\{0\}$ implies that $X\cong \AA^2_k\bs\{0\}$. Now, if $S$ is any Noetherian scheme with characteristic zero residue fields, then by the "anticipated" gluing lemma, we have that $X \simeq \AA^2_k\bs\{0\}$ implies that $X\cong \AA^2_k\bs\{0\}$ as $S$-schemes due to \cref{A2isunique}.

\pagestyle{plain} % comincia a mostrare il numero di pagina
\clearpage
\thispagestyle{empty}
\vspace*{\fill}
\begin{center}
Yay, you've reached the beginning of the end!
 \end{center}

% \afterpage{\blankpage
% \thispagestyle{empty}} 

\begin{savequote}
We are like children building a sand castle... The trick is to enjoy it fully but without clinging, and when the time comes, let it dissolve back into the sea.
\qauthor{"When Things Fall Apart" by Pema Chödrön}
\end{savequote}
\begin{appendix}
\chapter{All About Categories!}\label{app:category}
\markboth{Appendix}{}

\section{Grothendieck Topologies}\label{App;Groth-Top}
In this section, we will develop the notion of a Grothendieck topology. We will also illustrate the usefulness of the Nisnevich topology in comparison with the Zariski and the \'etale topologies. We shall be brief in our definitions and formulations, and we redirect readers to \cite{maclane2012sheaves} or \cite{vistoli2004notes} for more pedantic details.

\subsection{Sheaves on topological spaces}
For the readers' convenience, we shall recollect some basic notions of presheaves and sheaves from the topological setting. Let $X$ be a topological space and denote by $\mathcal{O}_X$ its collection of open subsets.
\begin{defn}
A \emph{presheaf} $F$ on $X$ is a contravariant functor 
       $$F:\mathcal{O}_X^{op}\to \Set$$
which assigns for every open subset $U\subset X$, a set $F(U)$. A \emph{sheaf} $\mathcal{F}$ on $X$ is a presheaf satisfying the following axiom: for every open covering $U= \bigcup_i U_i$ of an open subset $U$ of $X$, we have an equalizer diagram of the form 
\[\begin{tikzcd}
       {\mathcal{F}(U)} & {\prod_{i}\mathcal{F}(U_i) } & {\prod_{i,j} \mathcal{F}(U_i\times_U U_j)}
	\arrow["e", dashed, from=1-1, to=1-2]
	\arrow["p"', from=1-2, to=1-3]
	\arrow["q", shift left=3, from=1-2, to=1-3]
\end{tikzcd}\]
where for $t\in \mathcal{F}(U)$, $e(t) = \{ t{\vert_{U_i}} \mid i\in I \}$ and for a family $t_i\in \mathcal{F}(U_i)$, and set $U_{ij}:= U_i\times_U U_j$
   $$p\{t_i\} = \{ t_{i}|_{U_{ij}} \}, \quad q\{t_j\} = \{ t_{j}|_{U_{ij}} \}.$$
A morphism $\mathcal{F}\to \mathcal{G}$ of (pre-)sheaves is a natural transformation of functors.
\end{defn}

The above definition of a sheaf is equivalent to that of a presheaf satisfying the sheaf axiom (\cite[\href{https://stacks.math.columbia.edu/tag/006S}{Section 006S}]{stacks-project}). Of course, one can (and will) also talk about (pre)-sheaves of groups, rings, Abelian groups, modules, and so on by replacing the category of sets by the corresponding category in the above definition. All of these variants help one to talk about a notion of (pre-)sheaf in the concerned category.

\begin{example}
There are numerous examples of sheaves in nature:
\begin{itemize}
\item Let $X$ be a topological space and $k$ be a field (e.g., $\RR$, $\CC$,...). Then for every open $U\subseteq X$, the association  
    $$U \mapsto \{f:U \to k \mid f\ \text{is a continuous/smooth/differentiable $k$-valued function} \}$$
It is, in fact, a sheaf of rings.
\item The constant function on any topological space $X$ is a presheaf, while the locally constant function is a sheaf.
\item Let $X$ be a topological space and choose $x\in X$ and any set $S$. For any open set $U\subseteq X$, the presheaf 
\begin{equation*}
    \mathcal{F}(U) = 
    \begin{cases}
          S         & x \in U   \\
          \emptyset & x \notin U \\
    \end{cases}
\end{equation*}
is, in fact, a sheaf, called the \emph{sky scrapper sheaf}.
\end{itemize}
\end{example}

This notion of sheaves on a topological space can be generalized to any category $\mathscr{C}$. For any $X\in \mathscr{C}$, one considers maps of the form $\{U\to X\}$ that "cover $X$" (in a sense that will be made precise below). Then one can construct a topology with respect to these (open) coverings and define a sheaf as the corresponding contravariant functor that satisfies some analogous locality and gluing properties. This notion was established by Grothendieck and his collaborators in SGA IV for studying the cohomology theories of generalized spaces and is popularly known as the \emph{Grothendieck topology}. In a sense, it provides a way to "categorify" the notion of a sheaf from topological spaces to that of any small category $\mathscr{C}$. The collection of presheaves (resp. sheaves) on $\mathscr{C}$ forms an example of a \emph{functor category}. We denote the category of presheaf (resp. sheaf) on $\mathscr{C}$ by $\Psh(\mathscr{C})$ (resp. $\Shv(\mathscr{C})$).
\medskip

To define the notion of a Grothendieck topology, we shall formulate it via the notion of sieves. Let $\mathscr{C}$ be a (small) category.
\begin{defn}
Given an object $C\in \mathscr{C}$, a \emph{sieve} $S$ is a family of morphisms in $\mathscr{C}$, all with codomain $C$, such that 
        $$f\in S \Rightarrow f\circ g\in S$$
Whenever this composition makes sense. If $S$ is a sieve on $C$ and $h: D \to C$ is any arrow to $C$, then 
        $$h^*(S)= \{ g \mid \text{cod}(g) =D,\quad h\circ g \in S \}$$
is a sieve on $D$.
\end{defn}

There is a convenient alternative definition of a sieve: Let $U$ be an object of a category $\mathscr{C}$. A \emph{sieve} on $U$ is a sub functor of $h_U: \mathscr{C}^{op}\to \Set$, where $h_U:= \Hom(-,U)$ is the representable functor on $U$.

\begin{defn}
A \emph{Grothendieck topology} on a category $\mathscr{C}$ is a function $J$ which assigns to each object $C\in \mathscr{C}$, a collection $J(C)$ of sieves on $C$, in such a way that the following three axioms are satisfied:
\begin{enumerate}
\item (Maximality): The maximal sieve $t_C = \{f \mid \text{codomain}(f) = C\}$ is in $J(C)$,
\item (Stability): If $S\in J(C)$, then $h^*(S) \in J(D)$, for any arrow $h:C\to D$,
\item (Transitivity): If $S \in J(C)$ and $R$ is any sieve on $C$ such that $h^*(R) \in J(D)$, for all $h: D \to C$ in $S$, then $R\in J(C)$.
\end{enumerate}
\end{defn}

\begin{defn}
A \emph{site} is a pair $(\mathscr{C}, J)$ of a small category $\mathscr{C}$ equipped with a Grothendieck topology $J$. If $S \in J(C)$, one says that $S$ is a covering sieve, or that "$S$ covers $C$".
\end{defn}

\begin{example}
We shall now give some examples of sites.
\begin{itemize}
\item Let $X$ be a topological space; denote by $X_{op}$ the category in which the objects are the open subsets of $X$, and the arrows are given by inclusions. Then we get a Grothendieck topology on $X_{op}$ by associating with each open subset $U\subseteq X$, the set of open coverings of $U$. This defines a site of a topological space.
\item Let $\mathscr{C}= \rm{Top}$ is the category of topological spaces. If $U$ is a topological space, then a covering of $U$ will be a jointly surjective\footnote{that is, the set-theoretic union of their images equals $U$} collection of open embeddings\footnote{an open continuous injective map $V\to U$} $\{U_i\to U$\}. This is also called the \emph{classical topology} on $X$.
\item In the example above, if in addition, the covering families $\{U_i\to U $\} are also local homeomorphisms, then it is called the \emph{\'etale site} on $X$.
\end{itemize}
\end{example}

\subsection{Sheaves on Schemes}
We shall now develop the notion of sheaves on a site with a view towards application to algebraic geometry and eventually to the category of motivic spaces $\Spc_S$. Sheaves on a site can be defined in much the same way as sheaves on a topological space. Recall that a \emph{ringed space} is a pair $(X,\mathcal{O}_X)$, where $X$ is a topological space and $\mathcal{O}_X$ is a sheaf of rings on $X$. It is \emph{locally ringed} if, in addition, each stalk is a local ring. A scheme $(X,\mcal{O}_X)$ is a canonical example of a locally ringed space which is locally isomorphic to an affine scheme in the Zariski topology. In other words, for every point $p\in X$, there exists an open neighborhood $U$ such that $(U,\mathcal{O}_X|_{U})$ is isomorphic to an affine scheme $(\Spec R, \mathcal{O}_{\Spec R})$. One equips the category of schemes with a different Grothendieck topology that varies depending on the nature of the open covering defined.

\begin{example}
Let $T$ be a scheme with a family of coverings $\{f_i:T_i \to T \}_{i\in I}$ such that $T= \bigcup_i f_i(T_i)$, then we say that such a covering is called
\begin{itemize}
\item \emph{Zariski} if each $f_i$ is an open immersion,
\item \emph{\'etale} if each $f_i$ is \'etale, 
\item \emph{smooth} if each $f_i$ is smooth,
\item \emph{fppf}\footnote{in French for “fidèlement plat de présentation finie”.} (faithfully flat and finitely presented) if each $f_i$ is flat, locally of finite presentation,
\item \emph{fpqc}\footnote{in French for "fidèlement plat" and quasi-compact.} (faithfully flat and and quasi-compact) if each $f_i$ is flat and for each affine open $U\subseteq T$ there exists a finite set $K$, a map $\textbf{i}:K \to I$ and affine opens $U_{\textbf{i}(k)} \subset T_{\textbf{i}(k)}$ such that $U = \bigcup_{k\in K} f_{\textbf{i}(k)}(U_{\textbf{i}(k)})$.
\end{itemize}
\end{example}

We will now describe another important topology called the \emph{Nisnevich topology} that is well-suited for performing $\AA^1$-homotopy theory.

\subsection{The Nisnevich Site}\label{App:Nisnevich-top}
We begin with an alternative definition of the Nisnevich topology from the one presented in \cref{sect:Nisnevich}. This is due to \cite[Definition 3.7.1.1]{lurie2018SAG} (see also \cite[Lemma 1.5]{MV99}):

\begin{defn}
The \emph{Nisnevich topology} on $\Sm_S$ is the topology generated by those finite families of \'etale morphisms $\{p_i:U_i\to X\}_{i\in I}$ such that there is a finite sequence 
    $$\emptyset \subseteq Z_n \subseteq Z_{n-1}\subseteq \dots \subseteq Z_1 \subseteq Z_0 = X$$ 
of finitely presented closed subschemes of $X$ such that
        $$\bigsqcup_{i\in I} p_i^{-1} (Z_m\bs Z_{m+1})\to Z_m\bs Z_{m+1} $$
admits a section for $0\le m\le n-1$.            
\end{defn}

For a justification that this definition is indeed equivalent to that of being surjective on $k$-points, one can refer to \cite{hoyois2016trivial}. Yet another equivalent definition of the Nisnevich topology is the following: it is a Grothendieck topology that is generated by the elementary distinguished squares (\cite[Remark 2.3.3]{asok2021A1}). In terms of covering sieves, we have that the Nisnevich topology is the coarsest such that the empty sieve covers $\emptyset$, and for EDS  as above, the sieve on $X$ generated by $U \to X$ and $V \to X$ is a covering sieve. For any chosen definition, we have that the Nisnevich topology is finer than the Zariski topology but coarser than the \'etale topology:
\begin{align*}
    \text{\'etale} >> \text{Nisnevich} > \text{Zariski} 
\end{align*}
This dual nature makes the Nisnevich topology more suitable for doing motivic homotopy theory.

\subsubsection{Advantages of Nisnevich topology}\label{App:Nisnevich-useful}
\begin{enumerate} 
\item The Nisnevich cohomological dimension (and even the homotopy dimension) of a scheme with Krull dimension $d$ is equal to $d$ (this is similar to the Zariski topology): $H^i_{Nis}(S,\mathcal{F}) = 0$ for all $i>d$ and any sheaf $\mcal{F}$ of Abelian groups (\cite[Proposition 1.8]{MV99}).
\item For a field $k$, we have that any field has no non-trivial Nisnevich cohomology (similar to the Zariski topology): $H^i_{Nis}(\Spec k,\mathcal{F}) = 0$, for any sheaf $\mathcal{F}$ of Abelian groups.
\item Algebraic $K$-theory has Nisnevich descent (similar to the Zariski topology),
\item Pushforward along finite morphisms is an exact functor on Nisnevich sheaves of abelian groups (this behaves like the \'etale topology).
\item Nisnevich cohomology can be computed using \v{C}ech cohomology (similar to the \'etale topology): For any sheaf of abelian groups $\mathcal{F}$ on $Sm_S$ and any $n\ge 0$, we have an isomorphism of groups $\check{H}_{Nis}^n(S,\mathcal{F})\cong H_{Nis}^n(S,\mathcal{F})$ (\cite[Proposition 1.9]{MV99}), where the left hand side is the \v{C}ech cohomology of a space $X$. This fails for Zariski topology (see \cite[Example 1.10]{MV99}).
\item For closed immersions, Nisnevich sheaves exhibit purity isomorphism (like \'etale topology). As a consequence, we have that any smooth pair $(Z, X)$ is locally equivalent in the Nisnevich topology to a pair of the form $(\AA^m, \AA^n)$: See \cref{purity-theorem}.
\item Since the Nisnevich topology is coarser than the \'etale topology, it inherits the property of being subcanonical from \'etale topology. This has the consequence that any representable presheaf in the Nisnevich topology is automatically a representable sheaf (see e.g., \cite[\S 2.4]{asokmorel2011} or \cite[\href{https://stacks.math.columbia.edu/tag/03O4}{Remark 03O4}]{stacks-project}).
\end{enumerate}

% -----------------------------------------------------------------------------
\section{Model Categories}\label{app:model-cats}
The language of model categories was established by Quillen (\cite{quillen1967homotopical}) as a suitable framework to deal with homotopy theoretic formulations. This framework has nowadays become almost indispensable to study any category up to some homotopical context; the first step one does is to garnish the desired category with a model structure and re-frame the original question in this setup. The far-reaching adaptability of model categories is due to their concise description that allows one to study morphisms in a given category by just studying two specified classes of maps: weak equivalences and either fibrations or cofibrations. This is furthermore enhanced by a certain adjoint pair of (Quillen) functors that are defined at the level of the homotopy categories. The formalism of localization from the commutative algebra allows one to formally invert a chosen multiplicative system. In application to algebraic geometry, localizations help us to formally invert any chosen class of maps, producing a new, larger category where these chosen maps become isomorphisms. In this section, we will give a brief introduction and background to this machinery with a focus on the $\AA^1$-homotopy theory. For most parts, we will follow \cite{lurie2009HTT}.

\begin{defn}\label{defn:model-cats}
A \emph{model structure} on a category $\mathscr{C}$ consists of three distinguished classes of morphisms in $\mathscr{C}$
\begin{itemize}
\item \emph{cofibrations} $\mcal{C}$of
\item \emph{fibrations} $\mcal{F}$ib
\item \emph{weak equivalences} $\mcal{W}$eq
\end{itemize}
satisfying the following axioms: 
\begin{enumerate}
\item The category $\mathscr{C}$ admits (small) limits and colimits.
\item Given a composable pair of maps $X\xrightarrow{f} Y\xrightarrow{g} Z$, if any two of $g\circ f$, $f$, and $g$ are weak equivalences, then so is the third (this is called the \emph{two-out-of-three axiom}),
\item Suppose $f : X\to Y$ is a retract of $g : X\to Y$ : that is, suppose there exists a commutative diagram
\[\begin{tikzcd}
	X & {X'} & X \\
	Y & {Y'} & Y
	\arrow["i", from=1-1, to=1-2]
	\arrow["f"', from=1-1, to=2-1]
	\arrow["r", from=1-2, to=1-3]
	\arrow["g"', from=1-2, to=2-2]
	\arrow["f", from=1-3, to=2-3]
	\arrow["{i'}"', from=2-1, to=2-2]
	\arrow["{r'}"', from=2-2, to=2-3]
\end{tikzcd}\]
where $r\circ i =id_X$ and $r'\circ i' =id_Y$. Then
\begin{itemize}
\item[(i)]   If $g$ is a fibration, so is $f$,
\item[(ii)]  If $g$ is a cofibration, then so is $f$,
\item[(iii)] If $g$ is a weak equivalence, then so is $f$.
\end{itemize}
\item Given a diagram
\[\begin{tikzcd}
	A & X \\
	B & Y
	\arrow[from=1-1, to=1-2]
	\arrow["i"', from=1-1, to=2-1]
	\arrow["p", from=1-2, to=2-2]
	\arrow[dashed, from=2-1, to=1-2]
	\arrow[from=2-1, to=2-2]
\end{tikzcd}\]
A dotted arrow can be found rendering the diagram commutative if either
\begin{itemize}
\item[(i)] The map $i$ is a cofibration, and the map $p$ is both a fibration and a weak equivalence.
\item[(ii)] The map $i$ is both a cofibration and a weak equivalence, and the map $p$ is a fibration.
\end{itemize}
\item  Any map $X\to Z$ in $\mathscr{C}$ admits a functorial factorization as follows:
\begin{align*}
    &X\xrightarrow{f}Y \xrightarrow{g} Z \  &X\xrightarrow{f'}Y'\xrightarrow{g'} Z
\end{align*}    
where $f$ is a cofibration, $g$ is a trivial fibration, $f'$ is a trivial cofibration, and $g'$ is a fibration.
\end{enumerate}
A \emph{model category} is a category $\mathscr{C}$ equipped with a model structure.
\end{defn}

We will encounter plenty of examples of model categories in due course. Here is the simplest of the kind:
\begin{example}
Let $\mathscr{C}$ be any category that admits small limits and colimits. Then $\mathscr{C}$ can be endowed with the trivial model structure: 
\begin{itemize}
\item Weak equivalences $\mcal{W}$eq are isomorphisms of $\mathscr{C}$.
\item every morphism is a fibration $\mcal{F}$ib.
\item every morphism is a cofibration $\mcal{C}$of.
\end{itemize}
\end{example}

\begin{defn}
A morphism in a model category $\mathscr{C}$ is called a \emph{trivial cofibration} if it is both a cofibration and a weak equivalence. Similarly, $f$ is called a \emph{trivial fibration} if it is both a fibration and a weak equivalence.
\end{defn}

By axiom (1), any model category $\mathscr{C}$ has an initial object $\emptyset$ and a final object $*$. This gives us the following description:
\begin{defn}\label{cofib-replacement}
An object $X$ of $M$ is fibrant if $X\to *$ is a fibration, and $X$ is cofibrant if $\emptyset\to X$ is a cofibration. Given an object $X$ of $M$, an acyclic fibration $QX\to X$ such that $QX$ is cofibrant is called a \emph{cofibrant replacement}. Similarly, if $X\to RX$ is an acyclic fibration with $RX$ fibrant, then $RX$ is called a \emph{fibrant replacement} of $X$. 
\end{defn}
These replacements always exist by the axioms of model categories.

\subsection{Homotopy Category of a Model Category}
Broadly speaking, one has two different (yet equivalent) ways to construct a homotopy category associated with a model category. We will now briefly explain both of the viewpoints.

\subsubsection{Cylinder-Path Decomposition}
\begin{defn}
Let $\mathscr{C}$ be a model category containing an object $X$. A \emph{cylinder object} for $X$ is an object $C$ together with a diagram $X\sqcup X \xrightarrow{i} C \xrightarrow{j} X$ where $i$ is a cofibration, $j$ is a weak equivalence, and the composition $j\circ i$ is the fold map $X \sqcup X\to X$. Dually, a \emph{path object} for $Y\in \mathscr{C}$ is an object $P$ together with a diagram $Y\xrightarrow{q}P\xrightarrow{p} Y\times Y$ such that $q$ is a weak equivalence, $p$ is a fibration, and $p\circ q$ is the diagonal map $Y\to Y\times Y$.
\end{defn}

The existence of cylinder and path objects follows from the factorization Axiom (5) of \cref{defn:model-cats}. The following fact establishes an important tool for producing the homotopy category:
\begin{prop}\label{prop:cylinder-path-exists}
Let $\mathscr{C}$ be a model category. Let $X$ be a cofibrant object of $\mathscr{C}$, $Y$ a fibrant object of $\mathscr{C}$, and $f, g: X \to Y$ two maps. The following conditions are equivalent:
\begin{enumerate}
\item For every cylinder object $X\sqcup X \xrightarrow{j} C$ there exists a commutative diagram
\[\begin{tikzcd}
	{X\sqcup X} && C \\
	& Y
	\arrow["j", from=1-1, to=1-3]
	\arrow["{(f,g)}"', from=1-1, to=2-2]
	\arrow[from=1-3, to=2-2]
\end{tikzcd}\]
\item There exists a cylinder object $X\sqcup X \xrightarrow{j} C$ and a commutative diagram
\[\begin{tikzcd}
	{X\sqcup X} && C \\
	& Y
	\arrow["j", from=1-1, to=1-3]
	\arrow["{(f,g)}"', from=1-1, to=2-2]
	\arrow[from=1-3, to=2-2]
\end{tikzcd}\]
\item For every path object $P\xrightarrow{p} Y\times Y$ , there exists a commutative diagram
\[\begin{tikzcd}
	X && P \\
	& {Y\times Y}
	\arrow[from=1-1, to=1-3]
	\arrow["{(f,g)}"', from=1-1, to=2-2]
	\arrow["p", from=1-3, to=2-2]
\end{tikzcd}\]
\item There exists a path object $P\xrightarrow{p} Y\times Y$ and a commutative diagram
\[\begin{tikzcd}
	X && P \\
	& {Y\times Y}
	\arrow[from=1-1, to=1-3]
	\arrow["{(f,g)}"', from=1-1, to=2-2]
	\arrow["p", from=1-3, to=2-2]
\end{tikzcd}\]
\end{enumerate}
\end{prop}

If $\mathscr{C}$ is a model category containing a cofibrant object $X$ and a fibrant object $Y$, we say two maps $f,g: X \to Y$ are \emph{homotopic} if the hypotheses of \cref{prop:cylinder-path-exists} are satisfied and write $f\simeq g$. The relation $\simeq$ is an equivalence relation on $\Hom_{\mathscr{C}}(X,Y)$.

\begin{defn}
The homotopy category $h\mathscr{C}$ may be defined as follows:
\begin{itemize}
\item The objects of $h\mathscr{C}$ are the fibrant-cofibrant objects of $\mathscr{C}$.
\item The set $\Hom_{h\mathscr{C}}(X,Y)$ is the set of $\simeq$-equivalence classes of $\Hom_{\mathscr{C}}(X,Y)$, for all $X, Y \in h\mathscr{C}$.
\end{itemize}
\end{defn}

Note that the composition is well-defined in $h \mathscr{C}$: if $f \simeq g$, then $f\circ h \simeq g\circ h$ (this is clear from characterization (4) of \cref{prop:cylinder-path-exists}) and $h'\circ f\simeq  h' \circ g$ (this is clear from characterization (2) of \cref{prop:cylinder-path-exists}) for any maps $h,h'$ such that the compositions are defined in $\mathscr{C}$.

\subsubsection{Formal Inverses}
The other viewpoint involves formally adding a class of (inverse) maps to enlarge the class of weak equivalence $\mcal{W}eq$ of $\mathscr{C}$. Let $H(\mathscr{C})$ denote the category so obtained, then 
        $$H(\mathscr{C}) := \mathscr{C}[\mcal{W}eq]^{-1}.$$
If $X\in \mathscr{C}$ is cofibrant and $Y\in \mathscr{C}$ is fibrant, then homotopic maps $f,g : X\to Y$ have the same image in $H(\mathscr{C})$. Consequently, we obtain a functor $h\mathscr{C}\to H(\mathscr{C})$ which can be shown to be an equivalence. Thus, both of these methods produce similar homotopy theories at least up to an equivalence. One more important technique to point out in this direction is that there is a universal way of adjoining all the homotopy colimits to $\mathscr{C}$ to produce a homotopy category. This is called the \emph{universal homotopy theory} $U\mathscr{C}$ constructed by Dugger \cite{dugger2001universal}, which, though being a highly formal tool, has proven useful in the verification of certain generic properties of $\mathscr{C}$ by looking at certain universal cases. For our context, it can be shown that the universal homotopy associated to $\Sm_k$, $U(\Sm_k)_{\AA^1}$ is Quillen equivalent to that of the $\AA^1$-homotopy theory of Morel-Voevodsky $\mathscr{MV}_k$ (see \textit{ibid}, \S 8) for more details.

\subsection{Quillen Functors and Quillen Adjunctions}\label{app:Quillen-func}
Let $\mathscr{C}$ and $\mathscr{D}$ be model categories and suppose we are given a pair of adjoint functors 
        $$\mathscr{C} \xrightleftharpoons[G]{F} \mathscr{D} $$
(here $F$ is the left adjoint and $G$ is the right adjoint). The following conditions are equivalent:
\begin{enumerate}
\item The functor $F$ preserves cofibrations and trivial cofibrations.
\item The functor $G$ preserves fibrations and trivial fibrations.
\item The functor $F$ preserves cofibrations, and the functor $G$ preserves fibrations.
\item The functor $F$ preserves trivial cofibrations, and the functor $G$ preserves trivial fibrations.
\end{enumerate}

\begin{defn}
If any of this equivalent condition is satisfied, we say that the pair $(F, G)$ is a \emph{Quillen adjunction} between the model categories $\mathscr{C}$ and $\mathscr{D}$ and say that $F$ is a \emph{left Quillen functor} and that $G$ is a \emph{right Quillen functor}. In such a situation, one can show that $F$ preserves weak equivalences between cofibrant objects and $G$ preserves weak equivalences between fibrant objects.
\end{defn}

For a Quillen adjunction as above, one views the homotopy category $h\mathscr{C}$ as obtained from $\mathscr{C}$ by first passing to the full subcategory consisting of cofibrant objects and then inverting all weak equivalences. Applying a similar procedure with $\mathscr{D}$, we see that because $F$ preserves weak equivalence between cofibrant objects, it induces a functor $LF: h\mathscr{C} \to h\mathscr{D}$. This functor is called the \emph{left derived functor} of $F$. Dually, one defines the \emph{right derived functor} of $G$ as $RG: h\mathscr{D}\to h\mathscr{C}$. One can furthermore show that the derived functors $LF$ and $RG$ determine an adjunction between the homotopy categories $h\mathscr{C}$ and $h\mathscr{D}$:

\begin{prop}
Let $\mathscr{C}$ and $\mathscr{D}$ be model categories and let
                $$\mathscr{C} \xrightleftharpoons[G]{F} \mathscr{D} $$
be a Quillen adjunction. The following are equivalent:
\begin{enumerate}
\item The left derived functor $LF : h\mathscr{C}\to h\mathscr{D}$ is an equivalence of categories.
\item The right derived functor $RG : h\mathscr{D}\to h\mathscr{C}$ is an equivalence of categories.
\item For every cofibrant object $C\in \mathscr{C}$ and every fibrant object $D\in \mathscr{D}$, a map $C\to G(D)$ is a weak equivalence in $\mathscr{C}$ if and only if the adjoint map $F(C)\to D$ is a weak equivalence in $\mathscr{D}$.
\end{enumerate}
If the equivalent conditions listed above are satisfied, then we say that the adjunction $(F,G)$ gives a Quillen equivalence between the model categories $\mathscr{C}$ and $\mathscr{D}$.
\end{prop}
\begin{proof}
See \cite[Proposition A.2.5.1.]{lurie2009HTT}.
\end{proof}

\subsubsection{Bousfield Localizations}\label{app:Bousfield}
Given a model category $\mathscr{C}$, we will now discuss a procedure to formally invert a certain class of maps to enlarge $\mathscr{C}$ to include more weak equivalences via a process called \emph{localization}. The reader should take note that this process of localizing maps does not turn the maps themselves into an isomorphism, but rather, it makes the images of those maps in the homotopy category into isomorphisms. As the image of a map in the homotopy category is an isomorphism if and only if the map is a weak equivalence, localizing a model category with respect to a class of maps amounts to making maps into weak equivalences rather than isomorphisms.
\begin{defn}
Let $M = (\mcal{W}eq, \mcal{C}of,\mcal{F}ib)$ be a simplicial model category with class of weak equivalences $\mcal{W}eq$. Suppose that $I$ is a set of maps in $M$. 
\begin{itemize}
\item An object $X$ of $M$ is \emph{$I$-local} if it is fibrant and if for all $i : A \to B$ with $i \in I$, the induced morphism on mapping spaces $i^*: \Map_{M}(B,X) \to \Map_M(A,X)$ is a weak equivalence of simplicial sets.  
\item A morphism $f : A \to B$ is an \emph{$I$-local weak equivalence} if for every $I$-local object $X$, the induced morphism on mapping spaces
$f^* : \Map_M (B,X) \to \Map_M (A,X)$ is a weak equivalence. 
\end{itemize}
\end{defn}

Let $\mcal{F}ib_I$ denote the class of maps satisfying the right lifting property with respect to $W_I$-acyclic cofibrations, denoted as $\{\mcal{W}eq_I\cap \mcal{C}of\}$. Then the model category structure on $M$ given by $(\mcal{W}eq_I,\mcal{C}of,\mcal{F}ib_I)$, if it exists, is called the \emph{left Bousfield localization} of $M$ with respect to $I$. Whenever it forms a model category, we denote this localization by $\L_I M$.

\begin{remark}
When it exists, the Bousfield localization of $M$ with respect to $I$ is universal with respect to Quillen pairs $F: M \rightleftharpoons N: G$ such that $LF(i)$ is a weak equivalence in $N$, for all $i \in I$.    
\end{remark}

The following theorem portrays the existence of such localization in good cases and an elegant characterization of fibrant objects in those cases:

\begin{theorem}
If $M$ is a left proper and combinatorial simplicial model category and $I$ is a set of morphisms in $M$, then the left Bousfield localization $\L_I M$ exists and inherits a simplicial model category structure from $M$. Moreover, the fibrant objects of $\L_I M$ are precisely the $I$-local objects of $M$.
\end{theorem}
\begin{proof}
See \cite[Proposition A.3.7.3]{lurie2009HTT}.    
\end{proof}

For various other characterizations and existence of Bousfield localization for other model structures, one can refer to \cite[\S A.3.7]{lurie2009HTT} or \cite[Chapter 4]{hirschhorn2003model}. Dually for a right proper cellular model category, we have a theory of right Bousfield localization obtained by inverting the class of $\mcal{W}eq_I$-acyclic fibrations $\mcal{W}eq_I \cap \mcal{F}ib$. This is called \emph{cellularization} (see \cite[Chapter 5]{hirschhorn2003model}). The reader may also wish to refer to the classical reference \cite{bousfield1975localization} or a modern exposition \cite[\S 3.3]{antieau2017primer}.

We now turn to one of the most important model categories that has an elegant description and which has been well-studied in the literature.

\subsection{Combinatorial Model structures}
We will now briefly describe a special kind of model structure. A \emph{combinatorial model category} is a tractable model structure with a very strong control over the cofibrations in this model structure. It is a structure that is generated from small data = from a small set of (acyclic) cofibrations between small objects.

\begin{defn}
A model category $\mathscr{C}$ is \emph{left proper} if, for any pushout square
\[
\begin{tikzcd}
	{A} & {B} \\
	{A'} & {B'}
	\arrow["i", from=1-1, to=1-2]
	\arrow["j"', from=1-1, to=2-1]
	\arrow["{j'}", from=1-2, to=2-2]
	\arrow["{i'}"', from=2-1, to=2-2]
\end{tikzcd}
\]
in which $i$ is a cofibration and $j$ is a weak equivalence, the map $j$ is also a weak equivalence. Dually, $\mathscr{C}$ is \emph{right proper} if, for any pullback square
\[\begin{tikzcd}
	{X'} & {Y'} \\
	X & Y
	\arrow["{p'}", from=1-1, to=1-2]
	\arrow["{q'}"', from=1-1, to=2-1]
	\arrow["q", from=1-2, to=2-2]
	\arrow["p"', from=2-1, to=2-2]
\end{tikzcd}\]
in which $p$ is a fibration and $q$ is a weak equivalence, the map $q$ is also a weak equivalence. If every object in  $\mathscr{C}$ is cofibrant, then $\mathscr{C}$ is left proper.
\end{defn}

\begin{defn}\cite[Definition 5.5.0.1]{lurie2009HTT}
A category $\mathscr{C}$ is \emph{presentable} if it satisfies the following conditions:
\begin{enumerate}
\item The category $\mathscr{C}$ admits all (small) colimits,
\item There exists a (small) set $S$ of objects of $\mathscr{C}$ which generates $\mathscr{C}$ under colimits; in other words, every object of $\mathscr{C}$ may be obtained as the colimit of a (small) diagram taking values in $S$,
\item Every object in $\mathscr{C}$ is small. (Assuming (2), this is equivalent to the assertion that every object which belongs to $S$ is small.)
\item For any pair of objects $X,Y \in \mathscr{C}$, the set $\Hom_{\mathscr{C}}(X,Y)$ is small.
\end{enumerate}
\end{defn}

\begin{example}
For any model category $\mathscr{C}$, the category of presheaves $\Psh(\mathscr{C})$ is presentable. This follows from the fact that any presheaf is a colimit of representable ones (\cref{repr-psh-colimits}).
\end{example}

\begin{defn}
Let $\mathscr{C}$ be a presentable category and $\kappa$ a regular cardinal. We will say that a full subcategory $\mathscr{C}_0\subset \mathscr{C}$ is a \emph{$\kappa$-accessible subcategory} of $\mathscr{C}$ if the following conditions are satisfied:
\begin{enumerate}
\item The full subcategory $\mathscr{C}_0 \subset \mathscr{C}$ is stable under $\kappa$-filtered colimits.
\item There exists a (small) set of objects of $\mathscr{C}_0$ which generates $\mathscr{C}_0$ under $\kappa$-filtered colimits.
\end{enumerate}
We will say that $\mathscr{C}_0 \subset \mathscr{C}$ is an accessible subcategory if $\mathscr{C}_0$ is a $\kappa$-accessible subcategory of $\mathscr{C}$ for some regular cardinal $\kappa$.
\end{defn}

\begin{remark}
An $(\infty,1)$-category $\mathscr{C}$ is \emph{presentable} if $\mathscr{C}$ is accessible and admits small colimits.
\end{remark}

\begin{example}
The category of motivic spaces $\Spc_S \subset \Psh(Sm_S)$ is an accessible subcategory of $\Psh(Sm_S)$. This follows from \cite[Proposition 5.5.1.2]{lurie2009HTT} and the fact that $\Psh(Sm_S)$ is an accessible category.
\end{example}

\begin{defn}
A model category $\mathscr{C}$ is \emph{cofibrantly generated} if there are small sets of morphisms $I, J\subset Mor(\mathscr{C})$ such that
\begin{itemize}
\item $\mcal{C}of(I)$ is precisely the collection of cofibrations of $\mathscr{C}$,
\item $\mcal{C}of(J)$ is precisely the collection of acyclic cofibrations in $\mathscr{C}$,
\item $I$ and $J$ permit the small object argument.
\end{itemize}
\end{defn}
The key takeaway from the above definition is that in a cofibrantly generated model category, the collection of cofibrations and acyclic cofibrations can be characterized in a much simpler way: 
\begin{itemize}
\item[(a)] $\mcal{C}of(I)= llp(rlp (I)),$ 
\item[(b)] $\mcal{C}of(J) = llp(rlp (J)).$ 
\end{itemize}

And therefore the fibrations form precisely $rlp(J)$ and the acyclic fibrations precisely $rlp(I)$. Here, $llp = \text{left lifting property}$ and $rlp = \text{right lifting property}$. For equivalent definitions and various discussions of cofibrantly generated categories, one can refer to \cite[\href{https://ncatlab.org/nlab/show/cofibrantly+generated+model+category}{Cofib-gen}]{nLab-authors}. Since in practice these generators might be difficult to find, \cite[\S A.2.6]{lurie2009HTT} provides an alternative definition which puts more emphasis on the collection of weak equivalences in $\mathscr{C}$.

\begin{defn}
A model category $\mathscr{C}$ is \emph{combinatorial} if it is presentable as a category and cofibrantly generated as a model category.
\end{defn}
There are at least two brilliant ways of identifying combinatorial structure on $\mathscr{C}$:
\begin{enumerate}
\item A theorem due to Smith (\cite[Theorem 1.7]{beke2000}) characterizes combinatorial model categories in terms of weak equivalences and generating cofibrations, hence using only two-thirds of the input data explicitly required, which facilitates identifying combinatorial structures. For example, one can show that the left Bousfield localization of certain combinatorial model categories is again a combinatorial model category.

\item Dugger (\cref{app:duggers-thm}) identifies the combinatorial structures with those model categories that have a presentation in that they are Bousfield localizations of global model structures on simplicial presheaves, up to Quillen equivalence (\cite[Corollary 1.2]{dugger2001combinatorial}). This was an important tool in Lurie's proof to show that simplicial combinatorial model categories are precisely the models for locally presentable  $(\infty,1)$-categories.
\end{enumerate}

\begin{theorem}\label{app:duggers-thm}
Every combinatorial model category $\mathscr{C}$ is Quillen equivalent to a left Bousfield localization $\L_S \sPsh(K)$ f the global projective model structure on simplicial presheaves $\sPsh(K)_{proj}$ on a small category $K$
    $$\L_S \sPsh(K) \xrightarrow{\simeq} \mathscr{C} $$
where $S\subset \sPsh(K)$ is a well-defined subset.
\end{theorem}
\begin{proof}
This is the main result of \cite{dugger2001combinatorial}.
\end{proof}

\subsection{Homotopy Theory on Simplicial Presheaves}\label{App:model:simp-pshv}
% Ref: pick out relevant references from \cite[\S 5].{severitt2006motivic}, also \cite[\S 3.1]{antieau2017primer}, Follow \cite[\S 3]{severitt2006motivic}.
In this section, we will briefly review the category of simplicial sets and the simplicial presheaves and equip it with a model categorical structure following \cite{hovey2007model, Goerss2009simplicial, lurie2009HTT}.

\subsubsection{The Category of Simplicial Sets}
\begin{defn}
For each $n\ge 0$, we let $[n]:= \{0< 1< \dots n-1 <n \}$ denote the linearly ordered set. The category of (combinatorial) simplices, denoted $\Delta$, is the category with:
\begin{itemize}
\item Objects: linearly ordered sets $[n]$ (a.k.a. string of relations $1\to 2\to\dots\to n$).
\item Morphisms: $\theta : [m] \to [n]$ is an order-preserving set function.
\end{itemize}
The category $\Delta$ is also called the \emph{ordinal number category}.
\end{defn}

If $\mathscr{C}$ is any category, a simplicial object of $\mathscr{C}$ is a functor $ \Delta^{op}\to \mathscr{C}$. Dually, a \emph{cosimplicial object} of $\mathscr{C}$ is a functor $\Delta\to \mathscr{C}$.

\begin{defn}
The category of \emph{simplicial sets}, denoted as $\sSet$, is a simplicial object in the category of sets, that is, $\sSet:= \Fun(\Delta^{op}, \Set)$. 
\end{defn}

Since the category of sets has all (small) limits and colimits, it follows that the category of simplicial sets also has all (small) limits and colimits.

\begin{example}\label{eg:n-simplex}
There is a standard covariant functor   
\begin{align*}
        & \Delta \to \rm{Top}\\
        & n \mapsto  |\Delta^n| 
\end{align*}
The \emph{topological $n$-simplex} $\mid \Delta^n \mid \subset \RR^{n+1}$ is the space 
    $$|\Delta^n| := \{(t_0,\dots, t_n)\in \RR^{n+1}: \sum_{t=0}^{n} t_i=1, t_i\ge 0 \} $$
with the subspace topology. The map $\theta^*: |\Delta^n|\to |\Delta^m|$ induced by $\theta:[m]\to [n]$ is defined by 
            $\theta_*(t_0,\dots, t_m) = (s_0,\dots, s_n)$
where 
\begin{equation*}
s_i = 
\begin{cases}
  0                          & \text{if}\quad   \theta^{-1} = \emptyset \\
  \sum_{j\in \theta^{-1}(i)} & \text{if}\quad \text{else} 
\end{cases}
\end{equation*}
For any $T\in \sSet$, the \emph{singular set} $\Sing(T)$ is the simplicial set given by 
                $$[n]\mapsto \hom(|\Delta^n, T|)$$
\end{example}
By convention, we extend this construction by setting $\Delta^{-1}=\emptyset$. The standard 0-simplex $\Delta^0$ is a final object of the category of simplicial sets, that is, it carries each $[n] \in \Delta^{op}$ to a set having a single element.

\begin{remark}
There is an equivalent definition of a simplicial set that is useful for computations. A simplicial set $S_{\bullet}$ can be viewed as a collection of sets $\{S_n\}_{n\ge 0}$ endowed with two additional structures:
\begin{align*}
    &d^i: [n-1]\to [n], \quad 0\le i\le n\quad (\text{cofaces}) \\
    & s^j: [n+1]\to [n], \quad 0\le j\le n\quad (\text{codegeneracies})
\end{align*}
Both of these maps $d^i$ and $s^j$ are expected to satisfy a bunch of equalities, called the \emph{simplicial identities} that can be found in any of the standard references (e.g., \cite[Chapter 1, \S 1]{Goerss2009simplicial}). Every order-preserving map from $[n]$ to $[m]$ can be factored as a composition of face and degeneracy maps, and so the structure of a simplicial set $S$ is completely determined by the sets $S_n$ for $n\ge 0$ together with these face and degeneracy maps.
\end{remark}

\begin{example}
The \emph{standard $n$-simplex}, simplicial $\Delta^n$ in the simplicial set category $\sSet$ is defined by 
            $$\Delta^n:= \hom_{\Delta}(-, [n]).$$
Put alternatively, $\Delta^n$ is the contravariant functor on $\Delta$ which is represented by $[n]$.
\end{example}

A morphism $f: X \to Y$ of simplicial sets, called the \emph{simplicial map}, is a natural transformation of contravariant set-valued functors defined on $\Delta$. The Yoneda lemma implies that simplicial maps $\Delta^n \to Y$ classify $n$-simplices of $Y$ in the sense that there is a natural bijection
            $$ \hom_{\sSet}(\Delta^n,Y) \cong Y_n $$
between the set $Y_n$ of $n$-simplices of $Y$ and the set $\hom_S(\Delta^n,Y)$. There is a standard simplex 
            $$ \iota_n := \textbf{1}_n\in \hom_{\Delta}([n],[n])$$
            
\begin{example}
The category $\Delta^n$ contains two important sub-complexes: The $k$-th horn $\Lambda^n_k$ $0 \le k \le n$ which is a sub-complex of $\Delta^n$ generated by all faces $d_j(\iota_n)$ except the $k$th face $d_k(\iota_n)$ and the boundary $\partial\Delta^n$, which is the smallest sub-complex of $\Delta^n$ containing the faces $d_j(\iota_n)$, $0 \le j \le n$ of the standard simplex $\iota_n$.
\end{example}

\subsubsection{Model structure on Simplicial Sets}\label{app:model-on-sSets}
The category $\sSet$ of simplicial sets has a combinatorial, left proper, and right proper model structure, which is usually referred to as the \emph{Kan model structure} defined as follows:
\begin{enumerate}
\item A map of simplicial sets $f : X\to Y$ is a \emph{cofibration} if it is a monomorphism: if the induced map $X_n\to Y_n$ is injective for all $n\ge 0$,
\item A map of simplicial sets $f: X\to Y$ is a \emph{fibration} if it is a Kan fibration: if for any diagram
\[\begin{tikzcd}
	{\Lambda^n_i} & X \\
	{\Delta^n} & Y
	\arrow[from=1-1, to=1-2]
	\arrow[hook, from=1-1, to=2-1]
	\arrow["f", from=1-2, to=2-2]
	\arrow[dashed, from=2-1, to=1-2]
	\arrow[from=2-1, to=2-2]
\end{tikzcd}\]
It is possible to supply the dotted arrow rendering the diagram commutative, where $\Lambda^n_i$ is the $i$-th horn of $\Delta^n$.
\item A map of simplicial sets $f : X\to Y$ is a \emph{weak equivalence} if the induced map of geometric realizations $|X| \to |Y|$ is a homotopy equivalence of topological spaces.
\end{enumerate}

Studying model structures on the category of simplicial sets is the fundamental step towards building model structures for the category of motivic $S$-spaces $\Spc_S$. This is essentially since the category of simplicial presheaves on which the category of spaces $\Spc_S$ is built is, in fact, enriched over $\sSet$ (such that the enrichment is compatible with the model structure). Such enriched categories are called \emph{Simplicial Categories}. Briefly, speaking a model category $\mathscr{C}$ is \emph{simplicial} if the following are satisfied:
\begin{enumerate}
\item There exists mapping space functor $\Map : \mathscr{C}^{op}\times \mathscr{C} \to \sSet$ 
\item There exists a product functor $\otimes: \sSet\times \mathscr{C} \to \mathscr{C}$ 
\end{enumerate}
with a compatibility that the product is associative, that is, 
\begin{itemize}
\item $(L\times K)\otimes X \simeq L\times (K\otimes X)$, natural in $X \in \mathscr{C}$, $K,L \in \sSet$
\item $\Delta^0$ is a unit object, i.e. $\Delta^0 \otimes X \cong X$.
\item for a $X \in \mathscr{C}$ and $K\in \sSet$, there are adjoint functor pairs  
        \begin{align*}
        \otimes X:\sSet \leftrightharpoons \mathscr{C}: Map(-,X)\\
        K\otimes -: \mathscr{C}\leftrightharpoons \sSet: (-)^{K}
        \end{align*}    
\end{itemize}

\begin{example}
The category $\sSet$ together with the model structure introduced above is a simplicial model category if we take
$K\otimes X:= K\times X$ (\cite[Proposition 11.5]{Goerss2009simplicial}).
\end{example}

\subsection{Category of Simplicial Presheaves}
The category of simplicial presheaves $\sPsh(\mathscr{C})$ on a model category $\mathscr{C}$ is the contravariant functor 
        $$\sPsh(\mathscr{C}) := \Fun(\mathscr{C}^{op}\to \sSet)$$
We shall now briefly discuss some of the relevant model structures on the category of simplicial presheaves. Several model structures on the category of presheaves have been studied throughout history (See \cite{jardine1987simplicial,heller1988homotopy, Goerss2009simplicial, hirschhorn2003model, dugger2001universal,dundas2003motivicfunctors, isaksen2005flasque, hovey2007model, rondigs2025grothendieck}). Classically speaking, one equips this category with four possible model structures: 
\begin{itemize}
\item (Global) injective and local injective model structure
\item (Global) projective and local projective model structure 
\end{itemize}
The injective model structure is the one where cofibrations are detected pointwise\footnote{also called as sectionwise or objectwise}, and fibrations are defined as those that satisfy the left lifting property with respect to trivial (also called as acyclic) cofibrations. On the other hand, the projective model structure is the one where weak equivalences stay the same, but fibrations are detected pointwise, and cofibrations are those morphisms that satisfy the left lifting property with respect to the acyclic fibrations. Hence, the injective model structure is more suited to describe cofibrations and the projective model structure to describe fibrations. The local analogs are then obtained by (Bousfield) localizing with respect to a suitable class of morphisms to produce certain additional weak equivalences that allow one to study the category locally.
\medskip

For the current exposition below, we will follow \cite{severitt2006motivic, antieau2017primer}. We will only build this theory on the level of presheaves, noting that all of these can also be lifted to the category of simplicial sheaves via the adjoint functor pair 
            $$a: \sPsh(T) \rightleftarrows \sShv(T): \iota_* $$
where $\iota$ is the inclusion functor and $a$ is the associated sheaf functor that is defined objectwise.
\begin{defn}
Let $T= (\mcal{C}, J)$ denote a small Grothendieck site. Then the simplicial objects denoted by 
        $$\sPsh(T):= \Fun(\Delta^{op}, \Psh(T))$$
are called the \emph{simplicial presheaves} on $T$. That is, $\sPsh(T)= \Fun(\mcal{C}^{op}, \sSet)$ which are $\sSet$-valued presheaves on $T$. Furthermore, we have that the pointed simplicial presheaves are the same as $\sSet$-valued presheaves on $T$, i.e., $\sPsh_*(T):= \Fun(\mcal{C}^{op}, \sSet_*)$.
\end{defn}

\begin{prop}\label{repr-psh-colimits}
Every simplicial presheaf is canonically isomorphic to a colimit of representable simplicial presheaves $R_X:= \Hom_{\mcal{C}}(-, X)$.
\end{prop}
\begin{proof}
This is due to \cite[Chapter III, \S 7, Theorem 1]{mac1998categories}.
\end{proof}

\begin{corollary}\label{prop:Rep-psh-colimits}
Every motivic space is a colimit of representable ones.
\end{corollary}
\begin{proof}
This follows from the fact that $\L_{mot}$ preserves colimits (See \cref{Prop:L_mot-preserve}). See also \cite[Lemma 2.4]{dundas2003motivicfunctors}.
\end{proof}

\begin{defn}
Let $f: X \to Y$ be a map of simplicial presheaves. Then $f$ is called a
\begin{itemize}
\item \emph{local weak equivalence}, if $f_*:\pi_0(X)\to \pi_0(Y)$ is a bijection and for all $n>0$ and $U\in \mathcal{C}$, the induced map $\pi_n(X|_U, x)\to \pi_n(Y|_{U},f(x)$ is an isomorphism of sheaves for all base points $x\in X(U)_0$.
\item \emph{objectwise weak equivalence} (resp. (co-)fibrations if for all $U\in \mathcal{C}$, the induced map $f(U): X(U)\to Y(U)$ is a weak equivalence (resp. (co-)fibration) in $\sSet$.
\end{itemize}
\end{defn}

There is a closed symmetric monoidal structure on simplicial presheaves (\cite[Chapter XI, \S 1]{mac1998categories}). Let $X, Y$ be simplicial presheaves. Define the simplicial function complex 
        $$\Map(X,Y)_n := \Hom_{\sPsh(T)}(X\times \Delta^n, Y)$$
where $\Delta^n$ is taken as the constant simplicial presheaf. The internal function complex is given by 
        $$\ul{\Map}(X,Y)(U):= \Map(X|_{U}, Y|_U) $$
on simplicial sheaves, and one shows that this complex is again a simplicial presheaf. The category of simplicial presheaf, together with the categorical product
        $$ - \times - :\sPsh(T)\times \sPsh(t)\to \sPsh(T) $$
is symmetric monoidal, where $X\times Y$ is denoted by $X\otimes Y$. Moreover, they are closed since 
        $$ - \otimes X: \sPsh(T) \rightleftarrows \sPsh(T): \ul{\Map}(X,-) $$
The following is a version of the Yoneda lemma over $\sSet$.
\begin{lemma}
Let $X\in \mcal{C}$ and $F\in \sPsh(T)$. Then we have that 
        $$\Map(R_X, F) \cong F(X) \in \sSet.$$
\end{lemma}

\begin{lemma}
Let $F,G$ be simplicial presheaves and $X\in \mcal{C}$. Then 
$$\ul{\Map}(R_X, G) = G(-\times X) \quad \text{and} \quad  \ul{\Map}(F,G)(X) = \Map(F, G(-\times X)) $$
for presheaves.
\end{lemma}

All these can also be formulated for pointed simplicial presheaves. Furthermore, it is a closed symmetric monoidal category with respect to the smash product 
        $$- \wedge X: \sPsh_*(T) \rightleftarrows \sPsh(T) : \ul{\Map}_*(X,-)$$

\subsubsection{Model Structure on Simplicial Presheaves}
We are now ready to define the (local) model structures on simplicial presheaves. We begin with the following simplicial structure.
\begin{defn}
Let $X, Y$ be simplicial presheaves and $K$ be a simplicial set considered as a constant presheaf. Then define
                $$K\otimes X:= K\times X $$
and $\Map(X,Y)$ is the simplicial function complex declares as $X^K:= \ul{\Map}(K,X)$ and $\ul{\Map}(K,X)(U):= \Map(K, X(U))$. Then 
            $$-\otimes X: \sSet \rightleftarrows \sPsh(T): \Map(X,-) $$
and
            $$K\otimes -: \sPsh(T) \rightleftarrows \sPsh(T): (-)^K $$
are an adjoint pair of functors.
\end{defn}

We now define its homotopy category. 
\begin{defn}
Let $X$ be a simplicial presheaf. Define $\pi_0(X)$ as the sheafified presheaf 
            $$U\mapsto \pi_0(X(U)) $$
for $U\in \mcal{C}$, where $\pi_0(X(U))$ is just the connected components of the simplicial set $X(U)$. Let $U\in \mcal{C}, x\in X(U)_0,$ and $n>0$. Define $\pi_n(X|_U, x)$ as the sheafified presheaf 
        $$(V\to U) \mapsto \pi_n(X(V), x_V ) $$
of homotopy groups of simplicial sets on the site $T\downarrow U$ and $x_V\in X(V)_0$ is just the image of $x$ under $X(U)_0 \to X(V)_0$. Here, $X\downarrow U$ is the restricted pointed simplicial presheaf, which is defined via the composition
    $$X|_U:= (\mcal{C} \downarrow U)^{op} \xrightarrow{(\iota_U)^{op}} \mcal{C}^{op} \xrightarrow{X} \sSet_*$$
\end{defn}      

\begin{defn}
Let $f: X \to Y$ be a map of simplicial presheaves. Then $f$ is called a 
\begin{itemize}
\item \emph{local weak equivalence} if $f_*:\pi_0(X)\to \pi_0(Y)$ is an isomorphism of sheaves and for all $n>0$ and $U\in \mcal{C}$, the induced map $f_*: \pi_n(X|_U,x)\to \pi_n(Y|_U, f(x)$ is an isomorphism of sheaves for all base points $x\in X(U)_0$,
\item \emph{objectwise weak equivalence} if for all $U\in \mcal{C}$, $f(U): X(U)\to Y(U)$ is a weak equivalence in $\sSet$,
\item \emph{objectwise (co-)fibration} if for all $U\in \mcal{C}$, $f(U): X(U)\to Y(U)$ is a (co)-fibration in $\sSet$.
\end{itemize}
\end{defn}

\subsubsection{The Four Model Structures on Simplicial Presheaves}
The following theorems establish the (local) injective model structure on simplicial presheaves due to Jardine (\cite{jardine1987simplicial}):
\begin{theorem}
The category of simplicial presheaves on a small Grothendieck site together with the distinguished classes of maps 
\begin{itemize}
\item (local) objectwise weak equivalences
\item objectwise cofibrations
\item injective fibrations, those maps having the right lifting property with respect to acyclic (trivial) objectwise cofibrations.
\end{itemize}
form a proper simplicial cofibrantly generated model category and is called the \emph{(local) injective model structure} (cf. \cite[Theorem 1.1]{blander2001local}).
\end{theorem}

For instance, the $\AA^1$-homotopy theory as established by \cite{MV99} uses this model structure. There is an analogous model structure obtained by swapping the cofibrations with fibrations due to Bousfield and Kan (\cite{bousfield1972homotopy}):
\begin{theorem}
The category of simplicial presheaves on a small Grothendieck site together with the distinguished classes of maps 
\begin{itemize}
\item (local) objectwise weak equivalences
\item objectwise fibrations
\item projective cofibrations, those maps having the left lifting property with respect to acyclic (trivial) objectwise fibrations.
\end{itemize}
form a proper simplicial cellular/combinatorial model category and is called the \emph{(local) projective model structure} (cf. \cite[Theorem 1.4]{blander2001local}, \cite[Proposition 3.32]{antieau2017primer}).
\end{theorem}

The projective model category structure has a special universal property highlighted by Dugger \cite{dugger2001universal}: it is the initial model category into which $C$ embeds. It is essential to note that both of these model structures are Quillen equivalent (see e.g., \cite[Theorem 6.2.11]{severitt2006motivic}) and hence produce the same homotopy theory.
\medskip

We will end this chapter with a collection of useful notions from model categories.
\subsubsection{Kan Extensions}\label{app:Kan-ext}
Let us now define the Kan extensions. The Kan extension of a functor $F:\mathscr{C}\to \mathscr{D}$, if it exists, is one of the best approximations to finding an extension from another category $\mathscr{C}'$ to that of $\mathscr{C}$. In other words, the Kan extension of $F$ with respect to a functor $p: \mathscr{C'}\to \mathscr{D}$ is the functor that makes the following diagram
\[\begin{tikzcd}
	{\mathscr{C}} & {\mathscr{D}} \\
	{\mathscr{C}'} 
	\arrow["F", from=1-1, to=1-2]
	\arrow["p"', from=1-1, to=2-1]
	\arrow["{?}"', from=2-1, to=1-2]
\end{tikzcd}\]
commutative. In this way, it extends the domain of $F$ through $p$ from $\mathscr{C}$ to  $\mathscr{C}'$. Dually, there is also a notion of Kan lift that is analogously defined to lift any functor. Here we give an official definition of this fact (\cite[\href{https://kerodon.net/tag/02Y1}{Tag 02Y1}]{kerodon}):
\begin{defn}
Let $F:\mathscr{C}\to \mathscr{D}$ be a functor between categories and let $\mathscr{C}^0 \subset \mathscr{C}$ be a full subcategory. We say that $F$ is \emph{Left Kan extended from $\mathscr{C}^0$} if, for every object $C \in \mathscr{C}$, the collection of morphisms $\{F(u): F(C_0)\to F(C)\}_{u:C_0\to C}$ exhibits $F(C)$ as a colimit of the diagram
 $$(\mathscr{C}^0 \times_{\mathscr{C}} \mathscr{C}_{/C}) \to \mathscr{C}^0 \hookrightarrow \mathscr{C} \xrightarrow{F} \mathscr{D}.$$
\end{defn}
Such a functor, if it exists, is unique. In the above, $\mathscr{C}_{/C}$ we mean the \emph{overcategory} whose objects are maps $f: X\to C$ and morphisms between $f: X\to C$ and $g: Y\to C$ is a morphism $h: X\to Y$ such that the obvious commutative triangle commutes.

\begin{prop}
Let $F, G:\mathscr{C} \to \mathscr{D}$ be functors between categories, and suppose that $F$ is left Kan extended from a full subcategory $\mathscr{C}^0 \subset \mathscr{C}$. Then the restriction map 
 $$\{\text{Natural transformation from $F$ to $G$} \to {\text{Natural transformation from $F_{\vert{\mathscr{C}^0}}$ to $G_{\vert{\mathscr{C}^0}}$}} \} $$
is bijective. In particular, the functor $F$ can be recovered (up to canonical isomorphism) from the restriction $F_{\vert{\mathscr{C}^0}}$
\end{prop}

The existence of such an extension can be tested by looking at a certain colimit in $\mathscr{D}$.
\begin{prop}
Let $\mathscr{C}$ be a category, let $\mathscr{C}^0 \subset \mathscr{C}$ be a full subcategory, and let $F_0: \mathscr{C}^0\to \mathscr{D}$ be a functor between categories. Then the following conditions are equivalent:
\begin{enumerate}
\item There exists a functor$F :C\to D$ which is left Kan extended from $\mathscr{C}^0$ and satisfies $F_{\vert_{\mathscr{C}^0}}\to F_0$
\item For every object $C \in \mathscr{C}$, the diagram
    $$(\mathscr{C}^0 \times_{\mathscr{C}} \mathscr{C}_{/C}) \to \mathscr{C}^0 \xrightarrow{F_0} \mathscr{D}.$$
\end{enumerate}
\end{prop}

Among all the mathematical formulations, Kan extensions are one of the most ubiquitous occurrences in nature. Here are a few of them.
\begin{example}
One may think of left Kan extensions as a generalization of colimits and dually. Let $\mathscr{C}$ be any category that is both complete and cocomplete (eg, $\mathscr{C} = \Set$). Let $\mathscr{D}$ be any small category, then all functors $F: \mathscr{C}\to \mathscr{D}$ have both right and left Kan extensions along all functors $K: \mathscr{C}\to \mathscr{D}$.  In particular, any functor $F: \mathscr{C} \to \mathscr{D}$ has both right and left Kan extensions along $1_{\mathscr{C}}: \mathscr{C}\to \mathscr{C}$. By the universal property of Kan extensions, this in turn implies that both left and right Kan extensions are $F$ itself.
\end{example}

It is a general fact that the left adjoint to any pre-composition functor is a left Kan extension, and the right adjoint to any pre-composition functor is a right Kan extension. This follows since adjoints are unique up to unique isomorphism. The following illustrates this principle:

\begin{example}
Fix a group $G$ with a subgroup $H\le G$ along with the inclusion functor $i: H \to G$. The category $[G, Vect_k]$ of functors from the group $G$ to the $k$-vector spaces has $G$-representations as objects and $G$-equivariant linear maps as morphisms. As any representation of $G$ can be restricted through the inclusion map $i$ to a representation of $H$, we have a functor $\text{Res}: [G, Vect_k] \to [H, Vect_k]$ defined by precomposition with $i$. Now, any representation on $H$ can be extended to the (induced) representation of $G$ via the induction functor $\text{Ind}$ and vice-versa, extended to the coinduced representation of $G$, via the coinduction functor $\text{Coind}$. These extensions define functors, $\text{Ind, Coind}: [H,Vect_k] \to [G,Vect_k]$, which are left and right adjoints to the restriction functor $\text{Res}$. Due to the general principle stated above, we have that these functors define left and right Kan extensions along the $i$.
\end{example}

\subsubsection{Cofinality}
\begin{defn}\label{cofinal}
A subcategory $I\subseteq J$ is said to be \emph{cofinal} if for any functor $F: J\to \mcal{C}$, the induced map on colimits 
        $$ \text{colim}_I F\to \text{colim}_J F $$
is an isomorphism. Hence, to compute colimits in $J$, it suffices to compute it in the subdiagram $I\subseteq J$.
\end{defn}

\begin{example}
The subcategories $2\mathbf{N}\subseteq \mathbb{N}$ and $2\mathbb{N}+1\subseteq \mathbb{N}$ are cofinal in $\mathbb{N}$.
\end{example}

\subsubsection{A Very Weak Five Lemma}
We state the 5-lemma for a general pointed model category and elicit its effectiveness under contractibility assumptions. For more reference, we redirect the reader to \cite[Chapter 6]{hovey2007model}.
\begin{lemma}\label{weak5lemma}
Let $\mcal{C}$ be any pointed model category and consider the following commutative diagram in $\mcal{C}$ where the rows are cofiber sequences.
\begin{center}
\begin{tikzcd}
	A & B & C \\
	{A'} & {B'} & {C'} \\
	\\	{}
	\arrow["u", from=1-1, to=1-2]
	\arrow[swap, "f", from=1-1, to=2-1]
	\arrow["v", from=1-2, to=1-3]
	\arrow["g"', from=1-2, to=2-2]
	\arrow["h", from=1-3, to=2-3]
	\arrow[swap, "{u'}", from=2-1, to=2-2]
	\arrow[swap, "{v'}", from=2-2, to=2-3]
\end{tikzcd}
\end{center}
If $f$ and $g$ are weak equivalences, then so is $h$. In addition, assume that $f$ and $h$ are weak equivalences and $B$ is contractible in $\mcal{C}$, then $g$ is a weak equivalence in $\mcal{C}$ and consequently, $B'$ is contractible in $\mcal{C}$.
\end{lemma}
\begin{proof}
The proof is due to \cite[Proposition 6.5.3 (b)]{hovey2007model}. It suffices to show that $[B',X]\simeq *$ for every object $X\in \mathcal{C}$. Applying the representable Hom functor $[-, X]$ for any object $X\in \mathcal{C}$ to the above diagram, we get a commutative diagram of long exact sequences of pointed sets and groups 
    \[\begin{tikzcd}
	{[A,X]} && {[B,X]} && {[C,X]} && {[\Sigma A,X]} & \dots \\
	\\
	{[A',X]} && {[B',X]} && {[C',X]} && {[\Sigma A', X]} & \dots
	\arrow["{u^*}"', from=1-3, to=1-1]
	\arrow["{v^*}"', from=1-5, to=1-3]
	\arrow["\delta"', from=1-7, to=1-5]
	\arrow[from=1-8, to=1-7]
	\arrow["{f^*}", from=3-1, to=1-1]
	\arrow["{g^*}", from=3-3, to=1-3]
	\arrow["{(u')*}", from=3-3, to=3-1]
	\arrow["{h^*}", from=3-5, to=1-5]
	\arrow["{(v')^*}", from=3-5, to=3-3]
	\arrow["{(\Sigma f)^*}", from=3-7, to=1-7]
	\arrow["{\delta'}", from=3-7, to=3-5]
	\arrow[from=3-8, to=3-7]
    \end{tikzcd}\]
We conclude that the fiber $g^*$ is trivial by a simple diagram chase. As $[B,X]\simeq *$, it follows that $[B',X]\simeq *$ as well.
\end{proof}

The following lemma, due to Ken Brown, is extremely useful in model categories.
\begin{lemma}\label{ken-brown-lemma}
Suppose $\mathscr{C}$ is a model category and $\mathscr{D}$ is a category with a subcategory of weak equivalences that satisfies the two out of three axiom. Suppose $F: \mathscr{C} \to \mathscr{D}$ is a functor that takes trivial cofibrations between cofibrant objects to weak equivalences. Then $F$ takes all weak equivalences between cofibrant objects to weak equivalences. Dually, if $F$ takes trivial fibrations between fibrant objects to weak equivalences, then $F$ takes all weak equivalences between fibrant objects to weak equivalences.
\end{lemma}
\begin{proof}
This is due to \cite[Lemma 1.1.12]{hovey2007model}. 
\end{proof}

% \afterpage{\blankpage}

\begin{savequote}
We did not domesticate wheat. It domesticated us.
\qauthor{"Sapiens" by Yuval Noah Harari}   
\end{savequote}
\chapter{$\AA^1$-Algebraic Topology}\label{app:A1-alg-top}
\markboth{Appendix}{}

% $\AA^1$-Algebraic topology

In this section, we define and illustrate the first properties of strictly and strongly $\AA^1$-invariant following the book "\emph{$\AA^1$-algebraic topology over a field}" (\cite{morel2012A1topology}). In the sequel, we will also define the $\AA^1$-derived category and the $\AA^1$-homology sheaves, which we rely on in \cref{chp5}. Following this, we present the theory of Milnor-Witt sheaves and the associated Chow-Witt theory within the scope of this script.
\medskip

One has from algebraic topology that the set of path-connected components $\pi_0$ and the homotopy groups $\ pi_i$ are all discrete. The study of (strictly and strongly) $\AA^1$-invariant sheaves can be thought of as studying the analogous "discreteness" property of $\AA^1$-homotopy theoretic sheaves. 

\begin{defn}\label{defn:A1-inv-str-stro}
All the sheaves mentioned below are considered in the Nisnevich topology.
\begin{enumerate}
\item A sheaf of sets $\mcal{S}$ on $\Sm_k$ is said to be \emph{$\AA^1$-invariant} if for any $X \in \Sm_k$ ,the map
    $$\mcal{S}(X)\to \mcal{S}(X\times \AA^1)$$
induced by the projection $\AA^1 \times X \to X$, is a bijection.
\item A sheaf of groups $\mcal{G}$ on $\Sm_k$ is said to be \emph{strongly $\AA^1$-invariant} if for any $X \in \Sm_k$, the map
    $$H^i_{Nis}(X; \mcal{G}) \to H^i_{Nis}(X\times \AA^1; \mcal{G})$$
induced by the projection $\AA^1 \times X \to X$, is a bijection for $i \in \{0,1\}$.
\item A sheaf of groups $\mcal{M}$ on $\Sm_k$ is said to be \emph{strictly $\AA^1$-invariant} if for any $X \in \Sm_k$, the map
    $$H^i_{Nis}(X,\mathcal{M}) \to H^i_{Nis}(X\times_k \AA^1, \mathcal{M})$$
induced by the projection $\AA^1 \times X \to X$, is a bijection for $i\in \mathbb{N}$.
\end{enumerate}
\end{defn}

\begin{remark}
One also has the following characterization of these sheaves in terms of their associated algebraic invariants:
\begin{enumerate}
\item A sheaf of set $\mcal{S}$ is $\AA^1$-invariant if and only if it is $\AA^1$-local as a space,
\item A sheaf of groups $\mcal{G}$ is strongly $\AA^1$-invariant if and only if the classifying space $B(\mcal{G})= \textbf{K}(\mcal{G},1)$ is an $\AA^1$-local space,
\item A sheaf of abelian groups $\mcal{M}$ is strictly $\AA^1$-invariant if and only if for any $n\in \mathbb{N}$, the Eilenberg-MacLane space $\textbf{K}(\mcal{M},n)$ is $\AA^1$-local.
\end{enumerate}
\end{remark}

\begin{theorem}\label{thm:A1-htpysheaves-invariant}
Let $\mcal{X}$ be a pointed space. Then the $\AA^1$-fundamental sheaf of groups $\pi_1^{\AA^1}(\mcal{X})$ is strongly $\AA^1$-invariant. For any $n \ge 2$, the higher $\AA^1$-homotopy sheaves of abelian groups $\pi_n^{\AA^1}$ are strictly $\AA^1$-invariant.
\end{theorem}
\begin{proof}
This proof is due to \cite[Theorem 6.1, Corollary 6.2]{morel2012A1topology}.
\end{proof}

\begin{example}\label{eg:strictly-strongly}
For $i\ge 1$, the $\AA^1$-homotopy sheaves $\pi_i^{\AA^1}$ are strongly $\AA^1$-invariant, Rost cycle modules, $\AA^1$-invariant sheaves with transfers of Voevodsky, sheaf $\textbf{\underline{W}}$ of unramified Witt groups.
\end{example}

\subsection{$\AA^1$-derived category}
Let us denote by $\abek$ the abelian category of sheaves of abelian groups on $\Sm_S$ in the Nisnevich topology. Let $C_*(\abek)$ be the category of chain complexes in $\abek$. As in the case of classical homological algebra, one obtains the derived category of $\abek$, denoted by $D(\abek)$, from  $C_*(\abek)$ by inverting the class of quasi-isomorphisms between chain complexes. To establish the model categorical structure, one can use the following (Nisnevich) local model structure on the category of chain complexes $C_*(\abek)$ in $\abek$:

\begin{enumerate}
\item A morphism $f: A_* \to B_*$ in $C_*(\abek)$ is a \emph{weak equivalence} if it induces a quasi isomorphism over the sections of the smooth Henselian local schemes (stalks in the Nisnevich topology),
\item A morphism $f: A_* \to B_*$ is a fibration if it is an epimorphism,
\item A morphism $f: A_* \to B_*$ is a cofibration if it has the left lifting property with respect to the trivial fibrations.
\end{enumerate}

The derived category $D(\abek)$ is the associated homotopy category obtained by formally inverting the quasi-isomorphisms in the above model structure. We will now define the $\AA^1$-derived category. It is obtained by the left Bousfield localization, where we invert the projection maps with respect to $\ZZ(\AA^1_k)$ in the above local model structure, where $\ZZ(\AA^1_k)$ is the free abelian sheaf on $\AA^1_k$.

\begin{defn}
We will now define the $\AA^1$-derived category.
\begin{itemize}
\item A chain complex $D_* \in C_*(\abek)$ is called \emph{$\AA^1$-local} if and only if for any $C_* \in C_*(\abek)$, the projection $C_*\otimes \ZZ(\AA^1)\to C_*$ induces a bijection:
    $$Hom_{D(\abek)} (C_*,D_*)\to   Hom_{D(\abek)} (C_*\otimes \ZZ(\AA^1), D_*). $$
\item A morphism $f : C_* \to D_*$ in $C_*(\abek)$ is called an \emph{$\AA^1$-quasi isomorphism} if and only if for any $\AA^1$-local chain complex $E_*$, the morphism:
        $$ Hom_{D(\abek)}(D_*,E_*) \to Hom_{D(\abek)}(C_*, E_*) $$
 is bijective.
\item Finally, we define the $\AA^1$-derived category $D_{\AA^1}(\abek)$ as the category obtained by inverting the all the $\AA^1$-quasi isomorphisms.
\end{itemize}
\end{defn}
As in the case of the theory of simplicial presheaves, we have an (abelian) $\AA^1$-localization functor 
        $$L_{\AA^1}^{Ab}: C_*(\abek)\to C_*(\abek)$$ 
together with a natural transformation $\theta: Id \to L_{\AA^1}^{Ab}$ such that for any chain complex $C_*$,                           
        $$\theta_{C_*}: C_*\to L_{\AA^1}^{Ab}(C_*)$$ 
is an $\AA^1$-quasi isomorphism whose target is an $\AA^1$-local fibrant chain complex. There is natural functoriality that one derives as a formal consequence of such a framework (See \cite[\S 6.2]{morel2012A1topology}). 

\subsubsection{$\AA^1$-homology sheaves}
Recall that in algebraic topology (see e.g., \cite[Chapter 5]{may1992simplicial}), we have the notion of a normalized chain complex $N(C_*)$ associated to any chain complex $C_*$. Importing this machinery into algebraic geometry, one defined the notion of $\AA^1$-homology sheaves (\cite[\S 6.3]{morel2012A1topology}). We will briefly touch upon this now.
Let $\mcal{X}$ be any simplicial sheaf of sets with its normalized chain complex $\mcal{N}(\mcal{X})$ in $C_*(\abek)$, which is associated to the free simplicial sheaf of abelian groups $\ZZ(\mcal{X})$ on $\mcal{X}$. This defines a functor
         $$\mcal{N}_*: \Delta^{op}\Shv_{Nis} (\Sm_k)\to C_*(\abek).$$
This functor has a right adjoint 
        $$K: C_*(\abek) \to \Delta^{op}\Shv_{Nis}(\Sm_k) $$
called the \emph{Eilenberg-MacLane space functor}.

\begin{remark}
For an abelian sheaf $\mcal{A} \in \abek$ and an integer $n$, the pointed simplicial sheaf $K(\mcal{A},n)$ is defined by applying $K$ to the shifted complex $\mcal{A}[n]$, of the complex $\mcal{A}$ placed in degree 0. It is characterized by the following property:
\[\pi_i^{\AA^1}(K(\mcal{A},n)) \cong
        \begin{cases*}
             0           & if  $n < 0$  \\
             \mcal{A}    & if $i=n$ \\
             0           & if $i \ne n$
    \end{cases*} \]
\end{remark}

Moreover, it is a fact that the functor $\mcal{N}_*(-)$ sends simplicial weak equivalences to quasi-isomorphisms and $K(-)$  maps quasi-isomorphisms to simplicial weak equivalences. As a consequence, these induce a pair of adjoint functors
        $$\mcal{N}_*: \mcal{H}_s(k) \to D(\abek) $$
and 
        $$K: D(\abek) \to \mcal{H}_s(k).$$
Thus $\mcal{N}_*: \mcal{H}_s(k) \to D(\abek)$ maps $\AA^1$-weak equivalences to $\AA^1$-quasi isomorphisms and induces a functor
         $$\mcal{N}_*^{\AA^1}: \mcal{H}(k) \to   D_{\AA^1}(\abek) $$
defined by 
        $$\mcal{X} \mapsto L_{\AA^1}^{Ab}(\mcal{N}_*(\mcal{X})).$$

\begin{defn}
Let $\mcal{X}$ be a space and $n \in \ZZ$ be an integer. The \emph{$n$-th $\AA^1$-homology sheaf} associated to $\mcal{X}$, denoted as $H_n^{\AA^1}(\mcal{X})$, is defined as the $n$-th homology sheaf of the complex $\mcal{N}_*(\mcal{X})$. 
\end{defn}
If $\mcal{X}$ is pointed, one defines the \emph{reduced $n$-th $\AA^1$-homology sheaf} as the 
    $$\widetilde{H}_n^{\AA^1}(\mcal{X}) := Ker (H_n^{\AA^1}(\mcal{X})\to H_n^{\AA^1}(\Spec k))$$
and attributed to the vanishing of homology sheaves of $\Spec k$ in all non-zero dimensions, we have that 
    $$H_*^{\AA^1}(\mcal{X}) = \ZZ \oplus \widetilde{H_*}^{\AA^1}(\mcal{X})$$

The following is formal of the $\AA^1$-connectivity theorem (\cite[Corollary 6.31]{morel2012A1topology}):

\begin{corollary}
The $\AA^1$-homology sheaves $H_n^{\AA^1}(\mcal{X})$ of a space $\mcal{X}$ vanish for $n<0$ and are strictly $\AA^1$-invariant sheaves for $n\ge 0$.
\end{corollary}

\begin{corollary}
The $\AA^1$-localization functor commutes with  the suspension in $D(\abek)$ and so there exists a canonical suspension isomorphism for any pointed space $\mcal{X}$ and any integer $n\in \ZZ$:
        $$\widetilde{H}_n^{\AA^1}(\mcal{X}) \cong \widetilde{H}_{n+1}^{\AA^1}(\Sigma_s \mcal{X}) $$
where $\Sigma_s \mcal{X}:= S^1\wedge \mcal{X}$ is the simplicial suspension of $\mcal{X}$.
\end{corollary}

We end this section with one powerful application of these sheaves for $\AA^1$-simply connected spaces that we use in Chapter 5.
\begin{theorem}
Let $k$ be any perfect field. Let $f: X\to Y$ be a morphism between $\AA^1$-simply connected pointed smooth $k$-schemes and let $d= \text{max}\{\dim \ X+1, Y\}$. If $f$ induces an isomorphism $H_i^{\AA^1}(X) \xrightarrow{\cong} H_i^{\AA^1}(Y)$ for all $2\le i <d$ and an epimorphism $H_d^{\AA^1}(X)\to H_d^{\AA^1}(Y)$, $f$ is an $\AA^1$-weak equivalence in $Spc_k$.
\end{theorem}
\begin{proof}
See \cite[Theorem 1.1, \S 3]{shimizu2022relative}.
\end{proof}

\subsection{The Milnor-Witt Sheaves}\label{app:MWK}
The Milnor-Witt sheaves are examples of strictly $\AA^1$-invariant sheaves. In this section, we will complement the background for the (\cref{sect:Milnor-Witt}) closely following \cite{fasel2020Chow-Witt, deglise2023notes}. We will begin with the background on Grothendieck-Witt ring, Witt ring, and fundamental ideals. We will also then define the Milnor $K$-theory and state Milnor's conjecture. We will end with some additional knowledge in Milnor-Witt $K$-theory and Chow-Witt groups. In general, as far as the exposition here is concerned, the references will be \cite{lam2005quadratic, morel2012A1topology, deglise2023notes, bachmann2025MWmotives,fasel2020Chow-Witt}.

\subsubsection{Grothendieck-Witt Ring}\label{app:Gw-ring}
\begin{defn}
Let $F$ be a field. An \emph{inner product space} $(E,\phi)$ over $F$ is a finite $F$-vector space $E$ equipped with a bilinear form 
        $$\phi: E \otimes_F E \to F$$
which is symmetric and non-degenerate\footnote{this means $E \to E^{\vee},\ x\mapsto \phi(x,-)$ is an isomorphism}. A morphism $f:(E,\phi) \to (F,\psi)$ of inner spaces is a $F$-linear morphism $f:E\to F$ such that $\phi(f(v),f(w) = \psi(v,w)$. We define the \emph{rank} of $(E,\phi)$ to be its dimension of the vector space $E/F$.
\end{defn}
The category of inner spaces admits direct sums $\oplus$ and tensor products $\otimes$ on its elements. Thus, the set of isomorphism classes of inner spaces, denoted by $I_F$, is a commutative monoid for $\oplus$ and a commutative semi-ring $\oplus,\otimes$. The following definition using inner spaces is due to \cite{milnor1973symmetric}.

\begin{defn}
The \emph{Grothendieck-Witt ring} of field $F$ is defined as the group completion of the monoid $(I_F,\oplus)$, with products induced by the tensor product $\otimes$. The rank map $rk$ defined above induces a ring homomorphism $rk: \GW(F)\to \ZZ$.
\end{defn}

Alternatively, one also defines the $\GW(F)$ to be the isomorphism classes of quadratic forms of $F$, exploiting the one-to-one correspondence between the symmetric bilinear forms and the theory of quadratic forms. But this will no longer hold whenever char $F=2$ (cf. \cite[Remark 2.1.3]{deglise2023notes}). Nevertheless, the definitions based on the inner spaces are valid irrespective of the characteristic of $F$, and so we will adopt this approach for all the formulations in the $\AA^1$-homotopy theory in our thesis. The $\GW(F)$ can also be presented as explicit generators and relations (See \cite[Theorem 2.1.11]{deglise2023notes}). The $\GW(F)$ can be presented as ring via an explicit element $\epsilon = - \Lin -1\Rin$ which provides $\GW(F) = \ZZ[\epsilon]/(\epsilon^2-1)$.

\begin{remark}
For an unit $u\in F^{\times}$, the inner space $F\otimes F\to F$ defined by $(x,y)\mapsto u \cdot xy$ is of rank 1. We denote its class in $\GW(F)$ by the symbol $\Lin u \Rin$. This gives a relation $\Lin uv^2 \Rin = \Lin u \Rin$. If we have $n$-units, $u_1,\dots,u_n$, we write $\Lin u_1,\dots,u_n \Rin = \Lin u_1 \Rin +\dots \Lin u_n \Rin$.
\end{remark}

\begin{defn}
A field $F$ is said to be \emph{quadratically closed} if every element of $F$ is a square in $F$. Equivalently, every quadratic polynomial has a root in $F$. We denote the quadratic classes of $f$ by  $Q(F):= F^\times/(F^\times)^2$. For any field $F$, there is a canonical map $\phi: Q(F) \to \GW(F)$ defined by the relation in the above remark.
\end{defn}

\begin{prop}
A field $F$ is quadratically closed if and only if the map $rk: \GW(F)\to \ZZ$ is a ring isomorphism.
\end{prop}
\begin{proof}
This is an easy consequence of the isomorphism of the map $\phi$ above, see \cite[Proposition 3.1]{lam2005quadratic}.
\end{proof}

\begin{example}
The following are some of the first computations of Grothendieck-Witt rings.
\begin{itemize}
\item If $F$ is an algebraically closed field, then the rank map $rk: \GW(k)\to \ZZ$ is an isomorphism. More generally, $rk$ is an isomorphism whenever $(-1)$ is a square in $k$ (since they are quadratically closed).
\item Let $F=\RR$, then $\GW(\RR)\cong \ZZ \oplus \ZZ$. Indeed, we know that the signature determines a quadratic form over a real vector space. Hence, any class $\sigma\in \GW(\RR)$ can be uniquely written as $\sigma = p\cdot \Lin 1\Rin + q\cdot \Lin -1 \Rin$ with $rk(\sigma)= p+q$ and the signature of $\sigma$ is the pair $(p,q)$. Hence, the map $\GW(\RR)\to \ZZ\oplus \ZZ$ defined by $\sigma\mapsto (p,q)$ is an isomorphism.
\item Let $F=\FF_q$ be the finite field with $q=p^n$ and $p$ prime. Then we have that
    \[\GW(\FF_q)\cong 
        \begin{cases*}
            \ZZ              & q \text{even} \\
            \ZZ/2 \oplus \ZZ & q \text{odd}
        \end{cases*}\]
\end{itemize}
\end{example}
The Grothendieck Witt ring is intimately related to the Milnor-Witt ring by the following fact: for $a\in F^\times$, let us denote by $\Lin a \Rin$ the class of symmetric bilinear forms on $F$ defined by $(X, Y) \mapsto a XY$. Then we have a fundamental statement due to \cite[Lemma 3.10]{morel2012A1topology}.

\begin{lemma}\label{lemma:GW--MW}
The map 
\begin{align*}
    \GW(F) \to K^{\MW}_0 (F)\\
    \Lin a \Rin \mapsto 1+\eta [a]
\end{align*}
passes to a well-defined isomorphism.
\end{lemma}

\begin{remark}
Due to this result, we will denote by $\Lin a \Rin$ the symbol $1+\eta[a]$ of degree 0. For example, the relation 2 in \cref{MW:relations} then can written as
        $$[ab] = [a] +  \Lin a \Rin[b] = [b] + \Lin b \Rin [a]  $$
for any $a,b\in F^\times$.  In particular, we have that $ \Lin b^2 \Rin =1 $ yields $[ab^2] = [a]+[b^2]$. Moreover, the fourth relation becomes 
\begin{align}\label{hyberbolicform}
\eta[-1]+2 =  \Lin -1 \Rin + 1 =  \Lin -1,1 \Rin=: h\quad\ \text{implies} && \eta h =0    
\end{align}
Here, $h$ is called the \emph{hyperbolic form} on $F^2$.
\end{remark}

The following element in $\GW(F)$ is of paramount importance to describe the commutativity of $K^{\MW}_*(F)$.
\begin{lemma}\label{MW:epsilon-defn}
Let $\epsilon:= -\Lin -1 \Rin \in \GW(F)$. For any $n\in \ZZ$, let 
\[ 
n_{\epsilon} = 
\begin{cases*}
\sum_{i=1}^{n} \Lin (-1)^{i-1} \Rin  & if $n> 0$.  \\
0 & if $n = 0$.  \\
\epsilon \sum_{i=1}^{-n} \Lin (-1)^{i-1} \Rin  & if $n<0.$
\end{cases*} 
\]
Then the following holds.
\begin{enumerate}
\item For any $a\in F^\times$, we have $[a,a]=[-1,a]=[a,-1]$ and $[-a,a]=0$
\item For any $\alpha\in K^{\MW}_n(F)$ and $\beta\in K^{\MW}_m(F)$, we have $\alpha\beta = \epsilon ^{mn}\beta\alpha$
\item For any $n\in \ZZ$ and $a\in F^\times$, we have $[a^n] = n_{\epsilon}[a]$.
\end{enumerate}
\end{lemma}

\subsubsection{Witt Ring}\label{app:Witt-ring}
The \emph{Witt ring} is defined as the quotient ring 
                    $$\W(F):= \GW(F)/ (h).$$
The hyperbolic form $h$ has a rank 2, the rank map $rk$ induces a morphism of rings $\W(F) \to \ZZ/2$. The negative part of the Milnor-Witt $K$-theory is governed by the Witt ring as described by the following fact (\cite[Lemma 3.10]{morel2012A1topology}):
\begin{lemma}
For all $n\le -1$ and $a\in F^\times$, the morphism 
$$\gamma_n: K^{\MW}_n(F)\to \W(F)\quad \text{defined by}\quad  \Lin a \Rin \eta^{-n}\mapsto \Lin a \Rin$$
is an isomorphism of abelian groups.
\end{lemma}
As a consequence, we have that for all $m,n\le -1$, the following diagram
\[\begin{tikzcd}
	K^{\MW}_m(F)\times K^{\MW}_n(F) & K^{\MW}_{m+n}(F) \\
	\W(F)\times \W(F) & \W(F)
	\arrow["{\times}", from=1-1, to=1-2]
	\arrow["{\gamma_m\times \gamma_n}",swap, from=1-1, to=2-1]
	\arrow["{\gamma_{m+n}}", from=1-2, to=2-2]
	\arrow["{\times}",swap, from=2-1, to=2-2]
\end{tikzcd}\]
is commutative.

\subsubsection{Fundamental ideal}
The kernel of the rank map $\Ker(rk):= \mathcal{I}$ is called the \emph{fundamental ideal}, which is additively generated by the classes of the form $\Lin -1, a \Rin-h$ in $\GW(F)$. In a similar vein, there is another rank map in the Witt theory, $rk: \W(F)\to \ZZ/2$, and we set $Ker(rk):= \mathcal{I}(F)$, called (by abuse of naming) \emph{Fundamental ideal}. The $\mathcal{I}(F)$ is additively generated by the classes of $\Lin -1,a \Rin$ in $\W(F)$. For $n\le0$, one often sets $\mathcal{I}^n:= \W(F)$ and considers the quotient groups 
        $$\overline{\mathcal{I}}^n:= \mathcal{I}^n(F)/\mathcal{I}^{n+1}(F)$$
for all $n\in \ZZ$ (by convention, $\overline{\mathcal{I}}^n=0 $, if $n\le -1$). The fundamental ideals are described conveniently using the Pfister forms. For any $u\in F^\times$, the Pfister form is defined as $\Lin\Lin u\Rin\Rin:= 1- \Lin u\Rin$. 
\begin{prop}
A field $F$ is quadratically closed if and only if the map $rk: \W(F)\to \ZZ/2$ is a ring isomorphism.
\end{prop}

\begin{example}
The following are some simple examples of Witt rings and their fundamental ideals.
\begin{itemize}
\item If $F$ is algebraically closed field, then $\W(F)\cong \ZZ/2$ and so its fundamental ideal $\mcal{I}(F)$ is trivial,
\item If $F=\RR$, then any quadratic form is equivalent to a diagonal form $\Lin a_1,\dots,a_n\Rin =(a_1x_1^2, \dots,a_nx_n^2)$ and hence, we have that the induced rank map $\W(\RR)\to \ZZ$ is an isomorphism. Consequently, the kernel consists of even forms, and so we have $\mcal{I}(\RR)= 2\ZZ$.
\item Let $F=\FF_q$ be the finite field with $q = p^n$, for some prime $p$. Then using the computation of $\GW(\FF_q)$, we have that, 
\[ I(\FF_q)\cong \begin{cases*}
      0      & q \text{even} \\
    \ZZ/2    & q \text{odd}
\end{cases*}\]
Consequently, we have that 
\[ \W(\FF_q)\cong \begin{cases*}
    \ZZ/2               & q \text{even} \\
    \ZZ/2[t]/(t-1)^2    & q=1 \text{mod} 4\\
    \ZZ/4               & q=3 \text{mod} 4
\end{cases*}\]
The higher fundamental ideals all vanish, $I^n(\FF_q)=0$, for $n \ge 2$.
\end{itemize}
\end{example}

\subsubsection{Milnor $K$-theory}\label{App;MilnorK-theory}
To give more context, we will also briefly describe the Milnor $K$-theory (\cite{milnor1970algebraic}). The Milnor $K$-theory $K^{\M}(F)$ of field $F$ is defined as the $\ZZ$-graded algebra generated by the symbols $\{a\}$in degree +1, for $a\in F^\times$ modulo the relations:
\begin{enumerate}
\item[(M1)] $\{a,1-a\}= 0$, for all $a\in F\bs \{0,1\}$ (Steinberg relation)
\item[(M2)] $\{ab\}:= \{a\}\cdot \{b\} = \{a\}+\{b\}$ 
\end{enumerate}
The primordial properties that one has are those that $K^{\M}_0(F) = \ZZ$ and $K_1^{\M}(F)= F^\times$. There is a canonical map $K^{\M}_n(F)\to K_n(F)$ from the Milnor $K$-theory to the algebraic $K$-theory due to Quillen is an isomorphism if $n\le 2$ (for $n=2$, this is due \cite{matsumoto1969sous}). 
The fundamental ideal $\mcal{I}$ defined above forms an important part of the study associated with Milnor's conjecture (\cite[Question 4.3]{milnor1970algebraic}), now a theorem due to Kato \cite{kato1982symmetric} in characteristic 2 and Orlov-Vishik-Voevodsky \cite{orlov2007exact} and \cite{morel2005milnor} for all other cases. We will finish this section by stating the following beautiful result:
\begin{theorem}
Let F be an arbitrary field. Then for any $n\ge 0$, the map $F^\times\to I(F)$, $u\mapsto \Lin\Lin u \Rin\Rin$ induces a ring morphism :
     $$\mu: K^{\M}_*(F)/2K_*^M(F)\to \bigoplus_{n\ge 0} \mcal{I}^n(F)/ \mcal{I}^{n+1}(F)$$
which is an isomorphism.
\end{theorem}

\subsection{Residue morphisms in $K^{\MW}_*$}\label{app:residue-homom}
One of the quintessential tools nested with the Milnor-Witt $K$-theory is the \emph{residue homomorphism,} which allows us to define the differential between the Rost-Schmid complex that will be defined in the sequel, which leads us to the definition of the Chow-Witt groups. Let $F$ be a field and $v:F\to \ZZ\cup\{-\infty\}$ be a discrete valuation with residue field $\kappa(v)$ and valuation ring $\mathcal{O}_v$ and choose a uniformizing parameter $\pi = \pi_v$ of $v$. Then one defines the following group homomorphism, called the \emph{non-canonical residue morphism} due to \cite[Theorem 3.15]{morel2012A1topology}:

\begin{theorem}
There is a unique homomorphism of graded Abelian groups 
        $$\partial^{\pi}_v : K^{\MW}_*(F)\to K^{\MW}_{*-1}(\kappa(v))$$
of degree -1 such that $\partial_v^\pi$ commutes wit the multiplication by $\eta$ and 
\begin{enumerate}
\item $\partial_v^\pi([\pi, u_1,\dots,u_n]) = [\Bar{u_1},\dots,\Bar{u_n}]$
\item $\partial_v^\pi([u_1,\dots,u_n])=0$
\end{enumerate}
We call such a $\partial_v^\pi$ the \emph{residue homomorphism}.
\end{theorem}

It turns out that this homomorphism $\partial_v^{\pi}$ depends on the chosen valuation $\pi$ (hence, the naming) as described below. This is the reason for \cite{morel2012A1topology} to introduce the theory of \emph{twisted Milnor Witt $K$-theory}, which is the most suitable morphism to consider in this setup:
\begin{prop}\label{non-can-residue}
Let $\alpha\in K^{\MW}_*(F)$ and let $u\in \mathcal{O}_v^{\times}$. Then $\partial_v^\pi(\Lin u \Rin \alpha) =  \Lin \Bar{u}\Rin \partial_v^\pi(\alpha)$.
\end{prop}
\begin{proof}
This is \cite[Proposition 3.17]{morel2012A1topology}.
\end{proof}

\begin{remark}(\cite[Remark 1.9]{fasel2020Chow-Witt})\label{res-depends-on-pi}
The homomorphism $\partial_{v}^{\pi}$ depends not only on the valuation $v$, but also on the choice of the uniformizing parameter $\pi$.
\end{remark}
\begin{proof}
Suppose $\pi'=u\pi$, for some $u\in \mcal{O}_v^{\times}$. Then we have that $\partial_{v}^{\pi}([\pi,-1])= [-1]$, essentially by construction. However, we see that 
\begin{equation}
\begin{split}
\partial_{v}^{\pi'} ([\pi,-1])
    & =  \partial_{v}^{\pi'} ([u^{-1}\pi', -1])  \\    
    & =  \partial_{v}^{\pi'} ([u^{-1},-1]+ [\pi',-1]+\eta[u^{-1},\pi',-1] )    \\
    & =  \partial_{v}^{\pi'} ([u^{-1},-1]+ [\pi',-1]+\eta[\pi',u^{-1},-1] )     \\
    & =  [-1]+ \eta[u^{-1},-1]       \\
    & =  \Lin u^{-1} \Rin [-1]. 
\end{split}
\end{equation}
The class $[-1]\neq \Lin u^{-1} \Rin[-1]$ in general. For example, if $\kappa(v)=\RR$, then the homomorphism $K^{\MW}_1(\RR)\to \mcal{I}(\RR)$ shows that $[-1] \neq \Lin -1 \Rin [-1]$.
\end{proof}

\subsubsection{Twists in Milnor-Witt $K$-theory}\label{app:twists}
We will now explain this theory of twists in the Milnor-Witt ring. The twists are generally considered with respect to an invertible sheaf, a.k.a line bundle (in full generality, one uses the graded line bundles: (see \cite[\S 1.3]{fasel2020Chow-Witt}). We will now describe the twisted version of $K^{\MW}_*$ and, as a result, the twisted residue morphism $\partial_v^{\pi}$ and convince ourselves that this modification indeed eliminates the choice of the uniformizer.
\medskip

Let $V$ be an $F$-vector space of rank 1 with $V^\times$ comprising the non-zero elements of $V$. The group $V^\times$ acts transitively on $F^\times$, and this gives the free Abelian group $\ZZ[V^\times]$ a structure of a $\ZZ[F^\times]$-algebra. On the other hand, we know that $K^{\MW}_*(F)$ is a $K^{\MW}_0(F)$-algebra and there is a canonical map 
    $$F^\times \to K^{\MW}_0(F)\quad  \text{defined by}\quad  u \mapsto
    \Lin u \Rin .$$
Thus, all the groups $K^{\MW}_n(F)$ also inherits the structure of $\ZZ[F^\times]$-algebra. Now, we can conveniently set
the \emph{Milnor-Witt $K$-theory twisted by $V$} as        
        $$K^{\MW}_n(F,V) := K^{\MW}_n(F)\otimes_{\ZZ[{F^\times}]} \ZZ[V^\times]$$
With this definition, we can now consider the \emph{twisted residue homomorphism} (also called the \emph{canonical residue map}) with respect to a line bundle $L$. The homomorphism
        $$\partial_v: K^{\MW}_*(F,L)\to K^{\MW}_{*-1}(\kappa(v), (\mathfrak{m}_v/\mathfrak{m}^2_v)^*\otimes L)$$
is defined by 
$$\partial_v(\alpha \otimes l) = \partial_v^\pi(\alpha)\otimes \overline{\pi}^* \otimes l$$ 
for $l\in L$. Here, $(\mathfrak{m}_v/\mathfrak{m}^2_v)^*$ represents the relative tangent space, which, by definition, is the dual of the relative cotangent space $\mathfrak{m}_v/\mathfrak{m}^2_v$ (seen as $\kappa(v)$-vector spaces).

\begin{lemma}\label{app:lemma-residue-free}\cite[Lemma 1.19]{fasel2020Chow-Witt}
The homomorphism $\partial_v$ is well-defined and is independent of the uniformizer $\pi$.
\end{lemma}
\begin{proof}
For simplicity, we set the line bundle $L = \mcal{O}_v$. The morphism $\partial_{v}^{\pi'}$ is well-defined as $\partial_{v}^{\pi}$ is. Let $\pi'$ be another uniformizing parameter. Then there exists $u\in \mcal{O}_v^{\times}$ such that $u\pi = \pi'$. As in \cref{res-depends-on-pi}, we see that $\partial_{v}^{\pi} = \Lin u\Rin \partial_{v}^{\pi'}$. Since then we have that $\overline{\pi}^* = u^{-1}(\overline{\pi}')^*$, we have that 
$$p (\alpha) \otimes \overline{\pi}^* = \Lin u \Rin \partial_{v}^{\pi'} (\alpha)\otimes \overline{\pi}^* = \Lin u \Rin \partial_{v}^{\pi'} \otimes u^{-1} (\overline{\pi'})^*  = \partial_{v}^{\pi'} (\alpha) \otimes (\overline{\pi'})^* $$
\end{proof}

\subsubsection{Geometric and Canonical Transfers}
Similar to many other motivic theories, the Milnor-Witt $K$-theory also has \emph{transfers}, colloquially called the \emph{wrong way maps} in both untwisted (called the "geometric" transfer) and twisted (called the 'canonical' transfer) settings. The fact that we have it in the latter again emphasizes the fact that the transfer in the non-canonical setting again depends on something that we wouldn't want it to (choice of a section!). 
\medskip

In the untwisted setting, we have the following (\cite[Theorem  1.11]{fasel2020Chow-Witt}) split short exact sequence: for any $n\in \ZZ$
$$0\to K^{\MW}_n(F)\to K^{\MW}_n(F(t))\xrightarrow{\sum \partial_p^p} \bigoplus_p K^{\MW}_{n-1}(F(p))\to 0 $$
where $p$ runs over set of monic irreducible polynomials of $F[t]$ and $F(p) = F[t]/(p)$. The left-hand homomorphism is induced by the field extension $F\subset F(t)$ and the right-hand homomorphism is the sum of the residue homomorphisms associated to the $p$-adic valuation and the uniformizing parameter $p$. For any $\alpha \in K^{\MW}_n(F(t))$, one considers $[t]\alpha \in K^{\MW}_{n+1}(F(t))$ and use the residue map $\partial_t^t$. But one can also consider the other important valuation, $p=\infty$, that is, $v_{\infty}(f/g) = deg(g)-deg(f)$ with uniformizing parameter $-1/t$. Indeed, this gives rise to the definition of \emph{geometric transfers}, which are certain transfer maps with respect to valuation $\partial_{\infty}$ in the untwisted Milnor-Witt $K$-theory. For for any $n\in \ZZ$, the homomorphism
    $$\partial_{\infty}^{1/t}: K^{\MW}_n(F(t))\to K^{\MW}_{n-1}(F) $$
helps us defines the following (\cite[\S 4.2]{morel2012A1topology}): Let $p$ be a monic irreducible polynomial in $F[t]$, then the composite
$$\tau_F^{F(p)}: K^{\MW}_{n-1}(F(p))\subset \bigoplus_p K^{\MW}_{n-1}(F(p))\xrightarrow{s}K^{\MW}_n(F(t)) \xrightarrow{-\partial_{\infty}^{-1/t}} K^{\MW}_{n-1}(F) $$
is called the \emph{geometric transfer}, where $s$ is any section of $\sum \partial_p^p$ associated to the extension $F\subset F(p)$. The following lemma summarizes this story.

\begin{lemma}
The geometric transfers $\tau_F^{F(p)}$ are the unique homomorphisms which makes the diagram 
\[\begin{tikzcd}
	0 & {\mathrm{K^{MW}_n}(F)} & {\mathrm{K^{MW}_n}(F(t))} && {\bigoplus_p \mathrm{K^{MW}_{n-1}}(F(p))} & 0 \\
	\\
	&& {\mathrm{K^{MW}_{n-1}}(F)}
	\arrow[from=1-1, to=1-2]
	\arrow[from=1-2, to=1-3]
	\arrow["{\sum \partial_p^p}", from=1-3, to=1-5]
	\arrow["{-\partial_{\infty}^{-1/t}}"', from=1-3, to=3-3]
	\arrow[from=1-5, to=1-6]
	\arrow["{\sum \tau_F^{F(p)}}", from=1-5, to=3-3]
\end{tikzcd}\]
commutative. In other words, there are unique homomorphisms $f_p:K^{\MW}_{n-1}(F(p))\to K^{\MW}_{}n-1 (F)$ with $\sum (f_p\circ \partial_p^p)+ \partial_{\infty}^{-1/t} =0$.
\end{lemma}
But it turns out that these transfers depend on the chosen section $s$. This is the main reason for developing the theory of (canonical) transfers using the twisted residue morphism. We will now briefly touch upon this. The twists are with respect to the canonical sheaf of differential $\omega_{F/k}$ for a finitely generated field extension $F/k$. Note that the $F[t]$-module $\Omega_{F[t]/k}$ is free of rank equal to $\text{tr.deg}\ (F/k)+1$ (if $k$ is perfect). Now, given a monic polynomial $p \in F[t]$, consider the $p$-adic valuation to obtain twisted residue morphism
$$\partial_p: K^{\MW}_*(F(t), \det(\Omega_{F(t)/k})) \to K^{\MW}_{*-1} (F(p),          ((\mathfrak{m}_p)/(\mathfrak{m}_p^2)^*)\otimes_{F[t]} \det(\Omega_{F[t]/k}).$$
Then it follows that we have a split exact sequence as in the untwisted case (\cite[Proposition 1.20]{fasel2008groupeC-W}):
$$0\to K^{\MW}_n(F, \det(\Omega_{F/k}))\to K^{\MW}_n(F(t), \det(\Omega_{F(t)/k}))\xrightarrow{d} \bigoplus_p K^{\MW}_{*-1}(F(p), \det(\Omega_{F(p)/k}))\to 0 $$
where $d$ is the total residue homomorphism taken over all $p$. Now for any monic irreducible polynomial $p$ in $F[t]$, we have a homomorphism of the form
$$\rm{Tr}_F^{F(q)}: K^{\MW}_{n-1}(F(p), \det(\Omega_{F(p)/k}))\to K^{\MW}_{n-1}(F,\det (\Omega_{F/k})) $$
called the \emph{canonical transfer}. For field extensions $k\subset F\subset L$ with $L/F$ finite, this homomorphism is given by 
$$\rm{Tr}_F^L: K^{\MW}_{*}(L, \det (\Omega_{L/k})\to K^{\MW}_{*}(F,\det(\Omega_{F/k}))$$
        
\subsection{The Rost-Schmid Complex and the Chow-Witt theory}
As a result of twisted residue homomorphism and the canonical transfer, one can now define the \emph{Rost-Schmid complexes} and thereby, the \emph{Chow-Witt groups}.

\begin{defn}
Let $X$ be a finite type, separated scheme over any field $k$, and let $L$ be any line bundle. For any $j\in \ZZ$, the Rost-Schmid complex $\Tilde{C}(X,j, L)$ in weight $j$ is the complex (in homological dimension) whose term in degree $i$ is given by
    $$\Tilde{C}_i(X,j,L):= \bigoplus_{x\in X_{(i)}} K^{\MW}_{j+i}(\kappa(x), \omega_{\kappa(y)/k}\otimes L)  $$
One consequently verifies that for $x\in X_{(i)}$ and $y\in X_{(i-1)}$, there exists a well-defined notion of differential 
    $$d^x_y: K^{\MW}_{j+1}(\kappa(x), \omega_{\kappa(x)/k}\otimes L) \to K^{\MW}_{j+i-1}(\kappa(y),\omega_{\kappa(y)/k}\otimes L)   $$
defined using the residue homomorphism and the canonical transfer described above, and that we have 
    $$d_i: \Tilde{C}_i(X,j,L)\to \Tilde{C}_{i-1}(X,j,L) $$
\end{defn}
In all, the complex $(\Tilde{C}(X,j, L)$ is a (graded) Abelian group with a homomorphism $d$ of degree -1 such that $d_{i-1}\cdot d_i=0$. This is called the \emph{homological Rost-Schmid complex twisted by $L$} and we write $H_i(X,j, L)$ for the homology groups of this complex.
\medskip

Now, let us consider the special case when $X$ is smooth. In this case, we obtain the (graded) Abelian group whose component of degree $i$ is of the form 
        $$C(X,j,L)^i:= \bigoplus_{x\in X^{(i)}} K^{\MW}_{j-i}(\kappa(x), (\omega_{\kappa(x)/k})^{-1}\otimes L)$$
and we have $\Tilde{C}_i(X,j,(\omega_{\kappa(x)/k}^{-1})) = C(X,j+d_X)^{d_X-i}$, where $d_X$ is the dimension of $X$. It follows that we obtain a group homomorphism of (cochain) complex $C(X,j)$ with differentials $d$, i.e.. a pair $(C(X,j), d)$ of degree +1. We call this the \emph{cohomological Rost-Schmid complex}, and we write $H^i(X,j,L)$ for its cohomology groups.

\begin{remark}\label{app:Chow-Witt-reduced-coh}
If $Y\subset X$ is a closed subset of $X$, then one also considers subcomplexes supported on $Y$, and we denote them by $\Tilde{C}_Y(X,j,L)$ and $C_Y(X,j,L)$ respectively. Furthermore, if $i:Y\hookrightarrow X$ is a closed subscheme of $X$, then we have the following identification
    \begin{equation}\label{support:RostSchmid}
        \Tilde{C}(Y,j,i^*L) = \Tilde{C}_Y(X,j,L).
    \end{equation}
Furthermore, these groups are also invariant under taking reduction of schemes, that is, $\Tilde{C}(X,j, L) = \Tilde{C}(X_{red},j, L)$. And one can also relax that $X$ is essentially of finite type.
\end{remark}

We are now all equipped to comprehend the definition of the Chow-Witt groups.
\begin{defn}
Let $X$ be a separated scheme that is essentially of finite type over a field $k$. For any $i\in \mathbb{N}$ and any line bundle $L$ over $X$, the group           $$\widetilde{\CH}_i(X,L):= H_i(X,-i, L)$$
is called the \emph{homological Chow-Witt groups of $i$-dimensional cycles twisted by $L$}. Other names include the homological $i$-th Chow-Witt groups, or simply Chow-Witt groups. In a similar vein, one also has 
        $$\widetilde{\CH}^i(X,L) := H^i(X,i,L) $$
called the \emph{cohomological Chow-Witt groups of $i$-dimensional cycles twisted by $L$} or similarly, called as the homological $i$-th Chow-Witt groups, or simply Chow-Witt groups. If $Y\subset X$ is a closed subset, then we also have the supported version 
            $$\widetilde{\CH}^i_Y(X,L):= H^i(C_Y(X,i,L)) $$
called the \emph{cohomological Chow-Witt groups of $i$-dimensional cycles supported on $Y$ and twisted by $L$} or the Chow-Witt groups supported on $Y$.
\end{defn}

\begin{prop}
We will list some of the basic properties in this theory following \cite[\S 2.2, 2.3, 2,4]{fasel2020Chow-Witt}:
\begin{enumerate}
\item If $i>\dim(X)$ or $i<0$, then $\widetilde{\CH}_i(X,L)=0$.
\item If $X$ is smooth of dimension $d_X$, then we have $\widetilde{\CH}_i(X, \det(\Omega_{X/k})^{-1}\otimes L) \simeq \widetilde{\CH}^{d_X-i}(X,L)$.
\item For a proper morphism $f: X\to Y$ of finite type schemes, we have the pushforward (cf. \cite[Theorem 2.9]{fasel2020Chow-Witt})
    $$f_*: \widetilde{C}(X,j,f^*L)\to \widetilde{C}(Y,j,L)$$ 
and if $f$ is finite morphism of smooth $k$-schemes $f:X\to Y$ with $d=\dim(Y)- \dim(X)$, we have          
    $$f_*: \rm{C(X,j, \det(\Omega_{X/k})\otimes f^*L) \to C(Y,j+d, \det(\Omega_{Y/k})\otimes L)}$$
\item If $f:Y\to X$ is a flat morphism of essentially smooth $k$-schemes, then we have pullbacks (cf. \cite[Theorem 2.12]{fasel2020Chow-Witt}): 
    $$f^*: C(X,j,L)\to C(Y,j,f^*L).$$
For smooth morphism $f:Y\to X$ of essentially smooth $k$-schemes of finite type with $d=\dim(Y) - \dim(X)$, we get (\cite[Theorem 2.14]{fasel2020Chow-Witt})
    $$ f^*: \widetilde{C}(C,j,L)\to \widetilde{C}(Y,j-d,    \det(\Omega_{Y/X})^{-1})\otimes f^*L) $$
\item If $\iota: Y\subset X$ is a closed subscheme (with reduced structure on $Y$), then $\widetilde{C}(Y,j,\iota^* L) = \widetilde{C}_Y(X, j, L)$, for all $L$.
\item Letting $u:U\to X$ be the complement open subscheme of $Y$, we have an exact sequence of complexes
$$0 \to \widetilde{C}(Y,j,\iota^*L)\to \widetilde{C}(X,j,L)\to \widetilde{C}(U,j,u^*(L)) \to 0 $$
which induces a long exact sequence of Chow-Witt groups
$$..\cdots\to \widetilde{\CH}_i(Y,\iota^*(L))\to \widetilde{\CH}_i(X,L)\to \widetilde{\CH}_i(U,u^*(L))\to \cdots.. $$
In contrast with the classical Chow groups, the right-hand arrow $\widetilde{\CH}(X, L)\to \widetilde{\CH}(U,u^*(L))$ is not surjective, in general.
\item If $X$ is a smooth scheme, then the pullback map turns the group $\widetilde{\CH}_i(X,j, L)$ into a ring structure with the ring product acquired via the exterior product, which then extends to a ring homomorphism of smooth schemes (\cite[Proposition 3.13.]{fasel2020Chow-Witt}).
\end{enumerate}
\end{prop}

% \newpage

\begin{savequote}
Wealth is what you don’t see.
\qauthor{"The Psychology of Money" by Morgan Housel}
\end{savequote}
\chapter{Miscellaneous}\label{app:Scheme-theory}
\markboth{Appendix}{}

\section{Scheme Theory and Motivic Topology at Infinity}\label{app:schemes-alg-groups}

This section comprises miscellaneous concepts from scheme theory put together. We will present a brief background on the motivic version of the topology at infinity machinery. 

\subsubsection{Weil restriction functor}\label{app:weil-restriction}
Let $h: S' \to S$ be a morphism of schemes. Then, for any $S'$-scheme $X'$, the contravariant functor
\begin{align*}
    \mathfrak{R}_{S/S'}(X'): (\Sch_S)^{op}\to \Set \\
     T \mapsto \Hom_{S'}(T\times_S S', X')
\end{align*}
is defined on the category $\Sch_S$ of S-schemes. If it is representable, the corresponding $S$-scheme, again denoted by $\mathfrak{R}_{S'/S}(X')$ is called the \emph{Weil restriction} of $X'$ with respect to $h$. Thus, the latter is characterized by a functorial isomorphism
$$ \Hom_S(T, \mathfrak{R}_{S'/S}(X')) \xrightarrow{\sim} \Hom_{S'}( T\times_S S',X')$$ of functors in $T$ where $T$ varies over all $S$-schemes.

\subsubsection{Sheaf of relative differentials}\label{app:differentials}
The sheaf of relative differentials of a scheme is the algebraic analogy of K\"ahler differentials for rings and modules. For example, for a smooth variety $X$ over $\CC$,  the sheaf of differentials is essentially the same as the dual of the tangent bundle defined in differential geometry. For the short exposition here, we will refer to \cite[Chapter 2, \S 8]{hartshorne2013algebraic}. Let us recall the classical definition. Let $A$ be a ring (commutative, unital), let $B$ be an $A$-algebra, and let $M$ be a $B$-module. 

\begin{defn}
The \emph{module of relative differentials} of $B$ over $A$ is a $B$-module $\Omega_{B/A}$, together with an $A$-derivation $d:B \to \Omega_{B/A}$, which satisfies the following universal property: for any $B$-module $M$, and for any $A$-derivation $d':B\to M$, there exists a unique $B$-module homomorphism $f:\Omega_{B/A}\to M$ such that $d' = f\circ d$.
\end{defn}

\begin{example}
If $B = A[x_1,\dots, x_n]$ , then $\Omega_{B/A}$ is the free $B$-module of rank $n$ generated by $dx_1,\dots, dx_n$.
\end{example}

Let $f: X \to Y$ be a morphism of schemes. We consider the diagonal morphism  $\Delta: X\to X\times_Y X$ with the closed subscheme $\Delta(X)\subset X\times_Y X$.

\begin{defn}
Let $\mathcal{I}$ be the sheaf of ideals of $\Delta(X)$. Then we define the 
\emph{sheaf of relative differentials} of X over Y to be the sheaf                  $$\Omega_{X/Y}:= \Delta^*(\mathcal{I}/\mathcal{I}^2) $$
in $X$. The sheaf $\Omega_{X/Y}$ has a natural structure of $\mathcal{O}_X$-modules, which is a priori, quasi-coherent. If $Y$ is Noetherian and $f$ is of finite type, then $\Omega_{X/Y}$ is a coherent sheaf on $X$.
\end{defn}

\begin{example}
If $X = \AA^n$, then $\Omega_{X/Y}$ is a free $\mathcal{O}_X$-module of rank $n$, generated by the global sections $\{dx_1,\dots, dx_n\}$, where $x_1,\dots,x_n$ are affine coordinates for $\AA^n$.
\end{example}

\begin{defn}
The \emph{relative canonical sheaf} $\omega_{X/Y}$ is defined as the sheaf associated to $n$th-exterior product $\bigwedge^{n} \Omega_{X/Y}$, where $n =\dim (X)$. The canonical sheaf $\omega_{X/Y}$ is an invertible sheaf on $X$.
\end{defn}
\begin{example}
For a smooth projective variety $X\to \Spec k$, the canonical sheaf produces a birational invariant called its \emph{genus}, $\rho_g(X) :=\dim_k \Gamma(X,\omega_X)$. From the adjunction formula, one finds that $\omega_{\PP^n_k} \cong \mathcal{O}_X(-n-1)$. Since $\mathcal{O}(l)$ has 
no global sections for $l<0$, we find that $\rho_g(\PP^n) =0$ for any $n\ge 1$. Thus, if $Y$ is any rational variety, then $\rho_g(Y) = 0$.
\end{example}

\subsubsection{$\GG_a$-actions and Locally Nilpotent Derivations}\label{G_a-LND}
Let $k$ be a field of characteristic zero. Then for a $k$-algebra $B$, there is a bijection between the set of $\GG_a$-actions on $B$ and the locally nilpotent derivations on $B$ (\cite[\S 1.5]{freudenburg2006algebraic}). Let $X = \Spec B$ be the corresponding affine variety. Given $\partial \in \LND(B)$, then we have a group homomorphism
\begin{align*}
    \eta: (\ker \partial, +)\to \Aut_k(B)\\
        f \mapsto \exp (t \partial)(f)
\end{align*} 
where $\exp$ is the exponential map
    $$exp(t\partial)(f):= 1+ ({\partial f}) t + \frac{(\partial^2 f) t^2}{2!}+ \dots + \frac{(\partial^n f) t^n}{n!} + \dots .$$ 
In addition, if $\partial = 0$, then $\eta$ is injective. Restricting $\eta$ to the subgroup $\GG_a = (k,+)$, we obtain the algebraic representation $\eta : \GG_a \hookrightarrow \Aut_k(B)$.  Geometrically, this means that $\partial$ induces the faithful algebraic $\GG_a$-action $exp(t \partial)$ on $X$ ($t\in k$). Conversely, let $\rho : \GG_a\times X \to X$ be an algebraic $\GG_a$-action over $k$. Then $\rho$ induces a derivation $\rho'(0)$, where differentiation takes places relative to $t\in \GG_a$.
\medskip

The following lemma is useful for testing the smoothness of a scheme via fibers.
\begin{lemma}\label{smooth:over-arbitrary-base}
Let $S$ be any Noetherian scheme of finite type. Then for any affine $S$-scheme $X$, we have that the canonical morphism $f:X \to S$ is an affine morphism which is locally of finite presentation. Moreover, we have that $f$ is smooth if all the fibres of $f: X\to S$ are smooth.
\end{lemma}
\begin{proof}
This follows from \cite[\href{https://stacks.math.columbia.edu/tag/01SG}{Lemma 01SG}]{stacks-project} and \cite[\href{https://stacks.math.columbia.edu/tag/01V8}{Lemma 01V8}]{stacks-project}. 
\end{proof}

%------------------------------------------------------------------------
\section{Motivic Topology at Infinity}\label{app:Mot-top-infty}
The $\AA^1$-contractibility is a delicate notion as compared to its topological counterpart. To emphasize this further, we know in topology it is always true that if a space $X$ is contractible, then every vector bundle over $F\to E\to X$ is isomorphic to a trivial bundle $ E\cong X\times F$, where $F$ is the fiber. But this is in no sense straightforward in algebraic geometry. In particular, there are examples of smooth schemes that carry trivial vector bundles but are not $\AA^1$-contractible. On the flip side, there are examples of $\AA^1$-contractible schemes that carry non-trivial bundles (see the beautiful work by \cite{asok2008vector}); however, all such examples known so far are strictly quasi-affine schemes, and this wilderness can be contained by restricting the question to the affine realm. Altogether, these instances only reiterate the fact that there is so far no universal invariant that can uniquely identify the affine spaces among all the smooth affine $\AA^1$-contractible schemes.
\medskip

In connection with the motivic open Poincaré Conjecture (cf. \cref{intro:A1-cont-survey}), one approach is to develop a \emph{motivic topology at infinity}. Following the footsteps of topology, it seems instructive to first develop a notion of "motivic homology at infinity". This was proposed by \cite{wildeshaus2006boundary} under the name of \emph{boundary motive}, which takes place in Voevodsky's derived category of motives $\DM_{Nis}^{eff, -}(\Spec k,\ZZ)$.

\subsubsection{Motivic homology at infinity}
Let us quickly recall Voevodsky's derived category of motives over $k$ following \cite[Chapter 5, \S 2.1]{voevodsky2000cycles}. Fix a perfect field $k$ and consider the category of correspondences $\SmCor(k)$ whose objects are smooth schemes $X, Y \in \Sm_k$. A morphism $Y\to X$ in $\SmCor(k)$ is a \emph{finite correspondence}: the free Abelian group generated by integral closed subschemes of the following kind:
$$\Cor(Y,X):= \ZZ\{Z\subset Y\times_k X : Z\to Y \ \text{is finite, surjective}\} \subset Y\times X. $$
The composition law in this category is given by the natural pullback of schemes. The category $\Shv_{Nis}(\SmCor(k))$ of \emph{Nisnevich sheaves with transfer} is the category of those contravariant additive functors from $\SmCor(k)$ to Abelian groups, whose restriction to $\Sm_k$ is a sheaf for the Nisnevich topology. $\Shv_{Nis}(\SmCor(k))$ is an Abelian category and so contains the derived category $D^{-}(\Shv_{Nis}(\SmCor(k)))$ of complexes bounded above. This in turn contains the full subcategory of \emph{effective motivic complexes} $\DM_{-}^{eff}(\Shv_{Nis}(\SmCor(k)))$ over $k$. The latter is the one whose cohomology sheaves are homotopy invariant (\cite[Definition 3.1.10]{voevodsky2000cycles}). 
\medskip

There is another triangulated category $\DM_{gm}^{eff}(k)$ called the \emph{effective geometrical motives} (see \cite[Chapter 5, \S 2.1]{voevodsky2000cycles}). It has a canonical full triangulated embedding of $\DM_{gm}^{eff}(k)$ into $\DM_{-}^{eff}(k)$, mapping the geometrical motive of $X\in \Sm_k$ to $M_{gm}(X)$. Via this embedding, one identifies $M_{gm}(X)=: M(X)$ as an object of $\DM_{gm}^{eff}(k)$. The category of \emph{geometric motives} $\DM_{gm}(k)$ is the one obtained by inverting the Tate motive $\ZZ(1)$ in $\DM_{gm}^{eff}(k)$ (\cite[P. 192]{voevodsky2000cycles}). This latter category is also called the \emph{Voevodsky’s derived category of mixed motives}, denoted $\DM(\Spec k, \ZZ)$.
\medskip

The following version of the boundary motive suffices for our context. For a precise definition, one can refer to \cite[Definition 2.1]{wildeshaus2006boundary}.
\begin{defn}
Let $X\in \Sm_k$ be a smooth scheme over a perfect field $k$. Then the \emph{boundary motive} $\partial M_{gm}(X)$ can be defined as the one that fits into the following exact triangle 
$$\partial M_{gm}(X)\to M_{gm}(X)\to M_{gm}^c(X)\to \partial M_{gm}(X)[1] $$
where $M_{gm}(X)$ denotes the motive of $X$ and $M_{gm}^c(X)$ its motive with compact support in $\DM_{-}^{eff}(\Spec k)$.
\end{defn}

In particular, there are chain complexes $C_* \ZZ_{tr}(X)$ and $C_* \ZZ_{tr}^c(X)$ giving rise to the corresponding categories of motive $M(X)$ and compactly supported motive $M^c(X)$ of $X$ as objects in $DM_{Nis}^{eff, -}(X, \ZZ)$. This is obtained in analogy with "singular chain complex at infinity" as in topology (\cref{defn:sing-hom-infty}). Then one has a canonical map $\iota_X: C_* \ZZ_{tr}(X) \to C_* \ZZ_{tr}^c(X) $ induced from the underlying morphism of presheaves with transfers $\ZZ_{tr}(X)\hookrightarrow \ZZ_{tr}^c(X)$. Then one defines the \emph{motive at infinity}, denoted $M^{\infty}(X):= M_{gm}^{\infty}(X)$, as the motive associated with the complex $C_*(\text{coker}(\iota_X)[-1])$. Given this setup, we have the following result that strikes an analogy with that of topology \cref{eg:homol-infty-sphere}:

\begin{lemma}
Suppose $\XX$ is an $m$-dimensional $\AA^1$-contractible finite type smooth scheme. Then we have 
        $$M^{\infty}(\XX) \simeq \ZZ \oplus \ZZ(m)[2m-1]$$
(non-canonically) as objects in $\DM(\Spec k, \ZZ)$.
\end{lemma}
\begin{proof}
See \cite[Lemma 6.5]{asok2007unipotent}.    
\end{proof}

\subsubsection{Motivic Homotopy at Infinity}
Put differently, the motivic homology at infinity of $\XX$ is that of the motivic homology of the sphere of appropriate dimension. But to fully distinguish between $\AA^1$-contractible schemes, it is indispensable to have the more powerful theory of homotopy at infinity. The existence of homotopy purity allows one to do so, at least in the stable $\AA^1$-homotopy category $\SH(k)$. This was first described by \cite{levine2007motivic}. In addition to this, by developing the theory of punctured tubular neighborhoods in stable motivic homotopy theory, the authors in \cite{DDO2022punctured} have recently established a candidate for stable motivic homotopy at infinity that is in alignment with that of topology. 

\begin{defn}([\textit{ibid}, Eq. 1.1.0.a])
Let $S$ be a quasi-compact, quasi-separated base scheme. Let $f: X\to S$ be any separated scheme of finite type. Then the \emph{stable motivic homotopy type at infinity} is defined as the homotopy exact sequence
    $$\Pi_S^{\infty}(X) \to f_{!}f^{!}(\textbf{1}_S) \xrightarrow{\alpha_X} f_*f^{!} (\textbf{1}_S). $$
\end{defn}
Here $\textbf{1}_S$ is the motivic sphere spectrum over $S$, $f_!f^! \textbf{1}_S = \Pi_S(X)$ is the stable homotopy type of $X$ and $f_*f^! (\textbf{1}_S) = \Pi_S^c(X)$ is the properly supported stable homotopy type of $X$. The canonical morphism $\alpha_X$ is obtained from the six-functor formalism for the stable motivic homotopy category $\SH(S)$. The following fundamental properties show that this definition is a canonical choice.
\begin{itemize}
\item If $f$ is smooth, then $f_!f^!(\textbf{1}_S) = \Sigma^{\infty} X_+$ retrieves back the motivic suspension spectrum of $X$,
\item  If $f$ is proper, then $\alpha_X$ is an isomorphism (as $f_* \cong f_!$)
\item The morphism $\alpha_X$ is covariant with respect to proper morphisms and contravariant with respect to \'etale morphisms.
\end{itemize}

For a morphism $f: X\to S$ as above, recall that a \emph{compactification} is given by a system of open immersions $j: X\hookrightarrow \widetilde{X}$ into a proper $S$-scheme $\widetilde{f}:\widetilde{X}\to S$. The closed complement, known as the \emph{boundary}, is given by $\partial X:=  (\widetilde{X}\bs X)_{red}$ of $\widetilde{X}$ and comes equipped with closed immersion $i: \partial X \hookrightarrow X$. Moreover by setting, $\partial f = \widetilde{f} \circ i: \partial X\to S$, we get the following commutative diagram:
\[\begin{tikzcd}
	X && {\widetilde{X}} && {\partial X} \\
	&& S
	\arrow["j", hook, from=1-1, to=1-3]
	\arrow["f"', from=1-1, to=2-3]
	\arrow["{\widetilde{f}}"', from=1-3, to=2-3]
	\arrow["i"', hook', from=1-5, to=1-3]
	\arrow["{\partial f}", from=1-5, to=2-3]
\end{tikzcd}\]
This system gives rise to the following definition:

\begin{defn}([\textit{ibid}, Definition 4.1.1])
Let $(X, Z)$ be a closed $S$-pair and let $v$ be a (virtual) vector bundle on $X$. Then the \emph{punctured tubular neighbourhood} of $Z$ in $X$ relative to $S$ twisted by $v$, $\TN_S^{\times}(X,Z,v)$ is the homotopy fiber in $\SH(k)$ of the composite
    $$\beta_{X,Z}:\Pi_S(Z,i^{-1}v) \to \Pi_S(X,v)\to \Pi_S(X/X-Z,v) $$
\end{defn}

Among several other properties of $\TN_S^{\times}$, they prove that it coincides with that of stable homotopy at infinity (\textit{ibid} [Proposition 4.4.2]):
\begin{prop}
Let $(\widetilde{X},\partial X)$  be the closed $S$-pair associated with a compactification of a separated $S$-scheme of finite type. Then there exists a canonical isomorphism
      $$\Pi_S^{\infty}(X) \simeq \TN_S^{\times}(\widetilde{X},\partial X) $$
\end{prop}
Moreover, this isomorphism is in fact independent of the choice of compactification ([\textit{ibid}, Corollary 4.4.4]):

\begin{corollary}
Let $(X, Z)$ be a closed $S$-pair such that $X/S$ is proper. Then the punctured tubular neighborhood $\TN_S^{\times} (X,Z)$ is isomorphic to $\Pi_S^{\infty}(X\bs Z)$ and hence only depends on the open subscheme $X\bs Z$.
\end{corollary}

The authors have furthermore computed the stable motivic homotopy at infinity of several algebraic varieties, including the Du Val singularities on normal surfaces (\textit{ibid} [\S 5.4, Example 1]) and the Danielewski surfaces (\textit{ibid} [\S 5.4, Example 2]), among others. In principle, $\Pi_S^{\infty}$ can be seen to measure the formal difference between $\Pi_S(X)$ and its compact counterpart $\Pi_S^c(X)$. In unison with the context of our script, let us end with the following example (\textit{ibid} [Example 4.3.7]).

\begin{example}\label{eg:SH-infty-is-A^d}
Let $\XX$ be a stably $\AA^1$-contractible variety of dimension $d$ over a field $k$. Then the stable motivic homotopy at infinity of $\XX$ is seen to be that of the affine $d$-space, i.e.,
    $$\Pi_k^{\infty} (\XX) = \textbf{1}_k \oplus \textbf{1}_k(d)[2d-1] = \Pi_k^{\infty}(\AA^d_k).$$
\end{example}   
Thus, the stable motivic homotopy at infinity cannot distinguish between a stably $\AA^1$-contractible variety and $\AA^d_k$. Moreover, recall from topology that taking the product with Euclidean spaces (or any contractible space) kills off the homotopy at infinity. This is also true for $\Pi_S^{\infty}$ as the following illustrates (cf. \textit{ibid} [Proposition 4.3.9]):

\begin{example}
Let $f: Y\to S$ be a smooth stably $\AA^1$-contractible $S$-scheme of relative dimension $n$ with a trivial relative tangent bundle (e.g., take $Y = \AA^n_k$), then we have the following splitting of extension
        $$\Pi_S^{\infty}(X\times Y) \simeq \Pi_S(X) \oplus \Pi_S^c(X)(n)[2n-1]. $$
\end{example}

The authors in \cite{DDO2022punctured} furthermore believe that an unstable version of motivic homotopy at infinity should resolve this issue at stake. However, there are currently several difficulties in defining a candidate for the \emph{unstable motivic homotopy at infinity}; for instance, a coherent way to choose a base point at infinity or to formulate the "right definition" of motivic homotopy at infinity itself: since unstably, the homotopy fibre and the homotopy cofibre are generally different, it is unclear what should be the right notion to begin with. 
\medskip

\textsc{
{\fontfamily{bch}\selectfont
So far, it is only a dream to have such a distinguished motivic invariant in play; nevertheless, this dream too shall become a reality!
}}

% \afterpage{\blankpage}

\end{appendix}      
\backmatter

\printbibliography[heading=bibintoc, title={Bibliography}]
\let\cleardoublepage\clearpage

\chapter{Cover Page Courtesy}\label{app:Scheme-theory}
\markboth{Appendix}{}

I gratefully acknowledge the following sources for their contribution to the cover page of this thesis.

\begin{itemize}
\item Real Exotic 3-manifold, 
\href{https://www.londontsai.com/drawing/view/582083/1/679362}{Whitehead Manifold}
\item Henry Segerman, Oklahoma State University - \href{http://www.segerman.org/tshirts.html}{T-shirt designs}

\item The work of Maurits Cornelis Escher, \href{https://escherinhetpaleis.nl/en/about-escher/masterpieces/hand-with-reflecting-sphere}{Escher in Het Paleis}
\end{itemize}

\afterpage{\blankpage
\thispagestyle{empty}}

\end{document}